\documentclass[10pt,openany,leqno]{smfart}
\usepackage[english, francais]{babel}
\usepackage{smfthm}
    \usepackage[utf8]{inputenc}
\usepackage{amsmath,amsthm,amsfonts,amssymb,amscd,url, hyperref, adjustbox, color, graphics}
\usepackage{cleveref}
\usepackage{enumerate}
\usepackage[shortlabels]{enumitem}
\setlist{leftmargin=*}
\makeindex

\begin{document}
\begin{otherlanguage}{english}
\def\A{{\mathbb A}}
\def\R{{\mathbb R}}
\def\N{{\mathbb N}}
\def\Z{{\mathbb Z}}
\def\S{{\mathbb S}}
\def\u{ \underline }
\def\sA{{\mathfrak A}}
\def\sB{{\mathfrak B}}
\def\sC{{\mathfrak C}}
\def\sD{{\mathfrak D}}
\def\sE{{\mathfrak E}}
\def\sF{{\mathfrak F}}
\def\sG{{\mathfrak G}}
\def\sI{{\mathfrak I}}
\def\sJ{{\mathfrak J}}
\def\sL{{\mathfrak L}}
\def\sM{{\mathfrak M}}
\def\sN{{\mathfrak N}}
\def\sO{{\mathfrak O}}
\def\sP{{\mathfrak P}}
\def\sR{{\mathfrak R}}
\def\sS{{\mathfrak S}}
\def\sY{{\mathfrak Y}}
\def\sT{{\mathfrak T}}

\def\cV{{\mathcal V}}
\def\cI{{\mathcal I}}
\def\cD{{\mathcal D}}
\def\cH{{\mathcal H}}
\def\cY{{\mathcal Y}}
\def\cX{{\mathcal X}}

\def\sa{{\mathfrak a}}
\def\sb{{\mathfrak b}}
\def\sc{{\mathfrak c}}
\def\sd{{\mathfrak d}}
\def\sd{{\mathfrak d}}
\def\se{{\mathfrak e}}
\def\sg{{\mathfrak g}}
\def\sm{{\mathfrak m}}
\def\sn{{\mathfrak n}}
\def\sp{{\mathfrak p}}
\def\sq{{\mathfrak q}}
\def\sr{{\mathfrak r}}
\def\ss{{\mathfrak s}}
\def\st{{\mathfrak t}}
\def\sw{{\mathfrak w}}

\def\wt{\widetilde}
\def\wh{\widehat}

\def\arr{\overleftarrow}
\def\bm{{\boxdot_-}}
\def\bp{{\boxdot_+}}
\def \qand{\quad \text{and} \quad}

\def\Diff{\text{Diff}}
\def\Max{\text{max}}
\def\Log{\text{log}}
\def\loc{\text{loc}}
\def\inta{\text{int }}
\def\det{\text{det}}
\def\exp{\text{exp}}
\def\Re{\text{Re}}
\def\lip{\text{Lip}}
\def\leb{\text{Leb}}
\def\dom{\text{Dom}}
\def\diam{\text{diam}\:}
\newcommand{\ovfork}{{\overline{\pitchfork}}}
\newcommand{\ovforki}{{\overline{\pitchfork}_{I}}}
\newcommand{\whforki}{{\widehat{\pitchfork}_{I}}}
\theoremstyle{plain}

\newenvironment{sublemm}{\begin{enonce}{Sublemma}}{\end{enonce}}
\newenvironment{prob}{\begin{enonce}{Problem}}{\end{enonce}}
\newenvironment{fact}{\begin{enonce}{Fact}}{\end{enonce}}

\newenvironment{ques}{\begin{enonce}{Question}}{\end{enonce}}

\theoremstyle{remark}
\newenvironment{Claim}{\begin{enonce}{Claim}}{\end{enonce}}
\newenvironment{PartAns}{\begin{enonce}{Partial Answer}}{\end{enonce}}
\newenvironment{remas}{\begin{enonce}{Remarks}}{\end{enonce}}
\newenvironment{Examples}{\begin{enonce}{Remarks}}{\end{enonce}}

\title{Abundance of non-uniformly hyperbolic H\'enon-like endomorphisms}
\author[P. Berger]{Pierre Berger}
\address{CNRS-LAGA, Universit\'e Paris 13.}
\email{berger(at)math.univ-paris13.fr}

\date{\today}

\begin{abstract} 
For every $C^2$-small function $B$, we prove that the map $(x,y)\mapsto (x^2+a,0)+B(x,y,a)$ leaves invariant a physical, SRB probability measure, for a set of parameters $a$ of positive Lebesgue measure. 
When the perturbation $B$ is zero, this is the Jakobson Theorem; when the perturbation is a small constant times $(0,x)$, this is the celebrated Benedicks-Carleson Theorem.

In particular, a new proof of the last theorem is given, based on devellopment of the  combinatorial formalism of  the Yoccoz puzzles. By  adding new geometrical and combinatorial ingredients, and restructuring classic analytical ideas,  we are able to carry out our proof in the $C^2$-topology, even when the underlying dynamics are given by endomorphisms.
\end{abstract}
\begin{altabstract} 
Pour toute petite $C^2$-function $B$, nous prouvons que pour un ensemble de paramètres $a$ de mesure de Lebesgue positive, l'application $(x,y)\mapsto (x^2+a,0)+B(x,y,a)$ pr\'eserve une mesure de probabilité qui est physique et SRB. Quand l'application $B$ est nulle, il s'agit du théorème de Jakobson ; quand la perturbation est égale à une petite constante fois $(0,x)$, on obtient le célèbre théorème de Benedicks-Carleson.

Nous donnons en particulier une nouvelle preuve de ce dernier théorème, basée sur le formalisme combinatoire des pièces de puzzle de Yoccoz. En ajutant de nouveaux ingrédiants géométriques et combinatoires, et en restructurant des idées analytiques classiques, nous arrivons à prouver notre résultat en topologie $C^2$, et cela, même quand la dynamique est un endomorphisme.
\end{altabstract}
\maketitle


Our aim is to prove the existence of a non-uniformly hyperbolic attractor for a large set of parameters $a\in \mathbb R$, for the following family of maps:
\[f_{a\,B}:\; (x,y)\mapsto (x^2+a,0)+B(x,y,a),\]
where $B$ is a fixed $C^2$-map of $\mathbb R^3$ to $\R^2$ close to $0$. We denote by $b$ an upper bound of the uniform $C^2$-norm of $B|[-3,3]^2$ and of the determinant of $Df_{a\,B}$. For $B$ fixed, we prove that for a large set $\Omega_B$ of parameters $a$, the dynamics $f_{a\,B}$ is \emph{strongly regular}.  This has many consequences, among which is the following theorem.
\index{$a$,$b$,$B$}

\begin{theo}[Main]\label{Main}
For any $\eta>0$, there exist constants $a_0>-2$ and $b>0$ such that the following property holds: for any $B$ with $C^2$-norm less than  ${b}$, there is a subset $\Omega_B \subset [-2,a_0]$ of relative measure greater than $1-\eta$, such that for any $a\in \Omega_B$, $f_{a\,B}$ leaves invariant a physical, SRB measure. 
\end{theo}
 This answers a question of  Pesin-Yurchenko for reaction-diffusion PDEs in applied mathematics \cite{Yurch}. This solves also a step of the program of Yoccoz stated at his first lecture at Collège-de-France in 1997 \cite{Yolecture1}. The present manuscript was also, following his own words \cite[-1'37'']{Yo16}, the main source of inspiration of his last lecture at Collège-de-France. Nevertheless the writing of our text has been deeply revised since this time.   
 
  To the author's knowledge,  this is the first result showing the abundance of  non-uniformly hyperbolic, surface, (non-invertible and not expanding) endomorphisms leaving invariant an SRB measure. It seems also to be the first result proving the abundance of  non-uniformly hyperbolic surface maps  (invertible or not) for families in a $C^2$-open set (all previous results need three derivatives even \cite{YW08}).
{\small  \tableofcontents }
\section*{Introduction}
\subsection{History}
 The birth of chaotic dynamical systems goes back to Poincar\'e in his study of the 3-body problem.  At the time, the prevaling belief was that dynamical systems are always deterministic: small perturbations do not change the long term behavior. Let us recall a (simplified version) of his famous counterexample.

The idea is to consider two massive planets of equal mass and in circular orbits around $0$ in the complex plane $\mathbb C$. One views $\mathbb C$ as embedded into $\mathbb C\times \mathbb S^1$ via the inclusion $\mathbb C\approx \mathbb C\times \{1\}\hookrightarrow \mathbb C\times \mathbb S^1$.
Put a planet $P$ in $\{0\}\times \mathbb S$ with a vertical initial speed $v\in \{0\}\times \mathbb R$. On one hand, assuming $P$ has negligible mass, the motions of both massive planets remain circular and included in $\mathbb C\times \{0\}$. On the other, the dynamics of the planet remains in the circle $\{0\}\times \mathbb S$; in this way the dynamics resemble the (time one) pendulum map $f$ of the tangent space of the circle $T\mathbb S=\mathbb S\times\mathbb R$. We identify $T\mathbb S$ with the punctured  plane in the phase diagram drawn at the left of Figure \ref{fig:poincare}.   
It turns out that the point $M=(-1, 0)\in T\mathbb S$ is a hyperbolic fixed point: the differential of $f$ at $M$ has two eigenvalues of modulus different from 1.
Moreover, the stable and unstable manifolds of $M$ are equal to a same curve $W^s(M)=W^u(M)$. This is a \emph{homoclinic tangency}.

Contrary to the case of the pendulum, we can perturb the system in such a way that not only it holds $W^s(M)\not=W^u(M)$ but also the intersection $W^s(M)\cap W^u(M)$ contains a point $N\not= M$ where the intersection is transverse (see the second picture of Figure \ref{fig:poincare}). Indeed, we can  assume that the two massive planets have an elliptic orbit centered at 0 with small eccentricity $e\not=0$. The hyperbolic point $M$ persists, its stable and unstable manifolds $W^s(M)$ and $W^u(M)$ are  no longer equal, the intersection $W^s(M)\cap W^u(M)$ contains all the iterates of $N$. 

The global picture of $W^s(M)$ and $W^u(M)$ is then extremely complex (see the third picture of Figure \ref{fig:poincare});  with it, the field of chaotic dynamical systems is born.

 Even today, we do not know how to describe this picture mathematically. In fact, we do not even know if the closure of $W^s(M)\cap W^u(M)$ can have positive Lebesgue measure. 
\begin{figure}[h]
    \centering
        \includegraphics[width=\textwidth]{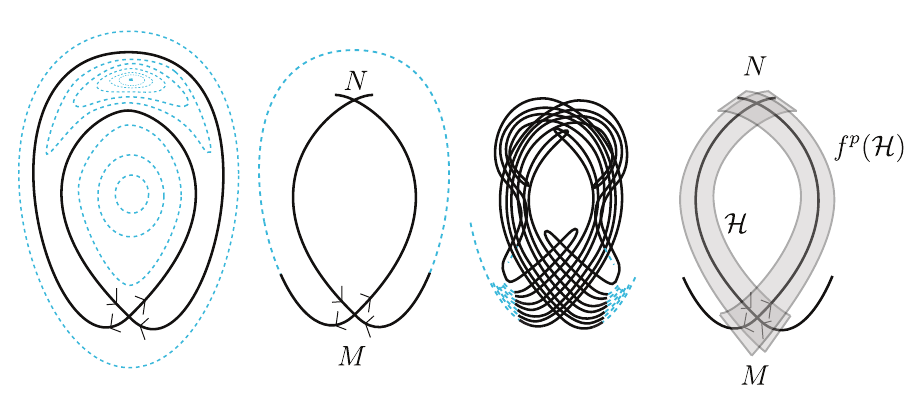}
    \caption{Homoclinic tangle arising from a 3-body problem.}
\label{fig:poincare}
\end{figure}

In the sixties, Smale remarked that a thin neighborhood $\mathcal H$ of the segment of $W^s(M)$ containing $M$ and $N$ is sent by an iterate $f^p$ onto a neighborhood of the segment of $W^u(M)$ containing $M$ and $N$ (see the last picture of  Figure \ref{fig:poincare}). If $\mathcal H$ is sufficiently thin, the maximal invariant $K=\cap_{n\in \mathbb Z} f^{pn}(\mathcal H)$ is a  compact hyperbolic set: the celebrated Smale Horseshoe. \emph{Hyperbolicity} means that the space $T\mathbb R^2|K$ is endowed with two $Df^p$-invariant directions, one expanded, the other contracted by $Df$.

 The theory of hyperbolic dynamical systems was largely developed by the schools of Smale and Sinai. It can be considered as more or less complete \cite{Sm}. A hyperbolic set $K$ is an attractor if it is transitive and { if} $K=\cap_{n\ge 0} f^n(V)$ for some neighborhood $V$ of $K$.  Hyperbolic attractors are well understood through the following properties.
\medskip

\noindent{\bf Persistence.} A uniformly hyperbolic attractor is \emph{persistent} { if every} $C^1$-perturbation $f'$ of $f$ leaves invariant a uniformly hyperbolic attractor $K'$ homeomorphic to $K$, via a homeomorphism which conjugates the dynamics $f|K$ and $f'|K'$.
\medskip

\noindent{\bf Geometry.} Every uniformly hyperbolic attractor supports a lamination { whose leaves are} unstable manifolds. This means that $K$ can be covered by finitely many open sets $(U_i)_i$ whose intersection with $K$ is homeomorphic to the product of $\mathbb R^d$ with a compact set $T$, such that $\mathbb R^d\times \{\st\}$ corresponds to a local unstable manifold, for every $t\in T$.
\medskip

\noindent{\bf SRB, physical measure.} An \emph{SRB measure} of $K$ is an $f$-invariant probability measure $\nu$ supported by $K$, such that 
the conditional measure with respect to every local unstable manifold is absolutely continuous with respect to the Lebesgue measure. Whenever the dynamics is of class $C^{1+\alpha}$, there exists a unique, ergodic,  $SRB$ measure supported by $K$.  By Birkhoff's Theorem, $\nu$-almost every point is $\nu$-\emph{generic} { with orbit intersecting every Borel subset $U$ in mean with proportion $\nu(U)$}. Moreover every SRB measure is \emph{physical}: the set of $\nu$-generic points -- called \emph{the basin of $\nu$} -- is of positive Lebesgue measure on a neighborhood of the attractor.  
\medskip

\noindent{\bf Coding.} Every uniformly hyperbolic attractor has a Markov partition. This provides a semi-conjugacy of the dynamics with a subshift of finite type. The conjugacy is 1-1 on a generic set supporting all the measure with large entropy. This implies the existence and uniqueness of the maximal entropy measure, it is also a key point to construct the \textquotedblleft thermodynamic" formalism.
\medskip

The persistence and the existence of an SRB measure show that deterministic dynamical systems may have (robust) statistical behaviors.

These properties enable a { deep} understanding  of \emph{uniformly hyperbolic dynamical systems}, or more precisely \emph{Axiom A diffeomorphisms}. { By definition, the latter are those dynamical systems} whose  non-wandering set is
locally maximal and hyperbolic.

Many dynamical systems do not satisfy Axiom A.  It is indeed easy to see that an Axiom A,  conservative dynamical system is necessarily Anosov (the whole manifold is hyperbolic).  Newhouse { also found} a (dissipative) $C^2$-surface diffeomorphism, robustly not Axiom A \cite{Newhouse}, by building an example of horseshoe of the plane whose local stable and unstable manifolds are robustly tangent.

In the meantime, H\'enon \cite{Henon} { numerically exhibited} a \textquotedblleft strange attractor"  for the family of maps $(x,y)\mapsto (x^2+a+y,bx)$. { Later,} Benedicks-Carleson \cite{BC2} proved its existence mathematically: for a  set of parameters $(a,b)$ of positive Lebesgue measure (with $b$ very small), they proved that there exists a topological attractor which is not (uniformly) hyperbolic. 
 Viana-Mora  \cite{MV93} showed that this proof can be adapted to an open set of $C^3$-perturbations of the Hénon family, occurring in the study of the unfolding of some homoclinic tangencies. Later Benedicks-Young showed the existence of an SRB measure for these parameters \cite{BY}. These works were generalized by Wang-Young \cite{WY01, YW08};  they showed further properties of the SRB measure. A recent work of Takahashi shows the existence of strange attractors for endomorphisms of the plane \cite{Ta11}, nonetheless he did not show the existence of an SRB measure for these maps.  

The work of \cite{BC2} was one of the greatest achievements in dynamical systems of the last decades, especially for the analysis developed therein.
Unfortunately the maps they deal with are defined by a long induction ($\ge 100$). This makes their ideas difficult to be understood and generalized.  Moreover, to state a property on such dynamics, one has to recall the whole induction process. That is why we propose a new approach to this problem based on a development of the Yoccoz puzzle pieces            .
\subsection{Current developments}
We will work with the following family of maps indexed by a parameter $a\in \R$:
\[f_{a\,B}:\; (x,y)\mapsto (x^2+a,0)+B(x,y,a),\]
where $B$ is a fixed $C^2$-map of $\mathbb R^2$ close to $0$. We denote by $b$ an upper bound of the uniform $C^2$-norm of $B|[-3,3]^3$ and of the determinant of $Df_{a\,B}$. We observe that for $B=0$, the dynamics is the product of $x\mapsto x^2+a$ with $y\mapsto 0$. Note that for $B(x,y,a)=( b y , \pm b x)$, the   dynamics $f_{a\, B}$ is conjugated to the Hénon map $(x,y)\mapsto (x^2+a +y, \pm b^2x)$, for every $(a, b)$. 

For any $B$ such that $b$ is sufficiently small, the main theorem gives the existence of a physical SRB measure for $f_{a\, B}$ for every parameter $a$ in a set of Lebesgue measure positive. Contrarily to all the previous generalizations \cite{MV93, WY01, YW08, Ta11} of \cite{BC2}, our approach is basically different: it is a generalization and development of  Yoccoz' concept of strong regularity. The combinatorial formalism developed in this work is in my opinion its main novelty. We will give a rough idea of it in the sketch of proof at the  next section. For now let us just mention that the definition of strongly regular Hénon-like endomorphisms can be rigorously stated and rather quickly (if one skips the one-dimensional study). This definition is purely combinatorial and topological, and intrinsic (it does not depend on the family of maps). Moreover, the formalism will enable us to define combinatorially and topologically some Pesin manifolds and  Pesin sets.  Furthermore,  we will give some  bounds on the box dimension of the SRB measure,  the expansion and contraction along the Pesin manifolds and the distortion along the unstable manifold, from purely combinatorial and topological hypotheses. 
This will enable us to show that the strongly regular Hénon-like endomorphisms  display analogous properties to those of   uniformly hyperbolic attractors. Let us summary these. \medskip
  
\noindent {\bf Abundance.} We will prove that the strongly regular Hénon-like endomorphisms are \emph{abundant}. This means that for every $B\in C^2(\R^3, \R^2)$ of $C^2$-norm bounded by $b$ small, there is a parameter set $\Omega_B \subset  \mathbb R$ of positive Lebesgue measure, such that for every $a\in \Omega_B$, the map $f_{a\, B}$  is strongly regular.\medskip
\medskip
  
\noindent {\bf Geometry.} When dealing with an endomorphism $f$ of a manifold $M$, we consider its inverse limit space $\arr M_f$:
$$\arr M_f:=\{(z_i)_{i\le 0}\in M^{\Z^-}: z_{i+1}=f(z_i),\; \forall i<0\}\; .$$ 
We endow $M_f$ with the induced product topology.

The non-uniformly hyperbolic counterpart to the uniformly hyperbolic local stable and unstable manifolds are the following:
\begin{defi}[Local Pesin manifolds]\index{Pesin manifolds}  \label{pesin}
An embedded manifold $W\subset M$ is a \emph{Pesin local stable manifold} of $ x \in  M$ if there exists a neighborhood $ U$ of $ x$ such that $ W$ is the connected component of $ x$ in
$ \{ y\in  U:  \limsup_{i\to +\infty} \frac 1i \log  d(f^i(x),f^i(y))<0\}\; .$

An embedded manifold $\arr W\subset \arr M_f$ is a \emph{Pesin local unstable manifold} of $\arr x=(x_i)_{i\le 0} \in \arr M$ if there exists a neighborhood $\arr U$ of $\arr x$ such that $\arr W$ is the connected component of $\arr x$ in
$ \{\arr y=(y_i)_{i\le 0} \in \arr U:  \limsup_{i\to -\infty} \frac 1i \log  d(x_i,y_i)>0\}\; .$ 
\end{defi}
Strongly  regular will be defined by asking tangencies between combinatorially defined Pesin, local stable and unstable manifolds. \medskip

\noindent {\bf  SRB, physical measure.} We will show that every strongly regular Hénon-like endomorphism $f$ leaves invariant an SRB, hyperbolic, physical  measure. To define these concepts in the endomorphism case, we shall consider the dynamics $\arr f$ on the inverse limit $\overleftarrow{M}_f$. Let  $\pi: \arr M_f\to M$ be the zero coordinate projection. We recall that by  \cite{Ro67}, the map $\arr \mu\mapsto \mu:=\pi_* \arr \mu$ is a bijection from the set of $\arr f$-invariant probability measures $\arr \mu$ onto the set of $f$-invariant probability measures $\mu$.

\begin{defi}[Hyperbolic measure]\index{Hyperbolic measure} An $f$-invariant  ergodic probability measure $\mu$ is \emph{hyperbolic} if there exists for $\arr \mu$ a.e. point $\overleftarrow  z = (z_i)_i$, a splitting $E^s_{z_0}\oplus E^u_{\overleftarrow  z}=T_{z_0} M$ which depends measurably on $\overleftarrow  z$ and such that:
\begin{equation}
D_{z_0} f (E^s_{z_0})\subset E_{f(z_0)}^s\qand D_{z_0}f (E^u_{\overleftarrow  z})= E_{\overleftarrow  f(\overleftarrow z)}^u,\end{equation}\begin{equation}
\limsup_{n\to +\infty} \frac1n \log \|D_{z_0}f^n (E_{z_0}^s)\|<0\qand 
\liminf_{n\to +\infty} \frac1n \log \|D_{z_{-n}}f^n (E_{\overleftarrow f^{-n}(\overleftarrow  z)}^u)\|>0\; .
\end{equation}
\end{defi}
By ergodicity, the dimension $d_u$ of $E^u_x$ is $\mu$-a.e. constant. Let us assume $d_u\neq 0$ and let $B_{d_u}$ be the unit ball of $\R^{d_u}$.
\begin{defi}[SRB]\label{defiSRB}\index{SRB measure} The measure $  \mu$ is an \emph{SRB measure} if there exists:
\begin{itemize}
\item  a compact metric space $T$ and a $C^0$-embedding
 $\phi: B_{d_u}\times T\hookrightarrow \arr M$ such that for every $t$, $\phi(B_{d_u}\times \{\st\})$ is a Pesin local unstable manifold. 
\item the pull back $\nu$  of $\arr \mu$ by $\phi$ is not  identically zero.
 \item  the conditional measure of $\nu$ with respect to the fibers of $B_{d_u}\times T \to T$ is absolutely continuous w.r.t. the Lebesgue measure of $B_{d_u}$. \end{itemize}
\end{defi}
Let us recall the definition of the conditional measure in this context. First, the projection $ B_{d_u}\times T \to T$ pushes forward the measure $\nu$ to a measure $\hat \nu$ on $T$. Then by \cite{Ro67}, there exists a measurable family of measures $(\nu_t)_{t\in T}$ -- called \emph{conditional measure of $\mu$ w.r.t. $B\times \{\st\}\approx B_{d_u}$} such that for every Borel set $U$, it holds that $\nu(U)$ is the integration of the measurable function $t\mapsto \nu_t(U\cap B\times \{\st\})$ w.r.t. $\hat \nu$. 

\begin{defi}[Physical measure]\index{Physical measure} An ergodic  measure $\mu$ is \emph{physical} if 
its basin $B_\mu:=\{z\in \R^2: \frac1n \sum_{k=0}^n\delta_{f^n(z)}\rightharpoonup \mu\}$  has positive Lebesgue measure.
\end{defi}

\noindent {\bf Coding.} Among strongly regular Hénon-like \emph{diffeomorphism}, in \cite{Berentropy}, we showed that there is a compact set $\check {\mathcal R}\subset \mathbb R^2$ such that:
\begin{itemize}
\item For every ergodic, invariant probability measure with not too small entropy, the set $\check {\mathcal R}$ has a positive measure. 
\item there exists a countable Markov partition $(\check{\mathcal R}_\sg)_\sg$ of $\check {\mathcal R}$ such that for every $\sg$, the induced time $n_\sg$ of $\mathcal R_\sg$ is the first return time of every point in $\check{\mathcal R}_\sg$ into $\check{\mathcal R}$. 
\item this semi-conjugacy induces a H\"older conjugacy (mod 0) with a countable, strongly positive recurrent Markov shift.
\end{itemize}
We will give the combinatorial definition of this set in this manuscript, but we will not re-prove this result (which is not too hard once the sets are defined). 
 This result implies for every strongly regular Hénon-like diffeomorphism, the existence and the uniqueness of the  maximal entropy measure and its exponential mixing property.
 
 \subsection{Open Questions}
Many geometrical properties on the attractor are proved in \cite{Berentropy}, in particular the ergodic measures with support off $\check {\mathcal R}$ are supported by a unique set of small Hausdorff dimension and the Hausdorff dimension of $\check {\mathcal R}$ is close to $1$. To answer to the following question, it remains to study the support of points which are not generic.
\begin{ques} Does the Hausdorff dimension of the topological attractor of strongly regular Hénon-like diffeomorphism is close to 1 when $b$ is small? Is it smaller than $1+\frac1{|\log b|}$?
\end{ques}
In a recent work Matheus-Palis-Yoccoz \cite{MPY17} proved that  the Hausdorff dimension of the Palis-Yoccoz  non-uniformly hyperbolic horseshoes is what was expected.

In the diffeomorphism case and for Benedicks-Carleson parameters, Benedicks-Viana showed that the basin of the SRB measure has full Lebesgue measure in the neighborhood of the attractor \cite{BV}. There is no similar result for endomorphisms with singularities, neither local diffeomorphisms.
\begin{ques}  Does the basin of the SRB measure is of full Lebesgue measure in the neighborhood of the attractor of Pesin-Yurchenko strongly regular map?
\end{ques}

We will explain in \cref{Dream}  that the techniques and concept of this manuscript are hopefully useful for the following:
\begin{prob}  Find an abundant family of non-uniformly hyperbolic attractors of Hausdorff dimension approaching 3/2.\end{prob}

\subsection{Sketch of proof}

The proof is split into three parts. All the inductions done are independent and rather short (at most one page long). In order to do so a combinatorial formalism is introduced to encode all the operations done on pieces and graph transforms. The first part introduces and states the main concept of this work: the strongly regular Hénon-like endomorphisms. The second part shows the main result of this manuscript (abundance and existence of an SRB measure) by developing the combinatorial formalism. Together with the definition of strongly regular endomorphisms stated in \cref{section SR}, the second part contains certainly the most original ideas of this work. These sections contain some proposition, the proof of which are induction free and postponed to the last part. The third part is  independent to the second part, and does not need to understand the notion (nor the existence!) of strongly regular maps.  The last and third part is more classical. Let us describe in more details each of these parts. 
 \medskip
 
The concept of strong regularity was initiated by Yoccoz in his proof of the Jakobson theorem \cite{Y95}.  This concept was designed to be generalized to understand the abundance of the Hénon maps at the parameters considered by \cite{BC2},  as reported in his first lecture at Collège-de-France \cite{Yolecture1}. The concept of strong regularity is based on his concept of one-dimensional puzzle piece. In the one-dimensional, real case, a piece is a rather simple topological and combinatorial object: the pair of an integer $n$ and an interval  sent diffeomorphically by the $n^{th}$-iterate of the quadratic map onto the interval bounded by a fixed point and its preimage. Such pieces are used to encode the critical orbits. 
Yoccoz' definition of strongly regular quadratic maps is formulated using this coding. It is recalled in \cref{Yoccoz SR}. We give also a variation of this definition in \cref{our SR in dimension 1} that will be generalized in dimension 2 to prove our main result.

The 2-dimensional situation is more complex: first by its topology and also because the $2$-dimensional counterpart of the critical point 0 -- which is a set of homoclinic tangency between Pesin stable and unstable manifolds -- is more tricky to manipulate: it is a Cantor set of positive dimension and each of its points may vanish by small perturbation of the dynamics. In the work \cite{PY01, PY09}, Palis and Yoccoz defined a two-dimensional counterpart of these pieces: the affine-like iterations, and two operations on these puzzle pieces: the $\star$-product and the parabolic products $\boxdot_\pm$.  Their definition of strongly regular non-uniformly hyperbolic horseshoes was based on these operations: roughly speaking, they considered dynamics for which these operations can be iterated forever, and with some room in the parameter space depending on the order of iterations. As a matter of fact, their definition of strongly regular maps depends on the considered family of maps (and does not include the Hénon family).  

In this work we consider a similar generalization of puzzle piece and operations. We define the $2$-dimensional piece as the pairs  $(Y,n)$  of a box $Y$ and an integer $n$ called the order. A box  is a subset  $Y\subset \R^2$ and an integer $n$ such that $Y$ satisfies several geometrical assumptions, as well as the iterations $f^k|Y$ for $k\le n$. In particular, the set $Y$ is diffeomorphic to $[0,1]^2$ and bounded by two (nearly) vertical curves and two segments of two fixed affine, horizontal segments. The dynamical assumption asks for some expansions of the nearly horizontal vectors, in a way which is pointwise similar to \cite{BC2} and its developments \cite{MV93,WY01, Ta11}.   
The boxes are of heights uniformly wide as in \cite{PY09}, but now there are pieces  considered which are not puzzle pieces: in particular the width of $f^n(Y)$ is not uniformly wide among all the pieces $(Y,n)$. Such a generality seems necessary to deal with abundant attractors.  Furthermore, as for one-dimensional pieces,  we ask the vertical boundary of each box to be the union of  two arcs of the stable manifold of a fixed point. This dynamical assumption enables us to define canonically the hyperbolic continuation of such boxes. The definitions of the boxes and pieces are given in \cref{section pieces}. Similarly to \cite{PY09}, we define two operations. The first is the $\star$-product defined in \cref{section star prod}. 
  In \cref{section action of piece on curve}, we show how the $\star$-product enables us define Pesin manifolds. In \cref{section para prod}, we define the second operation: the parabolic product $\boxdot_\pm$. As the images of the boxes are not necessarily wide, the parabolic product is more tricky to define.
  Then the idea is to consider the pieces inherited from the Chebichev map $x\mapsto x^2-2$ and all the pieces constructed using the $\star$ and $\boxdot_\pm$ operations following some combinatorial rules. 

To encode the algebra of these operations and state the combinatorial rules we use an abstract alphabet $\hat \sA$ which does not depend on the dynamics. This alphabet enables us to encode both the puzzle pieces of the quadratic maps and some pieces of the Hénon-like maps, as words in the alphabet $\hat \sA$. In \cref{Algebraic preliminary}, we endow the set $\hat \sA^{(\N)}$ of words in the alphabet $\hat \sA $ with the concatenation rule $\cdot $ to form monoid, which is graded by the order function. Likewise, the set of pieces endowed with the operation $\star$ is a graded, pseudo-monoid. The encoding of these pieces is a homomorphism of graded pseudo-monoid. 
 In \cref{grammar}, we state the combinatorial rules.  The rules are formulated on words in the alphabet $\hat \sA$. The objects in the category of symbols which will be always denoted in $\mathfrak{Gothic}$.
 We illustrate this formalism in \cref{subsection 3.3} with the case of the quadratic maps.  
 
 In \cref{section Strong regularity in dimension 2}, we define the key subsets $\sR_k\subset \sR$  of $\sA^{(\N)}$ which index sets of pieces via an injective homomorphism. We define the sets $\arr \sR\subset \hat \sA^{\Z^-}$ and  $\overrightarrow \sR\subset \hat \sA^{\N}$ of concatenation of infinitely many words in $\sR$. We show that each element of these respective sets defines respectively an unstable or a stable  local,  Pesin  manifold. We define also the subset $\sL\subset \overrightarrow{\sR}$ of sequence which  two additional combinatorial rules: the strong regularity and common conditions.  Then we state the main concept of this work: the strong regularity condition for Hénon-like endomorphisms. This condition is satisfied if each $\st \in \arr \sR$ defines an unstable local  Pesin  manifold  which is sent tangent to a stable local,  Pesin  manifolds defined by a sequence $\sc\in \sL\subset \overrightarrow \sR$.  This definition of strong regularity is basically induction free and depends only on the dynamics (and not on the family in which it belongs). 
  Then we state the two results which imply the main theorem: the strongly regular Hénon-like endomorphisms are abundant and the strongly regular maps leave invariant an SRB, physical, ergodic probability measure.
 
 In \cref{subsection 3.5}, we prove that the definition of $\sR_k$
 is  purely combinatorial and topological. This implies that the sets  $\arr\sR$ and $\overrightarrow \sR$ are so, and thus that the  strong regularity condition on Hénon-like endomorphisms is purely  combinatorial and topological.
 The whole part is consistent with a general spirit of puzzle pieces in complex dynamics: from combinatorial and topological hypotheses, one obtains analytical bounds. A new aspect of the proof is that all the sets $\sR_k$ are well defined for a same open set of Hénon-like endomorphisms (formed by the $0$-strongly regular maps). This enables us to do most of the proofs without other hypotheses on the maps. This is why all the inductions of this work are at most one page long. We hope that this will make the argumentation much easier to verify by the reader. 
 

The second part of this manuscript proves the properties of the strongly regular maps. We begin  in \cref{Ideas of proof of the parameter selection} by giving the idea of the parameter selection (i.e. the proof of the abundance of strongly regular maps).  In dimension one, the idea is to put the critical value in a Pesin set of (positive Lebesgue measure), which ``varies'' slowly with the parameter. In dimension 2, the idea is to include a Cantor set of small box dimension into the stable foliation of a Pesin set, such that the picture ``varies'' slowly with the parameter. We give the abstract, sufficient conditions of \cite{BM13} (which were motivated by an earlier version of this work), such that a compact subset $K_1$ of $\R$ can be included in another compact subset $K_2$ of $\R$, for an abundant set of translation parameters. The conditions involve the box dimension $d$ of $K_1$, the diameters of $K_1$ and $K_2$, the Lebesgue measure of $K_2$ and the $L^{1-d}$ norm of the gaps of $K_2$.  
In \cref{section $k$-Strong regularity}, we truncate the definition of the strongly regular maps to some finite depth $k$ to form the definition of $k$-strongly regular maps. The symbolic formalism enables us to define rigorously and quickly the notion of $k$-combinatorial parameter interval, as those for which all the dynamics define the same set of words $\sR_k$. In \cref{Section: Transversality of the hyperbolic continuations}, we state that the hyperbolic continuation of any strongly regular piece varies with the parameter slower than the fold of any persisting Pesin unstable manifold indexed by an element of $\arr \sR$. 
 In \cref{section Geometry of the stable and unstable transverse spaces}, we define combinatorialy the subset 
of greatly regular words in $\sR$, which is included in the set strongly regular words in $\sR$. We show that these words define a stable foliation of positive Lebesgue measure, surprisingly without using the large deviation theorem nor distortion bounds. Such a proof is inspired from the work of \cite{Ts93} and \cite{Ta11}, although it seems that they used both the large deviation theorem and some distortion bounds. Then we state a combinatorial proposition, which implies that the box dimension of the transversal space to the unstable lamination indexed by $\arr \sR$ is small. This implies that the stable and the unstable laminations define subsets satisfying the assumptions of the aforementioned Theorem of \cite{BM13}.  We use all this material to achieve the proof of the parameter selection in \cref{section final proof abundance SR }. It is a new proof which does not use the large deviation argument (nor distortion bounds at this step). Instead, it counts the number of the combinatorial components of ``greatly regular maps'', the number of excluded parameter intervals and uses sharp bounds on their lengths. 
  In \cref{Structure of the transversal space to the unstable manifolds}, we prove the stated upper bound on the box dimension of the unstable  lamination  indexed by $\arr \sR$. The proof is new and crucial for the parameter selection because it defines combinatorialy some finite $\epsilon$-dense set in $\arr \sR_k'\subset \arr \sR \cap \sR_{k'} ^{\Z^-}$ with $k'$ small compared to $k$. As a matter of fact, this set is constant in any $k'$-combinatorial interval.  In order to introduce the finite $\epsilon$-dense set, we define in \cref{section righ divisibility} an order relation on the words of $\hat \sA^{(\N)}$, called the right divisibility. This defines an arithmetic distance on $\hat \sA^{\Z^-}$. A combinatorial counterpart of the Benediks-Carleson favorable times \cite{BC2} is defined and called the favorable divisors.
This definition is indeed purely combinatorial.  In \cref{section Application of the divisibility to Henon-like endomorphism}, we show that the favorable divisor defines an $\epsilon$-dense subset of $\arr \sR$, which is included in $\arr \sR_k$ and so depends only on the $k$-combinatorial interval for any $\epsilon\ge \eta^k$ with $\eta$ small. Furthermore, we show that the map which associates to an element of $\arr \sR$ its local unstable manifold  is Lipschitz for the $C^1$-topology. 
  In \cref{section Compactification of the transversal space to the long unstable manifolds}, we define the closure $\sT$ of $\arr \sR$ for the combinatorial, arithmetic distance, and associate  a Pesin local unstable manifold to each $\st\in \sT$ . This is useful for the proof of  the existence of the SRB measure done in  \cref{SRB}.
 In \cref{section rigidPPP}, we first define the one-dimensional piece of 2-dimensional maps and prove that under a combinatorial condition, the words in $\tilde \sR$ define pieces whose images have wide width: they are puzzle pieces. In \cref{section pour la distortion}, we give a combinatorial condition implying that the corresponding one-dimensional pieces have bounded distortions, and show that such pieces fill up the local unstable manifolds defined by $\arr \sR$ exponentially fast if the map is strongly regular.  In \cref{Construction SRB}, we use the previous section to deduce the existence of an SRB, physical, ergodic measure via a standard argument written in  \cite{BV}. 
  
The third part proves all the analytical ingredients of the proof. This part does not involve the definition of $(k)$-strongly regular maps.  It is also independent of the second part of this manuscript. 
 In \cref{section Estimate on the expansion of simple and strong regular pieces}, we recall Yoccoz' bounds on the expansion of the simple pieces, and deduce similar bounds for the corresponding two-dimensional pieces, as well as for the strongly regular pieces. In \cref{proof:regularpiece},  we first recall some analytic tools from \cite{BC2} and \cite{WY01}, which implies that the graph transform induced by a piece takes its value in the set of horizontal curves. Then we give some sharp bounds on the distortion of strongly regular  pieces.  We finish this section by recalling and applying some general lemmas of \cite{WY01} which study the most contracted direction in a piece. In \cref{section AF}, we start by presenting the latter bounds using the affine-like representation of these pieces. Such representations were introduced in \cite{PY01, PY09}, but the idea goes back to the generating functions and the Shilnikov cross coordinates \cite{Sh67}. We apply this to deduce that the graph transform induced by the pieces is contracting for the $C^1$-topology. Then we show that the graph transform induced by a parabolic product is also contracting. Furthermore, we prove the well-definedness of the parabolic product can be deduced from purely topological and combinatorial conditions. In \cref{section dependence piece with para}, we first give bounds on the parameter dependence of the pieces. Then we prove the bounds on the parameter dependence of the local unstable manifold indexed by $\arr \sR$. We finishes this manuscript by recalling the proof of a combinatorial lemma of \cite{BM13}. 

\textbf{An index  at the end gives the notations and the definitions involved.}   

{\footnotesize\textsl{Acknowledgments. } 
This work was done at IMS at SUNY Stony Brook NY, at CRM, at MFO, at IHES, at
Institute Mittag-leffler and at University Paris 13. I thanks these institutes for their hospitality. I am heavily indebted to M. Benedicks who presented me during  a long time his work with L. Carleson at the IMS. I am also extremely thankful to J.-C. Yoccoz for many discussions, his interesting advice (those implicitly contained in \cite{Y15} are explicitly mentioned).
 I wish to thank M. Lyubich who offered me all the best conditions to start this project.
I am also very grateful for Y. Ishii, M. Shishikura, H. Takahashi  and M. Tsujii for having carefully listened to this work during a mini-courses  in Tokyo and expressed valuable suggestions. I am finaly thanksful to  S. Crovisier, F. Ledrappier, E. Pujals,  R. Perrez Marco and J. Rivera-Letelier for their encouragements. I am grateful to F. Brumley for some of his English corrections. 
I am very thankful for the referees, their corrections and suggestions.  
}

\part{Strong regularity}
\label{rappeldesdef}

{
\section{Strongly regular quadratic maps}\label{SR dim 1}
For $a$ greater but close to $-2$, the quadratic map $P\colon  x\mapsto x^2+a$ has two fixed points $-1\approx \alpha_0< \beta\approx 2$ which are hyperbolic. The segment $[-\beta, \beta]$ is sent into itself by $P$, and its boundary bounds the basin of infinity. All the points of $(-\beta,\beta)$ are sent by an iterate of $P_a$ into $[\alpha_0,- \alpha_0]$. \index{$\alpha_0$}\index{$\beta$}
Yoccoz' definition of strongly regular maps is based on the position of the critical value $a$ with respect to the preimages of $\alpha_0$. To formalize this, he used his concept of puzzle pieces.

\subsection{Puzzle pieces}\label{piece1D}{
\begin{defi}[Piece and puzzle piece]\index{Piece and puzzle piece in dimension 1}
A \emph{piece} $(I, n)$ of $P$ is the 
data of a segment $I$ of $\R$  with non-empty interior and an integer $n$, such that:
\begin{enumerate}
\item   $P^{n}|I$ is a diffeomorphism onto its image included in $[\alpha_0,- \alpha_0]$.
\item the endpoints of $I$ are preimages of $\alpha_0$ but not of $0$:
$$\partial I \subset \bigcup_{n\ge 0} f^{-n}(\{\alpha_0\})\setminus \bigcup_{n\ge 0} f^{-n}(\{0\})\; .$$
\end{enumerate}
 The integer $n$ is called the \emph{order} of the piece. The piece $(I,n)$ is a \emph{puzzle piece} of $P$  if $P^{n}$ sends $I$ onto $[\alpha_0,- \alpha_0]$. 
\end{defi}
The set of pieces $(I,n)$ will be endowed below with a binary operation $\star$, defined by the composition of the iterations $f^n| I$. 
We will see in \textsection \ref{Algebraic preliminary} that the set of pieces endowed with $\star$ is a graded pseudo-monoid. This structure of pieces will be generalized in dimension 2.  One of the main idea behind the two dimensional argument is to describe the dynamics using the structure of this pseudo-monoid. In dimension one, since we will consider the return map to the interval $[\alpha_0,- \alpha_0]$ , we introduce the identity element of the pseudo-monoid by setting:
\[I_\se:= [\alpha_0,- \alpha_0]\qand n_\se=0\; .\]
\begin{exem} The pair $(I_\se, n_\se)$ is a puzzle piece, called \emph{neutral}.\index{$I_\se$, $n_\se$}
\end{exem}
To define the simple puzzle pieces, let us denote by $M$ the minimal integer such that $P^M(a)$ belongs to $I_\se$. Assume that $P^M(a)$ is not an endpoint of $I_\se$. 

 \begin{rema} We notice that $M$ is large since $a>-2$ is close to $-\beta\approx -2$. In \cref{propdefcM} \cpageref{propdefcM},
we will show that $a+\beta$ is of the order of $4^{-M}$.
\end{rema}}
\label{defalphai}
For $i\ge 0$, let $\alpha_i:= -(P|\R^+)^{-i}(-\alpha_0)$. Note that $(\alpha_i)_{i\ge 0}$ is decreasing and converges to $-\beta$. Also $[\alpha_{i+1}, \alpha_{i}]$ is sent bijectively by $P^{i+1}_a$ onto $I_\se$. The same holds for $[ -\alpha_i, -\alpha_{i+1}]$. \index{b @Segments $I_\se$, $I_\ss$}\index{$\alpha_i$}

\begin{figure}[h]
    \centering
        \includegraphics{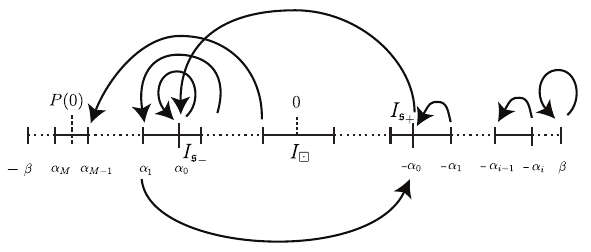}
    \caption{{\color{black} Dynamics of the first preimages of $\alpha_0$.} }
\label{geometricmodele}
\end{figure}

By definition of $M$, the critical value $a$ belongs to  $( \alpha_{M}, \alpha_{M-1})$.
 Hence for $2\le i\le M$, there are segments $I_{\ss^{i}_-}\subset \R^-$ and $I_{\ss^{i}_+}\subset \R^+$ both sent bijectively by $P$ onto $[\alpha_{i-1},\alpha_{i-2}]$.

\begin{defi}[Simple puzzle piece]
The pairs of the form $(I_{\ss^{i}_\pm}, i)$ for $2\le i\le M-1$ are puzzle pieces called \emph{simple}. We put $n_{\ss^{i}_\pm}=i$. \index{Simple piece}
 There are $2(M-2)$ such pairs. The set of symols of  simple puzzle pieces is denoted by $\sY_0=\{\ss^i_\pm ; 2\le i\le M-1, \pm\in\{-,+\}\}$. We put 
 $\sA_0:=   \sY_0\cup \{\ss^-_M, \ss^+_M \}$
with $(I_{ \ss^-_M},M)$ and  $(I_{ \ss^+_M},M)$ the pieces associated to $\ss_\pm ^M$. 
 \label{Psimple}\index{Symbolic set $\sY_0$} \index{Symbolic set $\sA_0$} 
\end{defi}
{\color{black}  In corollary \ref{coroYoc} 
p. \pageref{coroYoc}, we will show that:
 \begin{prop}\label{simple are hyp} for every $\ss\in \sY_0$, it holds $|D_xP^{n_\ss}|\ge 2^{(n_\ss -k)/3} |D_xP^k|$ for every $k\le n_\ss$ and $x\in I_\ss$. 
 \end{prop}

Puzzle pieces enjoy two fundamental properties:
\begin{enumerate}
\item Two puzzle pieces $(I,n)$ and $(I', n')$ are nested or with disjoint interior: 
\[I\subset I'\text{ or } I'\subset I \text{ or }  int\; I\cap int\; I'=\varnothing\;.\]
\item For every piece $(I,n)$, for every perturbation of the dynamics, the hyperbolic continuations of the corresponding preimages of the fixed point $\alpha_0$ define a puzzle piece for the perturbation.
\end{enumerate}

A natural operation on the pieces is the so-called \emph{simple product} $\star$:
\begin{defi}[$\star$-product]\index{Operation $\star$ in dim 1}
Let $(I,n)$ and $(I', n')$ be pieces. Also we put:
 $$(I,n)\star (I', n')=(I'', n'')\text{, with }I'':= I \cap f^{-n}(I')\text{ and }n''= n +n'.$$ 
The pair of pieces  $((I,n), (I', n'))$ is said \emph{admissible for $\star$} if $P^n(int \, I)$ intersects $I'$. 
\end{defi}
%

\begin{rema} \label{rema puzzle dim 1}We notice that if $(I,n)$ and $(I', n')$ are puzzle pieces with $I'\subset I_\se$, then 
they are admissible for $\star$ and 
their $\star$-product $(I, n)\star (I', n')$ is a puzzle piece: the map $P^{n''}$ sends bijectively $I''$ onto $I_\se$.\end{rema}

The following is immediate:
\begin{prop}\label{1-piece are pseudo}
The operation $\star$ is associative. For any pieces $(I,n)$, $(I',n')$, $(I'',n'')$ it holds:
\[((I,n)\star (I',n'))\star (I'',n'') = (I,n)\star ( (I',n')\star (I'',n''))=: (I,n)\star (I',n')\star (I'',n'')
\; . \]
The piece $(I_\se, n_\se)$ is the neutral element of the product $\star$:
\[(I,n)\star (I_\se, n_\se)= (I,n)\quad \text{ for every piece }(I,n).\]
\end{prop}
\begin{defi}[Prime piece]\index{Prime puzzle piece}
A puzzle piece $(I,n)$ such that $I\subset I_\se$ is \emph{prime} if it is not $(I_\se, 0)$ and if there is no puzzle piece $(I',n')$ such that $I\subsetneq I'\subsetneq I_\se$. 
\end{defi}
The following sheds light on the concept of prime piece:
\begin{prop}\label{primeproperty}
\begin{enumerate}[$(i)$]\item A puzzle piece is prime if and only if it is not the $\star$-product of two pieces both different to $(I_\se, 0)$. 
\item  If $(I,n)$ is a prime puzzle piece and $(I', n')$ is a  puzzle piece different to  $(I_\se, 0)$, then the interiors of $I$ and $I'$ are disjoint or $I'$ is contained in $I$. 
\end{enumerate}
\end{prop}
\begin{proof}
To show $(i)$ we notice that if a puzzle piece is a simple product  $(I,n)= (I',n')\star (I'',n'')$ of pieces different to $(I_\se,0)$ , then $I\subsetneq I'$, thus $(I,n)$ is not prime. Conversely, if two puzzle pieces $(I,n)$ and $(I',n')$ are strictly nested: say $I\subsetneq I'$, then $f^{n'}( I)\subsetneq I_\se$, and so $n>n'$. Furthermore, we notice that $(I'',n'') :=(f^{n'}(I), n-n')$ is a puzzle piece $\neq (I_\se,0)$ which satisfies $(I,n)= (I', n')\star (I'', n'')$. 
  
Statement $(ii)$ is an immediate consequence of the definition of prime and the fact that puzzle pieces are nested or with disjoint interiors. 
\end{proof}
A consequence of the proof is the well-definedness of:
\begin{defi}[Prime decomposition and depth]
Every  puzzle piece $(I,n)$ is equal to a unique product $(I_1, n_1)\star \cdots \star (I_k, n_k)$ of prime puzzle pieces. The chain $(I_i,n_i)_{1\le i\le k}$ is called the \emph{prime decomposition} of $(I,n)$. The integer $k$ is called the \emph{depth} of $(I,n)$. 
\end{defi}
\subsection{Strongly regular quadratic map after Yoccoz}\label{Yoccoz SR}
\begin{defi}[Strongly regular product of pieces]
\label{Stronglyregular sequence}\index{Strongly regular product of pieces}
A $\star$-product $(I_1, n_1)\star \cdots \star (I_m, n_m)$ of puzzle pieces $(I_j, n_j)$ is \emph{strongly regular} if each interval $I_j$ is included in $I_\se$ and it holds:
\begin{equation*}
\sum_{j\le i: \; (I_j, n_j)\text{ is not simple}} n_j \le 2^{-\sqrt M} \sum_{j\le i} n_j,\quad \forall m\ge  i\ge 1.
 \end{equation*}
Roughly speaking, the latter  inequality means that the time proportion of simple pieces in the product $(I_1, n_1)\star \cdots \star (I_j, n_j)$ is very close to $1$ for every $j\le m$. 
\end{defi} 
%
Let us recall:

\begin{defi}[ Yoccoz' definition of strong regularity]
A quadratic map $P$ is  \emph{Yoccoz'  strongly regular} if there exists a sequence  $(I_{j}, n_{j})_{ j\ge 1}$ of puzzle pieces with $I_{j}\subsetneq I_\se$ such that with $(I^{(k)},n^{(k)}):=
(I_{1}, n_{1})\star \cdots \star (I_{k}, n_{k})$ it holds:

\begin{enumerate}[($SR_1$)]
\item   the first return $P^{M+1}(0)\in I_\se$ belongs $I^{(k)}$ for every $k$: $ P^{M+1}(0)\in \bigcap_{k\ge 1} I^{(k)}$.
\item The product $(I_{1}, n_{1})\star \cdots \star (I_{k}, n_{k})$ is strongly regular for every $k\ge 1$.
%
\end{enumerate}
\begin{enumerate}[($SR_3^*$)]
\item For every $j$, the segment  $I_{j}$ has a neighborhood $\hat I_{j}$ which is sent bijectively by $P^{n_{j}}$ onto the neighborhood $[\alpha_1,-\alpha_1]$ of $I_\se$.
\end{enumerate}
\end{defi}

We are going to change condition $(SR_3^*)$ by a condition slightly more technical but which will simplify greatly the combinatorial study of those mappings.  The interest of $(SR^*_3)$ is the uniform  distortion bound it implies (by using the Schwarzian derivative of $P$, see \cref{distbound1} p. \pageref{distbound1}). In particular the contraction of $P^{n_{j}}|I_{j}$ is bounded from below by a constant independent of $M$. On the other hand, by \cref{simple are hyp},  for every simple piece $(I,n)$, the iteration $P^{n}|I$ is expanding by a uniform constant $>1$ at the power $n$. As most of the puzzle pieces  $(I_{j}, n_{j})_{ j\ge 1}$ are simple by strong regularity of the products (see definition \ref{Stronglyregular sequence}), it comes:
 \begin{prop}[Prop. 3.10 \cite{Y95}]
A  Yoccoz' strongly regular unimodal map satisfies the Collet-Eckmann condition, when $M$ is large enough:
\begin{equation}\tag{$\mathcal {CE}$}
\liminf_{n\to \infty} \frac1n\log \|D P^{n}(a)\| >0\; .
\end{equation}
\end{prop}}
%
%
\subsection{Strongly regular quadratic maps}\label{our SR in dimension 1}{\color{black}
Our definition needs one new combinatorial definition.  To make the notation less cluttered, we denote the symbols $\ss^2_-$ and $\ss^2_+$ by respectively $\ss_-$ and $\ss_+$. We recall that 
$I_{\ss_-}$ and $I_{\ss_+}$ are neighborhoods of respectively $\alpha_0$ and $-\alpha_0$ in $I_\se$.
The $\star $-product of $k$-times $(I_{\ss_-},2)$ is denoted by:
\[(I_{\ss_-^{\odot k}}, n_{\ss_-^{\odot k}}):= (I_{\ss_-},2)\star\cdots \star  (I_{\ss_-},2)\; .\]
We put also:
\[(I_{\ss_+\cdot \ss_-^{\odot k}}, n_{\ss_+\cdot \ss_-^{\odot k}}):= (I_{\ss_+},2)\star(I_{\ss_-^{\odot k}}, n_{\ss_-^{\odot k}})\; .\]
We notice  that the segments $I_{\ss_-^{\odot k}}$ and 
$I_{\ss_+\cdot \ss_-^{\odot k}}$ are neighborhoods of respectively $\alpha_0$ and $-\alpha_0$ in $I_\se$.\label{aleph}\index{$s_-^{\star \aleph(i)}$}

\begin{defi}[Common piece]
A puzzle piece $(I,n )$ is \emph{common} if its prime decomposition $(I_j, n_j)_{1\le j\le m}$ satisfies for every $j\ge 0$: 
 \begin{multline*}
(I_{j+1}, n_{j+1})\star \cdots \star (I_{j+m'}, n_{j+m'})\neq
(I_{  \ss_-^{\odot m'}}, n_{  \ss_-^{\odot m'}})
\text{ if }m\ge j+m'
\qand \\
(I_{j+1}, n_{j+1})\star \cdots \star (I_{j+m'+1}, n_{j+m'+1})\neq
(I_{  \ss_+\cdot \ss_-^{\odot m'}}, n_{  \ss_+\cdot \ss_-^{\odot m'}})
\text{ if }m\ge j+m'+1
,\end{multline*}
with $m':=\aleph (n^{(j)})$,  $n^{(j)}:= n_1+\cdots +n_j$ and  $\aleph(k):=\left \lfloor\frac k{12}+\frac M{24}\right\rfloor$ for every $k\ge 0$.\index{$\aleph$}
\end{defi}

The following definition is similar but not equivalent to Yoccoz' strong regularity definition. It differs by condition $(SR_3^*)$ which is replaced by $(SR_3)$.  
\begin{defi}[Strong regular quadratic map]
\label{defSRquadra}\index{Strong regular quadratic map}
A quadratic map $P$ is strongly regular if there exists a sequence  $(I_{j}, n_{j})_{ j\ge 1}$ of puzzle pieces with $I_{j}\subsetneq I_\se$ such that, with $(I^{(k)},n^{(k)}):=
(I_{1}, n_{1})\star \cdots \star (I_{k}, n_{k})$, it holds:
\begin{enumerate}[($SR_1$)]
\item   the first return $P^{M+1}(0)\in I_\se$ belongs $I^{(k)}$ for every $k$: $P^{M+1}(0)\in \bigcap_{k\ge 1} I^{(k)}.$
\item The product $(I_{1}, n_{1})\star \cdots \star (I_{k}, n_{k})$ is strongly regular for every $k\ge 1$.
\item The puzzle piece $(I^{(k)}, n^{(k)})$ is common and with depth $k$ for every $k\ge 1$.
\end{enumerate}
\end{defi}
\begin{rema} Asking that $(I^{(k)}, n^{(k)})$ is common for every $k$ implies that $0$ is not one of the preimages of $\alpha_0$. In particular, 
$P^{M+1}(0)$ does not belong to the boundary $\partial I$ of any puzzle piece $(I,n)$. 
\end{rema}
\begin{rema} Asking that each piece $(I^{(k)}, n^{(k)})$ is of depth $k$ in $(SR_3)$ is equivalent to ask that each piece $(I_j, n_j)$ is prime. This is not an extra condition on the map, because if a piece is not prime we can replace it by its prime decomposition. Indeed $(SR_1)$, $(SR_2)$ and the common condition are then still satisfied. However, this enables us to fix the structure: the point $P^{M+1}(0)$ is included in the (interior) of a unique puzzle piece of  depth $k$ which is $(I^{(k)}, n^{(k)})$. 
\end{rema}
To analyse the pieces of strongly regular maps, let us introduce the parabolic products. In order to do so, 
let us define:
 \[I_\boxdot:= cl(I_\se\setminus \bigcup_{\sa\in \sA_0} I_\sa)=P^{-1}_a([ a, -\alpha_{M-1}])\qand n_\boxdot=M+1\; .\]\index{b @Segment $I_\boxdot$}\index{b @Integer $n_\boxdot$}
As $a+\beta$ is of the order of $4^{-M}$ by \cref{propestimates1} \cpageref{propestimates1}, it holds:
\[|I_\boxdot|= O(2^{-M})\quad \text{when }a\to -2 \quad (\text{or equivalently } M\to \infty)\; .\]
The first return time of $I_\boxdot$ into $I_\se$ is $n_\boxdot$. The pair $(I_\boxdot,n_\boxdot)$ is not a puzzle piece because $P^{n_\boxdot}|I_\boxdot$ is not bijective. However, the \emph{parabolic products} $\boxdot_-$ and $\boxdot_+$ will enable us to define (combinatorially) pieces with subsegments of $I_\boxdot$.
First let us notice that $P$ has two inverse branches: 
$g_+: [a, \infty)\to \R^+$ and $g_-:[a, \infty) \to \R^-$. 
\begin{defi}[Parabolic products $\boxdot_-$ and $\boxdot_+$]\index{Parabolic products $\boxdot_\pm$ in dim 1}
Let $(I, n)$ and $(I', n')$ be two puzzle pieces such that $P^{M+1}(0)\in I'\subset I\subset I_\se$.  Assume that the left endpoints of $I$ and $I'$ are different (in particular $I'\subsetneq I$). Then the parabolic product $\boxdot_\pm (I-I')$ are called admissible and we define the segments:
\begin{multline*}
I_{\boxdot_+(I-I')}:=g_+(cl(I\setminus  I'))\quad \text{and}\quad 
 I_{\boxdot_-(I-I')}:=g_-(cl(I\setminus  I'))\; ,
 \end{multline*}
\index{b @Segment $I_{\boxdot_\pm(I-I')}$}
and for $\sp\in \{\boxdot_+(I-I'), \boxdot_-(I-I')\}$ , we put $n_\sp:=n_\boxdot +n$. 
The pair $(I_\sp, n_\sp)$ is a \emph{parabolic pieces}.
\end{defi}

\begin{fact}\label{primeend}For every
 $\sp\in \{\boxdot_+(I-I'), \boxdot_-(I-I')\}$, the segment $I_\sp$  is sent by $P^{n_\sp}$ onto a component of $I_\se\setminus I''$, with  $I''=P^n(I')$. The pair $(I_\sp, n_\sp)$ is not a puzzle piece and $(I'', n'-n)$ is a puzzle piece.
\end{fact}
\section{Pieces and operations for Hénon like-endomorphisms}
\subsection{Initial settings}\label{setting}
For every $n,m\ge 1$, every $r\ge 0$, every compact set $C\subset \R^n$, the space $C^r(C,\R^m )$ is formed by the maps $\phi\in C^0(C,\R^m)$ which can be extended to a $C^r$-map $\tilde \phi : \R^n\to \R^m $. Then we put:
\[\|\phi\|_{C^r}= \inf_{\tilde \phi \text{ extension of }\phi}  \sum_{j=0}^r \max_{z\in C} \|D^j_z \tilde \phi\|\; .\]
 The space $C^r(C,\R^m )$ endowed with this norm is a Banach space. 
 
In dimension $2$, we will deal with the following maps:
\begin{defi}[$0$-strongly regular map]
A $C^2$-map $f$ of $\R^2$ is $0$-\emph{strongly regular} if it is of the form:
\[f(x,y)= (P_a(x),0)+B(x,y)\quad \text{with } P_a(x)=x^2+a\]
and with $a$ and $B$ which satisfy that:
\begin{list}{$\bullet$}{}
\item The parameter $a>-2$ is close to $-2$ or equivalently the first return time $M$ of $a$ by $P_a$ in $I_\se=[\alpha_0,-\alpha_0]$ is large. Furthermore $P_a^M(a)$ does not belong $I_{ \ss_-^{\odot \lfloor M/24\rfloor}}$ (defined \cpageref{aleph}).
\item The $C^2$-norm of $B|[-3,3]^2$ and the determinant of $Df$ are bounded by  $b$.\index{1@Constant $b$ is the $C^2$-norm of $B$}
\item The integer $M$ is large, and $b$ is small depending on $M$.
\end{list}
\end{defi}
To fix the idea, we will assume:
 $$M\ge 10^6\qand  \theta:= 1/|\log b|\le 2^{-2^M}\; .\index{1@Constant $\theta= -1/\log_2\,  b$}$$ 

\begin{fact}\label{trivial estimate}For every $0$-strongly regular map $f$, it holds:
\begin{list}{$\bullet$}{}
\item $\|D^2f|[-3,3]^2\|\le 2+b$,
\item $\|Df| [-\beta, \beta]\times [-3, 3]\|< 4$.
\end{list}
\end{fact}

 \begin{enonce*}{Assumption on the families of maps}
Some  $0$-strongly regular maps will be selected among families $(f_a)_{a\in [-3,3]}$ of the form:
\[f_a(x,y)= (P_a(x),0)+B_a(x,y)\]
such that $(x,y,a)\in [-3,3]^3 \mapsto B_a(x,y)$ is of class $C^2$, with $C^2$-norm smaller than $b$.
\end{enonce*}
By definition, a $0$-strongly regular map $f_a$ is $b$-close to $\hat P_a:=(x,y)\mapsto (x^2+a,0)$ which preserves the line $\R\times \{0\}$ and whose restriction therein is equal to the quadratic map $P_a$.  Hence,  the fixed point $(\alpha_0,0)$ of $\hat P_a$ persists as a fixed point $A$ of $f_a$. 

\index{$A$}

Let $W^s_{loc} (\alpha_0 ; \hat P_a)= \{\alpha_0\}\times [- \theta, \theta]$ and $W^s_{loc} (-\alpha_0 ; \hat P_a)= \{-\alpha_0\}\times [- \theta, \theta]$. These are two segments of the stable manifold of $A$, whose union is sent into $W^s_{loc} (\alpha_0 ; \hat P)$ by $\hat P_a$. As $ f_a$ is $C^2$-close to $\hat P_a$, there are two  curves 
$W^s_{loc} (\alpha_0; f_a)$ and $W^s_{loc} (-\alpha_0 ; f_a)$ which are $C^2$-close to  respectively  $W^s_{loc} (\alpha_0 ; \hat P_a)$ and  $W^s_{loc} (-\alpha_0 ; \hat P_a)$, which are both sent into 
$W^s_{loc} (\alpha_0; f_a)$ by $f_a$ and with endpoints in $\R\times (\{-\theta\} \cup\{\theta\})$. 

Also a halve  local unstable manifold $W^u_{1/2} (A; f_a) $ of the fixed point $A$ of $f_a$ is  $C^2$-close to  $[\alpha_0, -\alpha_0]\times \{0\}=I_\se \times \{0\}$ and  with endpoints in $W^s_{loc} (\alpha_0; f)\cup W^s_{loc} (-\alpha_0; f)$. 

\begin{fact}\label{defi WuA} 
For every $f$ $0$-strongly regular we can assume moreover that 
$$W^u_{1/2} (A; f) =I_\se\times \{0\}\qand  
W^s_{loc} (\pm \alpha_0, f)=  \{\pm \alpha_0\}\times [-\theta, \theta]. $$ 
This can be also assumed at every $0$-strongly regular parameter of any $(f_a)_a$ satisfying the assumption on families.
\end{fact}
\begin{proof}
We will be short because this fact will be used only to simplify the notations of the manuscript. Let $J$ be the set of parameters for which $f_a$ is $0$-strongly regular.   
To prove this lemma, it suffices to show that:
\begin{itemize}
\item  $\bigcup_{a\in J} \{a\}\times W^s_{loc} (\pm \alpha_0, f_a)$ is $O(b)$-$C^2$-close to $\bigcup_{a\in J} \{a\}\times \{\pm \alpha_0(a)\}\times [-\theta, \theta]$. 
\item $\bigcup_{a\in J} \{a\}\times W^u_{loc} (A, f_a)$ is $O(b)$-$C^2$-close to $\bigcup_{a\in J} \{a\}\times I_\se(P_a)\times \{0\}$. 
\end{itemize}
Then we can conjugate $(f_a)_a$ via a family $O(b)$-$C^2$-close to the identity to a family of dynamics with the sough property . 
 
By hyperbolic continuation, the above sets are indeed $O(b)-C^0$-close surfaces. The $C^1$-bounds are shown by using the $Df$ forward and backward invariance of cones centered to respectively the direction  $\R^2\times \{0\}$ and $\R\times \{0\}\times \R$ with angle $O(b)$. The $C^2$-bounds are done using the invariance of $C^2$-jet cones \cite{berlam}[\textsection  6.6 and \textsection  7.6]. 
\end{proof}

%
%
We put:
$$Y_\se:= I_\se\times [- \theta, \theta],\quad  \partial^sY_\se:=\{\alpha_0, -\alpha_0\}\times [- \theta, \theta]\qand \partial^sY_\se:=I_\se\times \{- \theta, \theta\}\; .\index{$Y_\se$, $\partial^uY_\se$, $\partial^u Y_\se$}$$

$$W^u_{1/2} (A; f) = I_\se \times \{0\}\qand W^s_{loc}(A):= W^s(\alpha_0;f)\index{$W^s(A):= W^s(\alpha_0;f)\; .$}\; .$$

%
\subsection{Pieces}\label{section pieces} The definition of strong regularity is both combintatorial and topological. The formulation of the topological conditions will be done by using boxes and curves. Both need cones to be defined.

\begin{defi}[Horizontal and vertical cones]\index{Vertical and horizontal cones $\chi_h$ and $\chi_v$}\label{defchvchih}\index{Cone $\chi_h$}\index{Cone $\chi_v
$}
The \emph{horizontal cone} is $\chi_h:= \{(u,v): |v|\le \theta | u|\}$.
The \emph{vertical cone} is $\chi_v:= \{(u,v): |u|\le \theta | v|\}$. 
\end{defi}
\begin{defi}[Flat curve]\index{flat curve}
A $C^{1+Lip}$-embedded curve in $\Gamma$ in $\R^2$  is \emph{flat} if its curvature is at most  $\theta$. In other words, there exists parametrization $t\mapsto \gamma(t)$ of $\Gamma$ such that $\|\partial_t \gamma \|=1$ and the following holds true at every $t$:
\[\limsup_{s\to 0}  \left|\det\left(\frac{\dot \gamma(t)-\dot \gamma(t+s)}{s}, \dot \gamma(t)\right)\right|\le \theta\; .\]
\end{defi}
\begin{defi}[Horizontal and vertical stretched curves]\index{Horizontal curve}\index{Vertical curve}\index{Stretched curve}
By \emph{horizontal curve} we mean a  $C^{1+Lip}$-embedded curve $S$ included in  $\R\times [-\theta, \theta]$, which is flat and such that its tangent space is in $\chi_h$. A horizontal curve is said to be \emph{stretched} if both of its endpoints belong to the vertical sides $\partial^s Y_\se$ of $Y_\se$.

By \emph{vertical curve}, we mean a $C^{1+Lip}$-embedded curve $\mathcal C$ included in  $\R\times [-\theta, \theta]$ which is flat and such that its tangent space is in $\chi_v$. The vertical curve $\mathcal C$ is \emph{stretched} if both of its endpoints belong to $\R\times \{-\theta\}\sqcup \R\times \{\theta\}$. 
\end{defi}

From this definition, each of the two endpoints of a horizontal, stretched curve $S$ belongs to each of the two connected components of $\partial^s Y_\se$ respectively, such that the curve stretches across $Y_\se$. Then, note that $S$ is included in $Y_\se$. 

Similarly, each of the two endpoints of a vertical stretched curve belongs to each of the two connected components of $\R\times \{-\theta, \theta\}$, such that the curve stretches across $\R\times [-\theta, \theta]$.  We note that a vertical, stretched  curve is included in $\R\times [- \theta,\theta]$ but it is \emph{not} necessarily included in $Y_\se$. 

Similarly to the one-dimensional case, we will define subset of $\R\times [- \theta, \theta]$ using subsets of the stable manifold of $A$. We recall that $W^s_{loc}(A):=\{\alpha_0\}\times [- \theta,\theta]$ is a local stable manifold of $A$. Its preimages are not necessarily curves because $f$ is an endomorphism. To this end we need a transversality assumption.
\begin{defi}[Arc of $W^s(A)$]\index{Arc of $W^s(A)$}  An \emph{arc}  of $W^s(A)$ is a connected curve $\mathcal C\subset \R\times [-\theta, \theta]$ which is sent by an iterate $f^n$ into $W^s_{loc}(A)$, and such that  $D f^n$ does not vanish on the normal space to $\mathcal C$:
\[\forall z\in \mathcal C,\quad D_z f^n(\R^2)\not\subset   T_{f^n(z)} W^s_{loc}(A)\; .\]
\end{defi}
By transversality, \emph{arcs} are $C^2$-embedded curves. 
A first interest of the later object is that given two arcs of $W^s(A)$, they are either disjoint or  their union forms an arc of $W^s(A)$. A second interest is that the hyperbolic continuation enables us to follow them for perturbations of the dynamics. In dimension 2, the concept of piece takes the following form:
\begin{defi}[Box]\index{box}\label{boxdef}
A \emph{box} $Y$ is a compact subset of $\R\times [-\theta, \theta]$ such that there exist two real $C^2$-functions $\psi_+>\psi_-$ satisfying:
\[Y=\{(x,y)\in \R\times [-\theta, \theta]: \psi_-(y)\le  x\le \psi_+(y)\}\; .\]
and each of the curves $\{(\psi_\pm (y),y): y\in [-\theta, \theta]\}$ is a vertical, stretched,  arc of $W^s(A)$.

We remark that $Y$ is diffeomorphic to the filed square $[0,1]^2$. We define also: 
$$\partial^u Y:= \bigsqcup_{ \pm \in \{-,+\}}  \{(x,\pm \theta ): \psi_-(\pm \theta )\le  x\le \psi_+(\pm \theta )\}\; ,$$
$$\partial^sY:= \bigsqcup_{ \pm \in \{-,+\}} \{(\psi_\pm (y),y): y\in [-\theta, \theta]\}.$$ 
We notice that $\partial^s Y$ is formed by two disjoint, vertical, stretched curves. 

Also $\partial^u Y$ is formed by two horizontal curves and $\partial Y:= \partial^u Y\cup \partial^s Y$ is the topological boundary of $Y$. 
\begin{rema}\label{constructionbox} Conversely, a subset of $\R\times [-\theta, \theta]$  bounded by two disjoint vertical, stretched, arcs of $W^s (A)$ is a box.\end{rema} 

\end{defi}
An example of box is $Y_\se=I_\se\times [-\theta, \theta]$. To construct a box associated to each letter $\ss_n^\pm\in \sY_0$, we will use the following:
\begin{lemm}\label{conefact}
Let $U:= [a+2 b, \infty]\times [-\theta, \theta]$. For every box $Y\subset U$, the preimage 
$(f|\R\times [-\theta, \theta])^{-1}(Y)$ is formed by two boxes $g_-(Y) \subset (-\infty, -\sqrt b)\times [-\theta, \theta]$ and  $g_+(Y) \subset (\sqrt b , +\infty)\times [-\theta, \theta]$.
\end{lemm}
\begin{proof} Note that $f^{-1}(U)$ is in the complement of $[-\sqrt{b}, \sqrt b]\times \R$. We note that the differential $D_{x,y}f$ is $b$-close to $(u,v)\mapsto (2 xu, 0)$. Thus:
\begin{equation}\label{mappingcone} 
D_{z}f(\chi_h)\subset \chi_h\qand D_{z}f^{-1}(\chi_v)\subset \chi_v\; ,\quad \forall z\in f^{-1}(U) .\end{equation}
Consequently, given $y\in [-\theta, \theta]$, the curve $(\sqrt b , +\infty)\times \{y\}$ is sent by $f$ to a horizontal curve which stretches across $U$. It intersects transversally $\partial^s Y$ at two points. By transversality these two points depend $C^2$-on $y$. Also by \cref{mappingcone}, when $y$ varies in  $[-\theta, \theta]$, these intersection points form two disjoint vertical curves.  As by definition, these two vertical curves are stretched, by \cref{constructionbox}, the preimage of $Y$ in $(\sqrt b , +\infty)\times [-\theta, \theta]$ is a box $g_+(Y)$. 
The proof for  $g_-(Y)$ is the same. \end{proof}
We are going to construct boxes for each of these symbols:
\begin{eqnarray*}
\sY_0:= \{s^\pm_j: \pm\in \{-,+\}, 2\le j\le M-1\}\qand 
 \sA_0:= \sY_0\sqcup \{\ss^-_M, \ss^+_M\}\end{eqnarray*}

\begin{exem}[Simple boxes $Y_\ss$, $\ss\in \sY_0$] \label{simple boxes}\index{Simple boxes}
With the notation of \cref{conefact}, an induction on $i$ shows that  $g_+^i(Y_\se)$ is a box included in  $(0, +\infty)\times [-\theta, \theta]\subset U$. Actually each of these boxes is $O(b)$-close to $[-\alpha_{i-1}, -\alpha_i]\times [-\theta,\theta]$ when $i\ge 1$ (where the sequence $(\alpha_i)_i$ was defined in  \cref{piece1D} \cpageref{piece1D}). 
Similarly,  $g_- \circ g_+^i(Y_\se)$ is a box which is $O(b)$-close to $[\alpha_{i}, \alpha_{i-1}]\times [-\theta,\theta]$. 

Thus when $i\le M-2$, the box $g_- \circ g_+^i(Y_\se)$ is included in $U$. Thus we can define for every $\ss_{i+2}^\pm\in \sY_0$ the following simple box:
\[Y_{\ss_{i+2}^\pm}:= g_\pm\circ g_- \circ g_+^{i}(Y_\se)\; .\]
\end{exem}
\begin{exem}[Boxes $Y_{\ss_M^\pm}$, $Y_\sr$ and $Y_\boxdot$
] \index{$Y_\boxdot$}\index{$Y_\sr$}\label{boxYsM} Let $Y_\sr :=g_- \circ g_+^{M-1}(Y_\se)$. 

We recall that by $0$-strong regularity, it holds that $P^{M+1}(0)$ does not belong to 
$I_{ \ss_-^{\odot \aleph(0)}}$, with $\aleph(0)= \lfloor M/24\rfloor$. As $DP|[-\beta, \beta]$ is at most $4$, the preimage by $f^M|[\alpha_M, \alpha_{M-1}]$ of $I_{ \ss_-^{\odot \aleph(0)}}$ has length $\ge 2\cdot 4^{-M-M/12}$ and so $a-\alpha_{M-1}\ge  2\cdot 4^{-M-M/12}$. 
Consequently, $Y_\sr :=g_- \circ g_+^{M-1}(Y_\se)$ belongs to the open set $U$ defined in \cref{conefact}   because it is $O(b)$-close to $[\alpha_{M}, \alpha_{M-1}]\times [-\theta, \theta]$. Also by \cref{conefact}, the following are boxes:
\[Y_{\ss_M^\pm}:= g_\pm\circ g_- \circ g_+^{M-2}(Y_\se)\; .\]
The right component of $\partial^s Y_{\ss_M^-}$ and the left component of  $\partial^s Y_{\ss_M^+}$ are disjoint vertical, stretched curves. By \cref{constructionbox}, they bound a box that we denote by $Y_\boxdot$.
\end{exem}
We notice that $f(Y_\boxdot)\subset Y_\sr$ . Also it holds $Y_\se= \bigcup_{\ss\in \sA_0} Y_\ss \cup Y_{\boxdot}.$

\begin{figure}[h!]
 \centering
 \includegraphics[width=13cm]{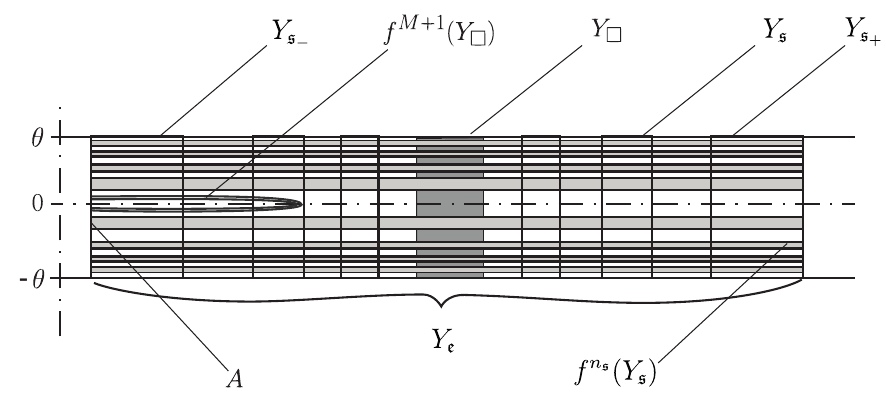}
 \caption{Simple pieces and their first return}
\label{notation_firstreturn simple piece}\end{figure}
\begin{defi}[Piece]\index{Piece in dimension 2}\label{piecedef}
A \emph{piece} $(Y, n)$ is the data of an integer $n$ and a box $Y$ such that the following conditions are satisfied:
\begin{enumerate}[$(i)$]
\item[$(o)$]The box $Y$ is sent into $Y_\se$ by $f^{n}$: $f^{n}(Y)\subset Y_\se$.
\item \label{defi_condi_hyp} For every $w\in \chi_h$,  for every $z\in Y$, for every $0\le m\le n$, it holds:
\[\| D_zf^{n}(w)\|\ge 2^{\frac{m}3 } \|D_{f^m(z)}f^{n-m}(w)\|\; .\]
\item  For every $w\in \chi_h$,  for every $z\in Y$, for every $0\le m\le n$ it holds:
\[\| D_zf^{m}(w)\|\ge b^{m/6} \|w\|\; .\]
 \end{enumerate}
\end{defi}
Using that the determinant of $Df$ is very small, we will show in sections \ref{proof:invariancecone}   and \ref{proof:invariancecone2}, Pages \pageref{proof:invariancecone} and \pageref{proof:invariancecone2}, the following:
\begin{prop}\label{invariancecone}
Every piece  $(Y, n)$  satisfies the following properties for every $z\in Y_\sa$:
\begin{enumerate}[(i)]
\item The horizontal cone is sent into itself by $D_zf^{n}$: %
$D_z f^{n}(\chi_h)\subset \chi_h\; .$ Moreover, given a horizontal curve $S\subset Y$, the curve $f^n(S)$ is horizontal (and so flat).
\item The vertical cone pulls back into itself by 
$D_zf^{n}$: 
$(D_z f^{n})^{-1}(\chi_v)\subset \chi_v\; .$
Moreover, 
for every vertical stretched curve $\mathcal C\subset Y_\se$ which is between the two components of $f^n(\partial^s Y)$, the set  $(f^{n}|Y)^{-1}(\mathcal C)$ is a vertical,  stretched  curve. 
 \end{enumerate}
\end{prop}
\begin{rema}\label{def piece canonical}
In the piece \cref{piecedef}, we put conditions $(i)$ and $(ii)$ instead of the above cone conditions $(i)$ and $(ii)$ (which would give a canonical generalization of the definition of piece in dimension 1), because they are stronger by \cref{invariancecone} and  the stated expansion will be obtained automatically for all the operations we will consider, provided that some combinatorial rules are satisfied. In other words, for all the pieces we will construct, these two pair of conditions are equivalent. 
\end{rema}
\medskip 
\begin{rema}\label{3.12}As $f$ is an endomorphism, the set $f^{n}(Y)$ is not necessarily diffeomorphic to the filed square $[0,1]^2$. However, $ f^{n}(\partial^u Y)$ is a union of two horizontal curves by \cref{invariancecone} $(i)$. These two curves may intersect. Conversely,  $f^{n}(\partial^s Y)$ is  the union of two disjoint curves (in general not vertical and not even embedded), by \cref{piecedef} $(i)$ and \cref{invariancecone} $(i)$ .\end{rema}
\begin{defi}[Puzzle piece]\index{Puzzle piece in dimension 2}
A piece $(Y, n)$ is a \emph{puzzle piece} if   $f^{n} (\partial ^s Y)$ is included in $\partial^sY_\se$. 
\end{defi}
\begin{exem}[Piece associated to $\se$] The pair $(Y_\se, n_\se)$ is a puzzle piece, with $n_\se=0$. \end{exem}
\begin{exem}[Pieces associated to $\ss\in \sA_0\cup\{\sr\}$]\index{Simple pieces in dimension 2}
Given $\ss=\ss_i^\pm\in \sA_0$, we recall that $n_\ss=i$. We put $n_{\sr}= M$. \label{simplespieceindim2}
We recall that in examples \ref{simple boxes} and \ref{boxYsM} we showed that 
$Y_\ss$ is a box for every $\ss\in \sA_0\cup\{\sr\}$.
We notice furthermore that $(Y_\ss, n_\ss)$ satisfies $f^{n_\ss}(Y_\ss)\subset Y_\se$ 
$f^{n_\ss}(\partial^s Y_\ss)\subset \partial^s Y_\se$. 
In  \cref{app to 2} p. \pageref{app to 2}, we will show moreover that $(Y_\ss, n_\ss)$ satisfies Conditions $(i)$ and $(ii)$ of \cref{piecedef}. Thus $(Y_\ss, n_\ss)$ is a puzzle piece for every $\ss\in \sA_0\cup\{\sr\}$. The puzzle piece is said \emph{simple} if $\ss\in \sY_0$. 
\end{exem}
\begin{rema}
The first return time of $Y_\boxdot$ in $Y_\se$ is $n_\boxdot=M+1$. We notice that $f^{n_\boxdot} (\partial^s Y_\boxdot)$ is included in the left hand side component of $\partial^s Y_\se$. Thus $f^{n_\boxdot}(\R\times\{0\}\cap Y_\boxdot)$ has  both endpoints in the left component of $\partial^s Y_\se$ and so it cannot be horizontal. Thus $(Y_\boxdot, n_\boxdot)$ is not a piece.
\end{rema}
The following is fundamental although elementary.
\begin{prop}
If $(Y,n)$ and $(Y',n')$ are puzzle pieces, then they are either nested or with disjoint interior:
Either  $Y\subset Y'$ or $Y'\subset Y$ or $Y\cap int(Y')=\varnothing$.
\end{prop}
\begin{proof}
Assume that $n \ge n'$ and  that $Y$ intersects $Y'$ at a set with non-empty interior. By coherence of the arcs of $W^s(A)$, the set  $I:=\partial^u Y\cap \partial^uY'$ is the union of two non trivial segments. If $Y$ is not included in $Y'$, then $f^{n'}(\partial^u Y)$ is not included in $Y_\se \supset f^{n'}(\partial^u Y')$.  As it contains $f^{n'}(I)\subset Y_\se$, the curves $f^{n'}(\partial^u Y)$ intersects $\partial^sY_\se$ in their interiors.  This property will hold also for any iterate  $f^{n'+k}(\partial^u Y)$, $k\ge 0$. 
  This contradicts the fact that $f^{n}(Y)\subset Y_\se$.
  Consequently $Y$ is included in $Y'$.
\end{proof}

\subsection{Operation  $\star$ on the pieces}\label{section star prod}
The mapping cone property \ref{invariancecone} enables to define a $\star$-product of two pieces as soon as a pure topological condition is satisfied.
\begin{defi}[$\star$-product]\index{$\star$-product in dimension 2}\label{starprod}
Let $(Y,n)$ and $(Y', n')$ be pieces with 
$Y'\subset Y_\se$.  The pair of pieces  $((Y,n), (Y', n'))$ is said \emph{admissible for $\star$} if $f^n(int \, Y)$ intersects $Y'$.
Also we put: $$(Y,n)\star (Y', n')=(Y'', n'')\text{, with }Y''= Y \cap f^{-n}(Y')\text{ and }n''= n +n'.$$ 
We depict this definition in \cref{fig_starprod}.
\end{defi}

\begin{figure}[h!]
 \centering
 \includegraphics[width=13cm]{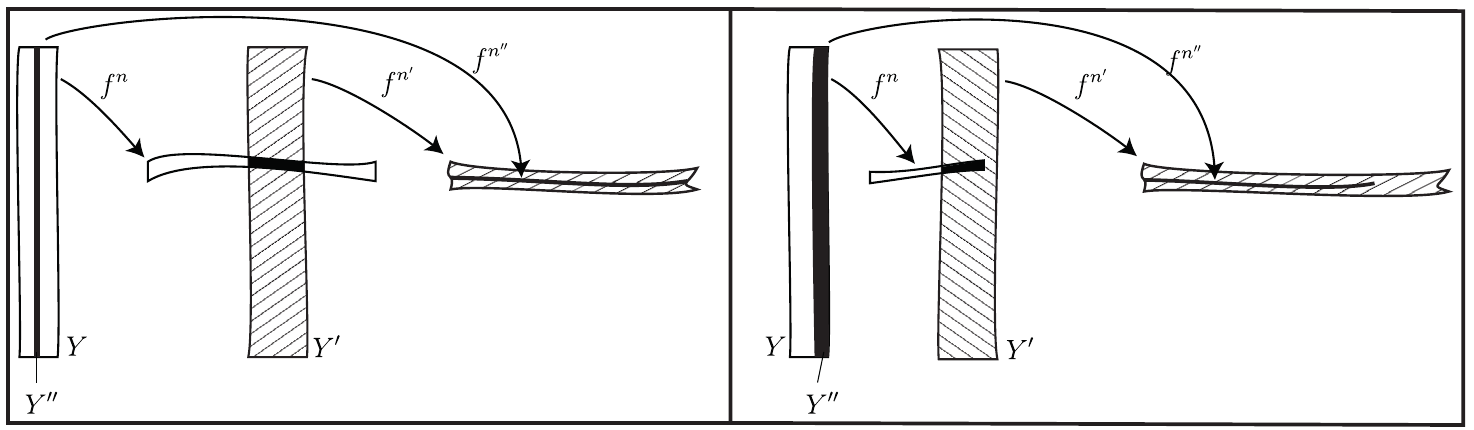}
 \caption{Two admissible configurations for the $\star$-product:
\[(Y'', n'')= (Y,n)\star (Y',n')\;.\]}
\label{fig_starprod}\end{figure}

\begin{prop}\label{admi2}
If $((Y,n), (Y',n') )$ is a pair of  pieces admissible for $\star$, then $(Y'',n'') = (Y,n)\star (Y',n')$ is a piece.
\end{prop}

\begin{proof}
We recall that $f^n(\partial^s Y)$ is formed by two components by \cref{3.12}. 
 As $\partial^s Y$ and $\partial^s Y'$ are formed by arcs of $W^s(A)$, each component of $f^n(\partial^s Y)$ is disjoint or included in  $\partial^s Y'$. Also $f^n(Y)$ is included in the $b$-neighborhood of $\R\times \{0\}$ and so it is disjoint from $\partial^u Y_\se$.

Let $y\in [-\theta, \theta]$ and $S(y) :=\R\times \{y\}\cap Y$. This curve has its endpoints in $\partial^sY$ and its tangent space in $\chi_h$. By  \cref{invariancecone} $(i)$, the curve $S'(y):=f^{n}(S(y))$ has its tangent space in $\chi_h$. Hence $S'(y)$ is transverse to  $\partial^sY'$ and intersects at most once each of its components (by the coherence of the arcs of $W^s(A)$ and since the image of $f$ is in $\R\times [-b,b]$).  The curve $S'(y)$ varies $C^2$ with $y$, and its endpoints cannot cross the curves $\partial^s Y'$ nor $\partial^u Y_\se$. 
Thus $S'(y)$ intersects in its interior a number of components of $\partial^s Y''$ which does not depend on $y\in [-\theta, \theta]$. 
If this number is 0, then $f^n(Y)$ must be included in $Y'$, and so $Y''=Y$ is a box. 
If a component of $\partial^s Y'$ intersects  the interior of $S'(y)$, then it intersects the interior of each curve $S'(y')$ for $y'\in [-\theta, \theta]$. Thus by \cref{invariancecone} $(ii)$, the preimage of this component is a vertical, stretched curve. This vertical curve bounds one side of $Y''$. The other side is shown to be a vertical curve using the same argument. 
This proves that $Y''$ is a box.

To verify conditions $(i)$ and $(ii)$ of \cref{piecedef}, it suffices to recall that $Df^{n}|Y$ sends $\chi_h$ into $\chi_h$ by \cref{invariancecone} and to use then conditions $(i)$ and $(ii)$ of $(Y,n)$ and $(Y',n')$. This shows that $(Y'',n'')$ is a piece.
\end{proof}

Note that the operation $\star$ is associative:
\begin{prop}
For any pieces $(Y,n)$, $(Y',n')$, $(Y'',n'')$ the two following assertions are equivalent:
\begin{itemize}
\item  $((Y,n), (Y',n'))$ is  admissible for $\star$ and  $((Y,n)\star (Y',n'), (Y'',n''))$ is admissible for $\star$,
\item $((Y',n'), (Y'',n''))$ is admissible for $\star$ and 
$( (Y,n), (Y',n')\star (Y'',n''))$ is admissible for $\star$.
\end{itemize}
Furthermore it holds:
\[((Y,n)\star (Y',n'))\star (Y'',n'') = (Y,n)\star ( (Y',n')\star (Y'',n''))=: (Y,n)\star (Y',n')\star (Y'',n'')
\; . \]
\end{prop}
The proof of this proposition follows from the commutativity of the operations $\bigcap$ and $f^{-1}$ on sets. 
\begin{exem}\label{productsimplepiece}
For $\sg= \sa_1\cdots \sa_m\in \sA_0^{(\N)}$, the following product is admissible: 
\[(Y_{\sg}, n_{\sg}):= 
(Y_{\sa_1}, n_{\sa_1})\star \cdots \star (Y_{\sa_m}, n_{\sa_m})\; .\]
Furthermore, $(Y_{\sg}, n_{\sg})$ is a puzzle piece. 
\end{exem}
\begin{exem}\label{ss-j}
We recall that $\ss_-:=\ss_2^-$ and $\ss_-:=\ss_2^+$.  
Given $j\ge 0$, we denote:
\[( Y_{  \ss_-^{\star j }}  ,n_{  \ss_-^{\star j }} ):= 
\underbrace{ (Y_{\ss_-},2)\star \cdots \star (Y_{\ss_-},2)}_{j\text{ times}}\; 
\qand
( Y_{ \ss_+\cdot   \ss_-^{\odot  j }}  ,n_{ \ss_+\cdot  \ss_-^{\odot  j }} ):= 
 (Y_{\ss_+},2)\star(Y_{\ss_-^{\odot  j}},2j) 
 \]
We notice that the box  $Y_{  \ss_-^{\odot  j }}$ is a neighborhood in $Y_\se$ of the left component of $\partial^s Y_\se$. 
On the other hand $Y_{  \ss_+\cdot \ss_-^{\odot  j }}$ is a neighborhood in $Y_\se$ of the right component of $\partial^s Y_\se$. 
\end{exem}
\begin{exem}[Set $\partial_j Y$]\label{ss-j2}\index{$\partial_j Y$}
 Given a puzzle piece $(Y,n)$ let:
\[ (\partial^-_{j} Y, n+2j):= (Y,n)\star (Y_{\ss_-^{\odot j}}, 2j )\qand 
(\partial^+_{j} Y, n+2j +2):= (Y,n)\star (Y_{\ss_+\cdot \ss_-^{\odot j}}, 2j+2)\]
\[\partial_j Y:= \partial^+_{j} Y\sqcup \partial^-_{j} Y\]

 The set $\partial_j Y$ is a neighborhood of $\partial^s Y$ in $Y$. \end{exem}
\begin{rema}\label{estim_epaisseur}
For every $y\in [–\theta, \theta]$, the line $\R\times \{y\}$ intersects $\partial^j Y$ at two segments of length at least $4^{-n-2j}$. Indeed, such segments are sent by $f^{n+2j}$ to horizontal, stretched curves (whose length is approximately $2$), and along this orbit, the norm of $Df$ is at most $4$ by Fact \ref{trivial estimate}. \end{rema}

\subsection{Action of pieces on curves}\label{section action of piece on curve}
Puzzle piece induces natural graph transform on the set of 
 horizontal, stretched curves. They are also contracting for a complete  distance that we shall define. 
 \begin{defi}[Metric  space $\mathcal H$]\label{horizontal curve}\index{$\mathcal H$}
Let $\mathcal H$ be the set of horizontal, stretched curves. 
Given $S_+$ and $S_-$ in $\mathcal H$, we define:
\[d(S_+,S_-)= \|\rho_--\rho_+\|_{C^1},\]
where $S_\pm $ is the graph of $\rho_\pm $: $S_\pm =\{(x,\rho_\pm(x)):x\in I_\se\}$.

A \emph{graph transform} is a continuous map from $\mathcal H$ into $\mathcal H$.\index{Graph transform}
\end{defi}
We notice that $(\mathcal H, d)$ is a compact metric space. 

We recall that by the first item of \cref{invariancecone}, given a puzzle piece $(Y,n)$ in $Y_\se$ and a horizontal, stretched  curve $S\in \mathcal H$, the curve $f^n(S\cap Y)$ is horizontal and stretched. We will prove the following in \cref{proof Contraction of the graph transform}. 
\begin{prop}\label{contractionPP}
For every puzzle piece $(Y,n)$, the map   $S\in \mathcal H\mapsto f^{n}(S\cap Y)\in \mathcal H$ is $b^{n/3}$ contracting.
\end{prop} 
\begin{coro}\label{productsimplepiece2}
For every $\sg= \sa_1\cdots \sa_m\in \sA_0^{(\N)}$, with $(Y_\sg, n_\sg)$ the puzzle piece defined in \cref{productsimplepiece}, the following graph transform is well defined and $b^{n_\sg/3}$-contracting:
\[T_\sg : S\in \mathcal H \mapsto  S^\sg:= f^{n_\sg}(Y_\sg\cap S)\in \mathcal H\; .\]
\end{coro}

The next proposition shows that a sequence of nested puzzle pieces defines a Pesin local stable  manifold (see \cref{pesin} \cpageref{pesin}).
\begin{prop}\label{Pesin stable1}\index{$\pi_f$}
There is a Lipschitz map $\pi_f: \R\times [-\theta, \theta]\to \R$ such that:
\begin{enumerate}
\item The map $\pi_f$ is $b^{1/3}$-$Lip$-close to the first coordinate projection, and its fibers are $C^{1+Lip}$-curves.
\item The map $z\mapsto \ker D_z\pi_f$ is Lispchitz with constant smaller than $b^{1/3}$. 
\item For every puzzle piece $(Y,n)$, there are $(x_-, x_+)$ such that $Y=\pi^{-1}([x_-, x_+])$. Moreover, for every $z\in \partial^s Y$ and $n\ge 1$ it holds $ \|Df^n| T_z \partial^s Y\|\le b^{n/2}\; .$
\item For every sequence $C=(Y_k, n_k)_k$  of puzzle pieces $(Y_k, n_k)$ such that $Y_{k+1}\subsetneq Y_k$ for every $k\ge 0$,  there exists $x_C\in \R$ such that:
$$W^s_C:= \bigcap_{k\ge 0} Y_k  = \pi_f^{-1}(x_C)$$
 and for every $z\in W^s_C$ and $n\ge 1$ it holds $ \|Df^n| T_z W^s_C\|\le b^{n/2}\; .$
 In particular $W^s_C$ is a vertical stretched curve and a Pesin local stable manifold.   
\end{enumerate}
\end{prop}
We will prove this proposition in \cref{pesinstableprop}.

The next proposition shows how to construct a Pesin local unstable  manifold (see \cref{pesin} \cpageref{pesin}). Let $M:= \R\times (-\theta, \theta)$ and let $\arr M_f:= \{ (z_i)_{i\le 0}\in M^{\Z^-} :  f(z_{i-1})= z_i\; \forall i\}$ be the inverse limit of $f$. We recall that $\arr M_f$ is endowed with the induced product topology. We will prove the following in \cref{proof prop pesin instable gen} \cpageref{proof prop pesin instable gen}:
\begin{prop}\label{prop pesin instable gen} Let $ Y_\infty:=(Y_i, n_i)_i$ be a sequence of puzzle pieces of increasing order such that $f^{n_{i+1}-n_i}(Y_{i+1})\subset Y_{i}\subset Y_\se$ for every $i$. Then it holds:
\begin{enumerate}
\item the set  $W_u^{Y_\infty}:=\bigcap_i f^{n_i}(Y_i)$ is a horizontal stretched curve. 
\item the set $\arr W_u^{Y_\infty}=\{(z_i)_i\in M_f:  z_{-n_i}\in Y_i, \forall i\ge 0\}$ is a Pesin  local  unstable manifold which projects via the $0$-coordinate projection onto $W_u^{Y_\infty}$.
\end{enumerate}
\end{prop}
\subsection{Parabolic product}\label{section para prod}
In the one-dimensional case we define the parabolic pieces using that the first return $P^{n_\boxdot} (0)$ of the critical point in $I_\se$ belongs to 
$I_\sa\subset I_\sb$ for two puzzle pieces  $(I_\sa, n_\sa)$ and $(I_\sb, n_\sb)$. Actually, assumption $(SR_3)$ of strongly regular, quadratic map implies that the parabolic product used satisfies that $P^{n_\boxdot} (0)$ does not belong to $I_\sb\setminus (I_{\sb\cdot \ss_-^{\odot \aleph(\sb)}} \cup I_{\sb\cdot \ss_+\cdot \ss_-^{\odot \aleph(\sb)}})$. The following is the $2$-dimensional counterpart of this assumption. 
\begin{figure}[h!]
 \centering
 \includegraphics[width=7cm]{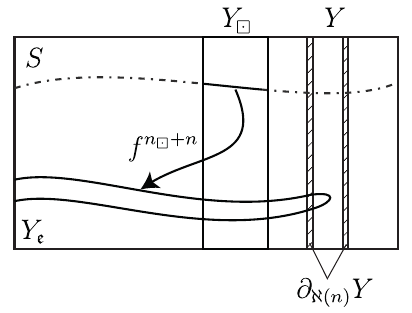}
 \caption{Critical position}
\label{fig:critical position}\end{figure}

\begin{defi}[Critical position]\index{Critical position} \label{defi Critical position}
A horizontal, stretched curve $S$ is \emph{in critical position} with a puzzle piece $(Y,n)$ if $f^{n_\boxdot}(S\cap Y_\boxdot)$ intersects $cl(Y\setminus \partial_{\aleph(n)} Y)$ and exactly one component of $\partial_{\aleph(n)} Y$, with $\aleph(n) =\lfloor \frac{n}{12}+\frac{M}{24}\rfloor$. This definition is depicted \cref{fig:critical position}.
\end{defi}

\begin{defi}[Domains $\cD(\boxdot(Y-Y'))$ and $\tilde \cD(\boxdot(Y-Y'))$] \label{domaine}\index{$\cD(\boxdot(Y-Y'))$ and $\tilde \cD(\boxdot(Y-Y'))$}
Let $(Y,n)$ and $(Y',n')$ be two puzzle pieces such that $Y'\subset Y \subset Y_\se$ and  $n'\le n +\max(M, n)$. Let $\cD(\boxdot(Y-Y'))$ be the set of horizontal stretched curves in critical position with $(Y',n')$ and let $\tilde \cD(\boxdot(Y-Y'))$ be the closed  $\theta^{n'}$-neighborhood  of
$\cD(\boxdot(Y-Y'))$ in $\mathcal H$. 
\end{defi}
\begin{rema}\label{C0 Dboxdot} The set $D(\boxdot(Y-Y'))$ is closed for the $C^0$-topology of $\mathcal H$, whereas $\tilde D(\boxdot(Y-Y'))$ is closed for the canonical, $C^1$-topology of $\mathcal H$. 
\end{rema}
\begin{rema}If the left components of $\partial^s Y'$ and $\partial^s Y$ are the same, then  $\cD(\boxdot(Y-Y'))=\tilde \cD(\boxdot(Y-Y'))=\varnothing$. 
\end{rema} 
\begin{rema}\label{estim_epaisseur bis} For every  $S\in \tilde \cD(\boxdot(Y-Y'))$ there exists $S'\in \cD(\boxdot(Y-Y'))$ which is $\theta^{n'}$-close. Then connected compact sets $f^{n_\boxdot} (S\cap Y_\boxdot)$ and $f^{n_\boxdot} (S'\cap Y_\boxdot)$ are $4^M\theta^{n'}$-close.  We recall that  $n'\le n +\max(M, n)$. Hence by \cref{estim_epaisseur}, the width of each component of $\partial_{\aleph(n')} Y'$ is large compared to $4^M\theta^{n'}$. Thus for every $S\in \tilde \cD(\boxdot(Y-Y'))$, the curve $f^{n_\boxdot} (S\cap Y_\boxdot)$ enter and exit at least once in $int\, Y'$ by the left component of $\partial^s Y'$, and does not intersect the right component of $\partial^s Y'$.
\end{rema}

\begin{prop}\label{def S_boxdot(a-b)}\label{remark pour la tranversalite} For every $S\in \tilde \cD(\boxdot(Y-Y'))$, the set $(f^{M+1}|S\cap Y_\boxdot)^{-1}cl(Y\setminus Y')$ consists of two disjoint segments. 
We denote by $S_{\boxdot_-(Y -Y')}$ the left hand side segment and by $S_{\boxdot_+(Y -Y')}$ the right hand side segment.

One endpoint of $S_{\boxdot_\pm(Y -Y')}$ is given by a transverse intersection with $\partial^sY$ via $f^{M+1}$ and the other endpoint is given by a transverse intersection with $\partial^sY'$ via $f^{M+1}$.
\end{prop}\begin{proof}
 As $S$ is a horizontal, stretched curve, it can be parametrized by $(t, \rho(t))_{t\in I_\se}$ for a function $\rho$ which is $2\theta$-$C^{1+Lip}$-small. 
 Its image by $f|Y_\boxdot$ is $O(b)$-close to $(t^2+a ,0)_{t\in I_\boxdot}$.
  Likewise,  $S^\boxdot:=f^{M+1}(S\cap Y_\boxdot)$ is $\theta$-$C^{1+Lip}$-close to be of the form $(S^\boxdot(t))_t=\{(Cst\cdot  t^2+P^{M}(a), 0): t\in I_\se\}\cap Y_\se$
  for $Cst<0$. Note that  $S^\boxdot$ is in the $b$-neighborhood of $I_\se\times \{0\}$ and so does not intersect $\partial^u Y_e$. Also  the endpoints of  $(S^\boxdot(t))_t$ are in the left component of $\partial^s Y_\se$.

Thus given any stretch vertical curve $V\subset Y_\se$ which intersects $(S^\boxdot(t))_t$, either the intersection is tranverse and occurs at exactly two points, or the intersection holds at a unique intersection point. 
Consequently, by \cref{estim_epaisseur bis}, the parametric curve $(S^\boxdot(t))_t$ intersects the left components of $\partial^s Y'$ transversally at exactly 2 points, but does not intersect the right component of $\partial^s Y'$.
\end{proof}

The curve  $f^{n_\boxdot +n} (S_{\boxdot_\pm (Y -Y')})$ is not stretched. However,  under extra assumptions, it is horizontal and  we will be able to extend it to a horizontal, stretched curve. Let us precise this:
\begin{defi}[Extension algorithm for a piece]\index{Extension of a horizontal curve}
If $S$ is a horizontal curve, an \emph{extension} $\hat S$ of $S$ is a horizontal, stretched curve which contains $S$: $\hat S\in \mathcal H$ and $\hat S\supset S$.
An \emph{extension algorithm for a piece}  $(Y,n)$ in $Y_\se$ is a continuous map
$T_Y:S\in  \mathcal H \to S^Y\in \mathcal H$ such that 
$$T_Y(S) \supset f^n(S\cap Y),\quad \forall S\in \mathcal H\; .$$ 
\end{defi}

We recall that the space of horizontal, stretched curves $\mathcal H$ is canonically endowed with a complete $C^1$-distance $d$ (see \cref{horizontal curve}).
\begin{prop}\label{graphtransformpara}
Let $(Y,n)$ and $(Y',n')$ be two puzzle pieces such that $Y'\subset Y \subset Y_\se$  with $n'\le n +\max(M, n)$.
If $\tilde \cD(\boxdot(Y-Y'))$ is non-empty, then for every $\pm\in \{-,+\}$, there exists a map:
\[T_{\boxdot_\pm (Y-Y')}: S\in \tilde \cD(\boxdot (Y-Y'))\subset \mathcal H\mapsto 
S^{\boxdot_\pm (Y -Y')}\in \mathcal H,\]
such that:
\begin{enumerate}[$(i)$]
\item  Then for every $S\in \tilde \cD(\boxdot (Y-Y'))$, the 
 horizontal, stretched curve $S^{\boxdot_\pm (Y -Y')}$ is an extension of $f^{n_\boxdot +n} (S_{\boxdot_\pm (Y -Y')})$:
$S^{\boxdot_\pm (Y -Y')} \supset f^{n_\boxdot +n} (S_{\boxdot_\pm (Y -Y')})\; .$
\item The map $T_{\boxdot_\pm (Y-Y')}$ is $b^{(n_\boxdot +n)/3}$-contracting.
\item The curve  $S^{\boxdot_\pm (Y-Y')}$ is $\theta b^{n/3}$-close to 
$f^n(S'\cap Y)$ for all $S\in \tilde \cD(\boxdot (Y-Y'))$ and  $S'\in \mathcal H$.
\end{enumerate}
\end{prop}
We will show this proposition in \cref{proof of disjoint box}. 
The proof will define explicitly the map $T_{\boxdot_\pm(Y\setminus Y')}$ 
in \cref{extension algo} \cpageref{extension algo}.
A consequence of this proposition is that each curve $f^{n_\boxdot +n} (S_{\boxdot_\pm (Y -Y')})$ is horizontal.  
\begin{figure}[h!]
 \centering
 \includegraphics[width=7cm]{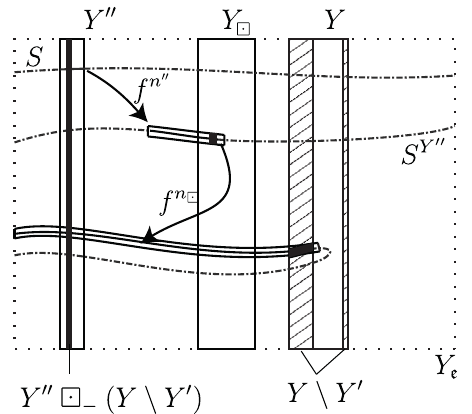}
 \caption{Admissible parabolic product.}
\label{notation_prodpara}\end{figure}
\begin{defi}[Admissible parabolic products]\label{def pre admissible}\index{Pre-admissible parabolic product}
Let $\pm\in \{-,+\}$.  Let $(Y,n), (Y',n'), (Y'',n'')$ be pieces and let
$T_{Y''}$ be an extension algorithm for $(Y'', n'')$.  The parabolic product  $\boxdot_\pm$ is  (resp. weakly) \emph{pre-admissible}  from these data if:
\begin{enumerate}
\item the pairs $(Y, n)$ and $(Y',n')$ are puzzle pieces and  
$M+1+n'\le 2^M n$,
\item the range of $T_{Y''}$ is included in  $\cD(\boxdot (Y-Y'))$ (resp.  $\tilde \cD(\boxdot (Y-Y'))$),
\item  $ f^{n''}(Y'')$ intersects $\boxdot_\pm (Y-Y')= \bigcup_{S\in \tilde \cD(\boxdot(Y-Y'))} S_{\boxdot_\pm (Y-Y')}$ at more than one arc of $W^s(A)$.
\end{enumerate}
A \emph{pre-admissible parabolic product} is (resp. weakly) \emph{admissible} if the pair:
$$(Y''\boxdot_\pm (Y-Y'), n''+n_\boxdot+n)$$ is a piece, with  $Y''\boxdot_\pm (Y-Y'):= (f^{n''}| Y'')^{-1}(\boxdot_\pm (Y-Y'))$.
This definition is depicted \cref{notation_prodpara}.\index{Admissible parabolic product}
\end{defi}
\begin{rema}\label{pre-admi=topo} If conditions 1 and 2 of the pre-admissibility are satisfied then:
$$Y''\boxdot_\pm (Y-Y'):= (f^{n''}| Y'')^{-1}(\bigcup_{S\in  \cD(\boxdot(Y-Y'))} S_{\boxdot_\pm (Y-Y')}),\quad  \forall \pm\in \{-,+\}.$$
Indeed $f^{n''}(Y'')\subset \bigcup_{S\in \mathcal H} T_{Y''}(S)\subset \cD(\boxdot(Y-Y'))\Subset \tilde \cD(\boxdot(Y-Y'))$. This enables us to see that the above preimage of the union over $\cD(\boxdot(Y-Y'))$ is equal to the one of the union over $\tilde \cD(\boxdot(Y-Y'))$. As $\cD(\boxdot(Y-Y'))$ is defined by a topological condition, by \cref{C0 Dboxdot}, it comes that pre-admissibility conditions 2 and 3 are also topological.
\end{rema}
%
%
%
%
\begin{prop}\label{disjoint box} If the parabolic product  is {weakly pre-admissible} for the data defined in \ref{def pre admissible}, then the set $Y''\boxdot_\pm(Y-Y')$ is a box.
If both parabolic products $\boxdot_- $ and $\boxdot_+$ are weakly pre-admissible, then the boxes  $Y''\boxdot_-(Y-Y')$ and $Y''\boxdot_+(Y-Y')$ are disjoint. 
\end{prop}
We will prove this proposition in \cref{proof disjoint box} p. \pageref{proof disjoint box}.
The parabolic product is more subtle than the simple product. We will need an extra condition to ensure the admissibility of a pre-admissible parabolic product: that the pair $( Y''\boxdot_\pm(Y-Y'), n''+n_\boxdot +n)$ is a piece. Namely we need that $Y$ satisfies sharper expansion  estimates than those given by piece's \cref{piecedef} $(i)$ and some distortion bounds.  This will be obtained by asking $(Y,n)$ and $(Y',n')$ to satisfy a pure combinatory condition: the strong regularity. 

The interest of this combinatory definition is to enable us to define algebraically the products of pieces which are well defined and satisfy these sharper expansion estimates on combinatorially and topologically defined open set of maps. In the parameter selection this will enable us to know when the parameter varies which pieces will be defined and varies continuously with the parameter. The formalism introduced in the next section is designed for this. 

%
%
%
%
%
%

\section{Symbolic formalism and strongly regular Hénon-like endomorphisms}\label{section SR}

Let  us recall some algebraic definitions which will be hopefully helpful to understand the notations and terminologies of our definition of strong regularity.
\subsection{Algebraic preliminary}\label{Algebraic preliminary}
\begin{defi}[Pseudo-monoid]\index{Pseudo-monoid}
A \emph{pseudo-monoid} $(\sM, \otimes)$ is an algebraic structure consisting of a set $\sM$ together with a binary operation $\otimes$ defined on a subset $D(\otimes) \subset \sM \times \sM$, which  satisfies the following axioms:
\begin{itemize}
\item \emph{Closure}:  For all $(\sa, \sb) \in D(\otimes)$, the result operation  $a \otimes b$ is also in $\sM$.
\item \emph{Associativity}: for all $(\sa, \sb) \in D(\otimes)$ and $(\sb, \sc)\in D(\otimes)$, we have:  $\sa \otimes (\sb\otimes \sc)=(\sa \otimes \sb)\otimes \sc $.
\item \emph{Identity element}: there exists an element $\se\in \sM$ such that for every element $\sa\in \sM$, we have  $(\sa,\se)\in D(\otimes)$ and $\sa\otimes \se=   \sa$.
\end{itemize}
\end{defi}
A simple example of pseudo-monoid is the semi-group $(\N,+)$ where $D(+)=\N$.

An immediate consequence of \cref{1-piece are pseudo} is:
\begin{exem} The set $\mathcal P_1$ of pieces  of a quadratic map endowed with the operation $\star$ with as domain $D(\star)$ the set of on admissible pairs for $\star$  is a pseudo-monoid.
\end{exem}

\begin{exem} The set $\mathcal P_2$ of pieces  of a Hénon-like endomorphism endowed with the operation $\star$ with as domain $D(\star)$ the set of on admissible pairs for $\star$  is a pseudo-monoid.
\end{exem}

To define the strongly regular Hénon-like endomorphisms we are going to define some combintatorial rules using a countable alphabet $\hat \sA$ which does not depend on the dynamics. 
\begin{defi}[Symbolic pseudo-monoid]\index{Symbolic pseudo-monoid}
Let $\hat \sA$ be a countable alphabet and let $\sR\subset \hat \sA^{(\N)}$ be a set of words in the alphabet $\hat \sA$. Then the set  $\sR$ endowed with the concatenation law is pseudo-monoid:  
Given two words $\sa, \sb \in    \sR$, we denote by $\sa\cdot \sb$ the concatenation of these two words, and define $D(\cdot):=    \{(\sa, \sb )\in    \sR^2: \sa\cdot \sb\in \sR\}$. The pair $(\sR, \cdot)$ is a pseudo-monoid with neutral element $\se$ the empty word. 
\end{defi}
The symbolic pseudo-monoids will be useful to do rigorously the parameter selection. They will encode some pieces of any 0-strongly regular Hénon-like endomorphism using an homomorphism:
\begin{defi}[Homomorphism of pseudo-monoids]
A \emph{homomorphism} \index{homomorphism   of pseudo-monoid}  from a pseudo-monoid $(\sM, \otimes)$ to a pseudo-monoid $(\sM',\otimes')$
 is a map  $\phi: \sM\to \sM'$ which satisfies:
\[\phi(\sa\otimes \sb)= \phi(\sa)\otimes' \phi(\sb)\qand \phi(\se)=\se'\; .\]
\end{defi} 
The homomorphism which will associate a piece to a symbolic word will be injective. 

A simpler example of (non-injective) homomorphism is the following. For $i\in \{1,2\}$, the map $\mathcal P_i\to \N$ which associates to a piece its order is a homomorphism of pseudo-monoid. This actually endows these pseudo-monoids with a structure of graded pseudo-monoid:
\begin{defi}[Graded pseudo-monoid] A pseudo-monoid $(\sM, \otimes) $ is \emph{graded} if it is endowed with a homomorphism to $(\N, +)$:\index{Graded pseudo-monoid}
$$\sa\in (\sM, \otimes) \mapsto n_\sa\in (\N, +).$$ \index{Graded pseudo-monoid}
A homomorphism of graded pseudo-monoid $\phi:(\sM, \otimes,n)\to (\sM, \otimes',n')$ is a homomorphism of pseudo-monoid which leaves invariant the order: $n'_{\phi(\sa)}=n_\sa$ for every $\sa\in \sM$. 
\end{defi}
This algebraic structure of the symbolic pseudo-monoid will be enriched with more structure which will take care of the parabolic operations $\boxdot_+$ and $\boxdot_+$.
 This will be shed light using an order relation called the \emph{right divisibility} $|$, which will be crucial for the parameter selection. The order relation will enable us to define in \textsection \ref{section righ divisibility} an ultra-metric distance on $\sA^{(\N)}$ (using a notion of GCD on $\hat \sA^{(\N)}$). This distance will bound from above the distance of two horizontal curves by using only their combinatorial definitions. Let us precise the definition of order relation in this setting:
\begin{defi}[Ordered pseudo-monoid]
A  pseudo-monoid $(\sM, \otimes) $ is \emph{ordered} if it is endowed with an order relation $\preceq$ satisfying the following property:\index{Ordered pseudo-monoid}

If  $(\sa, \sb, \sc)\in \sM^3$ is such that $\sa\preceq \sb$ and $(\sa, \sc), (\sb, \sc)\in D(\otimes)$ then $\sa\otimes \sc\preceq \sb \otimes \sc$. \index{Ordered pseudo-monoid}\end{defi}
Although the order relation on the symbolic monoid will be defined combinatorially, its restriction to the words associated to pieces
 will be equivalent to the following:
\begin{exem}
The pseudo-monoid $(\mathcal P_2, \star)$ is ordered by:
\[(Y',n')\preceq (Y,n)\Leftrightarrow (n'\le n\qand  f^{n-n'}(Y)\subset Y')\; .\]
\end{exem}

\subsection{The combinatorial rules}\label{grammar}
In the previous sections, we defined pieces for every symbol in $\sA_0$, and three operations $\star$, $\boxdot_-$ and $\boxdot_+$ on the pieces. 
These operations will be used following combinatorial   rules in the definition of the strong regularity.  We are going to define an alphabet $\hat \sA$ and several subsets of $\hat \sA^{(\N)}$ to encode these combinatorial rules. 
The alphabet $\hat \sA$ does not depend on $\hat f$ but only on $M$; this is an object in the category of symbols. Such a way of presentation enables us to shed light on the combinatorial rules that we will use and to split the combinatorial part of the proof to the topological part.\footnote{In the earlier version of this text, this symbolic alphabet was presented right after the definition of the strong regularity. It is put here before thanks J.C. Yoccoz and M. Shishikura who suggested me to extract the pure combinatorial part of the argument. I also might have been influenced by \cite{Y15}, where the combinatorial structure was presented at the very beginning of his note.} 

In the next section, we will associate to some words $\sg \in \hat \sA^{(\N)}$ a two-dimensional piece $(Y_\sg, n_\sg)$ if  some purely topological and combinatorial conditions are satisfied by $f$. These topological and combinatorial  conditions will imply nice analytical and geometrical  properties of the pieces. 

An important point is that this symbolic setting enables us to define combinatorially the continuation of the piece $Y_\sg$ when the parameter vary, in such a way that the analytical estimate (e.g. expansion of vectors in $\chi_h$) are preserved when the parameter varies. It will be used in a crucial way in the parameter selection to define the set of pieces which will persist for some parameter intervals. 
\bigskip

The alphabet $\hat \sA$  is defined  as an increasing union $\hat \sA=\bigcup_{k\ge 0} \hat \sA_k$ of alphabets $\hat \sA_{k+1}\supset \hat \sA_k\supset \cdots \supset \hat \sA_0:= \sA_0$   defined inductively. We recall that:
\begin{eqnarray*}
\sY_0:= \{s_\pm^j: \pm\in \{-,+\}, 2\le j\le M-1\}\qand 
 \sA_0:= \sY_0\sqcup \{\ss_-^M, \ss_+^M\}\; .\end{eqnarray*}
 For $\ss\in \sA_0$, we put $n_{\ss}= j$ if $\ss= \ss_\pm^j$.

 The induction on $k$ will associate to each symbol $\sa\in \hat \sA_k$ an integer $n_\sa\ge 0$. Given $\sa_1 \cdots \sa_m\in \hat  \sA^{(\N)}_k$ we put:
 \[\sa_1 \cdots \sa_m\in \hat \sA^{(\N)}_k=\bigcup_{m\ge 0} \hat \sA_k^m  \mapsto \sum_j n_{\sa_j}\; .\]
  The empty word is denoted by $\se$ and belongs to $ \hat \sA^{(\N)}_k$ for every $k$. We put $n_\se=0$. It is the neutral element for the pseudo-monoids $ \hat  \sA^{(\N)}_k$  endowed with the concatenation rule $\cdot \; $.
The inductive definition of $(\hat \sA_k)_k$ needs a few definitions. 
\begin{defi}[Prime and complete words] 
Let $\sg\in \hat \sA_k^{(\N)}$. It is the concatenation of finitely many  letters $\sa_i\in \hat \sA_k$.   We say that $\sg$ is \emph{prime} if $\sa_i\notin \sA_0$ for every $i<m$. The word $\sg$ is \emph{complete} if $\sa_m$ belongs to $\sA_0$. \index{Prime word} \index{Complete word}\end{defi}
\begin{defi}[Prime decomposition]
A word $\sg\in \hat \sA_k^{(\N)}$ has \emph{depth} $j$ if it contains exactly $j$ letters in $\sA_0$. \index{Depth of a word} If $\sg$ is complete then there exists a unique chain of prime, complete words $ \sg_1, \dots, \sg_j\in\hat  \sA_k^{(\N)}$ such that   $\sg= \sg_1\cdots \sg_j$. This is called the \emph{prime decomposition} of $\sg$. \index{Prime decomposition of a word}
\end{defi}

 We notice that a word which is prime has depth at most $1$, and a word which is complete has depth at least $1$.

Let us denote by $\ss_-^{\odot k}$ the concatenation of $k$-symbols $\ss_-^2$: $\ss_-^{\odot k} := \ss_-^2 \cdots  \ss_-^2\; .$
Given a word $\sg\in \hat \sA_k^{(\N)}$, we put 
$$\aleph (\sg):= \left\lfloor\frac{n_{\sg}}{12}+\frac M {24}\right\rfloor\; .$$

\begin{defi}[Regular words]\label{def regular words} A word  $\sg=\sa_1\cdots \sa_m \in \hat \sA_k^{(\N)}$ is \emph{regular} if the letter  $\sa_1$ belongs to $\sA_0$ and for every $2\le j\le m$ it holds:\index{Regular word}
\begin{equation}\label{regular ineq}n_{\sa_j}\le 2^{M}\sum_{m=1}^{j-1}n_{\sa_m}\; . \end{equation}
 A word  $\sg=\sa_1\cdots \sa_m \in \hat \sA_k^{(\N)}$ is \emph{weakly-regular} if the first letter satisfies  $n_{\sa_1}\le M\cdot 2^M$ and the same inequality \eqref{regular ineq}  holds true for  every $2\le j\le m$.
 \index{Regular word}
\end{defi}
We notice that a chain which is regular is weakly regular. Moreover, if we erase the first letter of a regular word, then the resulting word is weakly regular.

\begin{defi}[Strongly regular words]\index{Strongly regular words}\label{defi Strongly regular words} A word  $\sg \in \hat \sA_k^{(\N)}$ with prime decomposition $\sg=\sg_1\cdots \sg_m$ is \emph{strongly regular}  if it is complete and:\index{Strongly regular word}
\[\sum_{i\le  j: \sg_i\notin \sY_0} n_{\sg_i}\le 2^{-\sqrt{M}} \sum_{i\le  j} n_{\sg_i}, \quad \forall j\le m\; .\]\end{defi}
\begin{prop}\label{ordre SR}If $\sg$ is strongly regular with prime decomposition  $\sg=\sg_1\cdots \sg_m$ then for every $i\le m$, it holds:
\begin{equation*} n_{\sg_i}\le \max \left(2^{-\sqrt{M}}(1-2^{-\sqrt{M}})^{-1}n_{\sg_1\cdots \sg_{i-1}}, M-1\right) \qand n_{\sg}\le m\cdot ( M-1)(1+2^{1-\sqrt M}) .\end{equation*}
In particular a strongly regular word is regular. 
\end{prop}
\begin{proof}
The left inequality is a consequence of the bound $n_\sa\le M-1$ for every $\sa\in \sY_0$. To prove the second inequality, we consider the $J:= \{i\in \{1,\dots, m\}: \sg_i\in \sY_0\}$.  Then we have:
\[n_{\sg}\le (M-1)\cdot card( J)+  2^{-\sqrt{M}}(1-2^{-\sqrt{M}})^{-1} \cdot (M-1)\cdot card( J)\]
Then we conclude by bounding $ card( J)$ by $m$ and $ 2^{-\sqrt{M}}(1-2^{-\sqrt{M}})^{-1}$ by  $2^{1-\sqrt{M}}$.
\end{proof}

\begin{defi}[Common words]
A word  $\sg=\sa_1\cdots \sa_m \in \hat \sA_k^{(\N)}$ is \emph{common} if for every $0\le j\le m$ with $\sa_j\in \sA_0$ and $m':=\aleph (\sa_1\cdots \sa_j)$ it holds: 
\[\sa_{j+1}\cdots \sa_{j+m'}\neq  \ss_-^{\odot m'}
\text{ if }m\ge m'+j
\qand \sa_{j+1}\cdots \sa_{j+m'+1}\neq  \ss_+\cdot \ss_-^{\odot m'}\text{ if }m\ge m'+1+j
.\]\index{Common word}
\end{defi}
We are now ready to define the alphabet. \index{Alphabet $\hat \sA$}\index{Alphabet $\hat \sA_k$}
\begin{defi}[{$\hat \sA_k$} and $\hat \sA$]\label{sAk} We recall that $\hat \sA_0:= \sA_0$. Let $k\ge 1$.
The alphabet $\hat \sA_{k}$ is the union of $\hat \sA_{k-1}$ with the symbols $\boxdot_- (\sg'-\sg)$ and $\boxdot_+ (\sg'-\sg)$ among all the pairs of complete words $(\sg',\sg)\in (\sA_{k-1}^{(\N)})^2$ such that $\sg$ is strongly regular, common and of depth at most $k$ whereas $\sg=\sg'\cdot \sg''$ with $\sg''$ prime. Put:
\[n_{\boxdot_- (\sg'-\sg)}=n_{\boxdot_+ (\sg'-\sg)}=M+1+n_{\sg'}\; .\]
\end{defi}
We notice that $\sg'$ is also strongly regular and common; its depth is at most $k-1$. 
\begin{exem} Let us give the explicit expression of the first alphabets:
\begin{itemize}
\item  $\hat \sA_0= \sA_0$,
\item  $\hat \sA_1= \sA_0\cup\{ \boxdot_\pm (\se -\ss): \ss\in \sY_0\}$,
\item  For every $1\le k\le 2^{\sqrt M}$, the strongly regular words of depth $k$ are made only by simple pieces. Thus we have: 
$$\hat \sA_k = \sA_0\cup\bigcup_{0\le j\le k-1}\{ \boxdot_\pm (\sg -\sg\cdot \ss): \ss\in \sY_0\text{ and } \sg\in \sY_0^j\}\; .$$
\end{itemize}
\end{exem}
\begin{rema}\label{finite} We notice that $\hat \sA_k$ is finite for every $k$, and so $\hat \sA$ is countable. Also $\{\sa\in \hat \sA: n_\sa\le m\}$ is finite for every $m\ge 0$.  
\end{rema}

The alphabet $\sA$ defined for the one-dimensional case is included in $\hat \sA$.  The inclusion is strict: we have $(\hat \sA_{k+1}\setminus \hat \sA_k)\cap \sA$ at most of cardinality $2$, whereas $(\hat \sA_{k+1}\setminus \hat \sA_k)_k$ grows super exponentially. Nevertheless, in dimension 2, one could show that the number of letters we will use in  $((\hat \sA_{k+1}\setminus \hat \sA_k)\cap \sA)_k$ will grow at most exponentially fast with a factor close to $1$.

\subsection{Puzzle structure in dimension 1}\label{subsection 3.3}
The following section aims to make the reader comfortable with our combinatorial formalism by reformulating the definition of strongly regular quadratic maps. This formalism will be used intensively in dimension 2. However, none of the following one-dimensional statements will be used to prove the main theorem.
Let $P(x)=x^2+a$ be a quadratic map such that the first return time $M$ of the critical value $a$ in $I_\se$ is large. 
\begin{defi}[Definitions of $\sA_k$ and $\sG_k$]
 We define by induction on $k\ge 0$, a subset 
$\sA_k\subset \hat \sA_k$ of letters and a subset    
 $\sG_k\subset \hat \sA_k^{(\N)}$ of words $\sg$ associated to a pieces $(I_\sg, n_\sg)$.
For $k=0$ the set $\sA_0$ was already defined. The set  $\sG_0:= \sA_0^{(\N)}$ is formed by the words $\sg$ either equal to the neutral symbol $\se$ or to the concatenation of $j$-letters in $\sA_0$. If $\sg=\se$, it is associated to the piece $(I_\se,0)$.  If $\sg=\sa_1\cdots \sa_j\in \sA^j_0$, it is associated to the puzzle piece:
\[(I_\sg, n_\sg):=(I_{\sa_1}, n_{\sa_1})\star \cdots \star (I_{\sa_m}, n_{\sa_m})\; .\]

\medskip
Let $k\ge 1$. The set $\sA_k$ is the union of $\sA_{k-1}$ with the set of  symbols $\sa\in \hat \sA_k$ of the form 
$\sa= \boxdot_\pm (\sb-\sc)$ such that $\sb, \sc\in \sG_{k-1}$ and the pair of pieces $(I_\sb, n_\sb)$ and $(I_\sc, n_\sc)$ is admissible for the parabolic product. We put $(I_\sa, n_\sa):= \boxdot_\pm (I_\sb-I_\sc)$. 

A word $\sg\in \sA_{k}^{(\N)}$ belongs to $\sG_{k}$  if $\sg=\sg_1\cdots \sg_m$ is a concatenation of letters in $\sA_k$ and such that  the following $\star$ product is admissible:
\[(I_{\sg_1},n_{\sg_1}) \star  \cdots \star (I_{\sg_m},n_{\sg_m}) \]
Note that $(\sG_k)_k$ is increasing. We put $\sG:= \bigcup_{k} \sG_k$ and $\sA:= \bigcup_k \sA_k$. 
\end{defi}
We notice that the sets $\sG_k$ and $\sA_k$ can be defined even if $P$ is \emph{not} strongly regular.\medskip

The following sheds light on the correspondence between the combinatorial properties of the words in $\sG$ and those of their associated pieces.
\begin{prop}\label{primecomple=puzzle}\label{primeword=prime piece}
A word $\sg\in \sG$ is complete iff $(I_{\sg}, n_{\sg})$ is a puzzle piece. 
A complete word  $\sg\in \sG$ is prime  iff the puzzle piece  $(I_{\sg}, n_{\sg})$ is prime. 
\end{prop}
Before proving this proposition, let us state an immediate consequence of it.
\begin{coro}\label{coro correspondance formalism}
For every word $\sg\in \sG$, it holds:
\begin{itemize}
\item the piece $(I_\sg, n_\sg)$ is a common puzzle piece iff the word $\sg $ is common and complete.
\item the piece $(I_\sg, n_\sg)$ is a strongly regular puzzle piece iff  $\sg $ is strongly regular.
\item If $\sg$ is complete, the depth of $(I_\sg, n_\sg)$ is equal to the depth of $\sg$.\end{itemize}
 \end{coro}
 Another consequence is:
 \begin{coro}[Alternative definition of strongly regular quadratic maps]\label{def SR quadratic 2}
A quadratic map $P$ is strongly regular iff there exists $(\sc_k)_{k\ge 0}\in \sG^\N$  such that:
\begin{enumerate}[($SR_1$)]
\item  $P^{n_\boxdot}(0)$ belongs to $\bigcap_k I_{\sc_k}$. 
\item The word  $\sc_k$ is strongly regular, for every $k$. 
\item The word  $\sc_k$ is common with depth $k$, for every $k$.
\end{enumerate}
\end{coro}
\begin{proof}
By \cref{coro correspondance formalism} this reformation implies obviously the initial \cref{defSRquadra} of  strongly regular quadratic maps.  To show the other direction, using again \cref{coro correspondance formalism}, it suffices to show the next lemma.\end{proof}
\begin{lemm}\label{puzzle piece of sr quadratic map}
If $P$ is strongly regular, then for every puzzle piece $(I,n)$ in $I_\se$, there exists $\sg\in \sG$ such that $(I,n)=(I_\sg, n_\sg)$.\end{lemm} 
\begin{proof} We proceed by induction on $n$. If $I$ intersects the interior of $I_\sa$ for $\sa\in \sA_0$, then $I$ is included in $I_\sa$ by the nested property of puzzle pieces. Then by the induction hypothesis, the puzzle piece $(f^{n_\sa}(I), n-n_\sa)$ equals  $(I_{\sg'}, n_{\sg'})$ for $\sg'\in \sG$ and so $(I,n)=(I_{\sg}, n_\sg)$ with $\sg= \sa\cdot \sg'\in \sG$.  Otherwise, $I$ is included $I_\boxdot$. It cannot contain $0$, so it is included in $I_\boxdot\cap \R^\pm$ with $\pm\in \{-,+\}$. Then $(I', n'):=(f^{n_\boxdot}(I), n_\boxdot)$ is a puzzle piece included in $I_\se$. 
 There exists $k$ maximal such that $I'$ is included in $I^{(k)}$. By maximality of $k$ and the nested property of the puzzle pieces, it comes that $I'$ is included in $I^{(k)}\setminus int \,  I^{(k+1)}$. 
 Then $(I, n)$ is of the form $(\boxdot_\pm (I^{(k)}\setminus  I^{(k+1)}), n_\boxdot + n_{I^{(k)}})\star (I'', n'')$, 
 with $(I'', n'')$ a puzzle piece of order $n''<n$. By induction $(I'', n'')=(I_{\sg''}, n_{\sg''})$ and by using \cref{coro correspondance formalism}, 
 it holds $(I^{(k)},  n_{I^{(k)}})= (I_{\sc_k}, n_{\sc_k})$
  for a word $\sc_k\in \sG$ which is strongly regular and common of depth $k$. Then we use sub-lemma \ref{lemma:pp and prime word implies prime pp}, which implies that   $(I^{(k)},  n_{I^{(k)}})= (I_{\sc_{k+1}}, n_{\sc_{k+1}})$
  with $\sc_{k+1}\in \sG$. Again  \cref{coro correspondance formalism}, implies that $\sc_{k+1}$ is strongly regular and common of depth $k+1$. Thus $(\boxdot_\pm (I^{(k)}\setminus  I^{(k+1)}), n_\boxdot + n_{I^{(k)}})=(I_\sa, n_\sa)$ with $\sa= \boxdot(\sc_k-\sc_{k+1})\in \sA$ and so $(I,n)=(I_{\sa\cdot \sg''},n_{\sa\cdot \sg''})$ with $\sa\cdot \sg''\in \sG$. 
\end{proof}
A first interest of this reformulation is that it will be generalizable in dimension 2. Another interest of such a formulation (and of  the combinatorial formalism)  is its ability to study the combinatorial structure of strongly regular maps. From such a study, we will deduce analytic estimates on the pieces. Indeed:  
\begin{prop}\label{para are hyp}
Let $\sa\in \sA$. Then it holds:
\begin{equation*}
\forall z\in I_\sa\text{ and }j\le n_\sp: \\
|D P ^{n_\sa}(z)| \ge 2^{\frac{n_\sa-j}{3} } |D P^j(z)|\; .
\end{equation*}
\end{prop}
\begin{proof}
The case where $\sa\in \sY_0$ is given by \cref{simple are hyp}. The case $\sa\in \sA_0\setminus \sY_0$ is given by \cref{coroyoc3}. The case $\sa\in \sA\setminus \sA_0$ will be stated  in a more general form in \cref{Crutial prop}.
\end{proof}
An immediate consequence is:
\begin{coro}\label{preAspelling}
For every $\sg\in \sG$, the piece $(I_\sg, n_\sg)$ satisfies:
\begin{equation*}
\forall z\in I_\sg\text{ and }j\le n_\sp: \\
|D P ^{n_\sg}(z)| \ge 2^{\frac{n_\sg-j}{3} } |D P^j(z)|\; .
\end{equation*}
\end{coro}
Together with \cref{puzzle piece of sr quadratic map} this implies:
\begin{coro}\label{Coroh} 
If $P$ is strongly regular, then every puzzle piece $(I,n)$  it holds:
\[\forall z\in I\text{ and }j\le n: \\
|D_z P ^{n}(z)| \ge 2^{\frac{n-j}{3} } |Dz P^j(z)|\; .\]
\end{coro}
\begin{coro}\label{coro Pesin}
Let $\mathcal {SR}$ be the set of points in $I_\se$ which belongs to the segment of a  strongly regular piece of arbitrarily large order. Then it holds:
\[\liminf_{n\to \infty} \frac1n \log \| D_x P^n\|\ge \frac{\log2}3-2^{1-\sqrt M}\log4\; .\]  
\end{coro}
\begin{proof}
Let $x\in \mathcal {SR}$ and $(I'_j, n'_j)_j$ be a sequence of puzzle pieces such that $x$ belongs to 
$I'^{(j)}$ with $(I'^{(j)}, n'^{(j)})= (I'_1, n'_1)\star \cdots \star (I'_j, n'_j)$. By \cref{Coroh} it holds:
\[\left\{ \begin{array}{c}
\log \|D P^{n'_j}(P^{n'^{(j-1)}}(x)) \|\ge \frac{\log2}3 n'_{j}\, ,\\
{\log \|D P^{n'^{(j)}}(x) \|}\ge \frac{\log2}3{n'^{(j)}}\, .\end{array}\right.\]
As $DP\le 4$ along the orbit of $x$, the first inequality implies that $ \log \|D P^{m}(x') \|\ge - 2 n'_{i}$ for every $m\le n'_{i}$, $x'\in I'_{i}$ and $i\ge 0$. Finally we infer that the chain $(I'_i, n'_i)_i$ is strongly regular and so for $i$ sufficiently large it holds ${n'_{i}}\le 2^{-\sqrt M}n'^{(i)}$. Putting together these inequalities, we obtain:
\begin{multline*}{\log \|D P^{n'^{(i-1)}+ m}(x) \|}
\ge  \frac{\log2}3 n'^{(i-1)}-  \log4\cdot  n_i' \\
\ge  \frac{\log2}3 n'^{(i-1)}-2^{-\sqrt M} \cdot \log4\cdot  n'^{(i)} 
\ge  (\frac{\log2}3 -2^{-\sqrt M+1}\log4  )(n'^{(i-1)}+m)
 \; .\end{multline*}
\end{proof}
This implies as for Yoccoz' strong regularity definition:
\begin{coro}\label{CEP} 
If $P$ is strongly regular, then it satisfies the Collet-Eckmann  Condition $(\mathcal {CE})$. 
\end{coro}
The combinatorial formalism enables also to define immediatly the $k$-combinatorial interval, which are parameter intervals on which the symbolic set $\sG_k$ is constant, and so for which the pieces persist (with the same analytical estimate by \cref{preAspelling}). Really, an analytic definition of such parameter intervals turned out to be more tricky to state in dimension 2. Moreover, using mostly combinatorial arguments, we will evaluate the Lebesgue measure of the union of the parameter intervals for which $\sG_k$ is given by a  `$k$-strongly regular map', from which we will deduce that strongly regular maps are abundant:  
\begin{theo}
There exists $\Lambda_0\subset \R$ such that for every $a\in \Lambda_0$, the map $P_a$ is strongly regular and $\leb\, \Lambda_0>0$.
\end{theo}
This theorem is a consequence of Theorem \ref{SRabundant} stated below.
\begin{proof}[Proof of \cref{primecomple=puzzle}]
There are four implications to show. Two of them are given by lemmas \ref{pp implies complete} and \ref{pp and prime word implies prime pp}. These lemmas will be used to prove the two remaining implications in lemmas \ref{complete word implies puzzle piece} and \ref{last1dlemma}. 
\begin{lemm}\label{pp implies complete}
For every $\sg\in \sG$ if $(I_{\sg}, n_{\sg})$ is a puzzle piece, then $\sg$ is complete.
\end{lemm}
\begin{proof}
If $\sg=\sb_1\cdots \sb_k$ is not complete, then its last letter $\sb_k\in\sA$ is parabolic. Thus the piece $(I_{\sb_k}, n_{\sb_k})$ is parabolic and $P^{n_\sg}(I_\sg)\subset P^{n_{\sb_k}}(I_{\sb_k})$ is strictly included in $I_\se$ by fact \ref{primeend}. Thus $(I_\sg, n_\sg)$ is not a puzzle piece.
\end{proof}
\begin{lemm}\label{pp and prime word implies prime pp}
If $\sg\in \sG$ is such that $(I_\sg, n_\sg)$ is a puzzle piece and $\sg$ is prime, then $(I_\sg, n_\sg)$ is a prime puzzle piece.
\end{lemm}
\begin{proof}
If $(I_\sg, n_\sg)$ is not prime, there exist two puzzle pieces 
$(I,n)$ and $(I',n')$ such that   $(I_\sg, n_\sg)=(I,n)\star (I',n')$, with $I\neq I_\se\neq I''$. As the puzzle pieces are nested or disjoint,  it holds $I_\sg\subset I$. Then we apply the following lemma proved below. 
\begin{sublemm}\label{lemma:pp and prime word implies prime pp}
If  $(I,n)$ is a puzzle piece such that $I\supset I_\sg$ with  $\sg\in \sG$ complete, then there exists $\sg'\in \sG$ such that $(I,n)=(I_{\sg'}, n_{\sg'})$ and there exists $\sg''\in \sG$ such that $\sg= \sg'\cdot \sg''$. 
\end{sublemm}
By \cref{pp implies complete}, the word $\sg'$ is complete, and so $\sg$ is not prime. 
\end{proof}
\begin{proof}[Proof of sublemma \ref{lemma:pp and prime word implies prime pp}]
We assume that $(I,n)$ is not the neutral piece $(I_\se, 0)$ (since otherwise the lemma is trivial).  We proceed by induction on the number $k$ of $\sA$-letters of $\sg=\sa_1\cdots \sa_k$.  If $k=1$, then $\sg\in \sA_0$ and so $(I_\sg, n_\sg)$ is prime. Thus $(I,n)=(I_\sg, n_\sg)$. Let $k>1$. If $\sa_1\in \sA_0$, then $I \  I_{\sa_1}$ because $(I_{\sa_1}, n_{\sa_1})$ is a  prime puzzle piece. Then $(I,n) = (I_{\sa_1}, n_{\sa_1})\star (I', n')$ with $I'\supset I_{\sa_2\cdots \sa_k}$ and we conclude by induction.   If $\sa_1=\boxdot_\pm(\sb- \sc)$, then $I$ intersects the interior of $I_\boxdot$ and so $I\subset I_\boxdot\setminus \{0\}$. Thus $n\ge n_\boxdot$ and $(P^{n_\boxdot}(I), n-n_\boxdot)$ is still a puzzle piece. Also $P^{n_\boxdot}(I)$ must intersect the interior  of $I_\sb\setminus I_{\sc}$. By definition of admissible parabolic product, $I_\sb$ and $I_\sc$ are segments of puzzle pieces. Thus each of them is nested or disjoint from $P^{n_\boxdot}(int\, I)$. But  $P^{n_\boxdot}(I)$ cannot contain  $I_\sb$ nor $I_\sc$ because otherwise $I$ would contain $0$. Thus  $P^{n_\boxdot}(I)$ is included in $I_\sb\setminus I_{ \sc}$ and so $I$ is included in $I_{\sa_1}$; we conclude then by induction as above. 
\end{proof}
A third implication stated in \cref{primeword=prime piece} is given by the following.
\begin{lemm}\label{complete word implies puzzle piece}
If a word $\sg\in \sG$ is complete then $(I_{\sg}, n_{\sg})$ is a puzzle piece. 
\end{lemm}
\begin{proof} 
Put $\sg= \sb_1\cdots \sb_k$ with $\sb_i\in \sA$.  Let us show by induction on $k$ that $(I_\sg, n_\sg)$ is a puzzle piece.  
  If $k=1$, then  $\sg=\sb_k=\sb_1$ belongs to $\sA_0$, and so $(I_{\sg}, n_{\sg})$ is a puzzle piece.
Assume that $k\ge 2$ and put  $\sg':=\sb_2\cdots \sb_k$. 
Observe that $\sg':=\sb_2\cdots \sb_k$ is in $\sG$ and 
 $(I_\sg, n_\sg)= (I_{\sb_1}, n_{\sb_1})\star (I_{\sg'}, n_{\sg'})$ by associativity of the $\star$-product. 
 As $\sg'$ is complete,  by induction, the piece $(I_{\sg'}, n_{\sg'})$ is a puzzle piece. If $\sb_1$ is in $\sA_0$, then $(I_\sg, n_\sg)$ is a $\star$-product of two puzzle piece, and so  by \cref{rema puzzle dim 1}, it is a puzzle piece. 
 
 If $\sb_1\notin \sA_0$, then it is of the form $\sb_1=\boxdot_\pm(\sc-\sc')$, with $\sc, \sc'\in \sG$ such that 
 $(I_\sc, n_\sc)$ and $(I_{\sc'}, n_{\sc'})$ are puzzle pieces with $I_\sc\supset I_{\sc'}$. By definition $\hat \sA$, it holds that $\sc'=\sc\cdot \sc''$, with $\sc''$ prime. 
 We notice that $(P^{n_{\sc}}(I_{\sc'}), n_{\sc'}-n_{\sc})=(P^{n_{\sc}}(I_{\sc'}), n_{\sc''})$ is a puzzle piece which is included in $(I_{\sc''}, n_{\sc''})$. Thus $(I_{\sc''}, n_{\sc''})$ is a puzzle piece.  By \cref{pp and prime word implies prime pp}, the puzzle piece $(I_{\sc''}, n_{\sc''})$ is prime. As  the puzzle pieces $(I_{\sc''}, n_{\sc''})$ and $(P^{n_{\sc}}(I_{\sc'}), n_{\sc''})$ are nested and have the same order, they are equal: $I_{\sc''}=P^{n_{\sc}}(I_{\sc'})$.
   By Fact \ref{primeend}, the segment  $I^{\sb_1}:=P^{n_{\sb_1}}(I_{\sb_1})$ is a component of $I_\se\setminus P^{n_\sc}(int\, I_{\sc'})= I_\se\setminus I_{\sc''}$.  By admissibility, the segment  $I^{\sb_1}$ intersects the interior of the puzzle piece $I_{\sg'}$. Thus $I_{\sg'}$ cannot be included in $I_{\sc''}$ and since the latter is prime, $I_{\sg'}$ cannot contained strictly $I_{\sc''}$. Thus $I_{\sg'}$ is disjoint to the interior of $I_{\sc''}$ and  so $I_{\sg'}$ is included in  $I^{\sb_1}$. This implies that  
$(I_\sg, n_\sg)$ is a puzzle piece.
\end{proof}
The next lemma proves the fourth and the last implication of \cref{primeword=prime piece}. \end{proof}
\begin{lemm}\label{last1dlemma}
A complete word $\sg\in \sG$ is prime if  $(I_\sg, n_\sg)$ is prime. 
\end{lemm}
\begin{proof}
If $\sg\in \sG$ is not prime, then there are complete words  $\sg'$ and $\sg''$ in $\sG$ such that $\sg=\sg'\cdot \sg''$.  By \cref{complete word implies puzzle piece}, $(I_{\sg'}, n_{\sg'})$ and $(I_{\sg''}, n_{\sg''})$ are  non-trivial puzzle pieces. Note that  $(I_{\sg}, n_{\sg})=(I_{\sg'}, n_{\sg'})\star (I_{\sg''}, n_{\sg''})$ is not prime. 
\end{proof}

We recall that $\mathcal P_1$ denotes the set of the pieces for the quadratic map $P$. For the sake of completeness let us finish this section by the following:
\begin{prop}\label{reciavenir}
The map $ \sg \in \sG\mapsto (I_\sg, n_\sg)$ is an injectif  homomorphism of graded pseudo-monoid  from $(\sG, \cdot)$ into
$(\mathcal P_1, \star)$.
\end{prop} 
\begin{proof}
We already noticed that $ \sg \in \sG\mapsto (I_\sg, n_\sg)$ is homomorphism of graded pseudo-monoid  from $(\sG, \cdot)$ into
$(\mathcal P_1, \star)$. To show its injectiveness by induction, it suffices to show that for every $\sa\neq \sb\in \sA$, it holds $int(I_\sa)\cap  int(I_\sb)=\varnothing$. If $\sa$ or $\sb$ belong to $\sA_0$ it is clear. Otherwise, both are of the form $\boxdot_\pm (\sc_k-\sc_{k+1})$ and  $\boxdot_{\pm'} (\sc_j-\sc_{j+1})$, where $\sc_m$ denotes the prime piece of depth $m$ which contains $P^{n_\boxdot}(0)$. As
 $I_\sa$ intersects $int\, I_\sb$, they must be on the same side of $0$ and so $\pm'=\pm$. Furthermore, the image by $P^{n_\boxdot}$ of the interior of these segments must intersect each other. Thus $I_{\sc_k}-I_{\sc_{k+1}}$ intersects $I_{\sc_j}-I_{\sc_{j+1}}$. This implies that $j=k$ and so $\sa=\sb$. 
\end{proof}
\subsection{Puzzle struture and strongly regular endomorphisms in dimension 2}\label{section Strong regularity in dimension 2}
Let $f$ be $0$-strongly regular (see \textsection\ref{setting}). The strong regularity condition on $f$ will be stated using the following  subset $\sR\subset \hat \sA^{(\N)}$ and its action onto the spaces of graph transforms and pieces. 

\begin{defi}[Definition of $\sR_k$ and $\tilde \sR_k$]\label{admissibleword}\index{$\sR_k$, $\tilde \sR_k$}
\index{Admissible regular words}\index{Admissible weakly regular words}
 We define by induction on $k\ge 0$, two subsets $\sR_k$ and $\tilde \sR_k\subset \hat \sA^{(\N)}$ of words $\sg$ associated to a piece $(Y_\sg, n_\sg)$ and an extension algorithm $T_\sg$. By definition, $T_\sg$  satisfies  $f^{n_\sg}(S\cap Y_\sg)\subset  T_\sg (S)\in \cH$ for every $S\in \mathcal H$.
\medskip

Let $\tilde \sR_0=\sR_0:= \sA_0^{(\N)}$. Every word $\sg\in \sR_0$ is either the neutral letter $\sa=\se$ or a product of $j$-letters in $\sA_0$ for $j\ge 1$. In \cref{productsimplepiece} p.\pageref{simplespieceindim2}, we defined the puzzle piece $(Y_{\sg}, n_{\sg})$ and in \cref{productsimplepiece2}, its extension algorithm  $T_\sg : S\in \mathcal H \mapsto  f^{n_\sg}(Y_\sg\cap S)\in \mathcal H$.
 \medskip

Let $k\ge 1$. A word $\sg\in \hat \sA_{k}^{(\N)}$ belongs to $\sR_{k}$ (resp.  $\tilde \sR_{k}$) if it is regular (resp. weakly regular) and one of the two following conditions hold true:
\begin{enumerate}[$(i)$]
\item $\sg=\sg_1\cdots \sg_m$ is a concatenation of words $\sg_1, \sg_2, ..., \sg_m$ in $\sR_{k-1}$ (resp. in $\tilde \sR_{k-1}$) such that if $m\ge 2$ the following $\star$-product is admissible:
\[(Y_{\sg_1},n_{\sg_1}) \star  \cdots \star (Y_{\sg_m},n_{\sg_m})\; . \]
We denote by $(Y_{\sg},n_{\sg})$ its product. Put $T_\sg:= T_{\sg_m}\circ \cdots \circ T_{\sg_1}$. 
\item $\sg= \sa \boxdot_{\pm}(\sb-\sc)$  with $\pm\in \{-,+\}$ and where $\sa, \sb, \sc$ are words in $\sR_{k-1}$ (resp. $\sa \in \tilde \sR_{k-1}$), the letter $\boxdot_{\pm}(\sb-\sc)$ belongs to $\hat \sA_{k}$, the pieces  $(Y_\sa, n_\sa), (Y_\sb, n_\sb), (Y_\sc, n_\sc)$ and the extension algorithm  $T_\sa$ form a 4-tuple  (resp. weakly) admissible for the parabolic product $\boxdot_\pm$.
We denote $(Y_{\sg},n_{\sg}):= (Y_\sa \boxdot_\pm (Y_\sb-Y_{\sc}), n_\sa+n_\boxdot+n_\sb)$ and put $T_{\sg }:= T_{\boxdot_\pm(Y_\sb-Y_\sc)} \circ T_\sa$
(see prop. \ref{graphtransformpara}).

\end{enumerate}
Note that the sequences  $(\sR_k)_k$ and $(\tilde \sR_k)_k$ are increasing. Also 
$\sR_k\subset \tilde \sR_k$ for every $k$.  We put $\sR:= \bigcup_{k\ge 0} \sR_k$ and $\tilde \sR:= \bigcup_{k\ge 0} \tilde \sR_k$. \index{$\sR$, $\tilde \sR$}
\end{defi}
We recall that the set $\mathcal P_2$ of pieces endowed with the $\star$-product, the space of graph transforms $C^0(\mathcal H, \mathcal H)$ endowed with the composition rule $\circ$, and the set of symbols $\tilde \sR$ endowed with the concatenation $\cdot$ are pseudo-monoids.
A direct consequence of \cref{admissibleword} $(i)$ is:
\begin{fact}\label{fact:funct}
The maps $\sg\in (\tilde \sR, \cdot)\mapsto (Y_\sg, n_\sg)\in (\mathcal P_2, \star)$  and $\sg\in (\tilde \sR, \cdot)\mapsto T_\sg\in (C^0(\mathcal H, \mathcal H),\circ)$ are homomorphisms of pseudo-monoids.\end{fact}
\begin{rema}
By definition of $\sR_k$, for every word $\sg=\sa_1\cdots \sa_m\in \sR_k$, it holds $\sa_i\in \sA_{k-1}$ if $i<m$ and $\sa_m\in \sA_k$. By \cref{sAk} of $\sA_k$, if $\sa_i\notin \sA_0$ then $\sa_i$ is of the form $\sa_i= \boxdot_\pm (\sb-\sc)$ with $depth\,\sb  +1=depth\, \sc\le k$. If $i<k$, it holds furthermore   $depth \, \sc \le k-1$.
\end{rema}
Let us emphasis that the sets $\sR$ and $\tilde \sR$ are well defined for \emph{any} $0$-strongly regular map $f$. The set $\sR$ of regular words will be useful to define the strongly regular maps. For such maps, the set $\sR$ will be also useful to encode the dynamics of a set of positive Lebesgue measure of points.
The set $\tilde \sR$ of weakly regular words is technical: it will be useful only for the proofs. Here\footnote{In \cite{Berentropy}, it has been used to show the existence of a uniform lower bound $m>0$ on the Lyapunov exponent of every invariant measure of any strongly regular map of the form $(x,y)\mapsto (x^2+y+a, 0)+B(x,y)$.}
 it will deserve to show the abundance of strongly regular mappings. 

The following is the 2-dimensional counterpart of \cref{complete word implies puzzle piece}. It will shown in  \cref{rigidPP_proof} \cpageref{rigidPP_proof}.
\begin{prop}\label{rigidPP}  For every $\sg \in \sR$, the piece 
 $(Y_\sg, n_\sg)$ is a puzzle piece if and only if $\sg$ is complete. 
\end{prop}
\medskip

The definition of strong regularity is based on vertical and horizontal, stretched curves defined via the space $(\sR, \cdot)$. 
These curves are defined by taking limits of words in $\sR$. There are two kinds of limits possible: negative or positive. The negative limits will be used to define combinatorially horizontal, unstable Pesin manifolds.  The positive limits will be used to define  combinatorially stretched, vertical, local stable Pesin manifolds.  We will use both to define the strong regularity condition on the dynamics $f$, by asking a tangency condition between them. 
 \begin{defi}\index{$\overrightarrow
 \sR$} Let $\overrightarrow \sR$ be the set of 
positive one-sided sequences $\sc:= \sa_1\cdots \sa_m \cdots \in \hat \sA^\N$ which are the concatenation of infinitely many words in $\sR$: There exists an increasing sequence $(n_i)_i$ satisfying for every $i$: $\sa_1\cdots \sa_{n_i}\in \sR$ and $\sa_{1+n_i}\cdots \sa_{n_{i+1}}\in \sR$.
For every $\sc=\sa_1 \cdots \sa_{m} \cdots \in \overrightarrow \sR$, put
 $$W^s_\sc= \bigcap_{m\ge 0} Y_{\sa_1 \cdots \sa_{m}}\qand W^s_{\sr \cdot \sc}= \bigcap_{m\ge 0} Y_{\sr \cdot \sa_1 \cdots \sa_{m}}.$$\end{defi}
\begin{prop}
\label{Pesin stable}
For every $\sc\in \overrightarrow \sR$, the sets $W^s_\sc$ and $W^s_{\sr \cdot \sc}$ are vertical stretched curves and local Pesin stable manifolds which satisfy for every unit, tangent vector $u$:
$$  \|Df^n(u)\|\le   b^{n/2}\; ,\quad \forall n\ge 1\; .$$
\end{prop}
\begin{proof} 
As $\sc$ is in $\overrightarrow \sR$, there are infinitely many $m$ such that $\sa_1 \cdots \sa_{m} $ is complete and so such that $(Y_{\sa_1 \cdots \sa_{m} }, n_{\sa_1 \cdots \sa_{m} })$ is a puzzle piece by \cref{rigidPP}. Thus \cref{Pesin stable1}.4 implies the statement. 
\end{proof}
Conversely, we consider:
\begin{defi}[
Space  $\overleftarrow \sR$
]\label{def_arr_sR} \index{$\arr \sR$}
Let $\arr \sR$ be the set of negative one-sided sequences  $\st:= \cdots \sa_{-i} \cdots \sa_0\in \hat \sA^{\Z^-}$ which are the concatenation of infinitely many words in $\sR$:
There exists an increasing sequence $(n_i)_i$ satisfying for every $i$: $\sa_{1-n_i}\cdots \sa_{0}\in \sR$ and $\sa_{1-n_{i+1}}\cdots \sa_{-n_{i}}\in \sR$.
\end{defi}


The following proposition shows that the horizontal curves combinatorially encoded by  $\overleftarrow \sR$ are  Pesin local  unstable manifolds (see \cref{pesin} \cpageref{pesin}). 
\begin{prop}\label{arr Rk2Wu}
For every $\st=\cdots \sa_{-k}\cdots \sa_0 \in \overleftarrow  \sR$,  the set:
 $$\hat W_u^\st:= \bigcap_{\{k: \sa_{-k}\cdots  \sa_0 \in \sR\}} \{ T_{\sa_{-k}\cdots  \sa_0}(S): S\in \mathcal H\}$$
 is a horizontal stretched curve which extends  the following horizontal curve:
\[ W_u^\st:= \bigcap_{\{k: \sa_{-k}\cdots  \sa_0 \in \sR\}} f^{n_{\sa_{-k}\cdots  \sa_0}}(Y_{\sa_{-k}\cdots  \sa_0})
\subset \hat W_u^\st
 \; .\] 
Moreover, the following is a Pesin local unstable manifold:
$$\overleftarrow W_u^\st:=\{ (z_i)_{i\le 0}\in \overleftarrow M_f :  \exists k\le 0 \text{ arbitrarily large s.t. } \sg:=\sa_{k}\cdots \sa_0\in \sR\text{ and }z_{-n_\sg} \in Y_\sg\}\; .$$
and its $0$-coordinate projection of $\overleftarrow W_u^\st$ is equal to $W_u^\st$. 
\end{prop}
To prove this, we use the following generalization of  \cref{contractionPP} \cpageref{contractionPP}:
\begin{lemm} \label{contractingrate}
For every $\sg \in \tilde \sR$, the map  $T_\sg:S\in \mathcal H\mapsto S^\sg$ is $b^{n_\sg/3}$-contracting. 
\end{lemm}
\begin{proof} We prove this proposition by induction on $k$ such that $\sg\in \tilde \sR_k$. By \cref{contractionPP}, for every $\sg\in \sR_0$, the map $T_\sg$ is $b^{n_\sg/3}$-contracting. 
The step $k\to k+1$ is an immediate consequence of the induction  in the case $(i)$ of the definition of $\sR_{k+1}$. 
It is given by \cref{graphtransformpara} \cpageref{graphtransformpara} for the case $(ii)$. 
\end{proof}
\begin{proof}[Proof of \cref{arr Rk2Wu}]
By definition of $\arr \sR$, there exists   $(\sg_{-m})_{m\ge 0}\in \sR^\N$ 
such that $\st = \cdots \sg_{-m} \cdots \sg_{0}$. 
By \cref{contractingrate}, the set $\{ S^{{\sg_{-m}\cdots \sg_0}}: {S\in \mathcal H}\}$ is the image of $\mathcal H$ by a composition of contractions of factor $b^{n_{\sg_{-m}}/3 }\cdots b^{n_{\sg_{0}}/3 }$. By compactness of $\cH$, the subset  
$\{ S^{{\sg_{-m}\cdots \sg_0}}: {S\in \mathcal H}\}$ is compact in $\cH$. Its diameter is smaller than $2\theta \cdot b^{n_{\sg_{-m}\cdots \sg_0}/3}$, and decreases with $m\ge 0$. Hence the intersection $\bigcap_{m\ge 0}\{ S^{{\sg_{-m}\cdots \sg_0}}: {S\in \mathcal H}\}$ consists of a unique horizontal stretched curve $\hat W_u^{\st}$.
We notice that:
$$  W_u^\st = \bigcap_{m\ge 0} f^{n_{\sg_{-m}\cdots \sg_0}}(Y_{\sg_{-m}\cdots \sg_0})\subset \bigcap_{m\ge 0} \bigcup_{S\in \mathcal H} S^{\sg_{-m}\cdots \sg_0}=\hat W_u^\st\; .$$

By the regularity, the sequence $\st=\cdots \sa_i\cdots \sa_0$ displays infinitely many letters in $\sA_0$. Thus by $\arr f$-invariance of the set of Pesin unstable manifolds, we can assume that $\sa_0$ is in $\sA_0$. Then $\sg_{-i}\cdots \sg_0$ is complete for every $i$. Thus by \cref{rigidPP}, the pair $(Y_i, n_i)=(Y_{ \sg_{-i}\cdots \sg_0}, n_{\sg_{-i}\cdots \sg_0})$ is a puzzle piece. Furthermore its holds $f^{n_{i+1}-n_i}(Y_{i+1})\subset Y_i$ for every $i$. Thus  we can apply \cref{prop pesin instable gen} which gives the last statement of the proposition.  
\end{proof} 

\begin{defi}[Strongly regular Hénon-like map]\label{defSRHL}\index{Strong regular Hénon-like map}
The map $f$ is strongly regular if 
for every $\st\in \overleftarrow \sR$, 
there exists $\sc^\st = \sa_0^\st \cdots  \sa_m^\st \cdots \in  \overrightarrow \sR$  and $(m_k)_k$ such that:
\begin{enumerate}[($SR_1$)]
\item  The map $f^{n_\boxdot }|Y_\boxdot\cap \hat W_u^\st$
 is tangent to $W^s_{\sc^\st}$:
\begin{equation*}
\exists \zeta\in \hat W_u^\st\cap Y_\boxdot\text{ such that } \xi:= f^{n_\boxdot}(\zeta)\in W^s_{\sc^\st}\qand 
Df^{n_\boxdot} (T_\zeta  \hat W_u^\st) \subset T_{\xi} W^s_{\sc^\st}\; ,
\end{equation*}\index{Condition $(SR_1)$ in dim 2}
with $T_\zeta \hat W_u^\st$ and $T_{\xi} W_{\sc^\st}^s$ the tangent space of $\hat W_u^\st$ and $W_{\sc^\st}^s$ at $\zeta$ and $\xi$. 
\item The word $\sc_k:= \sa_0^\st \cdots  \sa_{m_k}^\st$ is strongly regular for every $k\ge 1$.
 \index{Condition $(SR_2)$ in dim 2}
\item The word  $\sc_k^\st$ is common with depth $k$ for every $k\ge 1$.
\end{enumerate}
\end{defi}

\begin{figure}[h!]
 \centering
 \includegraphics[width=10cm]{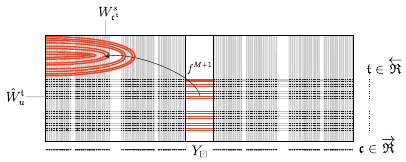}
 \caption{A map is strongly regular is some combinatorially defined Pesin stable and unstable manifolds are tangent.}
\label{firstidea}\end{figure}
Let us comment the axioms of this definition.

For every $\st \in \overleftarrow \sR$, the curve $\hat W_u^\st$  contains the canonical projection $W_u^\st$ of the Pesin local unstable manifold $\overleftarrow W_u^\st$. 
The tangency point $\xi\in \hat W_u^\st$ claimed in $(SR_1)$ does not belong necessarily to $W_u^\st$.

The tangency point $\xi$ of  $\hat W_u^\st$ given by $(SR_1)$ corresponds to the so-called critical point in other works on the dynamics of the Hénon maps \cite{BC2, MV93, WY01, Ta11}. An important point of the presented combinatorial formalism is that $\hat W_u^\st$ and $W^s_\st$ are here combinatorially and topologically defined, as we will see in the next subsection (in the all other works those are analytically defined).

Among one-dimensional dynamics, Condition $(SR_1)$ restated in \cref{def SR quadratic 2} was stating that the critical value belongs to a combinatorially defined Pesin set. Analogously, among surface endomorphisms, Condition $(SR_1)$ of \cref{defSRHL} states that any combinatorially defined Pesin unstable manifold are tangent to some combinatorially defined Pesin stable manifold. This is the first time that such geometrical picture (see figure \ref{firstidea})  is proposed.

Conditions $(SR_1-SR_2-SR_3)$ of \cref{defSRHL}  are the canonical generalizations of the same conditions for one-dimensional dynamics (see \cref{def SR quadratic 2}).   Conditions $(SR_1, SR_2, SR_3)$ of strong regularity  should be seen as a system of axioms. It is not easy to show that there exists one diffeomorphism of the plane which satisfies these axioms. The main result  of this work is the existence of strongly regular dynamics with positive probability among any family $(f_{a\, B})_{a\in \R}$ with $f_a(x,y)= (x^2+a,0)+B(x,y,a)$ with $B\in C^2([-3,3]^3,\R^2)$ small:
\begin{theo}[Abundance of strongly regular mappings]\label{SRabundant}
Strongly regular maps are abundant in the following meaning: For every $\epsilon>0$, there exists $b>0$, such that for every $B$ of $C^2$ norm less than $b$, there exist $\eta>0$ and a subset $\Pi_B\subset [-2,-2+\eta]$ with $\frac{\leb \Pi_B}{\leb [-2,-2+\eta]}\ge 1-\epsilon$  such that for every $a\in \Pi_B$, the map $f_{a\, B}$ is strongly regular.
\end{theo}
This theorem will be shown in \cref{proof abundance}.  A second new result of this work is that every strongly regular endomorphism leaves invariant an ergodic, physical, SRB probability measure. 
We will be shown in \cref{SRB}:
\begin{theo}[Existence of SRB measure]\label{existenceSRB}
If $f$ is strongly regular then it leaves invariant an ergodic, physical, SRB probability measure $\mu$.
\end{theo}

\begin{rema}[Maximal entropy measure]
If $f$ is a diffeomorphism, then the union  $\bigcup_{\st\in \arr \sR} W^\st_u$ is disjoint. The intersection
 $\Lambda = \bigcup_{\sc\in \arr \sR}  W^\st_u\cap \bigcup_{\sc\in  \overrightarrow \sR} W^s_\sc$ is endowed with a canonical Markov structure. 
 Indeed, given $\sg\in \sR\setminus \{\se\}$ which is not the concatenation of two non-trivial regular words,  we put $\sL_\sg:=\{\sc=\sg\cdot \sc'\in  \overrightarrow \sR: \text{ with } \sc'\in   \overrightarrow \sR\}$. 
In \cite{Berentropy}, we showed (in the case   strongly regular Hénon-like diffeomorphisms of the form $f(x,y)=(x ^2+a+{b} y, 0) +B(x, {b} y)$) that 
$(\sL_\sg)_\sg$ is a strongly positive recurrent  Markovian partition of $\Lambda$ whose induced times is conjugated is the first return time in $\Lambda$. Moreover we showed that any ergodic probability measure whose entropy is not too small has its support included in the orbit of $\Lambda$ and furthermore, that most of the periodic orbits intersect $\Lambda$. From this we deduced that $f$ displays a unique probability measure of maximal entropy which is supported by the orbit of  $\Lambda$ and which is equi-distributed on the periodic points.
\end{rema}

\subsection{Topological and combinatorial nature of the Puzzle structure}\label{subsection 3.5}
Let us focus on the properties of $\sR$. 
  \begin{Claim}
The map $\sg\in (\sR, \cdot) \mapsto (Y_\sg, n_\sg)\in (\mathcal P_2, \star)$ is a homomorphism of graded  pseudo-monoids which is {\bf injective}.
\end{Claim}
The functorial property was already observed in Fact \ref{fact:funct}.
 The injectiveness stated is not obvious, and given by the next proposition \ref{injectif}.
 Let us notice before this, that in general the homomorphism is not onto. Actually $\sR_k$ index a subset of pieces whose  good analytical properties are implied by  pure topological and combinatorial conditions as we will see in Claim \ref{combitop}.

\begin{prop}\label{injectif} For all $\sg, \sg'\in \tilde \sR$, the pieces $(Y_\sg,n_\sg)$ and $(Y_{\sg'},n_{\sg'})$ are nested or with disjoint interior.
Furthermore, if $Y_{\sg'}\subset Y_\sg$, then there is $\sg''\in \sA^{(\N)}$ such that   $\sg'= \sg\cdot \sg''$. Moreover,  $Y_{\sg'}= Y_\sg$ if and only if $\sg=\sg'$.
\end{prop}
Before proving this proposition, let us state the converse:
\begin{lemm}\label{facttronc}
For every $k\ge 0$ and $\sg=\sa_1\cdots \sa_m\in \tilde \sR_k$ with $\sa_i\in \hat \sA_k$,  for every $j\le m$, the word $\sa_1\cdots \sa_j$ belongs to $\tilde \sR_k$ and it holds:
\[Y_{\sa_1\cdots \sa_j}\supset Y_{\sa_1\cdots \sa_m}\; .\] 
\end{lemm} 
\begin{proof} The lemma is an immediate consequence of an induction on the number of letters of $\sg$ via the rules $(i)$ and $(ii)$ of \cref{admissibleword}. \end{proof}
\begin{proof}[Proof of \cref{injectif}]
By \cref{facttronc}, to prove the proposition, it suffices to show by induction on $k\ge 0$  that if $\sg_-,\sg_+\in \tilde \sR_k$ are words with the same number $m$ of letters, then:  
$$Y_{\sg_-}\cap int(Y_{\sg_+})\neq\varnothing\Longleftrightarrow Y_{\sg_-}=Y_{\sg_+}\; .$$
Only the `ìf'' part of this equivalence is not obvious.
If $k=0$ then $ \sg,\sg'\in \sA_0^m$, and the induction hypothesis follows  from the fact that pieces indexed by $\ss\in \sA_0$ have disjoint interiors.

 Let $k\ge 1$.    Let us put  $\sg=\sg_0\cdot \sa$ and  $\sg'=\sg'_0\cdot \sa'$ for $\sg_0, \sg_0'\in \sR_{k-1}$ and $\sa, \sa'\in \sA_k$. The words $\sg_0=\sg_0'$ have the same number $m-1$ of letters. If $m=1$, then $\sg_0=\sg_0'=\se$. Otherwise, as $Y_{\sg}$ intersects the interior of    $Y_{\sg'}$,  by \cref{facttronc}, the box $Y_{\sg_0}$ intersects the interior of  $Y_{\sg'_0}$. Then by induction on $m$, ${\sg_0}=\sg_0'$.
  Let us show that $\sa=\sa'$.  First note that $f^{n_{\sg_0}}(Y_\sg)$ is included in $Y_\sa$ if $\sa\in \sA_0$ or in $Y_\boxdot$ otherwise. The same occurs for $f^{n_{\sg_0}}(Y_{\sg'})$ with ${\sa'}$ instead of $\sa$. Thus if $\sa\in \sA_0$ then $\sa'=\sa$ (and we are done!), and if $\sa\notin \sA_0$, then $\sa'\notin \sA_0$. 
In the later case, we have $\sa= \boxdot_\pm (\sb-\sc)$ and $\sa'= \boxdot_{\pm'} (\sb'-\sc')$. Also $f^{n_{\sg_0}+n_\boxdot}(Y_{\sg})$ is included in $Y_{\sb}\setminus Y_{\sc}$ and  $f^{n_{\sg_0}+n_\boxdot}(Y_{\sg'})$ is included in $Y_{\sb'}\setminus Y_{\sc'}$. 

By pre-admissibility of the parabolic product, the connected set  $
f^{n_\boxdot}(T_{\sg_0} (\R_\se\times \{0\}\cap Y_\boxdot)$
 intersects only the left component of the sets:  $\partial^s Y_\se$, 
$\partial^sY_{\sb}$, $\partial^sY_{\sb'}$, $\partial^sY_{\sc}$ and $\partial^sY_{\sc'}$. 
Thus  $Y_{\sb}$, $Y_{\sb'}$, $Y_{\sc}$ and $Y_{\sc'}$ are nested. 
Up to swapping the notations, we assume that $Y_{\sb}\supset Y_{\sb'}$. 
Also $(Y_{\sb}\setminus Y_{\sc})\cap (Y_{\sb'}\setminus Y_{\sc'})$  contains the non-empty set $f^{n_\boxdot}(Y_\boxdot \cap f^{n_{\sg_0}}(int Y_{\sg}\cap Y_{\sg'}))$. Thus 
\begin{equation}\label{inclusions} Y_{\sb}\supset Y_{\sb'} \supsetneq Y_{\sc'} \supset Y_{\sc}
\quad \text{ or }  \quad 
 Y_{\sb}\supset Y_{\sb'} \supsetneq Y_{\sc} \supset Y_{\sc'}\; .\end{equation}
 Now we recall that $\sb, \sb', \sc, \sc'$ are complete words by \cref{sAk} \cpageref{sAk}.  By the  inclusions of \eqref{inclusions} and the induction hypothesis, 
if $\sb\neq \sb'$, there are non-empty complete words $\sb_1$ and $\sb_2$ such that  $\sb'= \sb\cdot \sb_1$ and $\sc=\sb\cdot \sb_1\cdot \sb_2$. This is a contradiction with the fact $\sc$ is the concatenation of $\sb $ with a prime, complete word. The same argument shows that $\sc=\sc'$.   Finally we recall that $Y_{\sg_0}\boxdot_+(Y_\sb-Y_\sc)$ is disjoint from $Y_{\sg_0}\boxdot_-(Y_\sb-Y_\sc)$ by \cref{disjoint box} thus $\pm=\pm'$ and so $\sa=\sa'$. 
\end{proof}
The argument of the latter proof shows moreover that for every $\sg\in \tilde \sR$ and $j\ge 2$, there are at most two letters $\sa_{j,-}$ and $\sa_{j,+}$ in $\hat \sA$ such that $\sg\cdot \sa_{j,\pm }$ are in $\tilde \sR$ and $n_{\sa_{j,\pm}}=j$. Moreover every letter in $\sA$ has its order at least $2$.  This implies:
\begin{coro}\label{ican} The recurrence relation 
$i_{can}(\se)=\varnothing$ and  $i_{can}(\sg\cdot \sa_{j,\pm})=i_{can}(\sg)\cdot (\pm j)$ defines an inclusion $i_{can}:\tilde \sR\hookrightarrow  (\Z\setminus \{-1,0,1\})^{(\N)}$. 
The cardinality of $\{\sg\in \tilde \sR: n_\sg\le j\}$ is at most $2^j$ for every $j\ge 0$.
 \end{coro}
\begin{proof} Let $C_j$ be $ Card \{ (n_i)_{i}\in (\Z\setminus \{-1,0,1\})^{\N)}: \sum |n_i|\le j\}$. By induction on $j$ we assume that $C_i\le 2^i$ for every $i<j$. Then it holds:
\[C_j\le 2+\sum_{k=2}^{j-1} 2\cdot C_{j-k} \le 2+ \sum_{k=2}^{j-1} 2^{j-k+1}=2^j\; .\]
\end{proof}
Let us explain the following.
\begin{Claim} \label{combitop} The definition of $\sR_k$ is purely topological and combinatorial.\\The strong regularity condition is purely combinatorial and topological. \end{Claim}
As $\sR_k$ define $\arr \sR$ and $\overrightarrow \sR$, and since tangencies are topological conditions, the second statement of this claim is implied by the first. 
To explain the first statement, let us describe the nature of conditions $(i)$ and $(ii)$ of \cref{admissibleword}.

Condition $(i)$ is purely topological: indeed the admissibility of the product $(Y_{\sg_1},n_{\sg_1}) \star  \cdots \star (Y_{\sg_m},n_{\sg_m})$ means that $\bigcap_i f^{-n_{\sg_1\cdots \sg_i}}(Y_{\sg_{i+1}})$ has non empty interior. 

Condition $(ii)$ sates that  $\sR_k$ includes all the words of the form 
$\sg=\sa\boxdot_\pm(\sb-\sc)$ such that $\sg$ is regular, $\sa, \sb, \sc$ are in $\sR_{k-1}$, $\boxdot_\pm(\sb-\sc)\in \sA_k$ and such that the tuple $[(Y_\sa, n_\sa), (Y_\sb, n_\sb), (Y_\sc, n_\sc),T_\sa]$ is admissible for the parabolic product. Such a condition (see def. \ref{def pre admissible})  assumes first the pre-admissibility of this tuple. Actually, the first condition of the pre-admissibility \cref{def pre admissible} is implied by the combinatorial rules. 
Indeed, by regularity of $\sg$, it holds $M+1+n_\sb\le 2^M n_\sa$. Furthermore, the words $\sb$ and $\sc$ are strongly regular and so complete. Thus by \cref{rigidPP}, the pair $(Y_\sb, n_\sb)$ and $(Y_\sc, n_\sc)$ are puzzle pieces.  Furthermore, $\sc= \sb\cdot \sb'$ for a word $\sb'\in \sA^{(\N)}$, and so by \cref{facttronc}, the box $Y_{\sc}$ is included in $Y_{\sb}$.

Thus the pre-admissibility condition of such  a tuple is given only by conditions 2 and 3 of \cref{def pre admissible}, which are purely  topological by \cref{pre-admi=topo}. Furthermore, the two latter topological conditions imply also the admissibility of the tuple by the following proposition proved in \cref{proof of Crutial prop} p. \pageref{proof of Crutial prop}:
\begin{prop}\label{Crutial prop}
Let $\sa :=\boxdot_{\pm}(\sb-\sc)\in \hat \sA$  with $\pm\in \{-,+\}$ such that $\sb, \sc$ belongs to $\sR$. Then 
for every $S\in \tilde \cD(\boxdot(Y_\sc-Y_\sc'))$,  it holds:
\[\forall z\in S_{\boxdot_{\pm}(\sb-\sc)}, \quad \forall u\in T_z S, \quad \forall k\le n_\sa, \quad 
\|D_zf^{n_\sa}(u)\|\ge 2^{k/3}\|D_zf^{n_\sa-k}(u)\|\; .\]
Moreover, if there exists  $\sg\in \tilde \sR$,  such that  the parabolic product $\boxdot_\pm (Y_\sb-Y_\sc)$ is {pre-admissible} for the piece $(Y_\sg, n_\sg)$, then the pair   $(Y_\sg \boxdot_\pm (Y_\sb-Y_{\sc}), n_\sg+n_\sa)$ is a piece.
\end{prop}
%

\part{Properties of strongly regular map}
\section{Proof of the abundance of strongly regular map: the parameter selection}\label{proof abundance}
In the last part we defined several subsets of $\R$, $\R^2$, $\hat \sA^{(\N)}$ and  $\hat \sA^{\Z^-}$ depending on $f$ or $P$.  As now the dynamics is going to vary, we shall display the dependence on $f$ of these subsets. For instance will denote $I_\se(P)$, $Y_\se(P)$, $\sR(f)$, $\arr \sR_k(f)$ etc.

\subsection{Ideas of the proof of the parameter selection}\label{Ideas of proof of the parameter selection}
\paragraph*{\bf The one-dimensional case}
Let $P(x)=x^2+a$ be a quadratic map such that the first return time $M$ of $a$ in $I_\se(P)$ is large. Let $I_{\mathcal L}(P)$ be the subset of points in $I_\se(P)\subset \R$  which belong to  the support of  strongly regular and common pieces of arbitrarily large orders. By \cref{coro Pesin}, for every $x\in I_{\mathcal L}(P)$, it holds $\liminf_{\infty} \frac1n \log\| D_xP^n\|>0$. Thus $I_{\mathcal L}(P)\subset I_\se(P)$ is a combinatorially defined Pesin set. 
We recall that $P$ is strongly regular iff $P^{M+1}(0)$ belongs to $I_{\mathcal L}(P)$. A consequence of next \cref{geo intuition para selec} is:
\begin{prop}\label{prop3.10}
Let $P$ be strongly regular. Then it holds for $M$ large:
\[\frac{\leb \, I_{\mathcal L}(P)}{\leb\,  I_\se(P)} =1+o(1)\; .\] 
\end{prop}
Thus at fortiori, it sounds intuitive that $P^{M+1}(0)$ belongs to $I_{\mathcal L}(P)$ with ``positive probability''; or in other words,   that strongly regular mappings are abundant. To prove this, our plan of proof is the following. First we will show  that the strongly regular pieces move uniformly slower than the critical value when the parameter $a$ varies. It is slightly more tricky here: all strongly regular pieces will not persist for every parameter $a$, however those of order $\le \max(M, 2^{\sqrt{M}}k)$ will persist when $P^{M+1}(0)$ belongs to a same strongly regular piece of depth $k$. This local constance of the combinatory will enable us to use estimates of the proof of  \cref{prop3.10} to deduce the abundance of strongly regular quadratic maps. 
\paragraph*{\bf The H\'enon-like case}
Let $f$ be $0$-strongly regular (see \textsection\ref{setting}).
Let $\sL(f)$ be the set of $\sc= (\sa_i)_{i\ge0}\in \overrightarrow \sR(f)$ such that $\sa_0\cdots \sa_m$ is strongly regular and common for infinitely many $m$. We put:
$$\mathcal L (f) := \bigcup_{\sc\in \sL(f)} W^s_{\sc}(f)\; .$$

\begin{defi}[$K^u(f)$]\index{$K^u(f)$}\label{Ku}
Let $K^u(f):= \pi_f (\mathcal L (f))$ with $\pi_f$ the map defined in 
\cref{Pesin stable1} which is $Lip$-close to the first coordinate projection, sends each $W^s_\sc$ to a point for every $\sc\in \overrightarrow \sR(f)$ and  whose restriction to $\R\times \{0\}$ is equal to the identity. 
\end{defi}
For every $\st\in \arr \sR$, we recall that the horizontal stretched curve 
$\hat W^\st_u$ is the graph of a function $2\theta$-$C^{1+Lip}$-small. 
This induces a parametrization of $f^{n_\boxdot}(\hat W^\st_u\cap Y_\boxdot)$ close to be of the form $\{(K\cdot x^2+K', 0): x\in I_\se(f)\}$. 
Thus $\pi_f\circ f^{n_\boxdot}(\hat W^\st_u\cap Y_\boxdot)$ is equal to a segment $[\alpha^0(f), \xi^\st (f)]$ where $\alpha^0(f)$ is the left endpoint of $I_\se(f)$ and $\xi^\st(f)\in I_\se(f)$. One can think $\xi^\st(f)$ as the point whose fiber $\pi_f^{-1}(\xi^{\st}(f))$ is `tangent' to $ f^{n_\boxdot}(\hat W^\st_u\cap Y_\boxdot)$. 
 This defines a subset of $\R$:
\[K^s(f):= cl\{\xi^\st(f) : \st \in \arr \sR\}\subset I_\se(f)\; .\]
We remark that the H\'enon-like map $f$ is strongly regular if and only if $K^s(f)$ is included in $K^u(f)$. Hence the following question is natural to study the abundance of strongly regular mappings:
\begin{ques}
Under which conditions, a translation of a precompact subset $K\subset \R$ is included in a precompact subset $\tilde K\subset \R$ for a set of positive Lebesgue measure of translations:
\[\leb \, \{\tau\in \R : K+\tau \subset \tilde K\}>0\; ?\]
\end{ques}
This question motivated the following result:
\begin{theo}[Thm 2.1  \cite{BM13}]\label{BM13}
Let  $K$ and $\tilde K$ be two precompact subsets of $\R$ satisfying the following properties: 
\begin{enumerate}[$(i)$]
\item  There exist  $0<d<1$ and $C_K>0$ such that $K$ can be covered by $C_K \epsilon^{-d}$ $\epsilon $-balls for every $\epsilon>0$,
\item  the lengths $(l_n)_n$ of the bounded components of $\R\setminus K$ satisfy:
\[\sum_{n: l_n> \diam K}(\diam K+l_n)+2C_K \sum_{n: l_n \le \diam K} (l_n)^{1-d} < \diam \tilde K - \diam K\; .\]
\end{enumerate}
Then the set of translations $\tau\in \R$ such that $K+\tau \subset \tilde K$ has positive Lebesgue measure:
\[\leb\{\tau\in \R: k+\tau\subset \tilde K,\; \forall k\in K\}>0\; .\] \end{theo}
\begin{rema} The first hypothesis of the theorem implies that $K$ has box dimension at most $d$. In the second hypothesis, we have $\tilde K=[a,b]\setminus \sqcup_n (a_n, b_n)$ with $l_n= b_n-a_n$ and $ \sum_n l_n^{1-d}<\infty$.\end{rema}

Our strategy to prove the abundance of strongly regular Hénon-like endomorphisms is first to show  that the strongly regular pieces move uniformly slower than all `critical values' $\xi^\st(f_a)$  when the parameter $a$ varies. Then we will define $\breve K^u(f_a)\subset K^u(f_a)$ which satisfies with $K^s(f_a)$ the assumptions $(i)$ and $(ii)$ of \cref{BM13}, whenever $f_a$ is strongly regular:
\begin{prop}\label{geo intuition para selec}
For every $0$-strongly regular map $f$,  it holds:
\begin{enumerate}
\item The set  $K^s(f)$ can be covered by $2^j$ balls of  radius $4^M \cdot b^{1+j/8M}$ for every $j\ge 0$. 
\item   If furthermore $f$ is strongly regular, there exists $\breve K^u(f)\subset \tilde K^u(f)$ such that $\leb(I_\se\setminus \breve K^u(f))$ is small when $M$ is large and the length $(l_n)_n$ of the bounded components of $I_\se(f)\setminus \breve K^u(f)$ satisfies:
\[\sum_n l_n^{1-d}=o(1)\quad\text{with  } d= 2^{-\sqrt M-3}. \]
\end{enumerate}
\end{prop}
The second part of this proposition will be shown in \cref{proof: geo intuition para selec} p. \pageref{proof: geo intuition para selec}. The first part will be restated in \cref{defsRk'} p. \pageref{defsRk'}.

Observe that the first statement implies that the diameter of $K^s(f)$ is smaller than $2\cdot 4^M\cdot b$ and the box dimension of $K^s(f)$ is smaller than $8M/|\log_2 b|$ with capacity  $C_{K^s(f)}\le 1$.
On the other hand, the second statement implies that the diameter of $\tilde K^u(f)$ is close to $2\approx \diam I_\se(f)$ and that assumption $(ii)$ of \cref{BM13} is satisfied. Thus  \cref{BM13} implies that  $K^s(f)+\tau$ is included in $\tilde K^u(f)$ for a set of numbers $\tau $ of positive Lebesgue measure. 

Thus, it sounds  intuitive that $K^s(f)\subset \breve K^u(f)$ with ``positive probability''; or equivalently that strongly regular mappings are abundant. 
Again it is a rough idea of the parameter selection: when the parameter $a$ varies, the geometries of both $K^s(f)$ and $\breve K^u(f)$ vary. 
To overcome this difficulty we will define in \textsection \ref{section $k$-Strong regularity} some combinatorially defined neighborhood of the set of strongly regular maps, called $k$-strongly regular, as those on which $\sR_k$ is constant. On such sets of maps we will show that strongly regular pieces of depth $\le   2^{\sqrt{M}}k $  as well as a $b^{1+k/(8M)}$-dense set of points in $K^s(f)$ vary smoothly with $f$. This will enable us to use similar estimates to those of the proof of  \cref{geo intuition para selec} to deduce the abundance of strongly regular Hénon-like maps. 

\begin{rema}\label{Dream}
In \cite{BM13} we studied several toy models for $\tilde K^u(f)$.  This let me hope that the strongly regular theory could be adapted to contain an example of attractor of the same dimension as the one of the initial H\'enon conjecture. 
\end{rema}

\subsection{$k$-Strong regularity}\label{section $k$-Strong regularity}
In order to prove \cref{SRabundant} on abundance of strongly regular maps, we will work with $0$-strongly regular map $f$ which satisfies moreover a truncated version of the strong regularity condition. 
This will be the definition of \emph{$k$-strongly regular dynamics}. The idea is basically to look at the subsets $\arr \sR_k$  of $\arr \sR$ of infinite concatenation of words in $\sR_k$, and to ask that  $\hat W^\st_u$ is in critical position with  $Y_{\sc^\st_{k}}$ for every $\st \in \arr \sR_k$.
\begin{defi}[Space  $\overleftarrow \sR_k$
]\label{def_arr_sR_k} 
\index{ $\arr \sR_k$}For every $k\ge 0$, 
let $\arr \sR_k$ be the set of negative one-sided sequences  $\st:= \cdots \sa_{-i} \cdots \sa_0\in \hat \sA^{\Z^-}$ which are the product of infinitely many words in $\sR_k$:
There exists an increasing sequence $(n_i)_i$ satisfying for every $i$: $\sa_{1-n_i}\cdots \sa_{0}\in \sR$ and $\sa_{1-n_{i+1}}\cdots \sa_{-n_{i}}\in \sR_k$.
\end{defi}
The union $\arr \sR':= \bigcup_k \arr \sR_k $ is in general strictly  included in $\arr \sR $, because it  could exist a sequence $\cdots \sg_k\cdots \sg_0\in \arr \sR$ with $\sg_k\in \sR\setminus \sR_k$ for every $k$.  However we have:
\begin{prop}\label{sRk dense} 
The set $ \{ W^\st_u: \st \in \arr \sR'  \}$ is dense in  $ \{ W^\st_u: \st \in \arr \sR\}$ for the topology induced by the space of horizontal stretched curves $\mathcal H$.\end{prop}
\begin{proof}  Indeed 
for every $\st=\cdots \sa_m\cdots \sa_0\in \arr \sR$, there exists $m$ large such that $ \sg :=\sa_m\cdots \sa_0\in \sR$. Thus there exists $k\ge 0$ such that  $ \sg \in \sR_k$. We can complete the latter by the sequence constantly equal to some $\sa\in \sA_0$ to form an element of $\arr \sR_k$:
$\st':= \cdots \sa\cdots \sa\cdot \sg\in \arr \sR_k$. By \cref{contractingrate}, the curves $\hat W^\st_u$ and $\hat W^{\st'}_u$ are $b^{n_\sg/3}$ close. 
\end{proof}
We recall that a horizontal curve $S$ is in critical position with a puzzle piece $(Y,n)$ if $f^{n_\boxdot}(S\cap Y_\boxdot)$ intersects $Y\setminus \partial_{\aleph(n)} Y$ and exactly one component of $\partial_{\aleph(n)} Y$ (see def. \ref{defi Critical position} p. \pageref{defi Critical position}). Also if $\sc_j\in \sR$ is strongly regular, then it is complete by definition and so by \cref{rigidPP}, the pair $(Y_{\sc_j}, n_{\sc_j})$ is a puzzle piece. 

\begin{defi}[ $k$-Strongly regular Hénon-like map]\label{defSRk}\index{$k$-Strongly regular Hénon-like map}
A $0$-strongly regular map $f$ (see \textsection\ref{setting}) is $k\ge 1$-strongly regular 
%
%
 if for every $1\le j\le k$, for every $\st\in \arr \sR_{k-1}$,  there exists $\sc_j^\st\in   \sR_{j-1}$  complete such that:
\begin{enumerate}[($SR_1^j$)]
\item  The horizontal stretched curve $\hat W_u^\st$ is in critical position with $(Y_{\sc_j^\st}, n_{\sc_j^\st})$.
\item The word $\sc_j^\st$ is strongly regular with depth $j$.
\end{enumerate}
\end{defi}
\begin{rema} \label{cjck}
By definition of the critical position, for every $i\le j\le k$, the boxes $Y_{\sc_i^\st}$ and  $Y_{\sc_j^\st}$
 are nested and so by  \cref{injectif}, there exists $\sg\in \hat \sA$ 
 such that $\sc_{j}^\st =\sc_i^\st\cdot \sg$.
 Then by definition of the critical position, it holds:
 \begin{enumerate}[($SR_3^j$)]
\item  The word $\sc_j$ is common for every $j\le k$.
\end{enumerate}
\end{rema}
 The following encodes the whole combinatorial part of the latter definition:
\begin{defi}[$k$-Strongly regular structure]
The $k$-strongly regular structure of $f$ is $(\sR_{k}, \arr \sR_{k-1}, (\sc_k^\st)_{\st\in \arr \sR_{k-1}})$.   We note that:
\[ \sR_{k}\subset \hat \sA^{(\N)}\; , \quad \arr \sR_{k-1}\subset \hat \sA^{\Z^-}\qand (\sc_k^{\st})_{t\in \arr \sR_{k-1} }\in \prod_{{\arr \sR_{k-1}}} \sR_{k-1}\; .\]
This structure belongs to the category of symbols, and is such that via the homomorphisms pure topological conditions are satisfied for  the dynamics $f$.
\end{defi}

The following proposition implies the uniqueness of this structure:
\begin{prop}[Uniqueness of $k$-strongly regular structure]\label{UNiqueness Structure} If $f$ is $k$-strongly regular, then for every $\st \in \arr \sR_{k-1}$, there is a unique strongly regular words $\sc_k\in  \sR_{k-1}$ such that $\hat W_u^\st$ is in critical position with $(Y_{\sc_k}, n_{\sc_k})$. In particular  the word $\sc_k^\st\in  \sR_{k-1}$ is uniquely defined. 
\end{prop}
We recall that the depth of a word $\sg \in \hat \sA^{(\N)}$ is the number of its letters in $\sA_0$. Also if a curve $W$ is in critical position with a piece $(Y,n)$, then the curve $W^{\boxdot}:=f^{n_\boxdot}(W\cap Y_\boxdot)$ intersects only the left component of $\partial^s Y$. Thus the latter proposition is an immediate consequence of the following:
\begin{lemm}\label{inclusion Yck}
Let $(Y_{\sg},n_\sg)$ and $(Y_{\sg'}, n_{\sg'})$ be two  pieces defined by two complete words  $\sg, \sg'$ in $\sR$.
Let  $W$ be a horizontal stretched curve such that $f^{n_\boxdot} (W\cap Y_\boxdot)$ intersects only the left component of both $\partial^s Y_\sg$ and   $\partial^s Y_{\sg'}$. It holds:
\begin{itemize}
\item If the depth  $\sg'$ is at least the depth of $\sg$, then $Y_{\sg'}\subset Y_\sg$, and there exists a words $\sg''\in \sA^{(\N)}$ such that $\sg'= \sg \cdot \sg''$.
\item If the depths of $\sg$ and $\sg'$ are equal, then $\sg=\sg'$ and $Y_\sg= Y_{\sg'}$.
\end{itemize}
 \end{lemm}
\begin{proof}
As $f^{n_\boxdot} (W\cap Y_\boxdot)$ is connected, the boxes $Y_{\sg}$  and $Y_{\sg'}$ are nested. Then by \cref{injectif},  $Y_{\sg'}\subset Y_\sg\Rightarrow \exists  \sg''\in \hat \sA^{(\N)}, \sg'= \sg\cdot \sg''$ and 
$Y_\sg\subset Y_{\sg'}\Rightarrow \exists  \sg''\in \hat \sA^{(\N)}, \sg= \sg'\cdot \sg''$. This proves the first part of the lemma. Also, as $\sg$ and $\sg'$ are complete so is $\sg''$. Thus if $\sg''\neq \se$ then the depth of $\sg''$ is positive and so the depths of $\sg$ and $\sg'$ are different. 
\end{proof}
For the parameter selection, we will need to understand the combinatory of the mapping, and to this end we will use the following similar lemma: 
\begin{lemm}\label{lemma same piece} Let $f$ be $k$-strongly regular.  We have:
\begin{enumerate}
\item  If $\st, \st'\in \arr \sR_{k-1}$ have their curves $\hat W_u^{\st}$ and $\hat W_u^{\st'}$ which are $2\theta^k$-close, then $\sc_j^{\st}=\sc_j^{\st'}$ for every $j\le k$.
\item If $\sg\cdot \sa\in \sR_{k-1}$, with $\sa=\boxdot_\pm (\sc'-\sc)$, then with $j\le k$ the  depth  of $\sc'$ it holds  $\sc'=\sc_j^\st$ for every $\st\in \arr \sR$ of the form $\st=\st'\cdot \sg$.
\end{enumerate}
\end{lemm}
\begin{proof} Put $\sc_j:= \sc_j^\st$.
We recall that the width of each component of $\partial_m Y_{\sr \cdot \sc_j}$ is $\ge 4^{-M-n_{\sc_j }-2m}$ by \cref{estim_epaisseur}. We recall that $n_{\sc_j}\le M\cdot  j$ by \cref{ordre SR} p. \pageref{ordre SR}. 
With $m=\aleph( n_{\sc_j })=\lfloor \frac{n_{\sc_j}}{12} +\frac M{24}\rfloor$ and $j\le k$, 
this thickness is larger than $4^{-(M+1)(j+1)}\gg b\cdot \theta^k$. 
By definition of the critical position, the curve $f(\hat W_u^{\st}\cap Y_\boxdot)$ intersects only one component of 
$Y_{\sr \cdot \sc_j}\setminus \partial_m Y_{\sr \cdot \sc_j}$. Thus 
$f(\hat W_u^{\st'}\cap Y_\boxdot)$ intersects only one component of $\partial^s Y_{\sr \cdot \sc_j}$ and so $\sc_j^{\st'}=\sc_j$ by  \cref{inclusion Yck}.
This proves the first item of the lemma. For the second item, we recall that $S\in \mathcal H\mapsto S^\sg$ is $b^{n_\sg/3}$-contracting by
 \cref{contractionPP}.  Note that $j\le n_{\sc'}\le n_\sa \le  2^M n_\sg$ by definition of the regular words \ref{def regular words}.  Thus $b^{n_\sg/3}$ is small compared to $\theta^j$ and so the first item implies the second. 
\end{proof}
The concept of $k$-strong regularity is designed to prove the abundance of strongly regular dynamics. To this end, we will use the following:
\begin{prop}\label{SRkimpliesSRinfty} If a map is strongly regular for every $k\ge 0$, then it is strongly regular.
\end{prop}
\begin{proof}
Let  $f$ be $k$-strongly regular for every $k\ge 0$. We put $\arr \sR':= \bigcup_{k\ge 0} \arr \sR_k$ and take  $\st\in \arr \sR'$. Note that $\st$ belongs to $\arr \sR_{k-1}$ for $k$ large enough. By $k$-strong regularity, there exists $\sc_k^\st  \in \sR$ which is strongly regular and of depth $k$ such that $\hat W^{\st}_u$ is in critical position with $Y_{\sc^\st_k}$.  By \cref{cjck}, $\sc_k^\st$ is common and  for every $j\ge k$, it holds $\sc_j^\st= \sc_k^\st \cdot \sg$ for $\sg\in \hat \sA^{(\N)}$. Thus there is
a sequence $\sc^\st= \sa_1\cdots \sa_m\cdots  \in \overrightarrow \sR$ such that for every $k$, there exists $m(k)$ satisfying $\sc_k^\st = \sa_1\cdots \sa_{m(k)}$ and it holds:
\[Y_{\sc^\st _0}\supsetneq\cdots \supsetneq Y_{\sc^\st_k}\supsetneq Y_{\sc^\st_{k+1}}
 \supsetneq\cdots\]
   As $\hat W^{\st}_u$ is in critical position with all these boxes, the map $f^{n_\boxdot }|Y_\boxdot\cap \hat W_u^{\st'}$ is tangent to $W^s_{\sc^\st}:= \bigcap_k Y_{\sc^\st_k}$ as stated in $(SR_1)$. By construction of $\sc^\st$, conditions $(SR_2)$ and $(SR_3)$ are also satisfied for every $\st \in \arr \sR'$. 
To show that $f$ is strongly regular, it remains only to prove that conditions  $(SR_1-SR_2-SR_3)$ are satisfied for every $\st \in \arr \sR\setminus \arr \sR'$. We infer the following property of $K^u$ (defined in  \ref{Ku}) shown below:
\begin{lemm}\label{SRktoSR}
The set $K^u\subset \R$ is compact.
\end{lemm}
By \cref{sRk dense}, for every $\st\in \arr \sR'$, there exists  $(t_n)_n$ such that 
$(\hat W_u^{\st_n} )_n$ converges to $\hat W_u^\st$. Thus the right endpoint of $\pi_f(f^{n_\boxdot} (\hat W_u^{\st_n}\cap Y_\boxdot)$ converges to the one of $\pi_f(f^{n_\boxdot} (\hat W_u^{\st} \cap Y_\boxdot)$. As $\pi_f(f^{n_\boxdot} (\hat W_u^{\st_n} \cap Y_\boxdot)$ belongs to $K^u$ for every $n$, it comes that 
 $\pi_f(f^{n_\boxdot} (\hat W_u^{\st} \cap Y_\boxdot)$ belongs to $K^u$. In other words, the map $f^{n_\boxdot }|Y_\boxdot\cap \hat W_u^{\st}$ is tangent to $W^s_{\sc}$ for $\sc$ strongly regular and common.
 \end{proof}
 \begin{proof}[Proof of \cref{SRktoSR}]
There are finitely many strongly regular and common  words of depth $k$ in $\sR$.  The images by $\pi_f$ of their associated pieces is a finite union of segments (which is compact). The intersection over $k$ of all these compact subset is compact and equal to $K^u$.\end{proof}

\begin{rema}\label{SRtoSRk} Conversely, it is immediate to show that $k$-strongly regular maps  are $j$-strongly regular for every $j\le k$. Also  a strongly regular map is $k$-strongly regular because  a strongly word $\sc_j\in \sR$ of depth $j$ is necessarily in $\sR_{j-1}$, as we will prove in \cref{link_btw_N_c&R_k}.  
\end{rema} 
A consequence of the proof of \cref{SRkimpliesSRinfty} is that $(SR_1-SR_2-SR_3)$ does not have to be checked for every $\st\in  \arr \sR$, but only for a dense subset  such as $ \arr \sR_0\odot \sR:=\{\st_0 \cdot \sg: \sg \in \sR, \st_0 \in \arr \sR_0\}\subset \arr \sR$ (see the proof of \cref{sRk dense}  for the density):
\begin{coro} The map $f$ is strongly regular if and only if
for every $\st \in \arr \sR_0\odot \sR$, there exists $\sc^\st=\sa_0 \cdot \sa _i\cdots \in \overrightarrow  \sR$ and $(m_k)_k\in \N^\N$ satisfying that:
\begin{enumerate}[$(SR_1)$]
\item The $f^{n_\boxdot}|Y_\boxdot\cap W^\st_u $ is tangent to $W^s_{\sc^\st}$, 
\item The word $\sc_k^\st:= \sa_0\cdots \sa_{m_k}$ is strongly regular for every $k\ge 1$.
\item The word  $\sc_k^\st$ is common with depth $k$ for every $k\ge 1$.
\end{enumerate}
\end{coro}
The above corollary states the definition strong regularity used in\cite{Berentropy}.
 \medskip
 
Let us now consider the $C^2$-{family} of maps $(f_a)_a$ satisfying the assumptions of \cref{setting}.

\medskip 

 \begin{defi}[$k$-Combinatorial interval]\index{$k$-combinatorial interval}\label{$k$-combinatorial interval}
 Let $\Omega_0$ be interval formed by the parameter $a$ at which $f_a$ is $0$-strongly regular. It is a $0$-combinatorial interval. 
 
Let $w\subset \Omega_0$ be an interval such that $f_a$ is $k$-strongly regular for every $a\in w$ with structure:
$$\left(\sR_{k}(f_a), \arr \sR_{k-1}(f_a), (\sc_k^\st(f_a))_{\st\in \arr \sR_{k-1}(f_a)}\right)\; .$$

The interval $w$ is \emph{$k$-combinatorial} if this structure is independent of $a\in w$. In other words, there exist $\sR_k(w)\subset \hat \sA^{(\N)}$, $\arr \sR_{k-1}(w)\subset \hat \sA^{\Z^-}$,  and  $(\sc_{k}^\st(w))_{t\in \arr \sR_{k-1}(w)}\in \prod_{{\arr \sR_{k-1}(w)}} \sR_{k-1}(w)$ such that for every $a\in w$, it holds 
\[\sR_k(f_a)=\sR_k(w)\, , \;\;  \arr \sR_{k-1}(f_a)=\arr \sR_{k-1}(w) \text{ and } (\sc_{k}^\st(f_a))_{\st\in \arr \sR_{k-1}(f_a)}= (\sc_{k}^t(w))_{\st\in \arr \sR_{k-1}(w)}\, .\]
\end{defi}
\begin{rema}
A $k$-combinatorial interval is also a $k'$-combinatorial interval for any $k'<k$.\end{rema}

\medskip

 We recall that for every $\sg\in \sR_k(w)$, the boundary $\partial^s Y_\sg(f_a)$ is formed by two arcs of $W^s(A)$. Thus by hyperbolic continuation these arcs persist for the perturbation $f_{a'}$, with $a'$ close to $a$. Together with $\{y=\pm \theta\}$ they bound a subset $  \tilde Y_\sg(f_{a'})$ close to $ Y_\sg(f_{a})$ that we call the \emph{hyperbolic continuation} of $ Y_\sg(f_{a})$. The next proposition states that for every $a'\in w$, this hyperbolic continuation is equal to the box  $ Y_\sg(f_{a'})$ defined by the structure $\sR_k(f_{a'})$: 
  $ Y_\sg(f_{a'})=  \tilde  Y_\sg(f_{a'})$.
\begin{prop}
Let $w$ be a $k$-combinatorial interval and $a\in w$. Then for every $\sg \in \sR_k(w)$, the family of boxes $(Y_\sg(f_{a'}))_{a'\in w}$ is the hyperbolic continuation of $ Y_\sg(f_{a})$. In particular, the curves  $\partial^s Y_\sg(f_{a'})$ depends $C^2$ on ${a'\in w}$.
\end{prop}
\begin{proof}
We proceed by induction on $k$. For $k=0$, we recall that $\sR_0(w)=\sA_0^{(\N)}$ and for $\sa\in \sA_0$, the box $Y_\sa(f_{a})$ is defined using the hyperbolic continuation of arcs of $W^s(A; \hat P)$, so are the boxes $(Y_\sg, n_\sg)$ for $\sg\in \sA_0^{(\N)}$ by definition of the $\star$-product. 
 
For $k\ge 0$, assume that for every $\sg\in \sR_k(w)$, the box $Y_\sg(f_{a'})$ is the hyperbolic continuation of  $Y_\sg(f_{a})$ for every $a'$ in a $k+1$ combinatorial interval $w\ni a$. 
Let $\sg\in \sR_{k+1}(w)$. If $\sg:= \sg_1\dots \sg_j$ with $\sg_1,\dots, \sg_j\in \sR_k(w)$.   By induction, note that the box $Y_{\sg}(f_{a'}):= \bigcap_{i=0}^{j-1} f_{a'}^{-n_{\sg_1}\cdots - n_{\sg_i}}(Y_{\sg_{i+1}}(f_{a'}))$ is still the hyperbolic continuation of $Y_{\sg}(f_{a})$. 
 If $\sg= \sa\boxdot_\pm (\sb-\sc)$, then the box $Y_\sg(f_{a'})$ is one of the two components of $(f^{n_\sa}| Y_\sa(f_{a'}))^{-1} (f^{M+1}_{a'}|Y_\boxdot)^{-1} cl(Y_\sb-Y_\sc)(f_{a'})$ and so  by induction, it is still the hyperbolic continuation of $Y_{\sg}(f_{a})$. 

 As the arc of $W^s(A; f_a)$ depends $C^2$ on $a$, we obtain the same regularity for  $a'\mapsto \partial^s Y_\sg(f_{a'})$. 
\end{proof}

\subsection{Transversality of the hyperbolic continuations}\label{Section: Transversality of the hyperbolic continuations}
We state here   that the hyperbolic continuations of the strong regular pieces move slower than the ``critical values''. 

\paragraph*{\bf The one-dimensional case} We recall that $P_a(x)=x^2+a$.
We shall show that the motion w.r.t $a$ of the hyperbolic continuation of any strongly regular piece $ (I_\sc(P_a),n_\sc)$  for $P_a$ is ``slower" than the one of $P_a^{M+1}(0)$.
As $a$ belongs to $ I_\sr(P_a)=[\alpha_M(a),\alpha_{M-1}(a)]$, to study this difference of speed we shall work in this interval. 
We recall that  $I_{\sr\cdot \sc}(P_a)$ denotes the support of the puzzle piece equal to $(I_\sr(P_a),M)\star (I_{\sc}(P_a),n_\sc)$, then it suffices to show that the motion of the endpoints of $I_{\sr\cdot \sc}(P_a)$ is different to 1, as given by:
\begin{prop}\label{trans dim1}
For every strongly regular word $\sc$ for $P_a$, if $I_{\sr \star \sc}(P_a)=: [x^-(a), x^+(a)]$, then for $M$ large, it holds:
\[\partial_a x^\pm(a) = \frac13 +O(2^{- M})\; .\]
\end{prop}
We will  generalize and prove this proposition in \cref{prop9.3} below:
\paragraph*{\bf The H\'enon-like case}
Let $(f_a)_a$ be satisfying  the assumptions of \textsection \ref{setting}.
Given $\sc\in\sR(f_a)$, we put $(Y_{\sr\cdot \sc}(f_a), n_{\sg}):=(Y_\sr(f_a), M)\star (Y_\sc(f_a), n_\sc)$. Note that $n_{\sg}= M+n_\sc$.  We recall that {$k$-combinatorial intervals} were defined in \cref{$k$-combinatorial interval} p. \pageref{$k$-combinatorial interval}.
\begin{prop}\label{preprop9.3}
Let $w$ be  a $k$-combinatorial interval for $(f_a)_a$.  Let $a\in w$ and let $\sc\in \sR_k(w)$ be strongly regular, and let $(Y_\sc(f_a), n_\sc)$ be its a puzzle piece. 
Then there exist two $C^2$-functions $\rho_\pm \colon w\times [-\theta, \theta] \to \R$ such that: 
\begin{enumerate}[$(i)$]
\item $ \partial^sY_{\sr\cdot \sc}(f_a) =\bigcup_\pm  \{(\rho_\pm (a, y), y): y\in [-\theta, \theta]\}$,
\item $\partial_a \rho_\pm = \frac13 +O(2^{- M})$ in the $C^0$-topology.
\end{enumerate}
\end{prop}
\begin{proof}
The proof contains one part which is combinatorial and one which is analytical. We prove here the following combinatorial part:
\begin{lemm}\label{decompo reguliere} Let $k\ge 1$ and let  $\sc$ be a strongly regular word in $\sR_k$. Then, there exist complete, regular words $\sg_1, \dots, \sg_m\in \sR_k$ such that:
\begin{enumerate}
\item   the concatenation of the regular, complete words $\sg_i$ is equal to $\sc$ : $\sc=\sg_1\cdots \sg_m$,
\item for every $i$,   either $\sg_i\in \sY_0$ or $n_{\sg_{i}}\le 2^{-\sqrt{M}/2} \sum_{j\le i} n_{\sg_j}$.
\end{enumerate}
\end{lemm}
\begin{proof}
Let $\sc=\sp_1\cdots \sp_m$ be the prime decomposition of $\sc$ (each $\sp_i$ is complete and prime). 
For every $i\le n_\sc$, put $X_i=1$ if there exists $l$ such that $i\in [ n_{\sp_1\cdots \sp_l}, n_{\sp_1\cdots \sp_{l+1}}-1]$ with $\sp_{l+1}\in \sA_0$ and put $X_i= 0$ otherwise.  

 By definition of strongly regular words, it holds for every $i$:
\[\sum_{j\le n_{\sp_1\cdots \sp_i}} X_j
=  n_{\sp_1\cdots \sp_i} - \sum_{j\le i: \sp_i\notin \sY_0 } n_{\sp_j}\ge
(1-2^{-\sqrt M}) n_{\sp_1\cdots \sp_i} \; .\]
This implies that $\sum_{j\le m} X_j\ge (1-2^{-\sqrt M}) m$ for every $m\le n_\sc$. 
We use now the following lemma with $A=1$, $c_2= 1-2^{-\sqrt M}$ and $c_1=1-2^{-\sqrt M/2}$. 
\begin{lemm}[Pliss  \cite{Pl72}]\label{Pliss}
Given $A\ge c_2> c_1>0$  such that for any real numbers $X_1,\dots, X_k$ in $(-\infty , A]$ satisfying $\sum_{i=1}^k X_i \ge c_2 k$, there exists $l\ge   \frac{c_2-c_1}{A-c_1} k$ and $1= n_1<\cdots < n_l< k$ such that 
\begin{equation}\label{pliss} \sum_{i=n_j }^{m} X_i \ge c_1 (m-n_j+1) \text{ for all } m\ge n_j\text{ and } j= 1,\dots , l.\end{equation}
\end{lemm}
This lemma implies that there exists a set of integers $n_i$ of density $\ge  (2^{-\sqrt M/2}-2^{-\sqrt M}) 2^{\sqrt M/2}=1-2^{-\sqrt M/2} $ of integers $n_i$ such that:
\begin{equation}\label{pliss2}\sum_{p=n_i}^j X_p\ge (1-2^{-\sqrt M/2} )(j-n_i+1)\quad \forall j\ge n_{i}\; .\end{equation}
Thus we can split $\sc$ as a maximal product $\sg= \sg_1\cdots \sg_m$ such that for every $j$, there exists $i$ satisfying that $n_i$ belongs to $[n_{\sg_1\cdots \sg_{j-1}}, n_{\sg_1\cdots \sg_{j-1}} +n_{\sb_1})$ with $\sb_1$ the first letter of $\sg_j$. By \eqref{pliss2}, we notice that $X_{1+n_i}=1$ and so  the first letter $\sb_1$ of $\sg_j$ is in $\sA_0$. This implies that $X_i=1$ for every $i\in [n_{\sg_1\cdots \sg_{j-1}}, n_i]$.  Thus  \eqref{pliss2} is satisfied with $n_i= n_{\sg_1\cdots \sg_{j-1}}$. This implies that the word $\sg_j=: \sb_1\cdots \sb_{m(j)}$ satisfies:
\[(1- 2^{-\sqrt M/2}) n_{\sb_1\cdots \sb_i}\le  \sum_{q=n_{\sg_1\cdots \sg_{j-1}}}^{n_{\sg_1\cdots \sg_{j-1}\cdot \sb_1\cdots \sb_i}} X_{q}
= n_{\sb_1\cdots \sb_i} - \sum_{ q\le i : \sb_{q}\notin \sA_0}n_{\sb_{q}}\quad \forall i\le m(j)\; .\]
In particular $\sg_j=: \sb_1\cdots \sb_{m(j)}$ is regular as stated in \cref{decompo reguliere}. Furthermore, as the set $\{n_j: j\}$ has density $\ge 1-2^{-\sqrt M/2}$, the last  property of \cref{decompo reguliere} holds true. 
\end{proof}
We recall that a given $\sg_i\in \sR_k$ complete, the pair $(Y_{\sg_i}, n_{\sg_i})$ is a puzzle piece by \cref{primecomple=puzzle}.  Thus the latter lemma together with the following combinatory free proposition (proved in \cref{proof prop9.3} P. \pageref{proof prop9.3}) implies \cref{preprop9.3}:
%
\begin{prop}\label{prop9.3}
Let $(f_a)_a$ be satisfying  the assumptions of \textsection\ref{setting}. Let $I$ be a parameter interval and let $(Y(f_a), n)= (Y_0(f_a), n_0)\star \cdots \star (Y_m(f_a), n_m)$  be a product of puzzle pieces for $f_a$, each of which persists for every $a\in I$ and  satisfies:
\begin{itemize}
\item $(Y_0, n_0)=(Y_\sr, n_\sr)$,
\item the box $Y_i$ is included in $Y_\se$, for every $1\le i\le m$,
\item either $(Y_i(f_a), n_i)= (Y_\ss(f_a), n_\ss)$ for $\ss\in \sY_0$ or  $n_{i}\le 2^{-\sqrt{M}/2} \sum_{1\le j\le i} n_{j}$, for every $1\le i\le m$.
\end{itemize}
Then there exist two $C^2$-functions $\rho_\pm \colon I\times [-\theta,  \theta] \to \R$ such that: 
\begin{enumerate}[$(i)$]
\item $ \partial^sY(f_a) =\bigcup_\pm  \{(\rho_\pm (a, y), y): y\in [- \theta,  \theta]\}$,
\item $\partial_a \rho_\pm = \frac13 +O(2^{- M})$.
\end{enumerate}
\end{prop}
 \end{proof}
 
Given $\st \in \arr \sR(f_a)$, we cannot consider the hyperbolic continuation of $\hat W^\st_u(f_a)$, because this curve contains the ``artificial part''  $\hat W^\st_u(f_a)\setminus  W^\st_u(f_a)$. However, the extension algorithm of \cref{graphtransformpara} P. \pageref{graphtransformpara} will be chosen such that the following proposition holds true: 
 \begin{prop}\label{transversity de Wu}
 Let $(f_a)_a$ be satisfying  the assumptions of \cref{setting}.
 Let $\st\in \sA^{\Z-}$ and $w\in \R$ such that $\st\in \arr \sR (f_a)$ for every $a\in w$. Then $\bigcup_{a\in w}\{a\}\times  \hat W_u^{\st }(f_a)$  is a Lipschitz surface in $\R^3$, and any vector tangent to it makes an angle less than $\theta$ with the plane $\R^2\times \{0\}$.  
\end{prop}
 This proposition we will be proved in \cref{deplacementdescourbes} p. \pageref{deplacementdescourbes}. 

Let us recall that $\hat W_u^{\st }(f_a)$ is a horizontal curve and so it is the graph of a function $x\in I_\se(P_a) \mapsto w_u^{\st }(a,x)$ which is $\theta-C^{1+Lip}$-small. Thus the image by $f_a$ of the surface $\bigcup_{a\in w}\{a\}\times  \hat W_u^{\st }(f_a)$ is:
\[\{ f_a(x, w_u^{\st }(a,x)): x\in I_\se(f_a)\}=
\{ (x^2+a ,0)+B_a(x, w_u^{\st }(a,x)): x\in I_\se(f_a)\}\; .\]
It is a family $\theta$ -$C^{1+Lip}$-close to  $(\{(x^2+a,0): x\in I_e(P_a)\})_a$.
\subsection{Geometry of the stable and unstable transverse spaces}\label{section Geometry of the stable and unstable transverse spaces}
We recall that the transverse space to the stable lamination is $K^u(f)=\pi_f(\mathcal L(f))$ with $\mathcal L(f)=\bigcup_{\sc\in \sL(f)} W^s_\sc(f)$ and $\sL(f)$ the subset of $\overrightarrow\sR(f)$ formed by the strongly regular and common sequences.  
We are going to study the Lebesgue measure and the $L^d$ norm of the gaps of a combinatorially defined subset $\breve K^u(f)\subset K^u(f)$. Namely, we are going to define $\overrightarrow \sS(f)\subset \hat \sA^\N$ such that $\breve K^u(f)=\pi_f(\bigcup_{\sc\in \sL\cap \overrightarrow \sS} W^s_\sc(f))$. The Lebesgue measure of $\breve K^u(f)$  will be shown positive by counting the number of its gaps and their lengths. Such a combinatorial way of evaluating the Lebesgue measure appears in \cite{Ts93, Ta11}. However, a new (and unexpected) aspect of our proof is that we will not use distortion estimates along the pieces nor the so-called large deviation argument. Instead we will use very sharp bounds on the expansion of the pieces associated to symbols in $\{\ss\in \sY_0: n_\ss\le M/2\}$, given by \cref{sharp bounds} stated below. This leads to the following definition:
\begin{defi}[Greatly regular word]
A word $\sg= \sa_1 \cdots \sa_m\in \hat \sA^{(\N)}$ is \emph{greatly regular} if:\index{Greatly regular word}
\begin{equation*}
\sum_{n_{\sa_j}> \frac M2,  \; j\le m} n_{\sa_j} \le\delta \sum_{j\le m} n_{\sa_j},\quad \forall j\le m \text{ with }\delta= 2^{-\sqrt {2 M}}.
 \end{equation*}
\end{defi}

Let us now consider $0$-strongly regular map $f$. As $f$ will not vary during this section, we do not display the dependence  on $f$ in the combinatorial subsets $\sR=\sR(f)$ and $\overrightarrow{\sR}=\overrightarrow{\sR}(f)$. 
\begin{defi}[$\sS$, $\overrightarrow \sS$, $\sS(> n)$, $\sS( m)$]\index{$\sS$, $\overrightarrow \sS$, $\sS(> m)$, $\sS( m)$}
Let $\sS$ be the subset of $\sR$ formed by words which are greatly regular. Let $\overrightarrow \sS$ be the subset of sequences $(\sa_i)_i\in \overrightarrow \sR$ such that $\sa_1\cdots \sa_j$ is greatly regular for some $j$ arbitrarily large.  Let $\sS(> m)$  (resp. $\sS( m)$)
be the subset of $\sS$ formed by the words $\sg$ of order $n_\sg> m$ (resp. $n_\sg = m$).
 \end{defi}
\begin{prop}\label{sS-sL} 
A sequence $(\sa_i)_i\in \overrightarrow{ \sS}$ is in $\sL$ if $\sa_1\cdots \sa_m$ is common for arbitrarily large $m$.
\end{prop}
\begin{proof}
By definition of greatly regular words, for every  $(\sa_i)_i\in \overrightarrow{ \sS}$, there exists an arbitrarily large $J$ such that 
$\sa_J$ belongs to $\sY_0$. Then the word  $\sa_1\cdots \sa_J$ is complete.  Let $\sp_1\cdots \sp_m$ be its prime decomposition. For each $j$, let $\sp_j= \sp_j'\cdot \sb_j$ with $\sb_j\in \sA_0$.  Again by definition of greatly regular words, it holds for every $j\le m$:
\[\sum_{i\le j} n_{\sp'_i}\le 2^{-\sqrt{2M}} \sum_{i\le j} n_{\sp_i}\; .\]
Note that if $\sp_i\notin \sA_0$, then $n_{\sp'_i}\ge M+1>M\ge n_{\sb_i}$. Thus 
$n_{\sp_i}/2\le n_{\sp'_i}$ and:
\[\sum_{i\le j: \sp_i\notin \sA_0} n_{\sp_i}\le 2^{1-\sqrt{2M}} \sum_{i\le j} n_{\sp_i}\; .\]
We infer that  $2^{1-\sqrt{2M}}$ is small compared to 
$2^{-\sqrt{M}}$, and so $\sa_1\cdots \sa_J$ is strongly regular. 

As there exists $I\ge J$ such that $ \sa_1\cdots \sa_I$ is common, the word $\sa_1\cdots \sa_J$ is also common and so $(\sa_i)_i$ is in $\sL$.\end{proof} 
A consequence of the proof if:
\begin{coro}\label{coro 4.29}A complete, greatly regular word is strongly regular. 
\end{coro}
To state the sharp estimate on the expansion of greatly regular pieces, we consider the  Riemannian metric $g$ whose value at $(x,y)\in \R^2$ is:
$$g_{(x,y)}:(u,v)\mapsto {\frac{u^2}{\beta^2-x^2}+ v^2}\; .$$\index{Riemannian metric $g$}
We denote by  $\|\cdot \|_g$ the norm associated to this metric. It is different to the Euclidean norm  $\| \cdot  \|: (u,v)\mapsto \sqrt{u^2+v^2}$.  In \cref{coro pour greatly} and \cref{Expansion SR piece}, we will show the following sharp estimates:
 \begin{prop}\label{sharp bounds}
 Let $f$ be $0$-strongly regular. It holds:
 \begin{enumerate}
 \item   For every  $\sa\in \sY_0$ and for every $z\in Y_\sa $ such that $n_\sa\in \{2,\dots,  \lfloor  M/2\rfloor \}$, it holds:
$$1-2^{-M} \le \frac1{n_\sa}\log_2 \frac{\|\partial_x f^{n_\sa}(z)\|_g}{\|\partial_x \|_g} \le 1+2^{-M}\; .$$
\item For every  strongly  regular  $\sg\in \sR$ and $z\in Y_{\sr\cdot \sg}$, it holds:
$$ 1-\frac1{\sqrt M} \le \frac1{2M+n_\sg}\log_2 \|\partial_x f^{M+n_\sg}(z)\| \; .$$
\end{enumerate}
\end{prop}
Note that the first bound from above is stated with the norm $\|\cdot \|_g$ whereas the second is stated with the Euclidean norm. 
A consequence of the first bound is:
\begin{coro}\label{expansion SSR}For every $0$-strongly regular map $f$, for every $\sa_1\cdots \sa_m\in \sR$, for every $z\in Y_\sg $, it holds:
\[ (1-2^{-M}) (n_\sg- n_*) +  (\frac13 -\frac1{\sqrt M}) {n_*}  
   \le \log_2 \|\partial_x f^{n_\sg}(z)\|_g 
   \quad \text{with }n_*:= \sum_{n_{\sa_i}> M/2} n_{\sa_i}\; . \]
\end{coro}
\begin{proof}
We have $\log_2 \|\partial_x f^{n_\sg}(z)\|_g $ greater than $ (1+2^{-M}) (n_\sg- n_*)$ plus the expansion (for the metric $g$) corresponding to the symbols of orders $\ge M/2$. The times spent in the latter symbols is $n_*$.   We recall that for every $\sq\cdot \sb\in \sR(f)$, by expansion stated in \cref{piecedef}.$(ii)$ of a piece, it holds:
 \[ \frac 13 \le \frac1{n_\sb}\log_2 \|\partial_x f^{n_\sb}(f^{n_\sq}(z))\|\; . \]
If $n_\sb\ge M/2$, as the metric $g$ is bounded on $Y_\se$, it holds:
 \[ \frac 13-\frac1{\sqrt{M}} \le \frac1{n_\sb}\log_2 \|\partial_x f^{n_\sb}(f^{n_\sq}(z))\|_\sg\; . \]
\end{proof} 
For the parameter selection, the following will be useful:
\begin{prop}\label{link_btw_N_c&R_k} If $\sg\in \sR$ is strongly regular, then $\sg$ belongs to $\sR_{\lfloor 2^{-1-\sqrt{M}} \cdot n_\sg \rfloor}$. 

If $\sg\in \sS$, then $\sg$ belongs to $\sR_{\lfloor 2^{-1-\sqrt{2M}} \cdot n_\sg\rfloor}$. 
\end{prop}
\begin{proof} We prove the two statement by the same argument. To this end, we put $\delta= 2^{-\sqrt {2M}}$ if $\sg\in \sS$ and $\delta= 2^{-\sqrt {M}}$ otherwise.  Let $\sg =\sa_1\cdots \sa_m$ with $\sa_i\in \hat \sA$  and note that:
 \[\sum_{ \sa_i\notin\sA_0} n_{\sa_i}\le \delta n_\sg\; .\]
  If $n_\sg\le M/\delta$, then $n_{\sa_i}\le M$ and so $\sa_i\in \sA_0$ for every $i$. This implies that $\sg\in \sR_0$.
  
   If $  n_\sg\ge M/\delta $, then $\sum_{ \sa_i\notin\sA_0} n_{\sa_i}\le \delta n_\sg$ and every $\sa_i\notin \sA_0$  is  of the form $\square_\pm(\sc-\sc')$ with:
  \[n_{\sc} =  n_{\sa_i}-M-1\qand n_{\sc'}\le M +(n_{\sa_i}-M-1)(1+2^{1-\sqrt M})< n_{\sa_i}(1+2^{1-\sqrt M})-1\; .\]
  Thus both $\sc$ and $\sc'$ have an order at most  $n_\sg (1+2^{1-\sqrt M})  \delta $ which is small compared to $n_\sg$. Thus by induction, both $\sc$ and $\sc'$ belong to  $\sR_{ \lfloor \delta n_{\sc'}/2\rfloor}\subset  \sR_{ \lfloor \delta n_{\sa_i}(1+2^{1-\sqrt M})/2\rfloor}$.
  
   As $n_{\sa_i}(1/2+2^{-\sqrt M}) \delta $ is small compared to $depth(\sc')\le (n_{\sa_i}-M-1)/2+1$ it comes that $\sa_i$ belongs to $\hat \sA_{depth(\sc')}\subset  \hat \sA_{\lfloor (n_{\sa_i}-M-1)/2+1\rfloor}=\hat \sA_{\lfloor (n_{\sa_i}-M+1)/2\rfloor}$.
    Let us now denote by $(i_j)_{1\le j\le N}$ the increasing sequence of integers such that $\sa_{i_j}$ does not belong to $\sA_0$. We have:
 \begin{itemize}
 \item    $\sa_1\cdots \sa_{-1+i_1}\in \sR_0$ and   $\sa_1\cdots \sa_{i_1}\in \sR_{ \lfloor \frac{n_{\sa_1}-M+1}2\rfloor}$,
      \item    $\sa_1\cdots \sa_{-1+i_2}\in \sR_{\lfloor  \frac{n_{\sa_1}-M+1}2 +1\rfloor}$ and   $\sa_1\cdots \sa_{i_2}\in \sR_{\lfloor 
\frac{n_{\sa_1}-M+1  +  (n_{\sa_2}-M+1}2+1\rfloor}$,
            \item \dots
            \item $\sa_1\cdots \sa_{i_N}\in \sR_{\lfloor N-1+\sum_j   \frac{n_{\sa_i}-M+1}2  \rfloor}$ and $\sg \in \sR_{\lfloor N+\sum_j   \frac{n_{\sa_i}-M+1}2  \rfloor}= \sR_{\lfloor \sum_j   \frac{n_{\sa_i}-M+3}2\rfloor}$.
\end{itemize}
Using $\sum_j n_{\sa_{i_j}}\le \delta n_\sg $, we obtain  that  $\sg$ belongs to     $\sR_{\lfloor \frac12 \sum_j n_{\sa_{i_j}} \rfloor}\subset \sR_{\lfloor \frac \delta 2n_\sg \rfloor}$.   
\end{proof}

We are going to study how fat $\breve K^u (f)= \pi_f(\bigcup_{\sc \in \overrightarrow \sS\cap \sL} W^s_\sc)$ is. To this end we shall study the number and the size of the components of 
$\mathcal E:= Y_\se \setminus \bigcup_{\sc \in \overrightarrow \sS} W^s_\sc	$ and those of
$\mathcal E':= \bigcup_{\sc \in \overrightarrow \sS\setminus \sL} W^s_\sc$
.  Observe that: 
\begin{Claim}\label{Expression des gaps} It holds $\breve K^u(f)= I_\se \setminus \pi_f(\mathcal E)\cup \pi_f(\mathcal E')\; $ and
\[\mathcal E:= \bigcup_{\sc \in \sS}  Y_\sc \setminus 
 \bigcup_{\sc' \in \sS(> n_\sc )} Y_{\sc'}\qand
  \mathcal E' = int^{Y_\se} \bigcup_{\sc'\in \{\sc\cdot \ss_-^{\odot \aleph (\sc)},\sc\cdot \ss_+\cdot \ss_-^{\odot \aleph (\sc)} \}\cap \sS}  Y_{\sc'}\; .\]
  with $int^{Y_\se} $ the interior in the topology of $Y_\se$ induced by $\R^2$.
 \end{Claim}
\begin{proof}The first equality is obtained by noting that $
Y_\se\setminus \bigcup_{\sc \in \overrightarrow \sS\cap \sL} W^s_\sc= \mathcal E \cup  \mathcal E'$. To prove the new expression of $\mathcal E$, we notice that a point is in $\mathcal E$ iff it belongs to a box $Y_\sc$ for a greatly regular word $\sc\in \sS$ but it does not belong to the box $Y_{\sc'}$ for a greatly regular word $\sc'\in \sS$ with order $n_{\sc'}>n_\sc$. To prove the new expression of $\mathcal E'$, we recall that by 
\cref{sS-sL}, $\overrightarrow \sS\setminus \sL$ is formed by the sequence $\bar \sc:= (\sa_i)_i\in \overrightarrow \sS$ which are not common. This means that there exists $(\sa_i)_{1\le i\le m} \in \bigcup_{\sc\in \sS} \{\sc\cdot \ss_-^{\odot \aleph (\sc)},\sc\cdot \ss_+\cdot \ss_-^{\odot \aleph (\sc)} \}\cap \sS$ for a certain $m$. 
\end{proof} 
   To  evaluate the length of the components of $\pi_f(\mathcal E)$ and $\pi_f(\mathcal E')$, we will use:
\begin{defi} Let  $\leb_g$ be the measure on $[-\beta, \beta]$ with density $\frac1{\sqrt{\beta^2-x^2}}$:\index{$\leb_g$}
\[  \leb_g(E):= \int_E (\beta^2-x^2)^{-1/2} dx\quad \text{for every Borel subset } E\subset \R\; .\]
\end{defi}
We will see that to each $\sc\in \sS$ is canonically associated at most 3 gaps of $\breve K^u(f)$ via Claim \ref{Expression des gaps}. We will bound their $\leb_g$-length by an exponential function of the integer $N_\sc$ defined below, and the number of the $\sc\in \sS$ will be bounded by a smaller exponential function of $N_\sc$. From this we will obtain \cref{geo intuition para selec}. 
\begin{defi}[$N_\sc$]
For every  $\sc=\sa_1\cdots \sa_m\in \sS$, let $N_\sc \ge n_\sc + \lfloor M/2\rfloor +1$ be minimal such that:
\[
N_\sc-n_\sc+ \sum_{i: n_{\sa_i}> M/2}  n_{\sa_i}  > \delta \cdot N_\sc\; 
\text{ with }\delta= 2^{-\sqrt{2 M}}.\]
\end{defi} 
Note that $N_\sc$ is greater than any $n_{\sc'}$, with $\sc'= \sc\cdot \sa\in \sS$ with $\sa\in \hat \sA$. 
\begin{rema}\label{remark on Nc}
If $N_\sc\le \frac{M}{2\delta }$ then $N_\sc-n_\sc= \lfloor \frac M2 +1\rfloor$. 
\end{rema}
\begin{prop}\label{ordre SSR}
For every $\sc\in \sS$ different to $\se$, it holds:
$$N_\sc \le \max(2 \delta n_\sc, \lfloor \frac M2\rfloor +1)+n_\sc\qand N_\sc \le (m +1)\cdot \lfloor M/2+1\rfloor\; ,$$
where $m$ is the number of $\hat \sA$-letters of $\sc$. 
\end{prop}
\begin{proof}
If $N_\sc-n_\sc>1+\lfloor M/2\rfloor $, then $N_\sc> M/(2\delta)$ and $n_\sc$ is large by definition of $N_\sc$. Then we have $N_\sc-n_\sc\le \delta N_\sc+1= \delta n_\sc +\delta( N_\sc-n_\sc)+1$. Thus:
\[N_\sc-n_\sc\le \frac{\delta n_\sc+1}{1-\delta}\le 2\delta n_\sc\; .\]
The second statement is proved by the same argument as for \cref{ordre SR}.
\end{proof}
Given $\sc\in \sS$ and  $\sc'\in \{\sc\cdot \ss_-^{\odot \aleph (\sc)},\sc\cdot \ss_+\cdot \ss_-^{\odot \aleph (\sc)} \}\cap \sS$, the following bounds the length of $\pi_f(Y_{\sc'})\subset I_\se\setminus \breve K^u(f)$ in function of 
$N_\sc$, as involved in the union $\mathcal E'$ of gaps corresponding of the common condition:
 \begin{prop}\label{badcommon}
 Let $\sc \in \sS$ and put $m:= \aleph(\sc)=\left\lfloor\frac{n_{\sc}}{12}+\frac M{24}\right\rfloor$. For every $\sc'\in 
 \{\sc  \cdot \ss_-^{m}, \sc  \cdot \ss_+ \cdot \ss_-^{m}\}\cap \sS$, with 
 $\delta= 2^{-\sqrt {2 M}}$,   it holds:
\begin{multline*}
\log_2 \frac{\leb_g \pi_f (Y_{\sc'})}{\leb_g I_\se}\le  -(1- 0,66\cdot \delta)\cdot N_\sc+\left\{\begin{array}{cc}
 0 &\text{if }N_\sc\ge  M/(2\delta)\\
 \lfloor M/2\rfloor -2  \lfloor M/24\rfloor+1&\text{otherwise.}\end{array}\right.
 \end{multline*}
  \end{prop}
  \begin{proof}
  By \cref{expansion SSR}, for every $z\in Y_{\sc'}$ it holds:
  \[
\log_2  \|\partial_x  f^{n_{\sc'}}(z)\|_g\ge ((1-\delta)(1-2^{-M})+ \delta(\frac1 3 -M^{-1/2}))n_{\sc}+ 
2m(1-2^{-M})\; .
\]
If $N_\sc\ge  M/(2\delta)$ then $n_\sc\ge M/\delta$, and so $m\gg M/2$. Then by the mean value theorem, it holds $2m\gg N_\sc-n_\sc$. This gives:
  \[
\log_2  \|\partial_x  f^{n_{\sc'}}(z)\|_g\ge  (1- 0,99\cdot\delta + \frac\delta 3)\cdot n_\sc+(N_\sc-n_\sc)(1-2^{-M})
\ge (1- 0,99\cdot\delta + \frac\delta 3)\cdot N_\sc \; .
\]

If $N_\sc<  M/(2\delta)$ then $N_\sc-n_\sc= \lfloor \frac M2 +1\rfloor$ by \cref{remark on Nc}. 
Thus   $2m\ge 2\lfloor M/24\rfloor $ and by the mean value theorem, it holds $2m\ge 2\lfloor M/24\rfloor -\lfloor M/2\rfloor -1+N_\sc-n_\sc$. Then we conclude similarly. 
\end{proof}
The following gives a bound from above of the length of the components of $\pi_f(\mathcal E')$ corresponding to the greatly regular condition:
 \begin{prop}\label{length badcomp} Let $f$ be $k$-strongly regular. 
Let $\sc\in \sS(f)$ be such that $N_\sc \le \max(M/(2\delta),  k/\delta)$. Then the set $\pi_f (Y_\sc \setminus  \bigcup_{\sc' \in \sS(> n_\sc )} Y_{\sc'})$ is a segment of $I_\se$ satisfying:
\[-\log_2 \frac{\leb_g \pi_f (Y_\sc \setminus  \bigcup_{\sc' \in \sS(> n_\sc )} Y_{\sc'})}{\leb_g I_\se}\ge (1- 0,66\cdot\delta )\cdot N_\sc \; .\]
Moreover $\sc$ belongs to $\sR_{k}$ and $Y_\sc \setminus  \bigcup_{\sc' \in \sS(> n_\sc )} Y_{\sc'}=Y_\sc \setminus  \bigcup_{\sc' \in \sS(> n_\sc )\cap \sR_{k}} Y_{\sc'}$.
\end{prop} 
\begin{proof}
 We recall that $\pi_f|I_\se\times\{0\}$ is the canonical inclusion onto $I_\se$. 

{\bf Case $N_\sc \le  \frac{M}{2\delta}$.} This implies that $n_\sc< \frac{M}{2\delta}$. By definition of greatly regular words, $\sc$ is a concatenation of symbols in $\{\ss\in \sY_0: n_\ss\le  M/2\}$. By \cref{remark on Nc},  $N_\sc-n_\sc=  \lfloor M/2\rfloor  +1$. As $\sc$ is a product of symbols in $\sA_0$, for every $\ss\in \sA_0$ is admissible from $\sc$. Consequently:
\[Y_\sc \setminus  \bigcup_{\sc' \in \sS(> n_\sc )} Y_{\sc'}=cl(Y_\sc\setminus 
\bigcup_{\ss\in \sY_0: n_\ss\le  \lfloor M/2\rfloor } Y_{\sc\cdot \ss} )
=(f^{n_\sc}|  Y_\sc)^{-1}(Y_\boxdot\cup  \bigcup_{\ss\in \sA_0: n_\ss \ge \lfloor M/2\rfloor  +1} Y_\ss)\; .
\]
Thus  $I_\se \times \{0\}$ intersects $Y_\sc \setminus  \bigcup_{\sc' \in \sS(> n_\sc )} Y_{\sc'}$ at a segment. By  \cref{coro pregreatly} P. \pageref{coro pregreatly}:
\begin{equation}\label{demi mauvais}-\log_2 \frac{\leb_g \pi_f (Y_\boxdot\cup  \bigcup_{\ss\in \sA_0: n_\ss \ge \lfloor M/2\rfloor  +1} Y_\ss)}{\leb_g I_\se}\ge (1- 2^{-M})(\lfloor M/2\rfloor  +1)\; .\end{equation}
Using \cref{sharp bounds}.1 for the product $\sc$ of symbols of order $\le M/2$, it comes,  $\log_2 \frac{\|\partial_x f^{n_\sc}(z)\|_g}{\|\partial_x \|_g}\ge  (1- 2^{-M})n_\sc$. Then with  \eqref{demi mauvais} and $\lfloor M/2\rfloor  +1=N_\sc-n_\sc$,  we obtain:
\[-\log_2 \frac{\leb_g \pi_f (Y_\sc \setminus  \bigcup_{\sc' \in \sS(> n_\sc )} Y_{\sc'})}{\leb_g I_\se}\ge (1- 2^{-M})N_\sc> (1- 0,66\cdot\delta )\cdot N_\sc \; .\]
{\bf Case $N_\sc >  \frac{M}{2\delta}$}. Let $\sc = \sa_1\cdots \sa_m$ with $\sa_i\in \hat \sA$ and put $n_*=\sum_{n_{\sa_i}>M/2} n_{\sa_i}$. It holds: 
\begin{equation}\label{ineq N}N_\sc-n_\sc +n_* =\max(\lfloor \frac M2\rfloor, \lfloor \delta N_\sc\rfloor)+1 = \lfloor \delta N_\sc\rfloor+1\ge \lfloor \frac M2\rfloor+1\; .\end{equation}
Sub-case $N_\sc-n_\sc \le M+1$. For the same reason as above, we have: 
\[Y_\sc \setminus  \bigcup_{\sc' \in \sS(> n_\sc )} Y_{\sc'}=
(f^{n_\sc}|  Y_\sc)^{-1}(Y_\boxdot\cup  \bigcup_{\ss\in \sA_0: n_\ss \ge  N_\sc -n_\sc} Y_\ss)\; .
\]
As $N_\sc -n_\sc\in [M/2+1, M+1]$,  by  \cref{coro pregreatly} P. \pageref{coro pregreatly}, it holds:
$$ \leb_g \pi_f(Y_\boxdot\cup  \bigcup_{\ss\in \sA_0: n_\ss \ge  N_\sc -n_\sc } Y_\ss)\le 2^{-(N_\sc-n_\sc) /3}\cdot \leb_g I_\se\; .$$
Thus by  \cref{expansion SSR},  it holds:
\[-\log_2 \frac {\leb_g  \pi_f (Y_\sc \setminus  \bigcup_{\sc' \in \sS(> n_\sc )} Y_{\sc'})}{\leb_g I_\se}\ge 
 (1-2^{-M}) ( n_\sc-n_*) +  (\frac13 -\frac1{\sqrt M})n_* + \frac {N_\sc-n_\sc}3\]
\begin{equation}\label{calcul refait} \ge  (1-2^{-M}) N_\sc
 + (1-2^{-M}) ( -N_\sc+n_\sc-n_*) +  (\frac13 -\frac1{\sqrt M})(N_\sc-n_\sc+n_* )\end{equation}
\[\ge  (1-2^{-M}) N_\sc +  (\frac13 -\frac1{\sqrt M}-1+2^{-M})(N_\sc-n_\sc+n_* )\; .\]
Then by \eqref{ineq N},  $N_\sc-n_\sc +n_* =  \lfloor \delta N_\sc\rfloor+1>\delta \cdot N_\sc$ and so:
\[-\log_2 \frac{\leb_g \pi_f (Y_\sc \setminus  \bigcup_{\sc' \in \sS(> n_\sc )} Y_{\sc'})}{\leb_g I_\se}>
(1- 0,66\cdot\delta)\cdot N_\sc \; .\]
Sub-case  $N_\sc-n_\sc > M+1$.  We recall that $N_\sc- n_\sc\le N_\sc- n_\sc+n_{*}=\lfloor \delta N_\sc\rfloor+1$.  Also by \cref{ordre SSR}, $N_\sc- n_\sc$ is at most $2\delta n_\sc$. As $f$ is $k$-strongly regular with $N_\sc \le \max (M/(2\delta), k/\delta)$ and $N_\sc > M/(2\delta)$, it comes $k\ge \delta N_\sc$ and so:
\begin{equation}\label{kNcnc} 2k\ge 2\delta N_\sc \ge 2\delta n_\sc \ge N_\sc-n_\sc\; .\end{equation}

 We recall that $ I_\se \times \{0\}$ is equal to $W^u_{1/2}(A)= W_u^{\st_0}$ with $\st_0=\cdots \ss_-\cdots \ss_-\in \arr \sR_0$ by Fact \ref{defi WuA} \cpageref{defi WuA}. We recall that $W^\st_u:=f^{n_\sc}(Y_\sc \cap W^{\st_0}_u)$ with $\st:= \st_0\cdot \sc$~. The curve  $W^\st_u$ is in critical position with a piece associated to a strongly regular word ${\sc_i^\st}\in \sR$ of depth $i$, for every $i\le k$. The piece  ${\sc_i^\st}$ has order at least $2i$, and so $n_{\sc_k^\st}\ge 2k\ge N_\sc-n_\sc$ by \eqref{kNcnc}. Thus there exists $j< k$ minimal such that:
 \[ M+1 + n_{\sc_j^\st }\ge  N_\sc-n_\sc\; .\]
 For every $i<j$, the curve $\hat W^\st_u$ is in critical position w.r.t.  $(Y_{\sc^\st_{i+1}}, n _{\sc^\st_{i+1}})$.  Thus, if 
$f^{M+1}(W^\st_u \cap Y_\boxdot) $ intersects  
 $int\,  Y_{\sc^\st_i}\setminus  Y_{\sc^\st_{i+1}}$, the corresponding letter  $\sa = \boxdot_\pm (\sc^\st_i-\sc^\st_{i+1})$ satisfies that 
 $\sc\cdot \sa$ belongs to $\sS$. Note that $N_{\sc\cdot \sa}=N_{\sc}$, and so $\sc\cdot \sa$ belongs to $\sR_k$ by \cref{link_btw_N_c&R_k}. This implies:
 $$I_\se\times \{0\} \cap Y_\sc \setminus  \bigcup_{\sc' \in \sS(> n_\sc )} Y_{\sc'} =(f^{n_\sc}| W^{\st_0}_u\cap Y_\sc)^{-1}(f| W^\st_u)^{-1}(Y_{\sr \cdot \sc_j^\st})\; .$$
By \cref{Pesin stable1} and horizontality of $\hat W_u^\st$, there exists $\phi$ $Lip$-close to the identity of $\R$ and $\psi$ $C^{1+Lip}$-close to the first coordinate projection $(x,y)\mapsto x$ such that    $\pi_f\circ  f|\hat W_u^\st: z\in \hat W_u^\st\mapsto \phi( \psi(z)^2+a)$.
 On the other hand,  $\pi_f(Y_{\sr \cdot \sc_j^\st})$ is an interval of length at most 
$2^{-(2M+n_{\sc_j})(1-1/\sqrt M)}\leb\, I_\se$ by \cref{sharp bounds}.2.
Thus its preimage by $\pi_f \circ  f| W_u^\st$ is a segment, and its length is bounded by:
$$ (1+o(1))\sqrt{2^{-(2M+n_{\sc_j})(1-1/\sqrt M)}\leb\, I_\se}\le  2^{-(2M+n_{\sc_j})/3}\; .$$
 As the Riemannian metric $g$ is bounded on $Y_\se$ and $M$ is large, the $\leb_g$ measure of this segment  is at most $2^{-(N_\sc-n_\sc) /3}\leb_g I_\se$.  
Again, by  \cref{expansion SSR} this implies that the set  $\pi_f(Y_\sc \setminus  \bigcup_{\sc' \in \sS(> n_\sc )} Y_{\sc'})$ is an interval, with $g$-length satisfying:
 \[-\log_2 \frac{\leb_g \pi_f (Y_\sc \setminus  \bigcup_{\sc' \in \sS(> n_\sc )} Y_{\sc'})}{\leb_g I_\se}>(1-2^{-M})(n_\sc -n_*)+(\frac13 -\frac1{\sqrt M})n_*  
 + \frac {N_\sc-n_\sc}3\]
 and then we conclude as we did in  \eqref{calcul refait}. Furthermore, we observe that $Y_\sc \setminus  \bigcup_{\sc' \in \sS(> n_\sc )} Y_{\sc'}=Y_\sc \setminus  \bigcup_{\sc' \in \sS(> n_\sc )\cap \sR_{k}} Y_{\sc'}$.
\end{proof}
Now we infer the following which gives a bound on cardinality of greatly regular words $\sc$ with the same $N_\sc$:
\begin{prop}\label{card bad comp}
For every $N\ge 0$, the set $\{\sc \in \sS: N_\sc= N\}$ has cardinality whose $\log_2$ is at most:
\[
\left\{\begin{array}{ll} 
\le (1-\delta \cdot 0.999) \cdot N,& \text{if } N\ge \frac{M}{2\delta}\\
&\\
\le N-M/2 ,& \text{if } N< \frac{M}{2\delta}\\\end{array}
\right.
\; .\]
\end{prop}
\begin{proof}
First we recall \cref{ican}. Given $\sc\in \sR$ and $j\ge 0$, there are 0, 1 or 2 symbols $\sa\in \hat \sA$ such that  $n_\sa= j$ and $\sc\cdot \sa$ belongs to $\sR$. If there are one or two symbols, then they are of the form $\ss_\pm^j$ or $\boxdot_\pm (\sc_i-\sc_{i+1})$. Note that $n_{\sa}\ge 2$. 
Thus the formula $i_{can }(\sc\cdot \sa) = i_{can} (\sc)\cdot (\pm j)$ 
defines by induction on the number of letters of $\sc$ an injection:
\[i_{can}: \sR\hookrightarrow (\Z\setminus \{-1,0,1\})^{(\N)}\; .\]
If $\sc$ belongs to $\sS$, we consider the concatenation of 
$i_{can}(\sc)$ with the symbol $N_\sc-n_\sc\ge \lfloor M/2\rfloor +1$. This defines:
\[j_{can}: \sc\in \sS\to i_{can}(\sc)\cdot (N_\sc-n_\sc)\in (\Z\setminus \{-1,0,1\})^{(\N)}=:R\; .\]\label{jcan}
We notice that $j_{can}$ is an inclusion of the set $\{\sc \in \sS: N_\sc= N\}$ into $ E(N)$ with 
\[E(N) :=\left\{ (n_j)_{j}\in R: \sum_j  |n_j| =N\; \& \; \sum_{ |n_j|> M/2}|n_j| = \lfloor \max( \delta N,M/2) \rfloor +1\right\}\; .\]
Then the bounds on the cardinality follows from the next lemma.
\end{proof}
\begin{lemm}[Lem. 5.4 \cite{BM13}]\label{cardSN}
The cardinality of $E(N)$ satisfies:
$$N^3 2^{N-\delta N}(e^2 M^2/(4\delta))^{2(\delta N+1)/M}, \quad\text{with } \delta= 2^{-\sqrt {2 M}}\; .$$
\end{lemm}
We recall the proof in  \cref{apendix combi} p. \pageref{apendix combi}. 
We are now ready to:
\begin{proof}[Proof of the second item of  \cref{geo intuition para selec}]\label{proof: geo intuition para selec}
Let $\breve K^u(f):= \pi_f( \bigcup_{\sc \in \overrightarrow \sS\cap \mathcal L} W^s_\sc)$.  We have obviously $ \breve K^u(f)\subset  K^u(f)=\pi_f( \bigcup_{\sc \in \mathcal L} W^s_\sc)$. 
Let us show that the Lebesgue measure of $\breve K^u(f)$  is positive.
By propositions \ref{length badcomp} and \ref{badcommon}, it holds:
\begin{multline*}\frac{\leb_g(I_\se\setminus  \breve K^u(f))}{\leb_g(I_\se)}\le  3
\sum_{N\ge M/(2\delta)} Card\, \{\sc\in \sS: N_\sc=N\} \cdot 2^{-N (1- 0,66\cdot\delta)}\\
+(2^{\lfloor M/2\rfloor-2\lfloor M/24\rfloor+2}+1)
\sum_{N<M/(2\delta)} Card\, \{\sc\in \sS: N_\sc=N\} \cdot 2^{-N (1- 0,66\cdot\delta)}\; .\end{multline*}
Then we infer \cref{card bad comp}. This gives:
\begin{multline}\label{def:uN}\frac{\leb_g(I_\se\setminus  \breve K^u(f))}{\leb_g(I_\se)}\le  
\sum_{N} u_N \\ \text{with }\left\{\begin{array}{ll}
u_N= (2^{\lfloor M/2\rfloor-2\lfloor M/24\rfloor+2}+1)\cdot 2^{-M/2}\cdot 2^{- 0,66\cdot\delta\cdot N}& \text{ if }N<\frac{M}{2\delta}\\
u_N= 3\cdot 2^{-N
  \frac\delta 3}& \text{ if }N\ge \frac{M}{2\delta}\end{array}\right.
\; .\end{multline}
This measure is small because:
\[  
\sum_{N\le \frac{M}{2 \delta}}2^{-M/2}\cdot 2^{
-0,66\cdot \delta \cdot N  }\le 
\frac{M}{2 \delta}2^{-M/2}=\frac{M}{2 }2^{-M/2+\sqrt{2 M}}=o(2^{-\lfloor M/2\rfloor+2\lfloor M/24\rfloor})\; ,
\]
and 
\[\sum_{N> \frac{M}{2 \delta}}2^{-N  \frac\delta 3}\le  
\frac{2^{ -\frac M 6}}{1-2^{- \frac\delta 3}}\sim \cdot  \frac3{\delta\ln 2} 
\cdot 2^{-\frac{M}6}=\frac 3 {\ln 2}\cdot
2^{\sqrt{M}-\frac{M}6}= o(1)\; .\]
Thus $\leb_g(I_\se\setminus  \breve K^u(f))$ is exponentially small when $M$ is large. 

\bigskip

Let us now now study the $L_p$-norm  $\sum_n l_n ^p$ where 
$(l_n)_n$  are the $\leb_g$-length of the gaps of $I_\se\setminus \breve K^u(f)$.  We are interested if this series converges for $p=1-d<1$.  From the above discussion, the convergence of $\sum_n l_n^{1-d}$ is equivalent to the convergence of the following series:
\[  \sum_N 
2^{(1-\delta \cdot 0.999) \cdot N}
2^{-N (1- 0,66\cdot\delta)(1-d)}\]
which converges if the following is positive:
 \[ (1- 0,66\cdot\delta )(1-d)-(1-\delta \cdot 0.999)>
\frac\delta 3 -d  (1- 0,66\cdot\delta ) \]
 this is equivalent to 
\[d<\frac\delta {3(1- 0,66\cdot\delta)}\; .\]
\end{proof}
The proof of the first statement of  \cref{geo intuition para selec} follows from a more general proposition:
\begin{prop}\label{defsRk'}
For every $k$ there exists a projection $\pi_k: \hat \sA^{\Z^-}\to \hat \sA^{\Z^-}_{\lfloor k \cdot 2^{-M}\rfloor}$ such that  for every  $f$ which is $\lfloor k\cdot 2^{-M}\rfloor$-strongly regular,
the set $\arr \sR_k'(f):=\pi_k(\arr \sR(f))$ satisfies:
 \begin{enumerate}
\item $\arr \sR_k'(f)$ is included in $\arr \sR_{\lfloor k\cdot 2^{-M}\rfloor}(f)$.
\item  For every $\st\in \arr \sR(f)$, with $\st':= \pi_k(\st)$, it holds
$ d(\hat W^{\st'}_u(f) , \hat W^{\st}_u(f))\le  
4\theta \cdot b^{ \lfloor  \frac{k 2^{-M-3}}{M}\rfloor }\; .$
\item It holds $ \log_2\,  Card( \arr \sR_k'(f))\le {2^{-M}k}\; .$
\item For every $\lfloor k\cdot 2^{-M}\rfloor$-strongly regular $\bar f$, if $\sR_{\lfloor k\cdot 2^{-M}\rfloor}(f)=\sR_{\lfloor k\cdot 2^{-M}\rfloor}(\bar f)$ then $\arr \sR_k'(f)=\arr \sR_k'(\bar f)$.
\end{enumerate}
\end{prop}
The proof of this proposition will occupy the full section \ref{Structure of the transversal space to the unstable manifolds}. 
Let us emphasis that the map $\pi_k$ is purely combinatorially defined and this  independently of the dynamics $f$. Thus  $\arr \sR'_k(f)$ is a combinatorially defined subset of $ \arr \sR(f)$. This proposition  aims to defined the $\epsilon$-dense set  $\{S^\st: \st\in  \arr \sR'_k(f)\}$ of $\{S^\st: \st\in  \arr \sR(f)\}$
(as stated in items (1-2-3)).  
Up to my knowledge, all the previous proof of the abundance of non-uniformly hyperbolic basic sets (of Hénon-like maps) involve also some $\epsilon$-dense sets of unstable manifolds, however they define them analytically (e.g. using the estimates on the expansion of some vectors by the dynamics) and not from pure combintatory. A tricky point is that these estimates might vary with the parameter, and so the corresponding set of curves also: some dramatic jump of curves may occur.  To avoid such dramatic scenario, a solution proposed in \cite{PY09} is to attach the full  structure of pieces and curves to parameter intervals which are inductively defined. As a matter of fact, their definition of strongly regular mappings depends on the choice of the family of maps and need furthermore a much  longer induction to be defined (as in \cite{BC2,MV93,WY01,Ta11}). This is not the case here: the definition of strongly regular mappings dependent only on the mapping (and not on the family of maps); this improvement is obtained thanks to the 4th item of the above proposition. 

\bigskip

Before proving \cref{defsRk'}, let us show how all what we already stated imply  our main theorem on the abundance of the strongly regular mappings.
\subsection{Proof of Theorem \ref{SRabundant} on abundance of strongly regular maps}\label{section final proof abundance SR }
Let $(f_a)_{a\in w}$ be a $C^2$-family of maps satisfying the assumptions of \textsection \ref{setting}. 
By the trasversality stated in \cref{trans dim1}, the maximal parameter interval $\Omega_0:=(a_-, a_+)$ satisfies that $P^M_{a_+}(a_+)$ belongs to the right endpoint of $I_{ \ss_-^{\odot \lfloor M/24\rfloor}}(P_{a_+})$ and $P^M_{a_-}(a_-)$ belongs to the right endpoint $-\alpha_0$ of $I_\se(P_{a_-})$. 

To prove Theorem  \ref{SRabundant} it suffices to show the existence of a Borel subset $\Omega\subset \Omega_0$ such that:
\begin{enumerate}[(1)]
\item  for every $a\in \Omega$, the map $f_a$ is strongly regular,
\item the relative measure $\leb (\Omega_0\setminus  \Omega)/\leb(\Omega_0)$ is small.  
\end{enumerate}
By \cref{SRkimpliesSRinfty}  and \cref{SRtoSRk}, this is equivalent to show the existence of a nested sequence $(\Omega_k)_{k\ge 0}$ of Borel subsets $w_k$ of $w$ such that:
\begin{enumerate}[(1')]
\item  for every $k$, for every $a\in \Omega_k$, the map $f_a$ is $k $-strongly regular,
\item  the sum $\sum_{k\ge 0}  \frac{\leb (\Omega_{k}\setminus \Omega_{k+1})}{\leb (\Omega_{0})}$ is small.
\end{enumerate}
Indeed \cref{SRkimpliesSRinfty}  implies that $\Omega=\bigcap_k \Omega_k$ is formed by strongly regular maps, and an elementary computation shows that (\emph{2'}) implies (\emph{2}).

We will give an `induction free' definition of $\Omega_k$ by introducing the following notion:
\begin{defi}[Greatly regular map]
For $k\ge 1$, a  $0$-strongly regular $f$  is \emph{$k$-greatly regular} if for every  $1\le j\le k$ and $\st\in \arr \sR_{j-1}'(f)$, there exists 
$\sc^\st_{j}\in \sR_{\lfloor  2^{-\sqrt{M}}(j-1)  \rfloor}(f)$ such that
\begin{enumerate}[$(GR_1^j)$]
\item   The curve $f(\hat W^\st_u )$ intersects  $Y_{\sr\cdot \sc_{j}^\st}$ without the $\theta^j$-neighborhood of $ \partial_{\aleph(\sg_j)}Y_{\sr\cdot \sc^\st_{j}}$ and $f(\hat W^\st_u )$ does not intersect the left component of $\theta^{j}$-neighborhood of $\partial_{\aleph(\sc^\st_{j})}Y_{\sr\cdot \sc^\st_{j}}$.
\item The word $\sc^\st_{j}$ is  greatly regular, complete and with depth $j$.
\end{enumerate}
\end{defi}
This definition state topological conditions on objects defined by the combinatory of 
$\lfloor 2^{-\sqrt{M}}(k-1) \rfloor$-puzzle structure, and such that the following holds:
\begin{prop} \label{sufficient SRk}
If $f$ is $k$-greatly regular then it is $k$-strongly regular. 
\end{prop}
\begin{proof}
As a complete greatly regular word is strongly regular by \cref{coro 4.29},  it suffices to show that for every $j\le k$, for every $\bar \st\in \arr \sR_{k-1}$, there exists a greatly regular, complete word $\sc_{j}\in \sR_{j-1}(f)$ of depth $j$ such that the curve  $\hat W^{\bar \st}_u$ is in critical position with $(Y_{\sc_{j}},n_{\sc_{j}})$. 

By  \cref{defsRk'}.2, there exists $\st\in \arr \sR'_{j-1}(f)$ such that $\hat W^{\st}_u$ is $4\theta \cdot b^{ \lfloor  (j-1) 2^{-M-3}/M\rfloor }$-close to $\hat W^{\bar \st}_u$. Then $f(\hat W^{\st}_u)$ and $f(\hat W^{\bar \st}_u)$ are $b \cdot 4\theta \cdot b^{ \lfloor  (j-1) 2^{-M-3}/M\rfloor }$-close for the Hausdorff distance.   By $(GR_1^j)$,
there exists a greatly regular, complete word  $\sc_j\in  \sR_{\lfloor  2^{-\sqrt{M}}j  \rfloor}(f)$ of depth $j$ such that the curve $f(\hat W^{\st}_u\cap Y_\boxdot) $ intersects $Y_{\sr\cdot \sc_{j}}$ without the $\theta^j$-neighborhood of $\partial_{\aleph(\sc_{j})}Y_{\sr\cdot \sc_{{j}}}$ and does not intersect the left component of $\theta^j$-neighborhood of $\partial_{\aleph(\sc_{j})}Y_{\sr\cdot \sc_{{j}}}$. 
As the 
distance between $f(\hat W^{\st}_u) $ and $f(\hat W^{\bar \st}_u) $ is small compared to $\theta^j$, the curve $\hat W^{\bar \st}_u$ is in critical position with $(Y_{\sc_{j}}, n_{\sc_{j}})$. 
\end{proof}
\begin{defi}[$\Omega_k$]\label{Omega_k} \index{$\Omega_k$}
For every $k\ge 1$, let $\Omega_k$ be the set of parameters $a$ at which
the map $f_a$ is $k$-greatly regular. 
 \end{defi}
 We remark that $(\Omega_k)_k$ is a decreasing sequence of compact subsets of $\Omega_0$.
 \begin{prop}
 Every component $w$ of $\Omega_k$ is a $k$-combinatorial interval.
 \end{prop}
 \begin{proof}
By \cref{sufficient SRk}, for every $a\in w$, the map $f_a$ is $k$ strongly regular. 
Let us show by induction that the $k$-puzzle structure is constant on $w$. This is obvious for $k=0$. Let $k>1$. By induction, the structure $\sR_{k-1}(f_a)$, $\arr \sR_{\max(k-2,0)}(f_a)$ is constant among $a\in w$.  Thus $\arr \sR_{k-1}(f_a)$ is also constant. Assume for the sake of contradiction, the existence of  $\st\in \arr \sR_{k-1}(w)$  such that $\hat W_u^\st(f_a)$ and $\hat W_u^\st(f_{a'})$ are in critical position with respectively 
$(Y_{\sc}(f_a), n_\sc)$ and $(Y_{\sc'}(f_{a'}, n_{\sc'})$ for two different strongly regular words $\sc, \sc'\in \sR_{k-1}$  of depth $k$. Note that for every $a''\in [a,a']$, the set 
 $Y_{\sc}\setminus \partial_{\aleph{\sc}}Y_{\sc}(f_{a''}) $ is disjoint from
$Y_{\sc'}\setminus \partial_{\aleph{\sc'}}Y_{\sc'}(f_{a''}) $ and at lower bounded positive distance. Using the fact that strongly regular words of depth $k$ form a finite set, we can take the limit $a'\to a$ such that this property occurs and we obtain a contradiction with the continuity of $a''\in w\mapsto \hat W_u^\st(f_{a''})$ given by the induction. 
 \end{proof}
 

We are going to prove \cref{SRabundant} by bounding the number and the lengths of the components of $\Omega_k\setminus \Omega_{k+1}$. To this end, the following will be used many times:
\begin{lemm}[Transfer phase-parameter spaces]\label{transfer} Let $k\ge 0$ and let $w$ be a component of $\Omega_{k}$.  Let $(\sc_-, \sc_+)\in \sR_{k}(w)$ be strongly regular words 
and let $Y_\sc(f_a)$ be the box bounded by the left and right component of respectively $\partial^sY_{\sc_-}(f_a)$
and $\partial^sY_{\sc_+}(f_a)$. For $\st\in \arr \sR_{k}(w)$, let $w'$ be the subset of parameters $a\in w$ such that $f(\hat W^\st_u (f_a))$ intersects the right component of $\partial^s Y_{\sr\cdot \sc}(f_a)$ but not the left component. 
Then $w'$ is an interval, and there exists a universal constant $C$ such that its length satisfies:
\[\frac{\leb\, w'}{\leb\, \Omega_0 } \le C \cdot \max_{a\in w} \leb_g\, (I_\sc (f_a))\quad \text{with }I_\sc(f_a):= Y_\sc(f_a)\cap \R\times \{0\}\]   
\end{lemm}
\begin{proof}
By \cref{prop9.3}, there exists a family $(\pi_{\sc, a})_a$ of functions 
$\pi_{\sc, a}: \R\times [- \theta,  \theta] \to \R$, which is $C^1$-close to the family of the first coordinate projections, and such that  for every $a$, $\pi_{\sc, a}|\R\times \{0\}$ is the canonical inclusion and $\pi_{\sc, a}(\partial^s Y_{\sr\cdot \sc}(f_a))$ is equal to a pair of points $\{x_a^-, x_a^+\}$, satisfying:
\[\partial_a x_a^-=\frac 13 +o(1)\qand \partial-a x_a^+ =\frac13 +o(1)\]  

 For every $\st\in  \arr \sR_k(w)$, let $L(\st, a)$ be the left endpoint   of $\pi_{\sc, a}(f_a(W^\st_u))$.  By \cref{transversity de Wu}, here exists $\eta$ small when $M$ is large such that:
\[(1-\eta)h \le L(\st, a+h)-
L(\st, a)\le (1+\eta) h\quad \text{when } h\to 0\; .\] 
 
The two latter motions are uniformly transverse, and the difference of the speed is close to $2/3$. Thus $w'$ is a (possibly empty) interval, and its length is at most $\max_{a\in w} \leb \, \pi_{\sc, a} (Y_{\sr\cdot \sc}(f_a))$ times a factor close to $3/2$.  We will see in  \cref{pre Expansion SR piece}, that $\leb \, \pi_{\sc, a} (Y_{\sr\cdot \sc}(f_a))$ is of the same order as  $4^{-M} \leb_g \, I_{\sc}(f_a)$.  This gives:
\begin{equation}\label{tranfers eq 1} \exists a\in w: \leb(w')\asymp   4^{-M}\leb_g \, (I_{\sc}(f_a))\; .\end{equation}
 Using the above bound with $f_a=\hat P_a$, $Y_{\sc_-}=Y_{\ss_-^{\aleph(0)}}$ and $Y_{\sc_+}=Y_\se$, we obtain:
  \begin{equation}\label{tranfers eq 2} \leb(\Omega_0)\asymp   4^{-M}\; .\end{equation}
Equations \eqref{tranfers eq 1} and \eqref{tranfers eq 2} imply the sought result. 
\end{proof}
\begin{defi}[ $\sS\sR^{(k)}$]
Given $f$ 0-strongly regular, let $\widehat{\sS\sR}^{(k)}(f)$ be the set of greatly regular words $\sg\in \sR(f)$ of depth $k$ and let ${\sS\sR^{(k)}}(f)$ be the subset of $\widehat{\sS\sR}^{(k)}(f)$ formed by complete words.
\end{defi}
\begin{prop} \label{order greatly regular}
Every $\sg\in \widehat{\sS\sR}^{(k)}(f)$ has its order $n_\sg\le  \lfloor \frac M2\rfloor\cdot (1 +2^{1-\sqrt {2M}})k$. 
Moreover, $\widehat{\sS\sR}^{(k)}(f)$ is included in $\sR_{ \lfloor 2^{-\sqrt {2M}} k \frac M2\rfloor}(f_a)\subset \sR_{ \lfloor 2^{-\sqrt {M}} (k-1)\rfloor}(f_a)$.
\end{prop}
\begin{proof}
The first statement follows from the same argument as for \cref{ordre SR}. 
The second is implied by the first and  \cref{link_btw_N_c&R_k} which states that any  $\sg\in \sS(f_a)$ belongs to  $\sR_{ \lfloor 2^{-1-\sqrt {2M}} n_\sg\rfloor}$.\end{proof}
A first consequence of \cref{transfer} and \cref{order greatly regular} is the following:
\begin{coro}\label{distance entre compo} Two different components $w_+$ and $w_-$ of $\Omega_k$ are at least $\leb(\Omega_0)  4^{ - 2Mk}$-distant. 
\end{coro}
\begin{proof}
Let $k\ge 1$ and put $k':=\lfloor k2^{-\sqrt M}\rfloor$. If  $w_+, w_-$ are in different components of $\Omega_{k'}$, then we can use the bound given by the induction at step $k'$. Otherwise, $w_+$ and $w_-$ are in a same component $w_0$ of $\Omega_{k'}$. Thus  by \cref{defsRk'}.3 and \cref{order greatly regular},  for every $\pm\in \{-,+\}$, it holds $\arr \sR_k'(w_\pm)=\arr \sR_k'(w_0)$ and $ \sS\sR^{(k)}(w_\pm)=\sS\sR^{(k)}(w_0)$.  Every  $\st\in \arr \sR'_k(w_0)$ satisfies $(GR_1^k)$  with a same $\sc_k^\st(w_\pm)\in  \sS\sR^{(k)}(w)$ for every $a\in w_\pm$. By \cref{transfer}, 
the set of parameters $a\in w_0$ such that each $\st\in \arr \sR'_k(w_0)$ satisfies $(GR_1^k)$  with $\sc_k^\st(w_\pm)$ is an intersection of intervals, and so it is connected. Consequently, there exists $\st\in \arr \sR'_k(w_0)$ such that $\sc^\st(w_+)\neq \sc^\st(w_-)$ and so the pieces $Y_{\sc^\st(w_+)}$ and $Y_{\sc^\st(w_-)}$ are disjoint. By \cref{estim_epaisseur}, the width of each component of  $\partial_{\aleph(\sc^\st (w_+)} Y_{\sc^\st(w_+)}$ is at least $4^{-n_{\sc^\st(w_-)}-2\aleph({\sc^\st(w_-)})}$. Thus \cref{transfer} and \cref{order greatly regular} imply the sough lower bound.
\end{proof}
Lemma \ref{transfer} will enable us to bound the component of $\Omega_k\setminus \Omega_{k-1}$ using the estimate of propositions \ref{length badcomp} and \ref{badcommon}. We also need a bound on the number of components. Our strategy is to use -- as for the proof of \cref{card bad comp} --  the inclusions $i_{can}$ (see \cref{ican}) and $j_{can}$ of $\sR_k(f_a)$ into $(\Z\setminus \{0,\pm\})^\N$. However if such maps are injective on each $\sR_k(f_a)$, they are not injective on $\bigcup_{a\in \Omega_k}\sR_k(f_a)$. We are going to show their injectivity when restricted to the set of used symbols.


\begin{prop}\label{La nouvelle}
For any component $w$ of $\Omega_{\lfloor k 2^{-M}\rfloor}$, element $\st\in \sR'_k(w)$, parameters  $a_+, a_-\in w$, words  $\sg_+\in \widehat{\sS\sR}^{(k)}(f_{a_+})$  and $\sg_-\in \widehat{\sS\sR}^{(k)}(f_{a_-})$, if:
\begin{itemize}
\item  $i_{can}(\sg_+)=i_{can}(\sg_-)$,
\item   $\forall \pm\in \{-,+\}$, the curve  $f_{a_\pm}(W^\st)$ intersects exactly one component of $\partial_s Y_{\sr \cdot \sg_\pm}$,
\end{itemize}
then  $a$ and $a'$ belong to the same component of $\Omega_{\lfloor k 2^{-\sqrt M}\rfloor}$ and $\sg_+=\sg_-$.
\end{prop}
\begin{proof}
First let us recall that each $ \sg_\pm $ belongs to $\widehat{\sS\sR}^{(k)}(f_{a_\pm})\subset \sR_{\lfloor (k-1) 2^{-\sqrt M}\rfloor}(f_{a_\pm})$ by \cref{order greatly regular}. The proposition is shown by induction on $k$.

 If $k< 2^{\sqrt M}$,  then $\lfloor k 2^{-\sqrt M}\rfloor=0$ and so 
for each $a_\pm$, the word $\sg_\pm$ belongs to $\sR_{0}(\Omega_0)$. Then the proposition follows from the injectivity of $i_{can}|  \sR_{0}(\Omega_0)$ (see \cref{ican}). 

Let $k\ge  2^{\sqrt M}$ and assume that the proposition holds true at step $k-1$.  
Let $\sg_\pm= \sg_\pm'\cdot \sg''_\pm$ with $\sg_\pm'\in \widehat{\sS\sR}^{(k-1)}(f_{a_\pm})$ and $\sg''_\pm$ which begin with a symbol in $\sA_0$. 
Note that $i_{can}(\sg_\pm')$ is the initial segment  of $i_{can} (\sg_\pm)$ which ends before the last integer of absolute value $\le M$. Thus $i_{can }(\sg'_+)=i_{can }(\sg'_-)$. 
Also $Y_{\sg'_\pm}$ contains $Y_{\sg_\pm}$, and so by induction $a_+$ and $a_-$ belong to the same component $w$ of $\Omega_{\lfloor (k-1) 2^{-\sqrt M}\rfloor}$. Again by injectivity of $i_{can}|  \sR_{\lfloor (k-1) 2^{-\sqrt M}\rfloor}(w)$ and since $\sg_-, \sg_+$ belong to $\sR_{\lfloor (k-1) 2^{-\sqrt M}\rfloor}(w) $ by \cref{order greatly regular},  it comes that $\sg_+$ and $\sg_-$ are equal to a same $\sg\in \sR_{\lfloor (k-1) 2^{-\sqrt M}\rfloor}(w) $. 
Furthermore, by \cref{transfer},  with $\sc= \sc^-=\sc^+= \sg$, it holds:
\[|a-a'| \le {\leb\, \Omega_0 } \max_{a\in w} \leb_g\, (I_\sg (f_a))\quad \text{with }I_\sg(f_a):= Y_{\sg}(f_a)\cap \R\times \{0\}\]   
Now we infer that $\leb_g\, (I_\sg (f_a))\le 2^{-k}$ by \cref{expansion SSR} (because  $n_{\sg}\ge 2k$ and $k$ is large). Thus $|a-a'|\le {\leb\, \Omega_0 }\cdot 2^{-k}$, and so by \cref{distance entre compo}, it comes that $a$ and $a'$ belongs to the same component of $\Omega_{\lfloor k 2^{-\sqrt M}\rfloor}$. 
\end{proof}

Still we need to bound the number of component of $\Omega_{k2^{-M}}$ to use the previous previous proposition. To this end we use the following:
\begin{prop}\label{card pi0 Omegak}
The number of components of $\Omega_k$ is at most $2^{k\frac M 2(1+ 2^{3-\sqrt{M}})}$.
\end{prop}
\begin{proof}
For every $k$, let $C_k$ be the cardinatility of the components of $\Omega_k$ and let $k':=\lfloor k 2^{-\sqrt{M}}\rfloor$. Let $\pi_0(\Omega_{k'})$ be the set of components of $\Omega_{k'}$ and let $w\in \Omega_{k'}$. 
 There are at most $C_{k'}$ such components. 

We are going to bound the number $C_k(w)$ of components of $\Omega_k\cap w$ and we will use the recurrence relation:
\begin{equation} \label{recu rel card}
 C_k\le C_{k'}\cdot \max_{w\in \pi_0(\Omega_{k'})} C_k(w)\; .\end{equation}
The interest of the above recurrence relation is that for each $w\in \pi_0(\Omega_{k'})$, the sets $\arr\sR_k'(f_a)$ and $\sS\sR^{(k)}(f_a)$ do not depend on $a\in w$ by \cref{defsRk'}.4 and \cref{order greatly regular}:  $\arr \sR_k'(f_a)=:\arr \sR_k'(w)$ and 
$\sS\sR^{(k)}(f_a)= \sS\sR^{(k)}(w)$. 
 
We notice that $\Omega_k\cap w$ is equal to the following intersection:
\[\Omega_k\cap w =  \bigcap_{ \st \in \arr \sR_k'(w)} \Omega_k^\st\cap w,\]
where $ \Omega_k^\st$ is the subset of parameters $a\in w$ for which there exists $\sc_k^\st(a) \in \sS\sR^{(k)}(w)$ satisfying $GR_1^k$ with $\st$ (and $GR_2^k$ by definition of $\sS\sR^{(k)}(w)$). By \cref{transfer}, for each $\sc_k\in  \sS\sR^{(k)}(w)$, the set of parameter for which $\sc_k^\st(a)=\sc_k$ is an interval.  

By \cref{order greatly regular}, the order of every $\sg \in \sS\sR^{(k)}(w)$  is at most $\frac M2(1 +2^{1-\sqrt{2M}})k$. Thus by \cref{ican}, the cardinality of $\sS\sR^{(k)}(w)$ is at most $2^{ \frac M2(1 +2^{1-\sqrt{2M}})k}$. Consequently the number of components of $\Omega_k^t$ is at most $2^{ \frac M2(1 +2^{1-\sqrt{2M}})k}$.  As the sets $\Omega_k^t$ are subsets of $ \R$, the number of components $C_k(w)$ of
$\bigcap_\st \Omega_k^\st$ is at most:
\[C_k(w)\le Card\, \arr \sR_k'(w) \cdot \max_{\st} Card\,  \pi_0   (\Omega_k^t)\le 
Card\, \arr \sR_k'(w) \cdot 2^{ \frac M2(1 +2^{1-\sqrt{2M}})k+1}
\; .\]
Now we infer that the cardinality of  $\arr \sR_{k}'(w) $ is at most $2^{2^{- M}k}$ by \cref{defsRk'}.3 to obtain:
\[ \log_2 C_k(w)\le  {2^{- M}k} + { \frac M2 (1+2^{1-\sqrt{2M}})k}\; .\] 
By \eqref{recu rel card}, this gives the following recurrence relation:
\[\log_2 C_k\le \log_2 C_{k'} + 2^{- M}k + \frac M2 (1+2^{1-\sqrt{2M}})k
\; .\]
Using this formula inductively with the relation $k'=\lfloor k 2^{-\sqrt M}\rfloor\le k\cdot  2^{-\sqrt M} $, it comes:
\[\log_2 C_k\le \sum_{ n\ge  0 }
(\frac M2(1 +2^{1-\sqrt{2M}})+2^{- M})\cdot 
k \cdot 2^{-n\sqrt{M}}
\; .\]
And so:
\[\frac{\log_2 C_k}{k} \le 
(\frac M2(1 +2^{1-\sqrt{2M}})+2^{- M})
 \sum_{ n\ge  0 } 2^{-n\sqrt{M}}\le  \frac M2(1 +2^{3-\sqrt{M}})\; .\]
\end{proof} 
We recall that the proof of  \cref{geo intuition para selec} p. \pageref{proof: geo intuition para selec} was  obtained using bound of the width of $Y_\sc \setminus  \bigcup_{\sc' \in \sS(> n_\sc )\cap \sR_k} Y_{\sc'}$ in function of $N_\sc$ for every $\sc\in \sS$. This will be used to obtain a bound on the difference of pieces  given by \emph{complete} words thanks to the following: 
\begin{lemm}\label{lemm relou}
Let $k\ge 1$,  $k'= \lfloor 2^{-\sqrt{M}} k\rfloor$,  and 
let $f$ be $k'$-strongly regular and $\sc\in \sS\sR^{(k-1)}$.
\[Y_\sc\setminus \bigcup_{\sc'\in \sS\sR^{(k)} }Y_{\sc'} = 
\bigcup_{\sq=\sc\cdot \sg \in \widehat{\sS\sR}^{(k-1)}\cap \sR_k} Y_\sq \setminus \bigcup_{\sq'\in \sS(>n_\sq)\cap \sR_k} Y_{\sq'}\; .\]
\end{lemm}
\begin{proof}
First let us recall that $N_\sc\le (M/2+1)(k+1)$. Thus $N_\sc \le \max(M/(2\delta),  k'/\delta)$, with $\delta=2^{-\sqrt{2M}}$. Moreover  for every $\sq=\sc\cdot \sg \in \widehat{\sS\sR}^{(k-1)}$, it holds $N_\sq= N_\sc$.  Thus by the last statement of \cref{length badcomp}, it holds $\sq\in \sR_k$ and 
\[\bigcup_{\sq=\sc\cdot \sg \in \widehat{\sS\sR}^{(k-1)}\cap \sR_k} Y_\sq \setminus \bigcup_{\sq'\in \sS(>n_\sq)\cap \sR_k} Y_{\sq'}=\bigcup_{\sq=\sc\cdot \sg \in \widehat{\sS\sR}^{(k-1)}}  Y_\sq \setminus \bigcup_{\sq'\in \sS(>n_\sq)} Y_{\sq'}\; .
\]
Let us show that the latter set is equal to $Y_\sc\setminus \bigcup_{\sc'\in \sS\sR^{(k)} }Y_{\sc'}$. 
If $z\in Y_\sc\cap \bigcup_{\sc'\in \sS\sR^{(k)} }Y_{\sc'} $, then 
$z\in Y_{\sc'}$, with $\sc'=\sc\cdot \sg\in \sS$ with $\sg$ complete and prime. In other words $\sg=\sg'\cdot \ss$ with $\sg'$ of depth $0$ and $\ss\in \sA_0$.  Then with $\sq= \sc \cdot \sg'$, it holds that $z$ belongs to $Y_\sq$ and to $\bigcup_{\sq'\in \sS(>n_\sq)} Y_{\sq'}$. Thus $z$ does not belong to $\bigcup_{\sq=\sc\cdot \sg \in \widehat{\sS\sR}^{(k-1)}}  Y_\sq \setminus \bigcup_{\sq'\in \sS(>n_\sq)} Y_{\sq'}$.

If $z\in Y_\sc\setminus \bigcup_{\sc'\in \sS\sR^{(k)}} Y_{\sc'}$, then let $\sq=\sc\cdot \sa_1\cdots \sa_j\in \sS$ be such that $z\in Y_\sq$ and $j$ is maximal. Note that $\sq \in \widehat{\sS\sR}^{(k-1)}$. Also by maximality of $j$, it holds $z\in  Y_\sq \setminus \bigcup_{\sq'\in \sS(>n_\sq)} Y_{\sq'}$. 
\end{proof}
\begin{proof}[Proof of Theorem \ref{SRabundant}]

To prove the theorem it suffices to show inequality $(2')$ which states that the sum $\sum_{k\ge 0}  \frac{\leb (\Omega_{k}\setminus \Omega_{k+1})}{\leb (\Omega_{0})}$ is small. We denote by $C$ a computational constant independent of $M$. Let $l_k:= \frac{\leb\,  (\Omega_{k}\setminus \Omega_{k+1})}{\leb\, \Omega_0}$ and $k'=\lfloor 2^{-\sqrt M} k\rfloor $. 
Given $\sc\in \sS(f_a)$, we put $I_\sc := Y_{\sc}\cap \R\times \{0\}$. If $\sc$ is furthermore complete, we put $\partial_{\aleph(\sc)} I_\sc := \partial_{\aleph(\sc)} Y_{\sc}\cap \R\times \{0\}$.

\noindent {\bf Case $k< 2^{M}$.} By \cref{defsRk'}.3, the set $\arr \sR_{k}'(f_a)$ does not depend on $a\in \Omega_0$ and it is formed by a unique element $\st $. 

Let $(w_i)_i$ be the components of $\Omega_{k'}$. By \cref{order greatly regular}, it holds $\widehat{\sS\sR}^{(k)}(f_a)\cup \widehat{\sS\sR}^{(k-1)}(f_a)\subset  \sR_{k'}(w_i)$ for every $i$ and $a\in w_i$.  Let $\widehat{\sS\sR}^{(k)}(w_i, \st)$ and $\widehat{\sS\sR}^{(k-1)}( w_i, \st)$ be the subsets of words $\sc$ in resp. $\widehat{\sS\sR}^{(k)}(w_i )$ and $\widehat{\sS\sR}^{(k-1)}( w_i )$ such that $f_a(W^\st_u)$ intersects the right component of $\partial^sY_{\sr\cdot \sc}(f_a)$ but not its left component for a certain $a\in w_i$. 

By \cref{transfer} we have
$l_k\le C l_k'+Cl_k''$ with
\begin{align*}l_k'&:=
\sum_i 
\sum_{\sc\in \sS\sR^{(k-1)}(w_i,\st )} \max_{a\in w_i}  \leb_g (I_\sc
\setminus \bigcup_{\sc'\in \sS\sR^{(k)}(w_i,\st )}I_{\sc'})(f_a)\; ,\\
l_k''&:= \sum_i \sum_{\sc\in \sS\sR^{(k )}(w_i,\st )} \max_{a\in w_i} \leb_g(\partial_{\aleph(\sc)}I_{\sc})(f_a)\; .
\end{align*}
The first inequality bounds  the measure of the set of parameters $a$ such that $f_a(W^\st_u)$ does not intersect the right component but not the left component of $\partial^s Y_{\sr \cdot \sc}$, with $\sc\in  \sS\sR^{(k)}$. The second inequality bounds the measure of the set of parameter $a$ such that $f_a(W^\st_u)$  intersects only the left component of $\partial^s Y_{\sr \cdot \sc}$ but $W^\st_u$ is not in critical position with $(Y_\sc, n_\sc)$. Actually, by taking the constant $C$ slightly larger, the latter inequality ensure also the $\theta^k$-room stated in  $(GR_1^k)$ because $\theta^k$ is small compared to the length of each component of $\partial_{\aleph(\sc)}I_{\sc}$. 
 In the bound of $l_k'$ we infere \cref{lemm relou} to obtain:
\[l_k'\le 
\sum_i 
\sum_{\sq\in \widehat{\sS\sR^{(k-1)}}(w_i,\st )} \max_{a\in w_i}  \leb_g (I_\sq
\setminus
 \bigcup_{\sq'\in \sS(>n_\sq)\cap \sR_{k'}(w_i)}I_{\sc'})(f_a)\]
Now we proceed as in the proof of \cref{geo intuition para selec} p. \pageref{proof: geo intuition para selec}:  we use \cref{length badcomp}  which imply:
\[l_k'\le 
\sum_i 
\sum_{\sq\in \widehat{\sS\sR^{(k-1)}}(w_i,\st )} 2^{-(1-0,66\cdot \delta)\cdot N_\sq}\]
For every $N$, by \cref{La nouvelle}, the  cardinality of $\sum_i Card\{\sq\in \widehat{\sS\sR^{(k-1)}}(w_i,\st ): N_\sq=N\}$ is the same of the one of $\bigcup_i i_{can} \{\sq\in \widehat{\sS\sR^{(k-1)}}(w_i,\st ): N_\sq=N\}$. The latter was bounded during the proof of \cref{card bad comp}, by the one of $E_N$. The  cardinality of $E_N$ was bounded by \cref{cardSN}. This leads to the following bound:
\[
\log_2 \sum_i Card\{\sq\in \widehat{\sS\sR^{(k-1)}}(w_i,\st ): N_\sq=N\} \le \left\{\begin{array}{ll} 
\le (1-\delta \cdot 0.999) \cdot N,& \text{if } N\ge \frac{M}{2\delta}\\
&\\
\le N-M/2 ,& \text{if } N< \frac{M}{2\delta}\\\end{array}
\right.
\; .\]
Thus we obtain the same bound as  in 
\eqref{def:uN} p. \pageref{def:uN}, where $u_{N}$ was defined:
\[ \sum_{k<2^{M}} l_k'\le \sum_{N\ge 0} u_N= o(1)\; .\]
And similarly  we have:
\[ \sum_{k<2^{M}} l_k''\le \sum_{N\ge 0} u_N= o(1)\; .\]
\noindent {\bf Case $k\ge 2^N$.} Put $k'':= \lfloor k2^{-M}\rfloor$ and let 
 $\pi_0(\Omega_{k''})$ be the set of connected components of $\Omega_{k''}$. 
Now $\arr \sR_k'(f_a)$ depends a priori on the component $w\in \pi_0(\Omega_{k''})$ which contains $a$. Furthermore, the set  $\arr \sR_k'(w)$ is in general not trivial.
Thus we have:
$l_k\le C l_k'+Cl_k''$ with: 
\begin{align*}l_k'&:= \sum_{\substack{w\in \pi_0(\Omega_{k''})\\ \st \in \arr \sR_k'(w)}}\quad 
\sum_{\substack{w_i\in \pi_0(\Omega_{k'}\cap w)\\\sc\in \sS\sR^{(k-1)} (w_i,\st )}} \quad \max_{a\in w_i}  \leb_g (I_\sc
\setminus \bigcup_{\sc'\in \sS\sR^{(k)}(w_i,\st )}I_{\sc'})(f_a)\\
l_k''&:= \sum_{\substack{w\in \pi_0(\Omega_{k''})\\ \st \in \arr \sR_k'(w)}}\quad 
\sum_{\substack{w_i\in \pi_0(\Omega_{k'}\cap w)\\\sc\in \sS\sR^{(k)} (w_i,\st )}}  \quad   \max_{a\in w_i} \leb_g(\partial_{\aleph(\sc)}I_{\sc})(f_a)
\end{align*}
Using the same argument as for the case $k<2^N$, we obtain:
\[l_k'+l_k''\le   \sum_{\substack{w\in \pi_0(\Omega_{k''})\\ \st \in \arr \sR_k'(w)}}\quad \sum_{N\ge k} u_N\]
We recall that the cardinality of $\arr \sR_k'(w')$ is $\le 2^{k\cdot 2^{-M} }$ by \cref{defsRk'}.3. By  \cref{card pi0 Omegak}
the number of components of $\Omega_{k''}$ is at most $2^{k2^{-M}  M }$.  Furthermore,  $u_j=3\cdot 2^{-j \delta/3}$ with $\delta=  2^{-\sqrt {2M}}$. This implies:
\[ l_k'+l_k''\le  2^{k\cdot 2^{-M} } 2^{k2^{-M}M}
\sum_{j\ge k} 2^{-j \delta}\le 
2^{2 k2^{-M} M }
  \frac{2^{-\delta k}}{1-2^{-\delta }}\le   \frac{2^{-\delta k/2}}{\delta \ln 2}\; .
\]
Now we sum this among $k$ to obtain:
\[ \sum_{k\ge 2^M} l_k' +l_k''\le  2\frac{2^{-\delta 2^M/2}}{(\delta \ln 2)^2 }=o(1)\; .
\]
\end{proof}
\section{Structure of the transversal space to the unstable manifolds}
\label{Structure of the transversal space to the unstable manifolds}
In this section we are going to define a combinatorial  and arithmetical structure on $\arr \sR$ from which we will deduce the proof of \cref{defsRk'}. As we mentioned,  this is a new and key argument for our proof of the parameter selection. 
\subsection{Right divisibility on $\hat \sA^{(\N)}$}\label{section righ divisibility}

We are going to define a partial order on $\sA^{(\N)}$ called \emph{the right divisibility}\index{Right divisibility} and denoted by  $|$. 
It aims to define a combinatorial upper bound of the $C^1$-distance  between horizontal curves $S^\st$ and $S^{\st'}$, for $\st, \st'\in \overleftarrow \sR$,  using \cref{contractionPP} and \cref{graphtransformpara} $(ii)$ and $(iii)$. Item $(iii)$ of \cref{graphtransformpara} will make the divisibility relation different from its usual meaning in the pseudo-monoid theory.  It takes care of the relation between each parabolic symbol $\boxdot_\pm (\sc_i-\sc_{i+1})$ and word $\sc_i$. This will define an ``arithmetic'' distance $dist$ on $\overleftarrow \sR$, so that $\st\mapsto \hat W_u^\st\in \mathcal H$ is $\theta$-contracting  and $\overleftarrow \sR$ is of small Haudorff dimension. 

Beyond this motivation, the following section is purely combinatory, and in particular does not depend on any mapping $f$ of $\R^2$.

\begin{defi}[Right divisibility] We define a relation between two elements $\sg$ and $\sg'\in \hat \sA^{(\N)}$, denoted by $\sg |\sg'$  by induction on $n_\sg\ge 0$.  If $\sg |\sg'$, we will say that $\sg$ is \emph{right-divisible} by $\sg'$.  
   If $\sg$ is the empty word $\se$ (and so $n_\sg=0$), then $\sg |\sg'$ iff $\sg'=\se$. 
 If $n_\sg>0$, then $\sg|\sg'$ if one of the following conditions hold:
\begin{enumerate}[$(D_1)$] 
\item ${ \sg}=\sg'$ or $\sg'=\se$ with $\se$ the empty word, 
\item $ \sg$ is of the form $ \boxdot_\pm (\sc_l-\sc_{l+1})$ and satisfies $\sc_l|\sg'$ with $\sc_l, \sc_{l+1}\in \sA^{(\N)}$,
\item there are splittings $ \sg=\sg_3\cdot \sg_2\cdot \sg_1$ and $\sg'=\sg'_2\cdot \sg_1$ into words $\sg_1, \sg'_2$ of $\hat \sA^{(\N)}$ such that $\sg_2|\sg_2'$ and $n_{\sg_3}+n_{\sg_1}\ge 1$.   
\end{enumerate}  
We notice that in $(D_2)$, it holds $n_{\sc_l}< n_\sg$ and in $(D_3)$ it holds $n_{\sg_2}< n_\sg$. Thus the induction on $n_{ \sg}$ is well defined.
\end{defi}
\begin{exem}\label{exemsimple}Let $\sg= \sa_1\cdots \sa_j$ with $\sa_i\in \sA_0$ for every $i$. Then for every $\sg'\in \hat \sA^{(\N)}$, it holds $\sg|\sg'$  iff  $\sg'= \sa_{j'}\cdots \sa_j$ for $j'\le j$ or $\sg'=\se$. Indeed we cannot use the rule $(D_2)$. 
\end{exem}
\begin{exem}\label{CvsNC} If $\sg = \sa_1\cdots \sa_j\boxdot_+(\sa_{j+1}\cdots \sa_{j+j'}-\sa_{j+1}\cdots \sa_{j+j'+1})$, with $\sa_i\in \sA_0$ for every $1\le i\le j+j'+1$, then 
$\sg|\sg'$ iff one of the following equality occurs:
\begin{itemize}
\item $\sg' = \sa_i\cdots \sa_j\boxdot_+(\sa_{j+1}\cdots \sa_{j+j'}-\sa_{j+1}\cdots \sa_{j+j'+1})$ for $i\le j$,
\item $\sg' = \boxdot_+(\sa_{j+1}\cdots \sa_{j+j'}-\sa_{j+1}\cdots \sa_{j+j'+1})$,
\item $\sg' = \sa_{j+i}\cdots \sa_{j+j'}$ for $i\ge 1$.
\end{itemize}
\end{exem}

The following Lemma is useful to show properties of the right divisibility:
\begin{lemm}\label{lemma for the division} For all $\sg, \sg'\in\hat \sA^{(\N)}\setminus \{\se\}$ with $\hat \sA$-spellings $\sg=\sa_{-j}\cdots \sa_0$ and $\sg'=\sa'_{-j'}\cdots \sa_{-1}'\cdot  \sa'_0$, it holds that $\sg|\sg'$ iff there exists $m\le  \min(j,j')$ such that $\sa_{-n}=\sa'_{-n}$ for every $n< m$ and 
$\sa_{-m}| \sa'_{-j'}\cdots \sa_{-m}'$.   
\end{lemm}
\begin{proof} We proceed by induction on $n_{\sg}$. 
If $n_\sg\le M$, then $\sa_{-j},\dots, \sa_0$ are in $\sA_0$ and we saw in \cref{exemsimple} that $\sg'= \sa_{-j'}\cdots  \sa_0$ and $j'\le j$. Thus the result holds with $m=j'$: we have indeed  $\sa_{-m}|\sa'_{m}$ by $(D_1)$. 
If  $n_{\sg}\ge M+1$, there are two possibilities:
\begin{itemize}
\item $\sg$ is equal to a single letter $\sa_0=\boxdot_{\pm} (\sc_l- \sc_{l+1})$ and so the lemma is satisfied with $m=0$,
\item $\sg$ can be split via  $(D_3)$ and so the lemma follows from the induction hypothesis.
\end{itemize}\end{proof}
Let us illustrate this lemma by a discussion on the nature of the last letter of a pair of words $(\sg,\sg')$ such that $\sg|\sg'$. 
\begin{coro}\label{divi complete}For every pair of non empty words $\sg, \sg'\in \sA^{(\N)}\setminus \{\se\}$ such that $\sg|\sg'$,  with $\sa_0$ and $\sa'_0$  the last $\hat \sA$ letters of respectively $\sg$ and $\sg'$, it holds:
\begin{itemize}
\item If $\sa_0\neq\sa'_0$ then $\sa_0$ is of the form $\square_\pm(\sc-\sc')$, $\sa'_0\in \sA_0$  and $\sc|\sg'$.  
\end{itemize}
\end{coro}
\begin{proof}
By Lemma \ref{lemma for the division}, $m$ cannot be greater than $0$ since otherwise $\sa_0=\sa_0'$. Thus $m=0$, and so $\sa_0|    \sg'$. 
%
\end{proof}
As announced, we have:
\begin{prop}The right divisibility relation $|$ is a partial order on $\hat \sA^{(\N)}$.
\end{prop}
\begin{proof} 
The reflexivity follows from the fact that $(D_2)$ and $(D_3)$ reduce the order. 
Let us prove the transitivity of the divisibility relation. Let $\sg,\sg',\sg''\in \hat \sA^{(\N)}$ be such that $\sg|\sg'$ and $\sg'|\sg''$. We are going to prove that $\sg|\sg''$ by induction on $n_{\sg}$.
We put $\sg= \sq\cdot \sa_0$, $\sg'=\sq'\cdot \sa'_0$ and $\sg'' =\sq''\cdot \sa''_0$, with $\sa_0,\sa_0',\sa_0''\in \hat \sA$. 

If $\sa_0=\sa_0' =\sa_0''$ then we remove this last letter to $\sg, \sg', \sg''$ and we still have the same divisibility relations by $(D_3)$. Thus we can conclude by induction on $n_{\sg}$.

If $\sa_0=\sa_0'$ but $\sa_0' \not =\sa_0''$, by \cref{divi complete} it holds  $\sa_0=\sa_0'$ and $\sa_0'|\sg''$. Using $(D_3)$ we get that $\sg| \sg''$.
If $\sa_0\not =\sa_0'$, by \cref{divi complete} it holds that $\sa_0$ is of the form $\boxdot _{\pm} (\sc_j-\sc_{j+1})$ and $ \sc_j|\sg'$. By induction on $n_\sg$, it comes$ \sc_j|\sg''$ and so $\sg| \sg''$.\end{proof}

The following states some other fundamentals properties of the divisibility relation. 
\begin{prop}\label{propdivi}  For all $\sg,\sg',\sg''\in \hat \sA^{(\N)}$, we have:
\begin{itemize}
\item[$(i)$] $\sg|\sg'\Rightarrow n_{\sg}\ge n_{\sg'}$ with equality only if $\sg=\sg'$,
\item[$(ii)$]  $\sg|\sg'$ and $\sg | \sg''$ and  $n_{\sg'}\ge n_{\sg''}\Rightarrow \sg'|\sg''$,
\item[$(iii)$] $\sg|\sg' \Leftrightarrow \sg\cdot \sg''| \sg'\cdot \sg''$.\end{itemize}
\end{prop}  
\begin{proof} Property $(iii)$ is an immediate consequence of Condition $(D_3)$. Property $(i)$ is shown easily by induction on $n_{\sg}$. To prove Property $(ii)$, we write the $\hat \sA$-spelling of the words $\sg$, $\sg'$, $\sg'' $:
  \[\sg=\sa_{-j}\cdots \sa_0,\quad  \sg'=\sa'_{-j'}\cdots \sa'_0,\quad \sg''=\sa''_{-j''}\cdots \sa''_0.\]
  By Lemma \ref{lemma for the division}, there exist maximal $m'\le\min(j,j')$ and $m''\le \min(j,j'' )$ such that:
  \begin{itemize}\item $\sa_{-n}= \sa'_{-n}$ (resp. $\sa_{-n}= a''_{-n}$) for every $n< m'$ (resp. $n< m''$),
  \item $\sa_{-m}|\sa'_{-j'}\cdots \sa'_{-m'}$ and $\sa_{-m''}|\sa''_{-j''}\cdots \sa''_{-m''}$.
  \end{itemize}
  If $m'>m''$ then $\sg'|\sg''$ by $(D_3)$, and we are done.  Similarly if $m'<m''$ then $\sg''|\sg'$, then by  $(i)$ and the assumption of $(ii)$, it holds $n_{\sg'}=n_{\sg''}$ and $\sg'=\sg''$; this contradicts $m'<m''$. 

If $m'=m''$ then the order of $\sa'_{-j'}\cdots \sa'_{-m'}$ is at least the order of $\sa''_{-j''}\cdots \sa''_{-m''}$. Also to show that $\sg'|\sg''$ it suffices to show that $\sa'_{-j'}\cdots \sa'_{-m'}|a'' _{-j''}\cdots \sa'_{-m'}$ by $(iii)$. 

If $\sa_{m}$ is in $\sY_0$ then both non-empty words  $\sa_{j'}'\cdots \sa'_{-m'}$ and $\sa_{j''}''\cdots \sa''_{-m'}$ must be equal to $\sa_{m}$ and we are done. If $\sa_{-m}\notin \sY_0$ then it is of the form $\boxdot_\pm (\sc_l-\sc_{l+1})$.  If $\sa_{j'}'\cdots \sa'_{-m'}= \sa_{-m}$ then we are done. Otherwise, by \cref{divi complete}, the word $\sc_l$ is divisible by both $\sa'_{-j'}\cdots \sa'_{-m'}$ and  $\sa''_{-j''}\cdots \sa''_{-m'}$. We conclude by induction on $n_{\sg}$.  
  \end{proof} 
  
  By properties $(i)$ and $(ii)$, we can define:
\begin{defi}[GCM] 
The \emph{greatest common divisor}\index{Greatest common divisor, $\wedge$, $\nu$} of $\sg$ and $\sg'$ is the element $\sd\in \hat \sA^{(\N)}$ dividing both $\sg$ and $\sg'$ with maximal order. We denote $\sd=:\sg\wedge \sg'$. 
For $\sg,\sg'\in\hat \sA^{(\N)}$, we put $\nu(\sg,\sg')=:n_{\sg\wedge \sg'}$.
\end{defi}

Let $\hat \sA^{\Z^-}$ be the set of negative, one-sided infinite sequences $\cdots \sa_{-j}\cdots \sa_0$ of letters in $\hat \sA$. Given $\st=\cdots \sa_{-i}\cdots \sa_0, \st'=\cdots \sa'_{-i}\cdots \sa'_0\in \hat \sA^{\Z^-}$ we observe that 
$(\nu(\sa_{-i}\cdots \sa_0, \sa'_{-i}\cdots \sa'_0 ))_{i\ge 0}$ is a non-decreasing  sequence. Let $\nu(\st, \st')$ be its limit (which is possibly equal to the infinity). 
\begin{defi}[Distance $dist$]\label{dist}\index{dist}  Let $dist: (\st, {\st'})\in \hat \sA^{\Z^-}\times\hat \sA^{\Z^-}\mapsto  b^{\frac{\nu(\st, \st')}{4}}$.
\end{defi}
\begin{prop}The function $dist$ is an ultra-metric distance.
\end{prop}
\begin{proof}
We observe that $dist $ is symmetric. Let us show that  $dist(\st, \st')=0$ implies $\st=\st'$. This is equivalent to show that  if $\nu(\st, \st')=\infty$ then $\st=\st'$. Put $\st= \cdots \sa_{-j}\cdots \sa_0$ and $\st'= \cdots \sa'_{-j}\cdots \sa'_0$. Let us show that  $a_0=\sa_0'$. To this end, we consider $\sg\in \hat \sA^{(\N)}$ such that $n_{\sg}\ge \max(\sa_0, \sa_0')$,  $\sa_{-j}\cdots \sa_{0}|\sg$ and  $\sa_{-j}'\cdots \sa_{0}'|\sg$ for some large $j$. By \cref{divi complete}, this implies that both  $\sa_0$ and $\sa_0'$ are equal to the last letter of $\sg$. Thus $\sa_0=\sa_0'$. Now we use $(D_3)$ to obtain $\nu( \cdots \sa_{-j}\cdots \sa_{-1}, \cdots \sa'_{-j}\cdots \sa'_{-1})=\nu( \cdots \sa_{-j}\cdots \sa_{0}, \cdots \sa'_{-j}\cdots \sa'_{0}) -n_{\sa_0}= \infty$ and so by the same argument we obtain $\sa_{-1}'= \sa_{-1}$. And so one by induction, we show that $\st=\st'$. 

Let us show that $dist $ is ultra-metric and so satisfies the triangle inequality. Let $\st= \cdots \sa_{-j}\cdots \sa_0, \st'= \cdots \sa'_{-j}\cdots \sa'_0$ and $\st''= \cdots \sa''_{-j}\cdots \sa''_0$ be distinct. Then for $j$ sufficiently large, with $\sg:= \sa_{-j}\cdots \sa_0$,  $\sg':= \sa'_{-j}\cdots \sa'_0$, and  $\sg'':= \sa''_{-j}\cdots \sa''_0$, it holds
$\nu(\st, \st')= \nu (\sg,\sg')$ and $\nu(\st', \st'')= \nu (\sg',\sg'')$. Assume for instance $\nu (\sg,\sg')\ge \nu (\sg',\sg'')$. We notice that $\sg'$ is divisible by both $\sg\wedge \sg'$ and $\sg'\wedge \sg''$. Thus by \cref{propdivi}.$(ii)$, it holds $\sg\wedge \sg'|\sg'\wedge \sg''$. Consequently both $\sg$ and $\sg''$ are divisible by $\sg'\wedge \sg''$. Then it comes that $\nu(\st, \st'')\ge \nu(\sg', \sg'')=\nu(\st', \st'')$ and so $dist$ is ultra-metric. 
\end{proof}

Let us say that $\st=\cdots \sa_{-j}\cdots \sa_0$ is \emph{right divisible} by  $\sg \in \hat \sA^{(\N)}$ if $\sa_{-j}\cdots \sa_0|\sg$ for some sufficiently large $j$ (actually $j=n_\sg$ suffices). Then we write $\st|\sg$.  

The following displays an arithmetic property of $dist$:
\begin{prop}\label{prepa completudesR}For every $\sg\in \sA^{(\N)}$, there exists $\bar \st\in \hat \sA^{\Z^-}$ such that the set $B:=\{\st\in \hat \sA^{\Z^-}: \st|\sg\}$ is equal to the ball centered at $\bar \st$ with radius  $b^{(n_\sg-\frac12)/4}$. In particular $B$  is both open and closed. \end{prop}
\begin{proof}
We recall that in an ultra metric space, the balls with strictly positive radius are both open and closed. Thus it suffices to find $\bar \st $ such that $B':=\{\st: dist(\st,\bar \st)< b^{n_\sg/4-1/8}\}$ is equal to $B$.

Let $\bar \st$ which is divisible by $\sg$. For instance take $\bar \st = \cdots \ss_-\cdots \ss_-\cdot \sg$. 
Obviously, for any $\st$ divisible by $\sg$, the distance between $\bar \st $ and $\st$ is at most $b^{n_\sg/4}<b^{n_\sg/4-1/2}$. Thus $B'$ contains $B$. 

Conversely, for any $\st\in \hat \sA^{\Z^-}$ which is not in $B$, then
$\nu(\bar \st,\st)\le n_\sg- 1 $ and so $\st$ belongs to the complement of $B'$. Indeed,  otherwise $\nu(\bar \st,\st)\ge n_\sg$ and so $\st$ and $\bar \st $ would be divisible by a certain $\sg'$ with $n_{\sg'}\ge n_\sg$, and as $\st|\sg'$ and $\st|\sg$, by \cref{propdivi} $(ii)$, it comes that $\sg'$ is divisible by $\sg$ and so $\st$ is divisible by $\sg$. A contradiction. 
\end{proof}
\begin{prop}\label{prepa completudesR2}
The space $\hat \sA^{\Z^-}$ endowed with the distance $dist$ is a complete metric space.\end{prop}
\begin{proof}
Let us show that any Cauchy sequence $(\st^m)_m$ converges. As  $(\st^m)_m$ is a Cauchy sequence, for every $m\ge 0$, there exists $\sg^m$ with order $\ge m$ such that every  $\st^k$ is divisible by $\sg^m$ for $k$ sufficiently large. Let us chose the sequence $(\sg^m)_m$ such that 
$(n_{\sg^m})_m$ is increasing. We notice that for every $m$, for $k$ sufficiently large in function of $m$, the element $\st^k$ is divisible by both 
$\sg^m$ and $\sg^{m+1}$. Thus by \cref{propdivi} $(ii)$, the word $\sg^{m+1}$ is divisible by $\sg^m$ for every $m$. 
Let $\sa_{-j(m)}^m\cdots \sa_0^m:= \sg^m$ be the $\hat \sA$ spelling of $ \sg^m$. Let us show by induction on $i$ the following hypothesis:
\begin{itemize}[$(H_i)$]\item 
For every $i\ge 0$, there exists $N(i)\ge 0$ such that  $i\le j(m)$ for every $m\ge N(i)$ and the sequence  $(\sa_{-i}^m)_{m\ge N}$ is constantly equal to a single letter $\sa_i\in \hat \sA$. 
\end{itemize}
Let us start with the case $i=0$. We observe that by \cref{divi complete}, given $k> m$, if $\sa_0^k\neq \sa_0^m$, then $\sa^k_0$ is of the form $\sa^k_0=\boxdot_\pm (\sc-\sc')$ with $\sc | \sa^m_0$ and $\sa^m_0\in  \sA_0$.  Then $\sa_0^k$ is not $ \sA_0$ and so by \cref{divi complete}, for every $j\ge k$, it holds $\sa_0^j=\sa_0^k$ .   Consequently, the sequence $(\sa_0^m)_m$ is indeed eventually constant. 

Let us assume the induction hypothesis $(H_i)$ for  $i\ge 0$. Then for $m$ large, the word $\sg^m$ is of the form $\sg'^m\cdot \sa_{-i-1}^m\cdot \sa_{-i}\cdots \sa_0$, with $\sa_{-i}\cdots \sa_0$ independent of $m$ and $\sa_{-i-1}^m\in \hat \sA$ for every $m$. By \cref{propdivi} $(iii)$, it holds that $\sg'^m\cdot \sa_{-i-1}^k$ is divisible by  $\sg'^m\cdot \sa_{-i-1}^m$ for every $k\ge m$. By proceeding as in the step $i=0$, we get that $(\sa_{-i-1}^m)_m$ is eventually constant.  This accomplishes the induction.

Therefore, we proved the existence of a 
$\st=\cdots \sa_{-i}\cdots \sa_0\in \hat \sA^{(\N)}$ such that for every $i\ge 0$, $m\ge N(i)$, the element $\st^m$ is divisible by $\sa_{-i}\cdots \sa_0$. Consequently, the sequence $(\st^m)_m$ converges to $\st$.
\end{proof}

In order to bound from above the box dimension of $\arr \sR$, we are going to endow the elements of $\hat \sA^{\Z^-}$ with favorable divisors. This is a combinatorial counterpart of the concept of favorable times in \cite{BC2}. However the purpose and the construction are here rather different. 

\begin{defi}[Favorable words and divisors]  \index{Favorable word}
A word  $\sg=\sa_1\cdots \sa_m \in \hat \sA_k^{(\N)}$ is \emph{favorable}\index{Favorable word} if it is weakly regular and its first letter letter $\sa_1$ belongs to $\hat \sA_1=\sA_0\cup\{ \boxdot_\pm (\se -\ss): \ss\in \sY_0\}$:
\[ \sa_1\in \hat \sA_1\qand  n_{\sa_j}\le 2^{M}\sum_{m=1}^{j-1}n_{\sa_m}\quad j\ge 2\; . \] 
The set of \emph{favorable divisors} $\tau_\st$ of $\st\in \hat \sA^{\Z^-}$ consists of the words $\sg\in {\hat \sA}^{(\N)}$ which are favorable and divide $\st$:
\[\tau_\st :=\{\sg\in  {\hat \sA}^{(\N)}: \st|\sg \text{ and } \sg \text{ is favorable}\}\; .\]
\end{defi}
\begin{rema}\label{Order} By \cref{propdivi} $(i)$ and $(ii)$, the set $\tau_\st$ is ordered by the divisibility relation: for every $\sg, \sg'\in \tau_\st$, it holds $\sg| \sg'$ if $n_\sg\ge n_\sg'$ and $\sg = \sg'$ if  $n_\sg= n_\sg'$.
\end{rema}

\begin{prop}\label{f-times}
Let $\st\in \hat \sA^{\Z^-}$. Then the following properties hold true:
\begin{enumerate}[$(i)$]
\item There exists a unique $\sg\in \tau_\st$ such that $\sg\in \hat \sA_1$. Note that $n_\sg\le M+1$. 
\item If $\sd\in \tau_\st$ satisfies $n_\sd\ge M+2$, then there exists $\sd'\in \tau_\st$ such that $n_\sd> n_{\sd'}\ge n_\sd/(M+2)$. 
 \item  If $\st, \st'\in  \hat \sA^{\Z^-}$, then their favorable divisors of order at most the order $\nu(\st, \st')$  of their greatest common divisor are equal: 
 $$\{\sd \in \tau_\st: n_\sd\le n_{\nu(\st, \st')}\} = \{\sd \in \tau_\st: n_\sd\le n_{\nu(\st, \st')}\}$$
\end{enumerate}
\end{prop}
\begin{proof}
Statement $(iii)$ is an immediate consequence of \cref{propdivi} $(ii)$. 

Let us prove $(i)$. Let $\sa_0$ be the last letter of $\st$. If $\sa_0$ belongs to $\hat \sA_1$, then $\sa_0$ is favorable and divides $\st$.  Otherwise, the last letter $\sa_0$ is of the form $\boxdot_\pm(\sc-\sc')$,  with $n_\sc\ge 2$ strongly regular and so complete. Thus the last letter $\sa_0'$ of $\sc$ is in $\sA_0$. Then $\st$ is divisible by $\sa_0'$.
Thus there is at least one favorable divisor in $\hat \sA_1$. As the divisor of $\st$ are ordered by the divisibility relation and as the elements of $\hat \sA_1$ are prime between each-other, there is at most one favorable divisor in $\hat \sA_1$.

Let us prove $(ii)$ by induction on $n_\sd\ge M+2$.  Let $\sd= \sa_{-j}\cdots \sa_0$ be the $\hat \sA$-spelling of $\sd$. We recall that $\sa_{-j}$ belongs to $\hat \sA_1$ and since $n_\sd\ge M+2$, we have $j\ge 1$.
We notice that $\sd_1:= \sa_{-j}\cdots \sa_{-1}$ is favorable. 
If $n_{\sd_1}\ge M+2$, then by induction there exists $\sd_1'$ with $n_{\sd_1}> n_{\sd_1'}\ge n_{\sd_1}/(M+2)$ and $\sd_1|\sd_1'$. We notice that $\sd$ is divisible by $\sd':= \sd_1'\cdot \sa_0$  and:
$$
n_{\sd'}:= n_{\sd_1'}+n_{\sa_0}\ge n_{\sd_1} /(M+2)+n_{\sa_0}> n_\sd/(M+2)
$$
If $n_{\sd_1}<M+2$ then $\sd_1$ belongs to $\hat \sA_1$ and   $n_{\sd_1}\le M+1$. 

If $\sa_0$ belongs to $\hat \sA_1$, then $\sd':= \sa_0$ is regular, divides $\sd$ and satisfies $n_{\sa_0}\ge 2$. Thus:
\[(M+2) n_{\sa_0}\ge M+2+  n_{\sa_0}\ge n_\sd= n_{\sd_1} +n_{\sa_0}\; .\]
\indent If $\sa_0$ does not belong to $\hat \sA_1$,  then $\sa_0$ is of the form $\sa_0=\boxdot_\pm(\sc-\sc')$ with $\sc$ a strongly regular word of order $\ge 2$. Thus $\sc$ is favorable and divides $\sd$. Also it holds:
\[(M+2) n_\sc\ge 2M+2+n_\sc\ge n_\sd= n_{\sd_1}+M+1+n_\sc\; .\]
\end{proof}
\begin{coro}\label{sAZ compact} The space  $(\hat \sA^{(\Z^-)}, dist)$ is compact. 
\end{coro}
\begin{proof} By \cref{finite} \cpageref{finite}, 
 the set $\{\sa \in \hat \sA: n_\sa\le m\}$ is finite for every $m\ge 1$. Thus the cardinality $C_m$ of $\{\sg \in \hat \sA^{(\N)}: n_\sg\le m\}$ is finite for every $m\ge 1$. By \cref{f-times}, the finite set $\{\st \cdot \sg \in \hat \sA^{\Z^-}: n_\sg \le (M+2) m, \st = \cdots \ss_-\cdots \ss-\}$ is $b^{m/4}$ dense, and so $(\hat \sA^{(\Z^-)}, dist)$ is totally bounded.  Together with the completeness of  $(\hat \sA^{(\Z^-)}, dist)$  given by \cref{prepa completudesR2}, this implies the sought  compactness.
\end{proof}

\subsection{Application of the divisibility to Hénon-like endomorphism}\label{section Application of the divisibility to Henon-like endomorphism}
Let $f$ be $0$-strongly regular (see   \cref{setting}). 
 We recall that $\sR$ and $\tilde\sR$  were defined in \cref{admissibleword} \cpageref{admissibleword} and $\arr \sR$ was defined in \cref{def_arr_sR} \cpageref{def_arr_sR}. 

\begin{defi}[Alphabet $\sA$] Let $\sA\subset \hat \sA$ be minimal such that $\tilde \sR\subset \sA^{(\N)}$. \index{Alphabet $\sA$ in dimension 2}\end{defi}
\begin{defi}[$T_\sa: \cD_\sa\to \cI_\sa$]\label{defi Ta}\index{$T_\sa: \cD_\sa\to \cI_\sa$}
If $\sa\in \sA_0$, we defined a map $T_\sa: S\in \mathcal H \mapsto S^\sa$ in \cref{admissibleword}. Let $ \cD_\sa:=\mathcal H$ be its domain and let  $\mathcal I_\sa$ be the image. If $\sa\in \sA\setminus \sA_0$ it is of the form $\sa=\boxdot_\pm (\sc-\sc')$ for $\sc, \sc'\in \sR$ such that the domain $\cD_\sa:= \tilde \cD(Y_\sc-Y_{\sc'})$ defined in \cref{graphtransformpara} non empty. Let $T_\sa:\;  S\in \cD_\sa\mapsto S^{\boxdot(Y_\sc-Y_{\sc'})}$ and let $\mathcal I_\sa$ be its image. Given $\sg=\sa_1\cdots \sa_m \in \sA^{(\N)}$ we denote $T_\sg:= T_{\sa_m}\circ \cdots\circ T_{\sa_1}$.
\end{defi}

An easy induction on the length of words shows the following:
\begin{fact}
If  $\sg\in \tilde \sR$ is divisible by $\sd\in \hat \sA^{(\N)}$ then $\sd$ belongs to $\sA^{(\N)}$.
\end{fact}

 As we mentioned, the metric $dist$ was designed for the following:
\begin{prop}\label{lip}
The function $\st\in \overleftarrow {\sR}\rightarrow \hat W^\st_u\in \mathcal H$ is $4\theta$-Lipschitz, for the $C^1$-distance on the space of flat stretched curves $\mathcal H$.
\end{prop} 
\begin{proof}
For every $\sg\in \sA^{(\N)}$, we   define  the subset $\cV_\sg$ of $\mathcal H$ by induction on the number $m$ of $\sA$-letters of $\sg=\sa_1\cdots \sa_m$, for every $j\ge n_\sg$: 
 \begin{enumerate}
\item If $m=0$, then $\sg=\se$ and put $\cV_\sg= \mathcal H$. 
 \item If $m\ge 1$, put $\sg= \sg'\cdot \sa$ and let  $\cV_\sg$ be the  $3\theta b^{n_\sg/3}$-neighborhood of $T_{\sa} (\cV_{\sg'} \cap \cD_\sa)$.
 \end{enumerate}
 Note that $\cV_\sg$ and $\mathcal I_\sg$ might be empty. An immediate induction shows:
\begin{fact}\label{fact1Vg}
The open set $\cV_\sg$ contains the $3\theta b^{n_\sg/3}$-neighborhood of $\mathcal I_\sg$. 
\end{fact}
We recall that the diameter of $\mathcal H$ is at most $4\theta$ and that the graph transform  $T_\sa$ is $b^{n_\sa/3}$-contracting for every $\sa\in \sA$ by propositions \ref{contractionPP} and \ref{graphtransformpara}.$(ii)$. Thus:
 \begin{enumerate}
 \item If $m=0$, the diameter of $\cV_\sg$ is at most $4\theta $.
 \item if $m>1$, the diameter of $\cV_\sg$ is at most $b^{n_\sa/3} \diam \cV_{\sg'}+3\theta b^{n_\sg /3}$.
 \end{enumerate}
As  $b^{n_\sa/3} 4\theta b^{n_{\sg'}/4}+3\theta b^{n_\sg /3}\le  4\theta b^{n_{\sg}/4}$  {when }$n_\sa=n_\sg-n_{\sg'}>0$, an induction shows:
\begin{fact}\label{fact2Vg} The diameter of $\cV_\sg$ is at most $4\theta b^{n_{\sg}/4}$.\end{fact}
In view of facts \ref{fact1Vg} and \ref{fact2Vg}, the next lemma implies  \cref{lip}.\end{proof}
\begin{lemm}  For every $\hat \sg, \sg\in \sA^{(\N)}$ such that $\hat \sg|\sg$, it holds $\cV_{\hat \sg} \subset \cV_\sg$. 
\end{lemm}
\begin{proof} We prove the lemma by induction on $n_{\hat \sg}$. If $\hat \sg\in \sY_0$, then $\hat \sg$ must be equal to $\sg$ or $\se$ and it suffices to recall that $\mathcal I_{\hat \sg}\subset \cV_{\hat \sg}\subset \mathcal H=\cV_\se$. 
If $\hat \sg$ is of the form $\boxdot_{\pm}(\sc-\sc')$, either $\sg=\hat \sg$ and we proceed likewise, or $\sc|\sg$. In this case, we recall that by \cref{graphtransformpara}.$(iii)$, the set 
$\mathcal I_{\hat \sg}$ is included in the $\theta b^{n_\sc/3}$-neighborhood of   $\mathcal I_\sc$. As $\cV_{\hat  \sg}$ is included in the $\theta b^{n_{\hat \sg}/3}$-neighborhood of $\mathcal I_{\hat \sg}$ and $3\theta b^{n_{\hat \sg}/3}<\theta b^{n_\sc/3}$, the set $\mathcal V_{\hat \sg}$ is included in the $3\theta b^{n_{\sc}/3}$-neighborhood of $\mathcal I_{\sc}$. We recall that $\cV_\sc$ contains in the $3\theta b^{n_\sc/3}$-neighborhood of $\mathcal I_\sc$ by fact \ref{fact1Vg}. Therefore  $\cV_{\hat  \sg}$ is included in $\cV_\sc $.  
As $\sc|\sg$, the set $\cV_\sc$ is included in $\cV_\sg$. Thus $\cV_{\hat \sg}$ is included in $\cV_\sg$. 
\medskip 

Now we take $\hat \sg= \hat \sg' \cdot \sa$ with $\hat \sg'\neq \se$ and $\sa\in \hat \sA$.

 If $\sa| \sg$, then by induction it holds $\cV_\sg\supset \cV_\sa$. Also $\cV_\sa$ is the $3\theta b^{n_\sa/3}$-neighborhood of $\mathcal I_\sa$, thus $\cV_\sa$ contains $\cV_{\hat \sg}$ equal to  $3\theta b^{n_\sg/3}$-neighborhood of $T_{\sa} (\cV_{\hat \sg'} \cap \cD_\sa)\subset \mathcal I_\sa$. Therefore $\cV_\sg\supset \cV_\sa\supset \cV_{\hat \sg}$. 

If $\sa\not | \sg$, then  $\sg= \sg'\cdot \sa$ and $\hat \sg'| \sg'$ by \cref{divi complete}. Then by induction 
$\cV_{\hat \sg'}\subset \cV_{\sg'}$, and so $T_\sa(\cV_{\hat \sg'}\cap \cD_\sa)\subset T_\sa(\cV_{\hat \sg}\cap \cD_\sa)$. Thus the $3\theta b^{n_{\hat \sg}/3}$-neighborhood $\cV_{\hat \sg}$ of $T_\sa(\cV_{\hat \sg'}\cap \cD_\sa)$  is contained in the $3\theta b^{n_{\sg}/3}$-neighborhood $\cV_{\sg}$ of $T_\sa(\cV_{\sg'}\cap \cD_\sa)$.
 \end{proof}     

We recall that a piece $(Y_\sg, n_\sg)$ is associated to each weakly regular word $\sg\in \tilde \sR$. In particular $\tilde \sR$ might contain words of the form $\boxdot_-(\se -\ss)=\se \boxdot_-(\se -\ss)$ for $\ss\in \sY_0$ because its order is smaller than $2^M$, without no more assumptions than the $0$-strong regularity. This might lead us to consider pieces of the form $Y_{ \boxdot_-(\se -\ss)}$. 
\begin{defi}\index{$\sF $}
Let $\sF $ be the set of favorable words $\sg$ which are in  $\tilde \sR $. 
\end{defi}

The following gives both a useful  geometric application and an interpretation of the favorable divisors:    
\begin{prop}\label{stab sRk} Let $f$ be $0$-strongly regular. Then for every $\sg\in \sR$, for every $\sb\in\hat \sA^{(\N)}$ such that 
  $\sb$ is favorable and  $\sg| \sb$, it holds $\sb\in \sF$ and $f^{n_{\sg}-n_\sb}(Y_\sg)\subset Y_\sb$. 
  \end{prop}
\begin{proof}
We proceed by induction on $n_\sg$. If $\sg$ is formed by a single letter, then $\sg\in \sA_0$, and so its has only two divisors : itself and $\se$ which both belong to $\sR\subset \tilde \sR$. Let us assume the induction hypothesis. Let $\sg=\sg'\cdot \sa$, with $\sa\in \hat \sA_k$ and $\sg'\neq \se$.
 Then we recall that the divisors of $\sg$ are the divisors $\sd$ of $\sa$ and those of the form $\sd'\cdot \sa$ with $\sd'$ a divisor of $\sg'$. 
 
If $\sa\in \hat \sA_1$, then $\sa$ is divisible only by $\sd\in \{\sa, \se\}\subset \sF$.  Note also  that  $f^{n_\sg-n_\sa}(int \, Y_\sg)\subset int\, Y_\sa$ by definition of the $\star$-product.
Let us now consider  a divisor of $\sg$ the form $\sd'\cdot \sa$ with $\sg'| \sd'$  and $\sd'$ favorable. Note that $\sd'\cdot \sa$ is favorable.  By induction, $f^{n_{\sd'}}(Y_{\sd'})\supset f^{n_{\sg'}}(Y_{\sg'})\supset f^{n_{\sg}-n_\sa}(Y_{\sg})$. Thus $f^{n_{\sd'}}(Y_{\sd'})$ intersects the interior of $Y_\sa$. Thus the $\star$-product $(Y_{\sd'}, n_{\sd'})\star (Y_\sa, n_\sa)$ is admissible and $\sd'\cdot \sa$ belongs to $\sF$. 

If $\sa\notin \hat \sA_1$.  Then $\sa$ is of the form $\sa=\boxdot _\pm (\sc-\sc')$ with $\sc, \sc'\in \sR$ strongly regular. If $\sd$ is   a favorable divisor of $\sa$ then it is  a favorable  divisor of $\sc$. Hence by induction, $\sd$ belongs to $\sF$. Moreover, the $f^{n_\sc-n_\sd}(Y_\sc)$ is a subset of $Y_\sd$. By definition of the parabolic product, it holds $f^{n_\sg-n_\sc}(Y_\sg)\subset Y_\sc$. Thus $f^{n_\sg-n_\sd}(Y_\sg)$ is a subset of $Y_\sd$. This proves the induction when $\sd$ is a divisor of $\sa$. 

Now consider the case where $\sd'\cdot \sa=\sd$ is a favorable divisor of $\sg'\cdot \sa=\sg$. By definition, $\sd'\cdot \sa$ is weakly regular and $\sg'|\sd'$. By induction $\sg'\in \sF$. It remains to show that $\sd'\cdot \sa$ belongs to $\tilde \sR$ to obtain $\sd\in \sF$. 
Let us show that $\sd'\cdot \sa$ is weakly regular and such that $\sd'\cdot \sa$ belongs to $\sF$. By \cref{Crutial prop}, it suffices to show that $\boxdot_\pm (Y_{\sc}-Y_{\sc'})$ is pre-admissible from $(Y_{\sd'}, n_{\sd'})$.  The case $\sd'=\sg'$ is obvious, let us assume $\sd'\neq \sg$.  By definition of the favorable words it holds $n_\sa\le 2^M n_{\sd'}$. By definition of the regularity of $\sg$, it holds that $S\in \mathcal H\mapsto S^{\sg'}$  has its image $\mathcal I_{\sg'}$ in $\cD(\boxdot (Y_{\sc}-Y_{\sc'})$. We recall that $diam\, \mathcal I_{\sg'}\le 4\theta b^{n_{\sg'}/3}$ and by \cref{lip},  the image of $\mathcal I_{\sd'}$ of $S\in \mathcal H\mapsto S^{\sd'}$ is in the $4\theta b^{n_{\sd'}/4}$-neighborhood of  $\mathcal I_{\sg'}$. Thus $\mathcal I_{\sd'}$ is included in the $5\theta b^{n_{\sd'}/4}$-neighborhood of $\cD(\boxdot (Y_{\sc}-Y_{\sc'})$. As $n_\sa\le 2^M n_{\sd'}$, $\mathcal I_{\sd'}$ is included in the $\theta^{n_\sa}$-neighborhood $\tilde \cD(\boxdot(Y_\sc-Y_{\sc'})$ of $\cD(\boxdot(Y_\sc-Y_{\sc'})$.  
Secondly by induction it holds $ f^{n_{\sg'}-n_{\sd'}}(Y_{\sg'})\subset Y_{\sd'}$.
  As $f^{n_{\sg'}-n_{\sd'}}(Y_\sg)\subset  f^{n_{\sg'}-n_{\sd'}}(Y_{\sg'})$ is contained in $\bigcup_{S\in \tilde \cD(Y_\sc-Y_{\sc'})} S_\sa$, it comes that
   $ f^{n_{\sg'}-n_{\sd'}}(Y_{\sd'})$ intersects the latter set at more than one arc of $W^s(A)$. 
 Consequently $\boxdot_\pm (Y_{\sc}-Y_{\sc'})$ is pre-admissible from $(Y_{\sd'}, n_{\sd'})$.
\end{proof}

Conversely, we will show that for every $\sg, \sb\in \sF$, if $f^{n_\sg-n_\sb}(Y_\sg)$ intersects the interior of $Y_\sb$ then $\sg|\sb$  in \cref{lemme_qui_torche} p. \pageref{lemme_qui_torche} 
\begin{coro} \label{f-times2}Let $f$ be 0-strongly regular.  For every $\st\in \arr \sR$, it holds $\tau_\st\subset \sF$. 
\end{coro}
The set $\sF$ will be useful to define combinatorially a subset of $\arr \sR_k$, of cardinality $\le 2^j$, which is $5 \theta b^{j/8M}$-dense and does not depend on $k\ge \max(0,j-M)$. In order to do so, let us show the following:
\begin{prop}\label{appartenance de P_j} 
If $f$ is $k$-strongly regular, then for every $\st\in \arr \sR_0$, $\sa_0\in \sA_0$  and $\sg\in \sF$ with $k =  \max(0, n_\sg-\lfloor M/2\rfloor )$, the sequence  $\st\cdot \sg$ belongs to $\arr \sR_k$.
\end{prop}
\begin{proof} Let $\sg=\sa_1\cdots \sa_m$ and $\st= \cdots \sa_{-i }\cdots \sa_{0}$.  We notice that the word $\sa_0\cdot \sg$ is regular. To show that $\st \cdot \sg$ is in $\arr \sR_k$, it suffices to prove by induction on $m$ that $\sa_0\cdot \sg$ belongs to $\sR_k$.   We recall that $\sa_0\in \sA_0$. 

Let $m=1$. If $\sa_1$ belongs to $\sA_0$, then $\sa_0\cdot \sa_1$ belongs to $\sR_0$.  If $\sa_1$ belongs to $\hat \sA_1\setminus \sA_0$, then $\sa_1$ is of the form $\boxdot_\pm(\se-\ss)$ for $\ss\in \sY_0$, $n_\sg=M+1$ and so $k= \lfloor M/2\rfloor +1$.  
Furthermore, as $\sa_1=\se\cdot \boxdot_\pm(\se-\ss)$ belongs to $\tilde \sR$, the tuple $[Y_\se, Y_\se, Y_{\ss}, T_\se =id_\mathcal H]$ is weakly admissible for the parabolic product $\boxdot_\pm$. This implies that every $S\in \mathcal H$ is $\theta^{n_\ss}$-close to be in critical position with $ Y_{\ss}$. By $1$-strong regularity,   $W_u^\st$ is in critical position with a simple piece, which must be $(Y_\ss, n_\ss)$ by uniqueness (see \cref{inclusion Yck}). By $\lfloor M/2\rfloor$-strong regularity of $f$, there exists $\sq\in \sY_0^{\lfloor M/2\rfloor }$, such that $W_u^\st$ is in critical position with  $(Y_{\ss\cdot \sq}, n_{\ss\cdot \sq})$. The boundary condition of the later position implies that any curve which is $4\theta$-close to  $W_u^\st$  is also in critical position with $(Y_\ss, n_\ss)$. Thus any horizontal curve in $\mathcal H\supset T_{\sa_0}(\mathcal H)$ is in $\cD_{\sa_1}$. This implies that $\sa_0\cdot \sa_1$ belongs to $\sR_1$. 

Let $m\ge 2$. Let $\sg= \sg'\cdot \sa_m$ . We notice that $\sg'$ belongs to $\sF$. Thus by induction $\sa_0\cdot \sg'$ belongs to $\sR_{k'}$, with  $k'=  \max(0, n_{\sg'}-\lfloor M/2\rfloor)$. 
As $(Y_{\sa_0},n_{\sa_0}) $ is a puzzle piece and $(Y_{  \sg'}, n_{ \sg'})\star (Y_{\sa_m}, n_{\sa_m})$ is admissible, the product  $(Y_{\sa_0\cdot \sg'}, n_{\sa_0\cdot\sg'})\star (Y_{\sa_m}, n_{\sa_m})$ is admissible. Thus $\sa_0\cdot \sg$ belongs to $\sR_k$. 

If $\sa_m\notin \sA_0$, then it is of the form  $\sa_m=\boxdot_\pm(\sc-\sc')$. Note that the order of $\sc$ is smaller than $k-\lfloor M/2\rfloor -1$. 
Thus the depth of $\sc'$ is smaller than $k-\lfloor M/2\rfloor $. By induction, $\st \cdot \sg'$ belongs to $\arr \sR_{k-1}$. By strong regularity, the curve $W_u^{\st\cdot \sg'}$ is in critical position with a piece of the same depth as $\sc'$.  By uniqueness, $W_u^{\st\cdot \sg'} $ is in critical position with $(Y_{\sc'}, n_{\sc'})$. Again by $k$-strong regularity, there exists $\sc''=\sc'\cdot \sq\in \sR_{k-1}$ with $depth (\sq)= \lfloor M/2\rfloor$ such that $W_u^{\st\cdot \sg'} $ is in critical position with  $(Y_{\sc''}, n_{\sc''})$. The boundary condition of the later position, implies that any curve which is $\theta^{n_{\sc''}}$-close to  $W_u^\st$  is also in critical position with $(Y_{\sc''}, n_{\sc''})$. The order $n_\sc$ and so $n_{\sc''}$ are bounded by a linear function of $n_{\sg'}$ by the weak regularity condition. 
 Thus, by $4\theta\cdot b^{n_{\sa_0\cdot \sg'}/3}$-contraction of $T_{\sa_0\cdot \sg'}$, any curve in $T_{\sa_0\cdot \sg'}(\mathcal H)$ is in $\cD_{\sa_m}$.
  As $\boxdot_\pm(Y_\sc-Y_{\sc'})$ is weakly admissible from $(Y_{\sg'}, n_{\sg'}), T_{\sg'}$ (because $\sg$ belongs to $\tilde \sR$) it comes that the tuple  $[(Y_{\sa_0\cdot \sg'}, n_{\sa_0\cdot \sg'}), (Y_{\sc}, n_\sc), (Y_{\sc'}, n_{\sc'}), T_{\sa_0\cdot \sg'}]$ is admissible for $\boxdot_\pm$. Consequently $\sa_0\cdot \sg'$ is in $\sR_k$. 
\end{proof}

In order to define $\pi_k$ and $\arr \sR_{k}'$ introduced in  \cref{defsRk'}, let us  denote by $\st_A:= \cdots \ss_-\cdots \ss_-$ the infinite sequence of $\hat \sA^{\Z^-}$ constantly equal to $\ss_-$. We notice that the curve $W_u^{\st_A}$ is the half unstable manifold
 $W^u_{1/2}(A)$ defined in Fact \ref{defi WuA}.

\begin{defi}[map $\pi_k$ and set $\arr \sR_k'$] \index{$\pi_k$}\index{$\arr \sR_k'$}
Let $\pi_k:\hat \sA^{\Z'}\mapsto \hat \sA^{\Z^-}$ be the projection which sends $\st$ to $\st_A\cdot \sg$, with $\sg$  the favorable divisor of $\st$ with maximal order $\le 2^{-M}\cdot k$: 
\[ \pi_k: \st \in \hat \sA^{\Z'}\mapsto \st_A\cdot \sg\in \hat \sA^{\Z'}\text{ with } \sg\in \tau_\st \text{ maximal s.t. }\sg \le k\cdot 2^{-M}\; .\]
Given a $0$-strongly regular Hénon-like endomorphism $f$, we put:
\[\arr \sR_k'(f):= \pi_k(\arr \sR(f))\; .\] \end{defi}

We are now able to prove \cref{defsRk'}.
\begin{proof}[Proof of \cref{defsRk'}]
The first item  is a direct consequence of \cref{appartenance de P_j}. 

Second item. If $k<M 2^{M+3}$, then $\lfloor k 2^{-M-3}/M\rfloor=0$. Thus it suffices to recall that that the diameter of $\mathcal H$ is at most $4\theta$. 
If $k\ge M 2^{M+3}$, then by Proposition \ref{f-times} $(ii)$, there exists $\sg\in \tau_t$ which is of order in $[k 2^{-M-1}/M, k 2^{-M}]$. Note that $\st_A\cdot \sg$ belongs to $\arr \sR_k'$ by \cref{appartenance de P_j}. Moreover, by \cref{dist} of $dist$, the point $\st_A
\cdot \sg$ is $b^{k 2^{-M-3}/M} $ close to $\st$. Then we conclude by using the $4\theta$-Lipschitzity of $\st' \mapsto \hat W^{\st'}_u$ stated in \cref{lip}.  

Third item. By \cref{appartenance de P_j}, the cardinality of $\arr \sR_k'$ is  at most the cardinality of $\{\sg\in \tilde \sR: n_\sg \le k\cdot 2^{-M}\}$. By  \cref{ican}, it is at most $2^{k\cdot 2^{-M}}$. 

Fourth item. If $\st \in  \arr \sR_k'(\bar f)$, then $\st$ belongs to  $\arr \sR_{\lfloor k\cdot 2^{-M}\rfloor} (\bar f)$ by the first item. Thus $\st$ belongs to  $\arr \sR_{\lfloor k\cdot 2^{-M}\rfloor} (f)\subset \arr \sR(f)$. As $\pi_k$ is a projection ($\pi_k\circ \pi_k= \pi_k$), it holds that $\pi_k(\st)=\st$. Thus $\st $ belongs to  $\sR_k'(  f)$. 
  \end{proof}
We recall that the \emph{box dimension} of $(\arr \sR_k,dist)$ is $\limsup_{\epsilon\to 0} -\log N(\epsilon)/\log\epsilon$, where $N(\epsilon)$ is the minimal number of  $\epsilon$-balls to cover $\arr \sR_k$.  By the second and third items of \cref{defsRk'}, we can take for every $k\ge 1$:
\[\log_2 \epsilon= \log_2(4\theta b)+ \lfloor k 2^{-M-3}/M\rfloor \log_2 b\qand \log_2 N(\epsilon)= 2^{-M} k\; .\]
This gives:
\begin{prop}\label{HDT*}If $f$ is strongly regular, then the box dimension of $(\arr \sR,dist)$ is at most $-\frac{8 M}{\log_2 b}$.
\end{prop}

\begin{rema} We recall that first $M$ is assumed large and then $b$ is assumed small in function of $M$. Hence the box dimension is small  in function of $b$.
\end{rema}
\begin{rema} The above estimate is very coarse. It should be possible to define a better combinatorial distance on $\arr \sR$, such that  $t\mapsto  W_u^\st$ is $1$-Lipshitz and whose Hausdorff dimension is nearly $\log 2 /\log b$.
When Yoccoz was preparing his last lecture at Coll\`ege de France with the present work \cite{Y15}, he stated definitions such that the GCD of 
two different simple pieces $\ss^i_\pm$ and $\ss^j_\pm$ of orders $i$ and $j$, 
should be $\min(i,j)-1$ and not 0 as we did. This should help to remove the factor $M$ in the above estimate on the box dimension.
  \end{rema}
\subsection{Compactification of the transversal space to the long unstable manifolds}\label{section Compactification of the transversal space to the long unstable manifolds}
Let $f$ be strongly regular. Let $\arr \sR_c$ be the set of $\st=\cdots \sa_{-m}\cdots \sa_0\in \arr \sR$ such that $\sa_0\in\sA_0$. Let $\sF_c$ be the set of $\sg=\sa_{-m}\cdots \sa_0\in \sF$ such that $\sa_0\in \sA_0$. 

\begin{defi}[Space $\sT$] \index{Space $\sT$}
Let $\sT$ be the closure of the set $\arr \sR_c$ in $\hat \sA^{\Z^-}$ endowed with the distance $dist$. 
\end{defi}
\begin{prop}
It holds:
$\sT=\{\st\in \hat \sA^{\Z^-}: \forall N\exists \sg\in \sF_c\;, n_\sg\ge N\; ,\;  \st| \sg\}\; .$
\end{prop}
\begin{proof}Let $\sT':=\{\st\in \hat \sA^{\Z^-}: \forall N\exists \sg\in \sF_c\;, n_\sg\ge N\; ,\;  \st| \sg\}$. We want to show that $\sT'$ is equal to $\sT$. 

We recall that \cref{appartenance de P_j}  implies 
that $\st\cdot \sg\in \arr \sR_c$ for every $\st\in \arr \sR_0$ and $\sg\in \sF_c$. Thus the closure of $\arr \sR_c$ contains the closure of $\sT'$. Moreover, by  \cref{f-times}, every point $\st\in \arr \sR_c$ is divisible by a favorable $\sg$ of arbitrarily large order and which belongs to $\sF$ by  \cref{stab sRk}. Moreover by \cref{divi complete}, the element $\sg$ belongs to $\sF_c$. Thus the closure of $\arr \sR_c$ is equal to  the closure of $\sT'$. 

Consequently,  it suffices to show that $\sT'$ is closed. We have:
\[\sT'= \bigcap_{N\ge 0} \{\st\in \hat \sA^{\Z^-}:\exists \sg\in \sF_c\;, n_\sg\ge N\; ,\;  \st| \sg\}\; .\]
By \cref{stab sRk} and \cref{divi complete},  this is equivalent to ask for every $m\ge 1$, the existence of $\sg\in \sF_c$ such that $n_\sg \in [m, 2Mm]$ and $\st| \sg$. Thus 
$$\sT'= \bigcap_{m\ge 1}\bigcup_{\{\sg\in \sF_c: n_\sg\in[m,2M m]\}} \{\st\in \hat \sA^{\Z^-}: \st| \sg\}\; .$$
We recall that each set $\{\st\in \hat \sA^{\Z^-}: \st| \sg\}$ is closed by \cref{prepa completudesR}. As $\{\sg\in \hat \sA^{(\N)}: n_\sg\in[m,2M m]\}$ is finite, the set  $\bigcup_{\{\sg\in \sF: n_\sg\in[m,2M m]\}} \{\st\in \hat \sA^{\Z^-}: \st| \sg\}$ is a finite union of closed set by \cref{prepa completudesR} and so $ \sT'$ is closed. 
\end{proof} 
The following associates to each element of $\sT$ a different Pesin unstable manifold. 
\begin{prop}\label{compactsT}
For every $\st\in \sT$, it holds:
\begin{enumerate}[(i)]
\item The limit $W^\st_u:=\lim_{\st'\in \arr \sR_c \to \st } W^{\st'}_u$ is a horizontal stretched curve.
\item  
The limit $\arr W^\st_u:=\lim_{\st'\in \arr \sR_c \to \st } \arr W^{\st'}_u$  is a Pesin local unstable manifold.
\item For every $\st \neq \st'$ in $\sT$, the curve  $\arr W^\st_u$ and  $\arr W^{\st'}_u$ are disjoint. 
\item The union $\bigcup_{\sT}  \arr W^\st_u $ is homeomorphic to the product  $\sT\times I_\se$. The first coordinate of this homeomorphisms is the composition 
$\arr M_f\to \R^2\to \R$  of the $0$-coordinate projection $\arr M_f\to \R^2$ with the first coordinate projection $\R^2\to \R$.
\end{enumerate}
\end{prop}
\begin{proof}
$(i)$ By \cref{rigidPP}, for every $\st'\in \arr \sR_c$, the horizontal curve $W_u^{\st'}$ is stretched. In particular it holds $W^{\st'}_u=\hat  W^{\st'}_u$.   By using the Lipschitzity of $\st' \mapsto \hat W^\st_u$ given by \cref{lip}, it comes that $W_\st^u$ is a well defined horizontal stretched curve for $\st'\to \st$. 

$(ii)$ Let $\st\in \sT$ and let $(\sg_i)_i\in \sF_\sc^\N$ such that $(n_{\sg_i})_i$ is increasing and $\st| \sg_i$ for every $i$. We notice that $\sg_{i+1}| \sg_i$ for every $i$ by \cref{Order}. 
Also $\{\st'\in \hat \sA^{\Z^-}: \st'| \sg\}$ is open by \cref{prepa completudesR}. Thus for every $i$, for every $\st'\in \sR_\sc$ close to $\st$, it holds $\st'|\sg_i$. Then by \cref{stab sRk}, the $-n_{\sg_i}$-coordinate of $\arr W^{\st'}_u$ is in $Y_{\sg_i}$. Thus it holds:
\[\arr W^\st_u:=\lim_{\st'\in \arr \sR_c \to \st } \arr W^{\st'}_u \subset \{(z_i)_i\in \arr M_f:\forall i,\;  z_{-n_{\sg_i}}\in Y_{\sg_i}\}\;.\]
Furthermore by the same proposition, for every $i$, it holds $f^{n_{\sg_{i+1}}-n_{\sg_i}}(Y_{\sg_{i+1}})\subset Y_{\sg_{i}}$. As each $\sg_i$ is complete, the piece $(Y_{\sg_i}, n_{\sg_i})$ is a puzzle piece by \cref{rigidPP}. Thus by \cref{prop pesin instable gen}, the following set is a Pesin local unstable manifold:
\begin{equation}\label{inclusion pesin unstable}
\{(z_i)_i\in \arr M_f:\forall i,\;  z_{-n_{\sg_i}}\in Y_{\sg_i}\}\;.\end{equation}
Furthermore, this set projects homeomorphically onto a horizontal stretched curve by the 0-coordinate projection. The same occurs for $\arr W^\st_u$ by continuity of the projection and the last statement of \cref{arr Rk2Wu}.
Consequently, the inclusion in \eqref{inclusion pesin unstable} is an equality. In particular, $\arr W^\st_u$ is a Pesin local unstable manifold.
 
$ (iv)$ The latter argument shows also that $\st\in \sT\mapsto \arr W^{\st}_u$ is continuous. For every $x\in I_\se$, let $w^\st(x)$ be the point of $\arr W^{\st}_u$ with $0$-coordinate in $\{x\}\times \R$.  We notice that $(\st,x)\in  \sT\times I_\se\mapsto w^\st(x)$ is continuous. As $\sT\times  I_\se$ is compact, to show that the latter map is a homeomorphism, it suffices to show $(iii)$. 
 
$(iii)$ If $\st$ and $\st'$ are different, there exists $\sg$ and $\sg'$ in $\sF$ such that $\st|\sg$, $\st'|\sg'$ and none of the words $\sg, \sg'$  divide the other. By \cref{stab sRk}, the curves $\arr W^{\st}_u$ and $\arr W^{\st'}_u$ are included in respectively $\{ (z_i)_{i\le 0}: z_{-n_\sg}\in Y_\sg\cap S\}$ and $\{ (z_i)_{i\le 0}: z_{-n_{\sg'}}\in Y_{\sg'}\cap S\}$. The two latter sets are disjoint by the next proposition.
\end{proof}
\begin{lemm}[\cite{Berentropy} Lem. 3.3
]\label{lemme_qui_torche} 
For all $\sp\in \sF$ and $\sg\in \sF$ if $f^{n'}( int\, Y_\sp)$ intersects $Y_\sg$ for $n'+n_\sg\ge n_\sp$, then $\sg=\sg'\cdot \sg''$ with $n_\sp=n'+n_{\sg'}$ and  $\sp|\sg'$. In particular, if $n'= n_\sp-n_\sg$, then $\sp|\sg$. 
 \end{lemm}
\begin{proof} If $n'=0$ then $int\, Y_\sp$ intersects $Y_\sg$ and so the proposition is given by \cref{injectif}. Assume that $n'>0$. 

We proceed by induction on the number of letters of $\sp$.
If $\sp$ is a single letter.  Then $\sp$ is in $\hat \sA_1$ and its unique return time in $Y_\se$ is $n'=n_{\sp}$. Thus $\sg'=\se$ satisfies the proposition. 

Assume that $\sp$ is formed by more than one letter. Put $\sp=\sp'\cdot  \sa$ with $\sa\in \hat \sA$.

 If $n'\le n_{\sp'}$, then we use the induction hypothesis with $\sp'$ instead of $\sp$. This implies the existence of $\sg'' $ such that $\sp'| \sg''$ and $\sg$ begins with $\sg''$. Note also that $f^{n_{\sg''}}(int\, Y_{\sg})$ intersects $f^{n_{\sp'}}(Y_\sp)$. Let $\sa'$ be the next letter of $\sg$ after $\sg''$. We want to show that $\sa'=\sa$. If $\sa\in \sA_0$, then $f^{n_{\sg''}}(int\, Y_\sg)$ intersects $Y_\sa$ and so is included in $Y_{\sa'}$. Hence,  the letter $\sa'$ cannot be parabolic and  $int\, Y_{\sa'}$ intersects $Y_\sa$. Thus  ${\sa'}=\sa$ . 
If $\sa\notin \sA_0$, then it is of the form $\sa=\boxdot_\pm(\sb-\sc)$. Thus $f^{n_{\sg''}}(Y_\sg)$ intersects $Y_\boxdot$ and so is included in $Y_{\boxdot}$. 
Thus ${\sa'}$ must be parabolic.
Put  $\sa'=\boxdot_\pm(\sb'-\sc')$.
 Moreover $f^{n_{\sg''}+n_\boxdot }(int\, Y_\sg)$ is included in $ Y_{\sb'}\setminus Y_{\sc'}$ 	and intersects $Y_\sb\setminus Y_{\sc}$. 
 By the same argument as for \cref{injectif}, it comes that ${\sa'}=\sa$. 

If $n'>n_{\sp'}$ and $\sa\in \hat \sA_1$. Since the first return time of $Y_\sa$ in $Y_\se$ is $n_{\sa}$, it comes that $n'=n_{\sp'}+n_\sa=n_\sp$ and so $\sg'=\se$ carries the lemma. 
If $n'>n_{\sp'}$ and $\sa=\boxdot_\pm(\sb-\sc)$,  then $n'\ge  n_{\sp'}+n_\boxdot$. Also $f^{n'}(Y_\sp)$ intersects $f^{n'}(Y_\sg)\subset f^{n'- n_{\sp'}-n_\boxdot}(Y_\sb)$. Thus we can apply the induction hypothesis with $\sb$ instead of $\sp$. This gives $\sg'$ s.t. $\sb|\sg'$ and so $\sp|\sg'$ because $\sp|\sb$.

Finally, if $n'=n_\sp-n_\sg$,  using $n_\sg=n_{\sg'}+ n_{\sg''}$ and $n_\sp=n'+n_{\sg'}$, it comes that $n_{\sg''}=0$, or equivalently $\sg''=\se$ and so $\sg=\sg'$. Thus we obtain $\sp|\sg$.  
\end{proof}
\section{Structure along the unstable manifolds}\label{SRB}
In this section we shall prove \cref{rigidPP} stating that complete words in $\sR$ define puzzle pieces and \cref{existenceSRB} stating that a strongly regular map leaves invariant an SRB measure which is ergodic and physical. 

\subsection{One-dimensional pieces and proof of \cref{rigidPP}}\label{section rigidPPP}
Let $f$ be $0$-strongly regular.
We are going to define one-dimensional pieces on each  horizontal curve $W^\st_u$ for $\st \in \arr \sR$.  
\begin{defi}[One-dimensional piece] \label{defi piece curve}
A \emph{one-dimensional piece} \index{one-dimensional piece in dimension 2}
$(S, n)$  is the data of an integer $n$ and a horizontal curve $S$ such that:
\begin{enumerate}[(1)]
\item  the endpoints of $S$ are transverse intersection points of $S$ with two arcs of $W^s(A)$,
\item the curve $S$  is sent by  $f^{n}$ into a horizontal, stretched curve $S'$,
\item for all $z\in S$, $w \in T_z S$ and $j\le n$, it holds$ \|D_{z}f^{n}(w)\|\ge 2^{\frac{n-j}{3}}\cdot \|D_zf^j({w})\|$.
\end{enumerate}
The piece $(I,n)$ is a \emph{puzzle piece} if $I$ is sent by $f^{n}$ \emph{ onto} the stretched curve $S'$.\end{defi}
\begin{exem}[Simple symbol admissible for a curve]\label{simple piece of a curve}
For every $\sa\in \sA_0$ and $S\in \mathcal H$, with $S_\sa:= S\cap Y_\sa$, the pair $(S_\sa, n_\sa)$ is a puzzle piece. We denote $S^\sa:=T_\sa(S)= f^{n_\sa}(S_\sa)$. We say that $\sa$ is \emph{admissible} from $S$. 
 \end{exem}
\begin{exem}[Parabolic symbol admissible for a curve]\label{para admi curve}
A parabolic symbol $\sa:=\boxdot_\pm(\sc-\sc')\in \hat \sA$ is \emph{admissible} from a horizontal curve $S$ if $\sc, \sc'$ belong to $\sR$ and $S$ belongs to $\tilde \cD(\boxdot(Y_\sc-Y_{\sc'}))$ (see def. \ref{domaine} p. \pageref{domaine}).
Put $S_{\sa}:= S_{\boxdot_\pm(Y_\sc-Y_{\sc'})}$. Also, in \cref{defi Ta}, we defined the horizontal, stretched  curve $T_\sa(S)$ which contains $f^{n_\sa}(S_\sa)$. Put $S^\sa:=T_\sa(S)$.
\begin{prop}\label{crutial for S}
Let $\sa:=\boxdot_\pm(\sc-\sc')\in \hat \sA$ such that $\sc, \sc'$ belong to $\sR$ and $S$ belongs to $\mathcal D_\sa:= \tilde \cD(\boxdot(Y_\sc-Y_{\sc'}))$. Then the  pair  $(S_\sa, n_\sa)$ is a piece.
\end{prop}
\begin{proof}
The pair $(S_\sa,n_\sa)$ satisfies properties $(2)$ and $(3)$ by respectively \cref{Crutial prop} and \cref{graphtransformpara}.$(i)$.
By \cref{remark pour la tranversalite}  the horizontal stretched curve $S$ is sent transversally to the arcs $\partial^sY_\sc \cup \partial^s Y_{\sc'}$ of $W^s(A)$, and the endpoints of $S_\sa$ are sent to two of these transverse intersection points. In particular property $(1)$ is satisfied.  
\end{proof}
\end{exem}
In \cref{defi Ta}, for $\sg\in \sA^{(\N)}$, we defined a subset $ \cD_\sg\subset \mathcal H$ and a graph transform $T_\sg: S\in \cD_\sg\to S^\sg\in \mathcal H$. 
\begin{defi}[Admissible word  from a curve]\label{admissiblewordfromcurve}
A word $\sg=\sa_1\cdots \sa_m\in \hat \sA^{(\N)}$ is admissible from a horizontal stretched curve  $S$ if $S\in D_\sg$ and 
the set:
\[S_\sg := (f^{n_{\sa_1}}| S_{\sa_1})^{-1}\circ (f^{n_{\sa_2}}|S^{\sa_1}_{\sa_2})^{-1} \cdots (f^{n_{\sa_m}}| S_{\sa_m}^{\sa_1\cdots \sa_{m-1}})^{-1}(S^\sg)\; \]
 is a segment of $S$ which contains more than one point. 
\end{defi}
We note that every $\sg\in \tilde \sR$ is admissible from every $S\in \mathcal H$.

By \cref{simple piece of a curve} and \cref{crutial for S}, it holds:
\begin{prop}\label{admissiblewordfromcurve3}
If $\sg$ is admissible from a horizontal stretched curve  $S$, then $(S_\sg, n_\sg)$ is a piece.
\end{prop}
\begin{prop}\label{prime are disjoint}
If $\sg$ and ${\sg'}$ are prime, complete words in $\hat \sA^{(\N)}$  which are admissible from $S\in \mathcal H$, then either $\sg={\sg'}$ or the interior of $S_\sg$ and $S_{{\sg'}}$ are disjoint.
\end{prop}
\begin{proof}
We proceed by induction on the number of letters of $\sg$. This enables us to assume $\sg$ prime.

 If $\sg$ is formed by a unique letter, then it belongs to $ \sA_0$ because $\sg$ is complete. Then $S_\sg\subset Y_\sg$ is disjoint from the interior of $\bigcup_{\sA_0\setminus \{\sg\}} Y_\ss\cup Y_\boxdot$. Thus $\sg'$ cannot begin with a parabolic symbol (since this implies $S_{\sg'}\subset Y_\boxdot$) nor a letter in $\sA_0\setminus \{\sg\}$. Hence the first letter of ${\sg'}$ must be $\sg$. As ${\sg'}$ is complete and prime, it is equal to $\sg$. 

Let us assume the induction hypothesis. Put $\sg= \sa\cdot \sg_0$ and ${\sg'}= \sa'\cdot {\sg'}_0$, with $\sa, \sa'\in \hat \sA$. If $\sa$ belongs to $\sA_0$, then $\sg$ is formed by a unique letter (because it is prime and complete); we already carried this case. Otherwise, $\sa$ is of the form $\sa= \boxdot_\pm (\sc-\sc\cdot \sb)$, with $\sc, \sc\cdot \sb\in \sR$. 
Likewise $\sa'$ is of the form $\sa'= \boxdot_\pm (\sc'-\sc'\cdot \sb')$, with $\sc, \sc\cdot \sb\in \sR$. 
By \cref{inclusion Yck} p. \pageref{inclusion Yck}, the boxes $Y_\sc$, $Y_{\sc'}$, $Y_{\sc\cdot \sb}$ and $Y_{\sc'\cdot \sb'}$ must be nested. By proceeding exactly as in the proof of \cref{injectif}, We obtain that $\sc=\sc'$, $\sc\cdot \sb= \sc\cdot \sb$ and 
$\sa= \sa'$. Then, by induction, we conclude that $\sg={\sg'}$. 
\end{proof}
The following will imply \cref{rigidPP}:
 \begin{prop}\label{coro a ques} If $\sg$ is a complete word and admissible from $S\in \mathcal H$, then $(S_\sg, n_\sg)$ is a puzzle piece.
 \end{prop}
 \begin{proof} We proceed by induction on the number of letters of $\sg$. This enables us to assume $\sg$ prime. If $\sg$ is formed by a single letter then it belongs to $\sA_0$, and we already saw in \cref{simple piece of a curve} that it defines a puzzle piece of $S$. 
 
 If $\sg= \sa\cdot {\sg'}$ with $\sg'\neq \se$, then $\sa$ cannot be in $\sA_0$ because $\sg$ is prime and complete. Thus $\sa$ is of the form $\sa= \boxdot_\pm (\sc-\sc\cdot \sb)$ where $\sb$ is a prime complete word (by definition of the strongly regular words \ref{defi Strongly regular words}). 
 
To prove the proposition it suffices to show that $S^\sa_{\sg'}$ is included in $f^{n_\sa}(S_a)$. 
 
 By \cref{graphtransformpara}.$(iii)$,  the curves $S^\sa$ and $S^{ \sc}$ are $\theta b^{n_\sc/3}$-close. We recall that $\sc\cdot \sb$ belongs to $\sR$ by \cref{admissibleword}$(ii)$. Let $y_0$ be the $y$-coordinate of the intersection point of $f(S_\sa)$ with $Y_{\sr\cdot \sc\cdot \sb}$. We can apply the following lemma with  $S'=S^\sa$, $\sp=\sc$,  $\sq=\sb$ and $\tilde S= f^{n_{\sr\cdot \sc}}(\R\times \{y_0\}\cap Y_{\sr \cdot \sc})$:
\begin{lemm}\label{lemme tres util}
For every $\sp\cdot \sq\in \sR$ and $\tilde S,S'\in \mathcal H$ which are  $\theta b^{n_{\sp/3}}$-close and  such that $\tilde S$ is of the form $\tilde S=S''^\sp$ for $S''\in \mathcal H$.  
Then the words $\sq$ is admissible from $S'$.  Moreover,
if the left (resp. the right) endpoint of $\tilde S_\sq$ belongs to $S'$ then it is the left (resp. right) endpoint of $S'_\sq$.   
 \end{lemm}
This lemma will be shown below. It implies that  $\sb$ is admissible from $S^\sa$. Let  $C$ be the right component of $\partial^s Y_{\sr\cdot \sc\cdot \sb}$. 
Let $z_0$ be the intersection point of $C$ with  $f(S_\sa)$ and let  $z_1:=f^{M+n_\sc}(z_0)$. Then $z_1$ is an endpoint of $f^{n_\sa}(S_a)$. Also by construction of $\tilde S$, $z_1$ is an endpoint of $\tilde S_\sb$.  Thus by the second part of the lemma,  $z_1$ is an endpoint of $S^\sa_\sb$. Note that furthermore, $f^{n_\sa}(S_a)$  and  $S^\sa_\sb$ are at different sides of $z_1$. 
Consequently, the interior of $f^{n_\sa}(S_a)$ is a component of 
$S^\sa\setminus S^\sa_\sb$. As $\sa\cdot \sg'$ is admissible from $S$, the segment $S^\sa_{\sg'}$ intersects the interior of  $f^{n_\sa}(S_a)$ and so the complement of $S_\sb^\sa$. Thus by \cref{prime are disjoint} applied to the prime words $\sg'$ and $\sb$,  the segment  $S^\sa_{{\sg'}}$ is disjoint from from the interior of $S^\sa_\sb$ and so it is included in $f^{n_\sa}(S_a)$.
\end{proof}
 \begin{proof}[Proof of \cref{lemme tres util}]  
 We are going to show by induction on the number $j$ of $\hat \sA$-letters in $\sq=\sq_1\cdots \sq_j$ that $\sq$ is admissible from $S'$ and  such that:
 \begin{enumerate}[$(i)$]
 \item for each endpoint $z_1$ of $f^{n_\sq}(\tilde S_\sq)$ there is an endpoint $z_1'$ of $f^{n_\sq}(S'_\sq)$ which is linked to $z_1$ by an arc of $W^s(A)$. 
 \item if an endpoint of $\tilde S_\sq$ belongs to $S'$, then it is also an endpoint of $S'_\sq$.
  \end{enumerate}
We notice that  $(ii)$ implies the last statement of \cref{lemme tres util}. 
If $j=0$, we have $\sq=\se$ and the induction hypothesis is obvious. 
Let $j>0$ and assume the induction hypothesis is valid for $j-1$ and put $\sq= \sg\cdot \sa$ with $\sa\in \hat \sA$.
Put:
\[W:= \tilde S^{\sg}, \quad W':= S'^{\sg}, \quad \breve W:= f^{n_{\sg}}(\tilde S_{\sg})\qand \breve W':= f^{n_{\sg}}(S'_{\sg})\; .\]
 By $(i)$ at step $j-1$, the endpoints of $\breve W$ and $\breve W'$ are linked by arcs of $W^s(A)$. 

If $\sa\in \sA_0$, then $W$ and $W'$ belong to the domain $\mathcal D_\sa=\mathcal H$. Also the interior of $Y_\sa$ intersects $\breve W$. 
Furthermore,  $\partial^s Y_{\sa}$ is formed by two arcs of $W^s(A)$ which are vertical, stretched curves. Thus by coherence of the arcs and $(i)$, the interior of $Y_\sa$ intersects $\breve W'$. Thus $\sq$ is admissible from $S'$. Let us prove $(i)$ and $(ii)$ at step $j$.

If $\partial^s Y_{\sa}$ does not intersect $\breve W$, then $\breve W$   is included in $Y_\sa$ and the same occurs for $\breve W'$.
Thus $\tilde S_\sq=\tilde S_{\sg}$ and $S'_\sq=S'_{\sg}$ and so $(ii)$ holds true by induction. Furthermore, 
$f^{n_\sq}(\tilde S_\sq)=f^{n_\sa}(W)$ and $f^{n_\sq}(\tilde S_\sq)=f^{n_\sa}(W)$ and so $(i)$ holds with the image by $f^{n_\sa}$ of the arcs given by $(i)$ at step $j-1$. 

If a component $C$ of $\partial^s Y_{\sa}$  intersects $\breve W$, then this occurs at a unique point $z$. Also the same occurs   with  $\breve W'$ for a unique point $z'$. 
Furthermore the preimage $z_0$ by $f^{n_{\sg}}|\tilde S_{\sg}$ of $z$ is an endpoint  $\tilde S_{\sq}$ and the preimage  $z_0'$ by $f^{n_{\sg}}|S'_{\sg}$ of $z'$ is an endpoint  $S'_{\sq}$.
 If $z_0$ is in $S'$, then by definition of the simple and parabolic operation, $z$ is in $W'$. Thus by uniqueness of the intersection point of $C$ with $W'$, it holds $z=z'$, and by uniqueness of the preimage $z_0=z_0'$. This prove $(ii)$. 
Also the image $z_1=f^{n_\sa}(z)$ and  $z_1'=f^{n_\sa}(z')$
are endpoints of respectively $f^{n_\sq}(\tilde S_\sq)$ and $f^{n_\sq}(S'_\sq)$. We notice that $z_1$ and $z_1'$ are linked by the arc $f^{n_\sa}(C)$. 
\medskip 

If $\sa\notin \sA_0$, then it is of the form $\sa=\boxdot(\sc-\sd)$. 
 By definition of $\sR\ni \sg\cdot \sa$, the curve $W=\tilde S^{\sg}$ belongs to  $\cD(\boxdot_\pm(Y_\sc-Y_\sd)$ and $n_\sa\le 2^M n_{\sp\cdot \sg}$.  Thus $\theta^{n_\sa}\ge \theta b^{(n_p+n_\sg)/3}$. Also the distance between  $W$ and $W'$ is at most $\theta b^{(n_p+n_\sg)/3}$ by assumption of the lemma and $b^{n_\sg/3}$-contraction of $T_\sg$. 
   Thus $W'$ belongs to $D_\sa=\tilde \cD(\boxdot_\pm(Y_\sc-Y_\sd)\Supset \cD(\boxdot_\pm(Y_\sc-Y_\sd)$.
As $\sq$ is admissible from $\tilde S$, the segment $W_\sa$ intersects the interior of $\breve W$. To show that $\sq$ is admissible from $S'$ is suffices to show that 
   $W'_\sa$ intersects the interior of $\breve W'$. 
To see this, we proceed as in the case $\sa\in \sA_0$: 
the endpoints of $f(W_\sa)$ and $f(W_\sa')$ are pairwise linked by arcs of $W^s(A)$ which are vertical stretched curves and the endpoints of $f(\breve W)$ and $f(\breve W')$ are pairwise linked by arcs of $W^s(A)$ by the induction hypothesis $(i)$. The coherence of the arcs of $W^s (A)$ implies that $\sq$ is admissible from $S'$. 
Furthermore, the same discussion with the box 
  $cl(Y_\sc\setminus Y_\sd)$ instead of $Y_\sa$, and the curves $f(W_\sa)$, $f(W_\sa')$, $f(\breve W)$ and $f(\breve W')$ instead of respectively  of $W_\sa$, $W_\sa'$, $\breve W$ and $ \breve W'$ in the case $\sa\in \sA_0$ shows $(i)$ and $(ii)$. 
 \end{proof}

\begin{proof}[Proof of Proposition \ref{rigidPP}]
\label{rigidPP_proof}
Let us show that $\sg\in \sR$ is complete iff $(Y_\sg,n_\sg)$ is a puzzle piece. If $\sg$ is not complete, then $\sg$ finishes by a symbol of the form $\boxdot_\pm (\sc-\sc')$, with $\sc, \sc'\in \sR$. Thus $f^{n_\sg}(Y_\sg)$ is included in $f^{n_\sc}(Y_\sc\setminus Y_{\sc'})$ which does not connect the two components of $\partial^s Y_\se$. Thus $(Y_\sg, n_\sg)$ cannot be a puzzle piece.

Let $\sg\in \sR$ be complete. Then by \cref{coro a ques}, for every $S\in \mathcal H$, $f^{n_\sg}(\partial S_\sg)\subset \partial^s Y_\se$. Consequently:
\[f^{n_\sg}(\partial^s Y_\sg) = \bigcup_{S\in \mathcal H}f^{n_\sg}(\partial S_\sg)\subset \partial^s Y_\se\; \]
and so $(Y_\sg, n_\sg)$ is a puzzle piece.
\end{proof}

\subsection{Distortion bound along the unstable manifold}\label{section pour la distortion}
The following combinatorial  property will enable us to display  nice distortion bounds:
\begin{defi}\index{Perfect}\label{perfect} A word $\sg:= \sa_1\cdots \sa_m\in \hat \sA^{(\N)}$ is \emph{perfect} if for every $i$ such that $\sa_i\notin \sY_0$, it holds:
\[\sum_{j=i+1}^m n_{\sa_j}\ge  16 n_{\sa_i}\; .\]
We denote by $\sP$ the set of perfect words. \index{$\sP$}
\end{defi}
We notice that a perfect word $\sg$ is complete. Thus if it is admissible from a horizontal curve $S$, then $(S_\sg,n_\sg)$ is a puzzle piece by \cref{coro a ques}.  Note that if $\sa_1\cdots \sa_m$ is perfect, then $\sa_j\cdots \sa_m$ is perfect for every $j\in [1,m]$. We will prove:

\begin{prop}\label{distparfaitcommon} 
Let $\sg\in \sP$ be a perfect word which is admissible from a horizontal curve $S$.  Then the pair $(S_{\sg}, n_{\sg})$ is a puzzle piece and for every $z,z'\in S_\sg$, it holds:
 
\[\log \frac{ \|D f^{n_\sa}|T_z S\|}{\|D f^{n_\sa}|T_{z'} S\| 
}\le 40\; .\]
\end{prop}
The proof that $(S_{\sg}, n_{\sg})$ is a puzzle piece was given in \cref{coro a ques}, the distortion bound will be proved in \cref{distparfaitcommon2} \cpageref{distparfaitcommon2}. Given $\st\in \sT$ we denote by $\leb$ the Lebesgue measure on $W^\st_u$. 
To prove the existence of an SRB measure, we shall bound the Lebesgue measure of the set $\mathcal E^j(\st)$ formed by the points $z\in  W_u^{\st}$ which do not belong to a piece defined by $\sP$ and of order $\le j$:
\[\mathcal E^j(\st):= W_u^{\st}\setminus \bigcup \{(W^\st_u)_{\sg}:\;  \sg\in \sP  \text{ is admissible from } W_u^{\st} \text{  s.t. } n_{\sg}\le j\}\; .\]
\begin{prop}\label{heart} Let $f$ be strongly regular. 
For every $\st\in \arr \sR$, and $M\le j$, the following estimate holds:
\[\frac 1j \log_2 \leb\, \mathcal E^j(\st)\le  -\frac 1{52}\; .\]
\end{prop}
We will prove this proposition at the end of this part. 

For every $\st\in \sT$, let $\sP(\st)\subset \hat \sA^{(\N)}$ be the set perfect words $\sa_1\cdots \sa_m$  which are admissible from $W^\st_u$ and such that $\sa_1\cdots \sa_j$ is not perfect for every $j<m$. We notice that for every $\sg\neq \sg'\in \sP(\st)$, the segments $(W^\st_u)_\sg$ and $(W^{\st'}_u)_\sg$ have disjoint interiors. 

\begin{prop}\label{semi-continuous} Let $f$ be strongly regular. Then 
 the following map is lower semi-continuous:
\begin{equation*} N: (\st, z)\in \bigcup_{\st \in \sT}   \{\st \}\times W^\st_u \mapsto \left\{
\begin{array}{cl}
n_\sg  & \text{if }  z\in  int\, (W^\st_u)_{\sg} \text{ for } \sg\in \sP(t). \\
0 &\mathrm{otherwise \;}.
 \end{array}\right.\end{equation*}
Moreover, there exists $L$ such that for every $\st\in \sT$, 
\[\int_{z\in W^\st_u}N(\st, z)\leb\le L\; .\]
\end{prop}
\begin{proof}
The semi-continuity along each horizontal curve is easy.  The semi-continuity transversally to these curves is more tricky.
Using the same argument as \cref{SRkimpliesSRinfty}, 
based on the continuity of $\st\in \sT \mapsto W^\st_u$ and the compactness of $K^u$ (see \cref{SRktoSR}), the curve  $W^\st_u$ satisfies $(SR_1-SR_2 -SR_3)$, for every $\st \in \sT$. 
Moreover by \cref{lemma same piece}, for every $n\ge 0$, there exists $\eta>0$ such that for every $\st, \st'\in \sT$, with $d(\st, \st')<\eta$, it holds:
\[\{\sg\in \hat \sA^{(\N)}: \sg \text{ admissible from }\hat W^\st_u\}= 
\{\sg\in \hat \sA^{(\N)}: \sg \text{ admissible from }\hat W^{\st'}_u\}\; .\]
and so:
$$\{\sg \in \sP(\st):n_\sg\le n \}=\{\sg \in \sP(\st'):n_\sg\le n \}\; .$$
An easy induction shows also that if $\sg$ is admissible from $S$ and $S'\in \mathcal H$ which are close, then $S_\sg$ and $S_{\sg'}$ are close. 
This implies that $N$ is lower semi-continuous. Furthermore, by  using \cref{heart}, we obtain the second part of the proposition:


\begin{equation}\label{unequattriemmenum}
\leb\, \{z: N(\st,z)> j\}=  \\leb\, \mathcal E^j(\st)\le
\lim_{n\to \infty}  \leb\, \mathcal E^j(\st_n)\le 2^{- j/50}
.\end{equation}
\end{proof}
\subsection{Construction of the SRB measure}\label{Construction SRB}
Let us consider the map:
 \[\begin{array}{ccl}
\arr F:
\bigcup_{\st \in \sT} \arr  W^\st_u   &\rightarrow &\quad \bigcup_{\st \in \sT} \arr W^\st_u \\
&&\\
\arr z=(z_i)_{i\le 0}&\mapsto &\left\{\begin{array}{cl}
\arr f^{n_\sg}(\arr z)&\mathrm{if}\; z_0\in int\, (W^\st_u)_\sg ,\; \sg\in \sP(\st),\\
                         \arr z & \mathrm{otherwise.}\end{array}\right.\end{array}\]
We remark that $\arr F$ is measurable, since it is the continuous map $\arr f$ at the measurable return time function $N$ defined in  \cref{semi-continuous}.
\begin{defi}[Same first itinerary]
For every $\st\in \sT$, we say that $\arr z=(z_i)_i$ and $\arr  z'=(z_i')_i $ in $\arr W^\st_u$ have \emph{the same first itinerary} if there exists $\sg\in \sP(\st)$ such that $z_0$ and $z'_0$ belong  to $(W^\st_u)_\sg$. They have the same $j\ge 1$-first itinerary if furthermore $\arr F(\arr z)$ and $\arr F(\arr  z')$ in $\arr W^{\st\cdot \sg}_u$ have the same $j-1$-first itinerary. \end{defi}

We identify $\bigcup_{\st \in \sT} \arr  W^\st_u $ to $\sT\times I_\se$ via the  homeomorphism given by \ref{compactsT}.$(iv)$.  A direct consequence of the distortion bound of \cref{distparfaitcommon} and the expansion property of \cref{defi piece curve} is:

\begin{prop}\label{distortionborne} The map $\arr F$ is uniformly expanding and has bounded distortion along unstable leaves:\begin{itemize}
\item[(1)] $|\partial_{x} (\arr F^k)	(\st, x)|\ge 2^{\frac{k}{3}}$, for every $k\ge 1$, $\st\in \sT$ and Leb. a.e. $x\in I_\se$. 
\item[(2)] $|\partial_x (\arr F^k)(\st, x)|/|\partial_x (\arr F^k)	( \st, x')|\le 40$ for any $k\ge 1$, $( \st, x), ( \st, x')\in \sT\times I_\se\approx \bigcup_{\st \in \sT} \arr  W^\st_u $ with the same $k$-first itinerary. 
\end{itemize} \end{prop}
We are now ready to show that every strongly regular dynamics leaves invariant an SRB probability measure.
\begin{proof}[Proof of \cref{existenceSRB}]We are going to proceed as in \cite{BV}. 
Let $\mathcal M$ be the $\sigma$-algebra on $\sT\times I_\se$ generated by the products $A\times B$ of Borel sets $A\subset \sT$ and $B\subset  I_\se$. Given a Borel measure $\nu$ on $\sT\times I_\se$, let 
$\hat \nu\otimes \leb$ be the measure on $\mathcal M$ defined by:
\[(\hat \nu\otimes \leb )(A,B)= \hat \nu(A)\cdot \leb(B) ,\]
where $\hat \nu$ is the push forward of $\nu$ via the projection $\sT\times I_\se\to \sT$. 

The existence of an SRB measure for $f$ will follow from the existence of an $\arr F$-invariant measure $\mu$ on $\sT\times I_\se$ of the form $\mu=\rho \cdot m\otimes \hat \mu$, with $\rho$ bounded. 
Let us fix any $\st\in\sT$. Let $\leb_\st$ be the product $\delta_\st\otimes \leb$, with $\delta_\st$ the Dirac measure of $\sT$ at $\st$. Let us recall the following lemma:

\begin{lemm}[Lemma 4.9 in \cite{BV}] There exists $K_0>0$ such that given any $\st\in\sT$, the sequence $\lambda_n=\arr F^n_* \leb_\st$ of push forward measures satisfies:
\[\lambda_n(A\times B)\le K_0 \cdot \hat \lambda_n(A)\cdot \leb(B),\]
for every $n\ge 1$, $A\times B\in \mathcal M$.
\end{lemm}
Since the measurable sets $A\times B$ generate the $\sigma$-algebra $\mathcal M$, the measure $\lambda_n$ is absolutely continuous along the unstable leaves, with Radon Nikodyn density $\rho_n$ bounded by $K_0$. Moreover the same is true for the sequence:
\[\mu_n:= \frac{1}{n} \sum_{j=0}^{n-1} \lambda_j= \frac{1}{n}\sum_{j=0}^{n-1} \arr F^n_* \leb_\st,\]
because $\mu_n(A\times B)=\frac1n\sum_j \lambda_j(A\times B)\le \frac{K_0 \leb(A)}{n} \sum_j \hat \lambda_j(B)=K_0 \leb(A) \hat \mu_n(B)$.

As $\arr F^n_* \mu_n-\mu_n=\frac{1}{n}(\lambda_n-\lambda_0)$, any accumulation point $\mu$ of $\mu_n$ is a $\arr F$-invariant measure. Such an accumulation point exists by compactness of the space of the probability measures. The absolute continuity of $\mu$ is shown in Corollary 4.10 of \cite{BV}.

We fix such an $\arr F$-invariant, absolutely continuous measure $\mu$.
We recall that by \cref{semi-continuous}, the integral $\int_{W_u^\st}  N d\leb $ is uniformly bounded among $\st \in \sT$.  Thus by absolute continuity, the  integral 
 $\int_{\bigcup_\sT\arr W^\st_u} N d\mu  = \sum_{0}^\infty \mu (\{N>j\})$ is finite. 

\[\mathrm{Let}\quad \arr \mu:= \frac{ \sum_{0}^\infty \arr f_*^j(\mu|N>j)}{\sum_{0}^\infty \mu (\{N>j\})}.\]

We remark that $\arr \mu $ is $\arr f$-invariant. Furthermore $\arr\mu$ restricted to $\bigcup\arr W^\st_u\approx \sT\times I_\se$ is absolutely continuous.  In particular $\arr \mu$ is an SRB measure.

Let $\nu$ be the push forward of $\arr \mu$ by the zero-coordinate projection $(z_i)_i\mapsto z_0$. The measure $\nu $ is invariant by $f$.  It is well known that a hyperbolic SRB measure of a $C^{1+\alpha}$-diffeomorphism is physical. As we deal with endomorphisms, we shall re-prove the physicallity of $\mu$ in our case.

\bigskip

\noindent{\bf Proof that $\nu$ is ergodic and physical}. 
The proof will use the map $\pi_f$ defined in \cref{Pesin stable1} \cpageref{Pesin stable1}.

Let $E$ be an $\arr f$-invariant set of $\arr \mu$ measure positive. By construction of $\arr \mu$, the set $E$ intersects $\bigcup_{\st\in \sT}\arr W^u_\st$ at a set of positive measure. For every $\st\in \sT$, let $E^\st=E\cap \arr W^\st_u$.  We identify $E^\st$ with its 0-coordinate projection in $W^\st_u$. 
\begin{lemm}\label{preergo}
There exists a sequence $(\st_i)_i\in \sT^\mathbb N$ such that $\leb(\arr W_u^{\st_i}\setminus E^{\st_i})\rightarrow 0$.
\end{lemm}
\begin{proof} Let $\st\in \sT$ be such that $E^{\st}$ has positive Lebesgue measure. By (\ref{unequattriemmenum}) p. \pageref{unequattriemmenum}, there exists a density point $\arr z\in E^{\st}$ and  a sequence of perfect, complete words $(\sg_i)_i$ admissible from $W_u^\st$ and such that $\{\arr z\}= \cap_{n\ge 0} (W_u^\st)_{\sg_n}$ and:
\[\lim_{n\to \infty}\frac{\leb [ (W_u^\st)_{\sg_n}\cap E^\st]}{\leb    (W_u^\st)_{\sg_n}}= 1\; .\] 
We remark that  $E^{\st\cdot \sg_i}\supset  f^{n_{\sg_i}}(E^\st\cap (W_u^\st)_{\sg_i})$ and so the sequence $(\st_i)_i\in \sT^\mathbb N$ is convenient by the distortion bound of the perfect pieces.
\end{proof}
By this lemma, there exists $\st\in \sT$ such that $\pi_f(E^{\st})$ is arbitrarily closed to be full in $I_\se$ in the sens of $\leb$. As $\breve K^u(f)\subset I_\se$ is nearly full by \cref{geo intuition para selec}.2, the subset $F:=\pi_f(E^{\st})\cap \breve K^u(f)\subset I_\se$ is near to be full by the Lipschitz property of $\pi_f$ stated in  \cref{Pesin stable1}.1.  

Let us apply this reasoning with $E$ an union of ergodic components and $E'$ another union of ergodic components of $\arr \mu$. Then there are points in $E$ and $E'$ which belongs to the same stable manifold $W^s_\sc$ for a certain $\sc\in \overrightarrow\sR$.   By \cref{Pesin stable1}.4, all the points in $W^s_\sc$ are asymptotic and so $E$ and $E'$ must have an ergodic component in common. A contradiction. Therefore, $\arr \mu$ is ergodic. 

Now let us apply this reasoning for $E$ equal to the set of points which are $\arr \mu$-generic. By Lipschitzness of the lamination $\bigcup_{\sc\in \overrightarrow \sR} W^s_\sc$ stated in propositions \ref{Pesin stable1}.1, the preimage by $\hat F:=\pi_f^{-1}(F)\subset Y_\se$  has   Lebesgue measure close to the one of $Y_\se$. By \cref{Pesin stable1}.4, each of the points in $\hat F$ has its orbit which is asymptotic to one of a point in $E^\st$ and so which is generic for $\nu$. Thus the basin of $\nu$ contains the set  $\hat F$ of positive Lebesgue measure, and so $\nu$ is physical. 
\end{proof}
\begin{proof}[Proof of \cref{heart}]
For every $\st\in \arr \sR$ and $1\le i\le j$, let $\mathcal E^{i,j}(\st)$ be the subset of points in $W_u^\st$ which do not belong to a puzzle piece defined by $\sg \in \sP$ of order  $n_\sg $ in $[i,j]$. We note that  $\mathcal E^{j}(\st)=\mathcal E^{1,j}(\st)$. Put:
\[P_i^j:= \sup_{t\in \arr \sR}\leb\big(\mathcal E^{i,j}(\st)\big)\; .\] It suffices to show by induction on $j$ that $P^j_1\le 2^{-j/52}$ for every $j\ge M$.
 
For $j= M-1$, the Proposition is a consequence of Lemma \ref{propestimates1} \cpageref{propestimates1}. Furthermore by the same lemma, there exists a universal constant $C>0$ such that:
\[P^j_1\le C\cdot 2^{-j}\le C\cdot 2^{-j/52}, \quad \forall j\le M-1\, .\]
Let $j>M$.  Let us remark the following formula:
\begin{equation}\label{P^i_j}
P^j_{i+1}\le P_i^j +  40\cdot  P_1^{j-i}\; .\end{equation}

Indeed if $z\in \mathcal E^{i+1,j}(\st)\setminus \mathcal E^{i,j}(\st)$ then the point $z$ belongs to a piece defined by $\sg'\in \sP$ such that $n_\sg=i$ and $f^i(z)$ does not belong to the support of a perfect word of order in $[1,j-i]$. The distortion estimate of \cref{distparfaitcommon}  gives formula (\ref{P^i_j}).
 
Using inductively (\ref{P^i_j}) it comes:
\begin{equation} P_i^j\le 40 \sum_{l=j-i+1}^j P_1^l\; .\end{equation}
The induction hypothesis at step $j'<j$ gives:
\begin{equation}\label{eq5} P_i^{j'}\le 40 C\sum_{l=j-i+1}^j 2^{-l/52} \le 40\cdot  C \frac{2^{-(j'-i+1)/52}}{1-2^{-\frac 1{52}}}\; .\end{equation}

Let $\sa_{-1}^\pm := \ss_M^\pm$ and $n_{-1}=M$. 
For $m\ge 0$,  let $\sa_m^\pm:= \square_\pm (\sc_m^\st-c_{m+1}^\st)$ and let $n_m$ be its order. Let $l\ge 0$ be maximal such that $n_l\le \frac{j}{16}$.
Let $\epsilon_{l+1}$ be the Lebesgue measure of the subset of points in $z\in W^\st_u\cap Y_\boxdot $ such that $f^{n_\boxdot}(z)$ is in $Y_{\sc_{l+1}}$ : 
 \[\epsilon_{l+1}:=\leb \left(W^\st_u\cap (f^{n_\boxdot} |Y_\boxdot)^{-1}(Y_{\sc_{l+1}})\right)\; .\]
 By the horizontal expansion along the parabolic pieces and the pieces in $\ss_M^\pm$,  we have:
\begin{equation} \leb\, \mathcal E^j(\st)\le 2\sum_{m=-1}^l 
2^{-\frac{n_{m}}{3} }
P_{16\cdot  n_{m} }^{j-{n_{m}}}+\epsilon_{l+1}\; .\end{equation} 
Using (\ref{eq5}) with $j'=j-n_{m}$, it comes: 
\begin{equation} P_1^j\le \frac{80\cdot C}{1-2^{-\frac 1{52}}} \sum_{m=-1}^l  2^{-\frac{n_{{m}}}{3} } 
 2^{- (j- 17\cdot n_{m}+1)/52}
+\epsilon_{l+1} \; ,\end{equation}  
\begin{equation} 
P_1^j\le \frac{80\cdot C\cdot   2^{-\frac j{52}}}{1-2^{-\frac 1{52}}} \sum_{k\ge M }^l   
 2^{(\frac{17}{52}-\frac13)k}
+\epsilon_{l+1} \le C'  2^{-\frac j{52}} 2^{(\frac{17}{52}-\frac13) M}
+\epsilon_{l+1} \; ,\end{equation}  
where $C'$ is a universal constant. Therefore:
\begin{equation}P_1^j \le \frac{1}{2} 2^{- j /52} +\epsilon_{l+1} \; .\end{equation}  
 As  $\pi_f\circ f|W_u^\st$ is equal to the composition of a map $C^1$-close to identity with $(x,y)\mapsto x^2+a$ and then $Lip$-map close to the first coordinate projection, the bound $\epsilon_{l+1}$ is smaller than the square root of 
$2 \cdot 2^{-(2M+ n_{\sc_{l+1}})(1-1/\sqrt M)}$ by \cref{Expansion SR piece} \cpageref{Expansion SR piece}.
Then we infer that  $n_{l+1}:=n_{\sc_{l+1}}+M+1$ is greater than $j/16$ because $n_l$ was maximal with the property $n_l\le j/16$. Thus it comes that 
$\epsilon_{l+1}$ is at most $\frac 12 2^{- j/52}$. \end{proof}

\part{Analytics bounds}
\section{Estimate on the expansion of simple and strong regular pieces}\label{section Estimate on the expansion of simple and strong regular pieces}
\subsection{Yoccoz' one-dimensional bounds}

Let $P_a(x)=x^2+a$ with $a\le 0$. 

In all estimates that follow, we use the letter $C$ to denote various constants {\bf independent of $M$}. The dependence on $M$ in the estimates will always be explicit.

We recall that the fixed points of $P_a$ are denoted by $\alpha_0(a)$ and $\beta(a)$. The sequence $(\alpha_i(a))_a$ of preimage of $\alpha_0(a)$ was defined \cref{defalphai} \cpageref{defalphai}. 

By \textsection 2.2 of \cite{Y95}, for $a=-2$, it holds:
\begin{equation}\label{valeur en -2}\alpha_n(-2)= -2\cos\frac\pi{3\cdot 2^n}\qand\tilde \alpha_n(-2)= -2\cdot \sin\frac{\pi }{3\cdot 2^{n+1}}\; .\end{equation}

\begin{prop}[Prop. 3.1 \cite{Y95} ]\label{propdefcM}
\begin{enumerate}
\item For $a \in [-2,-\frac 32]$, the preimage  $\alpha^{(m)}(a)$ belongs to $ [-2,-\frac 32]$ for $m>0$ and satisfies 
 $1/3 \leq \partial \alpha^{(m)}/\partial a \leq 1/2 \;$
 for $m\geq 0$.
 \item
 For $m\ge 1$, the equation
$P_a(0)= a = \alpha^{(m)}(a)$
has a unique root $a^{(m)}$ in $[-2, -\frac 32]$. This root is simple. The sequence $(a^{(m)})_{m>1}$ is decreasing. For $a \in (a^{(m+1)}, a^{(m)})$, the critical value $P_a(0)$ belongs to the interval
$(\alpha^{(m)} (a),\, \alpha^{(m-1)} (a))$.
\item The sequence $(a^{(m)})_{m>1}$ converges to $-2$. More precisely, one has, for some constant $C>0$ and all $m>1$
$$ C^{-1} 4^{-m} \leq a^{(m)} +2 \leq C \,4^{-m}.$$
\end{enumerate}
\end{prop}
An immediate consequence of \eqref{valeur en -2} and the latter proposition is
\begin{coro}\label{coro alphaM}For $a\in (a_{M} , a_{M-1} )$,
it holds $4^{-M}/C\le 2-\beta(a)\le C\cdot 4^{-M}$ and 
$4^{-M}/C\le \beta(a)-\alpha_m(a)\le C\cdot 4^{-M}$ for every $M-\le m\le M$. 
\end{coro}
Let $a\in (a_{M} , a_{M-1} )$ or equivalently assume that the first return time of the critical point $0$ in $I_\se$ is $M+1$. We have:
\begin{eqnarray*}
\alpha_{M}<& a=P_a(0)& < \alpha_{M-1},\\
- \alpha_{M-n}<& P_a^{n+1}(0)&<- \alpha_{M-n+1}, \qquad {\rm for}\; 1<n<M, \\
\alpha_0 <& P_a^{M+1}(0)&<- \alpha_0 .
\end{eqnarray*}

\smallskip

We define for $1\leq n \le M$ preimages $\wt \alpha_{n} \in I_\se$ by the conditions
$$P_a(\pm \,\wt \alpha_{n})= \alpha_{n-1}, \qquad \wt \alpha_{n} <0.$$

The finite sequence defined by these conditions verifies $\wt  \alpha_{1} = \alpha_0 $ and $\wt \alpha_{n} < \widetilde \alpha_{n+1}$ for $1 \leq n \le  M$. We remark that for $2\le  n \le  M$ :
$$I_{\ss^-_{n}} = [\widetilde \alpha_{n-1}  ,  \widetilde \alpha_{n}] \qand I_{\ss^+_{n}} = [- \widetilde \alpha_{n}  ,  - \widetilde \alpha_{n-1}].  $$


\smallskip

\begin{prop}[Prop. 3.4 \cite{Y95}]\label{propestimates1}
For $a \in (a^{(M)} , a^{(M-1)} )$, the following estimates hold:
\begin{eqnarray*}
C^{-1} 4^{-M} \leq &\beta(a) + P_a(0) &\leq C4^{-M},\\
C^{-1} 4^{-n} \leq &\alpha_{n}(a) - P_a(0) &\leq C 4^{-n}, \qquad \hbox{for} \; 0 \leq n < M-2,\\
C^{-1} 2^{-n} \leq& \vert \wt \alpha^{(n)}(a)\vert  &\leq C 2^{-n},  \qquad \hbox {for} \; 0 < n < M-1.
\end{eqnarray*}
\end{prop}

For $a \in (a^{(M)} , a^{(M-1)} )$, $x \in [-\beta(a), \beta(a)]$ we define
$$ h_a(x) := (\beta^2-x^2)^{-1/2}. $$

\begin{prop}[Prop. 3.5 \cite{Y95}]\label{propestimates2}
\begin{enumerate}
\item
The derivative of $P_a$ satisfies
$$ \vert DP_a (x) \vert  = 2 {h_a(x) \over h_a(P_a(x))} \; \Big(1 + {\beta + a \over x^2}\Big)^{-1/2}. $$
\item For $n>0$, $x \in [\alpha_{n}, \alpha_{n-1}]$
$$\Big\vert \log  \vert  DP^n_a(x)  \vert -\log  2^n {h_a(x) \over h_a(P_a^n(x))} \Big\vert \leq C \, n \, 4^{-M}.$$
\item For $1<n<M-1$, $x \in I_{\ss_n^{\pm}}$
$$\Big\vert \log \vert  DP^n_a(x)  \vert -\log  2^n {h_a(x) \over h_a(P_a^n(x))} \Big\vert \leq C4^{n-M}.$$
\end{enumerate}
\end{prop}

\begin{lemm}[lem. 2.8 \cite{Y95} ]
\label{distbound1}
For every piece $(I, n)$ such that there exists an interval $\hat I\supset I$ which is sent bijectively onto $[\alpha_1,-\alpha_1]\supset I_\se$. Then for every $x\in I$:
\begin{equation}\label{distU} |D\log |{DP_a^{n}(x)}||\le C_0 |DP_a^{n}(x)|,\end{equation}
with $C_0:= 2|\alpha_0-\alpha_1|^{-1}$. \end{lemm}
Note that  \cref{distU} holds true for any $(I,n)=(I_\sa, n_\sa)$ among $\sa\in \sY_0$.
We recall that $\alpha_0\approx -1$ and $\alpha_1\approx -\sqrt{3}$ when $M$ is large. Thus $C_0\approx 2/(\sqrt 3 -1)=2.7320\cdots$.

\subsection{Consequence of Yoccoz estimate in dimension 1}
 Here is a consequence of \cref{propestimates2}:
 \begin{coro}\label{coroYoc} For every $\ss\in \sY_0$, it holds for every $k\le n_{\ss}$: 
 $$|DP_a^{n_\ss}(x)|\ge 2^{k/2}  |DP_a^{n_\ss-k}(x)|\; .$$
 \end{coro}
 \begin{proof}
 Put $n:=n_\ss$.   By \cref{propestimates2}.2, for every $x\in I_{\ss_n^\pm}$ and $ k<n$, it holds:
  \begin{equation}\label{rondstar} 2^{k-Ck 4^{-M}}\sqrt{\beta^2 -  P_a^n(x)^2 \over 
\beta^2 -  P_a^{n-k}(x)^2}
  \le \left\vert \frac{ DP^n_a(x)}{DP^{n-k}(x)}  \right\vert\le
2^{k+Ck 4^{-M}}
\sqrt{\beta^2 -  P_a^n(x)^2 \over 
\beta^2 -  P_a^{n-k}(x)^2}
.\end{equation}  
As $| DP^n_a(x)|\le  |DP^{n-k}_a(x)|\le \beta$ and $k\le M$, we obtain:
  $$2^{k-CM 4^{-M}}  \le \left\vert \frac{ DP^n_a(x)}{DP^{n-k}(x)}\right\vert\; .$$
To prove the inequality at $k=n$, first we note that both $P^n(x)$ and $x$ are smaller than $1$ (because they are in $I_\se$, while $\beta\in [2-C4^{-M},2]$ by \cref{coro alphaM}. Thus  
$$\sqrt{\beta^2 -  P_a^n(x)^2 \over 
\beta^2 -  x^2}\ge \sqrt{3 -C4^{-M}\over 
4}= \frac{\sqrt{3}}2  -C4^{-M}>\frac{\sqrt 2}2 2^{Cn4^{-M}}$$
Then the sought inequality is a consequence \eqref{rondstar} (which is also valid for for $k=n$ by \cref{propestimates2}.3).
\end{proof}

Here is a consequence of \cref{coro alphaM} and   \cref{propestimates1}:
\begin{coro}\label{coroyoc2}
For every $x\in [\alpha_M, \alpha_{M-1}]$, it holds
\[ \frac{4^{M}}C \le |D_xP^M| \le  4^M\; .\]
For every $x\in [\alpha_{M-1}, \alpha_{M-2}]$, it holds
\[ \frac{4^{M-1}}C \le |D_xP^{M-1}| \le  4^{M-1}\; .\]
\end{coro}
\begin{proof} The upper bound is given by $DP_a|[-2,2]\le 4$. 
We recall that by \cref{coro alphaM}  $|\beta-\alpha_M|$ and $|\beta-\alpha_{M-2}|$ are of the order of $4^{-M}$. Thus for $x\in [\alpha_M, \alpha_{M-2}]$, $h_a(x)$ is of the order of $2^{M}$. On the other hand, $h_a| I_\se$ is bounded. Thus  \cref{propestimates2}.2  implies the corollary. 
\end{proof}
 Here is a consequence of the two above corollaries and \cref{distbound1}:
\begin{coro}\label{coroyoc3}
If $P^M(a)\in I_\se\setminus I_{\ss_-^{\odot \lfloor M/24\rfloor }}$, then for  every $x\in I_{\ss_M^\pm}$ and $k\le M$: 
 $$|DP_a^{M}(x)|\ge {2^{k/2}}\cdot |DP_a^{M-k}(x)|\; .$$

\end{coro}
\begin{proof}The same argument as for \cref{coroYoc} shows the case  $k<M$. It remains the case $k=M$. As the orbit $I_{\ss_-^{\odot \lfloor M/24\rfloor }}$ is included in $[\alpha_1,0]\subset [\sqrt 3,0]$, the length of $I_{\ss_-^{\odot\lfloor M/24\rfloor }}$ is at least $(2\sqrt 3)^{-M/12}$. Using the distortion bound of \cref{distbound1} for the piece $([\alpha_{M-1},\alpha_{M-2}], M-1)$, the length of $I_{\sr\cdot \ss_-^{\odot \lfloor M/24\rfloor }}$ is at least 
$|D_{P(x)}P^{M-1}|^{-1} \cdot (2\sqrt 3)^{-M/12}/C$.  Thus $|x|\ge \|D_{P(x)}P^{M-1}\|^{-1/2} \cdot  (2\sqrt 3)^{-M/24}/C$. It comes by \cref{coroyoc2}:
\[ |D_{x}P^{M}|\ge \frac{(2\sqrt 3)^{-M/24}}C \sqrt{|D_{P(x)}P^{M-1}| }\ge 
\frac{(2\sqrt 3)^{-M/24}}C 2^{M-1}\ge 2^{M/2}\; .\]
\end{proof}
The following is an immediate consequence of \cref{propestimates2}:
\begin{coro}\label{coro M'}
For every piece $\ss_n^\pm\in \sY_0$ with $n\le M/2$ and $x\in I_{\ss_n^\pm}$, it holds:
  $$2^{n(1-2^{-M-1})}\le \left\vert  DP^n_a(x) {h_a(P_a^n(x)) \over h_a(x) } \right\vert\le 2^{n(1+2^{-M-1})}\; .$$
\end{coro}
We recall that  $\leb_g$ is the measure on $[-\beta, \beta]$ with density $\frac1{\sqrt{\beta^2-x^2}}$:\index{$\leb_g$}
\[  \leb_g(E):= \int_E (\beta^2-x^2)^{-1/2} dx\quad \text{for every Borel subset } E\subset \R\; .\]
\begin{prop}\label{reste M'}With $M':=  \lfloor M/2\rfloor$, it holds:
$$C\cdot 2^{-M} \ge \log_2 \frac{\leb_g(I_\se \setminus \bigcup_{\ss\in \sY_0: n_\ss\le M/2}I_{\ss})}{\leb_g I_\se}-M'\le C\cdot 2^{-M}\; .$$
\end{prop}
\begin{proof}
By \cref{coro M'}, the set $E=\bigcup_{\ss\in \sY_0: n_\ss\le M/2}I_{\ss}$ satisfies:
\[2 \sum_{j=2}^{\lfloor M/2\rfloor } 2^{-n(1-2^{-M-1})}\leb_g\, I_\se\le \leb_g(E) \le 2 \sum_{j=2}^{\lfloor M/2\rfloor } 2^{-n(1+2^{-M-1})}\leb_g\, I_\se\; .\]
Thus with $M':= \lfloor M/2\rfloor$:
 \[2^{1-2(1-2^{-M-1})} \frac{1-2^{-M'(1-2^{-M-1})}}{1-2^{-1+2^{-M-1}}}\le \frac{\leb_g(E)}{\leb_g(I_\se)}\le 2^{1+2(1-2^{-M-1})} \frac{1-2^{-M'(1+2^{-M-1})}}{1-2^{-1+2^{-M-1}}}\; .\]    
 We develop using $2^{-1+2^{-M}}= 2^{-1}+O(2^{-M})$ and 
 $2^{-M'+M'2^{-M-1}}= 2^{-M'}+2^{-M'}O(M'2^{-M})=2^{-M'}+O(2^{-M})$:
  \[  (1-C\cdot 2^{-M})({1-2^{-M'}})\le \frac{\leb_g(E)}{\leb_g(I_\se)}\le  (1+ C\cdot 2^{-M})({1-2^{-M'}})\; .\]  \end{proof}  
  
\subsection{Application to two-dimensional pieces}\label{app to 2}
\begin{proof}[Proof of \cref{simplespieceindim2} \cpageref{simplespieceindim2}]
Let $\ss\in \sA_0\cup\{\sr\}$. To show that the pair $(Y_\ss, n_\ss)$ is a puzzle piece, it remains to prove properties $(i)$ and $(ii)$ of \cref{piecedef}. Namely that for every $z\in Y_\ss$, for every $w\in \chi_h$, for every $0\le m\le n_\ss$ it holds:
\begin{enumerate}[$(i)$]
\item  $\| D_zf^{n_\ss}(w)\|\ge 2^{\frac{m}3 } \|D_{f^m(z)}f^{n_\ss-m}(w)\|$,
\item $\| D_zf^{m}(w)\|\ge b^{m/6} \|w\|\; .$
\end{enumerate}
As $n_\ss$ is smaller than $M$ and since $\|Df\|\le 4$, statement $(i)$ at $m=n_\ss$ implies $(ii)$. To prove $(i)$, we recall that $Y_\ss$ is $O(b)$-close to $I_\ss\times [- \theta,  \theta]$. Also for every $(x,y)\in Y_\ss$ and unit vector $w\in \chi_h$, it holds:
\begin{itemize}
\item $f^j(z)$ is $O(\theta)$-close to $(P_a^j(x), 0)$,
\item  $D_zf^j(w)$ is $O(\theta)$-close to $(D_xP_a^j , 0)$.
\end{itemize}
Thus \cref{coroYoc}, implies $(i)$ for every $\ss\in \sY_0$.  It also implies $(i)$ for $m\le M-2$ when $\ss=\sr$. The case $\ss=\sr$ and $m\in\{M,M-1\}$ is a consequence of \cref{coroyoc2}.
The case $\ss\in \sA_0\setminus \sY_0$ is implied by \cref{coroyoc3}.
\end{proof}

\subsection{Sharp estimate on the expansion of strongly regular pieces}
We rare going to prove \cref{sharp bounds}.
Let $f$ be $0$-strongly regular.
 We recall that for every $z=(x,y)\in (-\beta, \beta)\times [-\theta, \theta]$ and $(u,v) \in \R^2$, given a vector $w\in \R$ pointed at $z$,  we defined the following Riemannian metric:
\[\|w\|_g= \sqrt{\frac{u^2}{\beta^2-x^2}+ v^2}\qand \|w\|_0= \sqrt{u^2+v^2}\; .\]

The first part of the following is an immediate consequence of  \cref{propestimates1} and the second part of \cref{reste M'}.
\begin{coro}\label{coro pregreatly}
For every $\lfloor M/2\rfloor \le j\le M+1$, the set  $ Y_\boxdot \cup \bigcup_{\ss\in \sA_0\, n_\ss\ge j} Y_{\ss}$ is a box, and it intersects $\R\times\{0\}$ at a segment of $\leb_g$-measure at most $2^{-j/3}$. Furthermore:
\begin{equation*} -\log_2 \frac{\leb_g \R\times\{0\}\cap (Y_\boxdot\cup  \bigcup_{\ss\in \sA_0: n_\ss \ge \lfloor M/2\rfloor  +1} Y_\ss)}{\leb_g I_\se}\ge (1- 2^{-M})\lfloor M/2\rfloor \; .\end{equation*}
\end{coro}

The following is consequence of \cref{propestimates2} and \cref{coro M'}:
\begin{coro}\label{coro pour greatly}
For every piece $\ss_n^\pm\in \sY_0$ and $z\in Y_{\ss_n^\pm}$, it holds:
  $$2^{n(1-(2M)^{-1/2})}\le \frac{
  \left\|  D_z f^n(u)  \right\|_g}{\|u\|_g}
  \le 2^{n(1+(2M)^{-1/2})}\; .$$
  If moreover $n\le M/2$, then  it holds:
    $$2^{n(1-2^{-M})}\le \frac{
  \left\|  D_z f^n(u)  \right\|_g}{\|u\|_g}
  \le 2^{n(1+2^{-M})}.$$
\end{coro}
An immediate consequence of \cref{coroyoc2} and $h_a|I_\se<1$ is:
\begin{coro}\label{pre Expansion SR piece}
For every  $z\in Y_{\sr}$, it holds:
  $$4^{M}/C\le \frac{
  \left\|  D_z f^M(u)  \right\|_g}{\|u\|_0}
  \le 4^{M}\; .$$
\end{coro}

We now ready to state:
\begin{prop}\label{Expansion SR piece}
Let $\sg\in \sR$ be a strongly regular word for $f$. 
Then for every $z\in Y_\sg$ and every non-zero vector  $u\in \chi_h$ it holds:
\[2^{n(1-1/\sqrt M)}\ge \frac{\|D_zf^{n_\sg}(u)\|_g}{\|u\|_g}\ge 2^{n(1-1/\sqrt M)}\; .\]

Also for every $z\in Y_{\sr\cdot \sg}$ and every unit vector $u\in \chi_h$ it holds:
\[ \|D_zf^{M+n_\sg}(u)\|_0\ge 2^{(n_\sg+2M)(1-1/\sqrt M)}\; .\]
\end{prop}
\begin{proof}
The second item follows from the first item and \cref{pre Expansion SR piece}. The first item follows from the following observations:
\begin{itemize}
\item a proportion of times $1-2^{-\sqrt M}$ is spent into the pieces in $\sY_0$, and therein the mean expansion is in $[2^{1-(2M)^{-1/2}}, 2^{1+(2M)^{-1/2}}]$ by \cref{coro pour greatly}.
\item for the remaining proportion of times the expansion is in $[2^{1/3}, 4]$ by \cref{Crutial prop}.
\end{itemize}
Thus  the mean of the expansion is $\ge 2^{(1-(2M)^{-1/2})(1-2^{-\sqrt M})}\ge 2^{1-M ^{-1/2}}$ and $\le 2^{n(1+1/\sqrt M)}$. 
\end{proof}
\section{Iterations of vertical and horizontal curves}\label{proof:regularpiece}

\subsection{Iteration of horizontal curves}
\label{proof:invariancecone}
Let us show the first item of \cref{invariancecone}. Namely, that given any piece $(Y,n)$, the cone $\chi_h|Y$ is sent by
 $Df^n$  into $\chi_h$ and if $S$ is a horizontal curve in $Y$ then $f^n(S)$ is  a horizontal curve. 
\medskip

\noindent{\bf The cone condition.} First let us show that for every $u\in \chi_h$, for every $z\in Y$, it holds  $D_zf^{n}(u)\in \chi_h$. 
If $n=0$, then the statement is obvious. Let $n\ge 1$,  $v=Df^{n-1}(u)$ and $(x,y)=f^{n-1}(z)$. We recall that $\|D_{(x,y)}f(v)\|\ge 2^{1/3}\|v\|$ by property \ref{defi_condi_hyp} of \cref{piecedef} p.\pageref{piecedef}. We recall that $D_{(x,y)}f$ is $b$-close to $(v_x,v_y)\mapsto (2x \cdot v_x,0)$. Thus $|v_x|\ge |v_y|$  since otherwise it would be contracted by $Df$ by \cref{defchvchih} of $\chi_v$. For the same reason $|x|\ge 1$. Thus by the form of $D_{(x,y)}f$, the vector $v$ is sent by $D_{(x,y)}f$ into $\chi_h$.
\begin{rema}\label{pour prod parabolic} Actually this proof shows even that $D f^n|Y$ sends $\chi_h$ into $\{(v_x,v_y): |v_y|\le {b}\cdot  |v_x|\}$. We recall also that the image $f$ -- which contains $f^n(Y)$ -- is in $\R\times [-b, b]$.\end{rema}
 \medskip
 
\noindent{\bf The flat condition.} To achieve the proof of \ref{invariancecone}.$(i)$, it remains to show that $f^n(S)$ is flat. As $f$ is of class $C^2$, by density we can assume that $S$ is of class $C^2$.  Then the result is a consequence of the following:
\begin{lemm}[Compare with Lem. 2.8 \cite{WY01}]\label{curvature}
Let $(Y, n)$ be a piece and let $S\subset Y$ be a $C^2$-curve with curvature at most $1$ such that $TS$ is included in $\chi_h$. Then $f^n (S)$ has curvature bounded by $C\cdot b$, for a universal constant $C>0$.\end{lemm}
\begin{proof}
Let $t\mapsto \gamma(t)$ be a parametrization of $S$ such that $\dot \gamma(t)\not=0$. Then the curvature of $S$ at $\gamma(t)$ is:
\[C_\gamma(t) := \frac{|\dot \gamma\times \ddot \gamma|}{|\dot \gamma|^3}(t)\; ,\]
where $u\times v= det(u,v)$. Thus the curvature of $\bar S:=f^n(S)$ at $\bar \gamma(t):= f^n\circ \gamma(t)$ is:
 
\[C_{\bar \gamma}(t) = 
\frac{
|Df^n \dot \gamma\times  Df^n\circ \ddot \gamma|+
|Df^n \dot \gamma\times 
\sum Df^{n-k-1} \circ D^2f (Df^k\circ 
\dot \gamma)^{\otimes 2}|
}{|Df^n\dot \gamma|^3}(t)\; .\]
Using that the determinant of $Df$ is smaller than $b$, it comes:
\begin{equation} \label{pour etude courbure}C_{\bar \gamma}(t) \le 
b^n\frac{| \dot \gamma\times \ddot \gamma|}{\|Df^n\dot \gamma\|^3}(t)+\sum_{k=0}^{n-1} \frac{
b^{n-k-1} \|Df^{k+1} \dot \gamma\times 
Df^2 (Df^k\circ 
\dot \gamma)^{\otimes 2}\|}{\|Df^n\dot \gamma\|^3}(t)
\; .\end{equation}
We now use Property \ref{defi_condi_hyp} of \cref{piecedef} p.\pageref{piecedef} to obtain:
\[C_{\bar \gamma}(t) \le \frac{b^n}{2^{n/3} }C_\gamma(t)+\frac{
|Df^{n} \dot \gamma\times 
Df^2 (Df^{n-1}\circ 
\dot \gamma)^{\otimes 2}|}{\|Df^n\dot \gamma\|^3}(t)+
\sum_{k=0}^{n-2} \frac{
b^{n-k-1} \| D^2f\| }{2^{\frac{n-k-1}3}\cdot 2^{\frac{2(n-k)}3}}(t)
\; .\]
The first and last term of the latter upper-bound are dominated by $b$. The second term  is at most:
$$\max_{|u|=1} 2^{-2/3} | Df( u)\times D^2f(u,u)|$$
Which is bounded by $4^2 \cdot 2^{-2/3} b$ because the second coordinates of $Df$ and $D^2f$ are both smaller than $b$ while their first coordinates are bounded by $4$.
\end{proof} 
\begin{proof}[Proof of \cref{prop pesin instable gen} ]\label{proof prop pesin instable gen}
Let us prove $1$. For every $i$, let $E_i:= \{f^{n_i}(S\cap Y_i): S\in  \mathcal H\}$.  As $f^{n_i}(Y_i)=\cup_{j\ge i}f^{n_j}(Y_j)$   is compact, it holds:
$$f^{n_i}(Y_i) = \bigcup_{S\in E_i} S=\bigcup_{S\in E_j, j\ge i} S =adh(\bigcup_{S\in E_j, j\ge i} S)\; .$$
$$W_u^{Y_\infty}= \bigcap_i adh(\bigcup_{S\in E_j, j\ge i} S)=
\bigcap_i  \bigcup_{S\in adh(\cup_{j\ge i} E_j),} S= \bigcup_{S\in \bigcap_i adh(\cup_{j\ge i} E_j),} S
\; ,$$
where the latter equality was obtained using the compactness of $\mathcal H$ for the canonical $C^1$-topology.  Thus $W_u^{Y_\infty}$ is an union of horizontal stretched curves. If there are two different curves $S$ and $S'$ in $W_u^{Y_\infty}$, then there is an $\epsilon$-ball of $\R^2$ in $W_u^{Y_\infty}$, and so in each $f^{n_i}(Y_i)$ for $i\ge 1$. This contradicts \cref{contractionPP} which states that   $E_i$ has diameter smaller than $4\theta b^{n_i/3}$. 

Let us prove $2$. Let $\overleftarrow z:= (z_i)_{i\le 0}\in \overleftarrow W_u^{Y_\infty}$ and let $ W^u_{loc}(\overleftarrow z)$ be the component of $\overleftarrow z$ in $\{(z_i')_{i\le 0} \in \overleftarrow M_f: \limsup_{-\infty} \frac1n \log d(f^n(z), f^n(z'))>0\text{ and } z_0'\in Y_\se\}$. By the proof of the first statement, for every $S\in \mathcal H$, the curve  $S_i:=f^{n_i}(S\cap Y_i)$ is close to $W_u^{Y_\infty}$ when $i$ is large. By property $(ii)$ of \cref{piecedef}, The curve $S_i$ is  $2^{-j/3}$-contracted by  $(f^{j}| f^{n_i-j}(S\cap Y_i))^{-1}$ for every $j\le i$. Hence  $\overleftarrow W_u^\st$ is included in $ W^u_{loc}(\overleftarrow z)$. 
 
 As the image of $M$ by $f$ is included into the interior of $M$, the components of $(f|M)^{-n_i}(Y_\se)$ are bounded by segments of $\R\times \{\pm\theta\}$ and by components  of the stable set of the fixed point $A$ intersected with $M$. This implies that  $Y_{n_i}$ is a component of $(f|M)^{-n_i}(Y_\se)$. Thus $\overleftarrow W_u^{Y_\infty}$ is a component of $\{(z_i')_{i\le 0} \in \overleftarrow M_f: z_0'\in Y_\se\}$, and so it contains  $ W^u_{loc}(\overleftarrow z)$. 
\end{proof}
\subsection{Distortion estimates along horizontal curves}\label{distortion estimate}

Let $f$ be $0$-strongly regular.  
In \cref{defi piece curve} we defined the one-dimensional pieces.

\begin{defi}[Distortion of a one-dimensional piece $(S,n)$]
\index{Distortion of a one-dimensional piece}
The \emph{distortion} $\Delta(S, n)$ is the minimal number $\Delta\ge 0$ such that for every $z\in S$:
\[\lim_{z'\in S\to z}
\left| \frac{\log{\|\partial_S f^{n}(z)\|} -\log \|\partial_S  f^{n}(z')\|}{ \|f^{n}(z)-f^{n}(z')\|}\right|
\le \Delta \, ,\]
with $\partial_S$ the derivative w.r.t. a unit vector tangent to $S$ oriented toward the right.
\end{defi}
In \cref{admissiblewordfromcurve}, we define the words $\sg\in \sA^{(\N)}$ which are admissible from a horizontal stretched curve $S$. We saw that they define the pieces $(S_\sg, n_\sg)$. 

Let us bound from above the distortion of the simple pieces:
\begin{prop}\label{distortion_Y0} For every $\sg\in \sY_0^{(\N)}$, with  $n_\sg\le 2M$, for every $S\in \mathcal H$,  the distortion of $(S_\sg,n_\sg)$ is at most $3$:
\[\Delta(S_\sg,n_\sg)\le 3\; .\]
\end{prop}
\begin{proof} We recall that every symbol $\sa\in \sY_0$ defines a puzzle piece which is regular in the sense of Yoccoz: there exists an interval $\hat I_{\sa}\supset I_\sa$ which is sent  bijectively by $P^{n_\sa}$ onto  $[\alpha_1,-\alpha_1]$. Thus the piece $(I_\sg, n_\sg)$ is also regular. By \cref{distbound1}, the distortion of the piece $(I_\sg, n_\sg)$ is  $\le 2.74$. As $b$ is small compare to $M$ and the distortion depends continuously on its $C^{1+Lip}$-entries, it holds $\Delta(S_\sg, n_\sg)\le 3$ for every $S\in \mathcal H$. As there are finitely many words $\sg \in   \sY_0^{(\N)}$ with order at most $2M$, we obtain the proposition. \end{proof}

 The following is rough estimate which will be useful to obtain a much finer estimate.
  \begin{lemm}\label{roughdistbound}
  Let $(S,n)$ be a one-dimensional piece. 
 Then its distortion satisfies:
 \[ \Delta(S, n)\le 3\cdot 2^{\frac {4}3 n }
  \; .
 \]
 
 \end{lemm}
 \begin{proof}
By definition of the one-dimensional pieces, for every $z\in S$ and every unit vector $u\in T_z S$, it holds:
  \begin{equation}\label{hyp_faible} 
\|D_{z}f^{n}( u)\|\ge 2^{n/3} \| u\|\; .\end{equation}
By density, we can assume $S$ of class $C^2$. Let $\gamma(t)$ be a parametrization of $S$ such that $\gamma(0)=z$, $\partial_t \gamma(0)=u$ and  $\|\partial_t\gamma(t)\|=1$ for every $t$.
Let $v= \partial^2_t \gamma(0)$. By flatness of $S$, it holds $\|v\|\le \theta$.  Note that it suffices to show:
 \begin{equation}\label{latoutepremiere} \frac{\|
\partial_t^2 (f^n\circ \gamma)(0) \|}{\|D_z f^n(u)\|^2}\le 4^{n}\; .\end{equation}
A computation gives:
\[\partial_t^2 f^n\circ \gamma=\sum_{m=0}^{n-1} D f^{n-m-1} \circ D^2 f\big (\partial_t (f^m\circ \gamma),\partial_t (f^m\circ \gamma)\big)+ Df^n \circ \partial_t^2 \gamma\; .\]
We now look at the value at $t=0$. We use the bounds $\|Df\|\le 4$ and $\|D^2f\|\le 2+2b$ stated in Fact \ref{trivial estimate} to obtain:
\[\|\partial^2_t (f^n \circ \gamma)(0)\|\le \sum_0^{n-1} ( 2+2b)\cdot 4^{n-m-1} \|D_z f^m (u)\|^2+\theta 4^{n}
\; .\]
Using inequality \eqref{hyp_faible}, it comes the sought inequality  \eqref{latoutepremiere}:
\[\frac{\|\partial^2_t (f^n \circ  \gamma)(0)\|}{\|D_zf^n (u)\|^2}\le \sum_0^{n-1} 3\cdot  4^{n-m-1} 4^{-\frac{n- m}3} +\theta 4^{n-\frac n 3}<   \sum_0^{n-1} (2+2b)\cdot  4^{\frac23 j} +\theta 4^{\frac23 n}< 3\cdot  2^{\frac {4}3n}
\; 
.\]
\end{proof} 
 We recall that by definition of the admissible words, if $\sa\cdot \sb\in \hat \sA^{(\N)}$ is admissible from $S\in \mathcal H$, then 
$\sa$ is admissible from $S$ and 
$\sb$ is admissible form $S^\sa\in\mathcal H$. The following gives a recurrence relation: 
\begin{lemm}\label{compodist}
Let $\sa\cdot \sb\in \hat \sA^{(\N)}$ be admissible from $S\in \mathcal H$. Then it holds:
\[\Delta(S_{\sa\cdot \sb}, n_{\sa\cdot \sb})\le 2^{-n_\sb/3}\cdot \Delta(S_\sa, n_\sa)+\Delta(S^\sa_\sb, n_\sb)\; .
\]
\end{lemm}
\begin{proof}
Let $z\in S_\sa$. We have:
\[\lim_{f^{n_\sa}(z')\in S^\sa\to f^{n_\sa}(z)}
\left| \frac{\log{\|(\partial_{S^\sa} f^{n_\sb})(f^{n_\sa}(z))\|} -\log \|(\partial_{S^\sa}  f^{n_\sa})(f^{n_\sa}(z'))\|}{ \|f^{n_\sb}(f^{n_\sa}(z))-f^{n_\sb}(f^{n_\sa}(z'))\|}\right|\le \Delta(S_\sb,n_\sb)\]
And so:
\begin{equation}\label{Dist1'} \lim_{z'\in S\to z}\left|
 \frac{\log{\|(\partial_{S^\sa} f^{n_\sb})(f^{n_\sa}(z))\|} -\log \|(\partial_{S^\sa}  f^{n_\sa})(f^{n_\sa}(z'))\|}{ \|f^{n_{\sa\cdot \sb}}(z)-f^{n_{\sa\cdot\sb}}(z')\|}
 \right|\le \Delta(S_\sb, n_\sb)\end{equation}
 By the expansion of this segment by $2^{n_\sb/3}$ we have for $z\in S_{\sa\cdot \sb}$:
\[\lim_{z'\in S\to z}
\left| \frac{ \|f^{n_{\sa}}(z)-f^{n_\sa}(z')\|}{\|f^{n_{\sa\cdot \sb}}(z)-f^{n_{\sa\cdot \sb}}(z')\|}\right|
\le 2^{-n_\sb/3}\]
Multiplying the latter limit with the definition of the distortion of $(S_\sa,n_\sa)$ this yields:
\begin{equation}\label{Dist2} \lim_{z'\in S\to z}
\left| \frac{\log{\|\partial_S f^{n_\sa}(z)\|} -\log \|\partial_S  f^{n_\sa}(z')\|}{ \|f^{n_{\sa\cdot \sb}}(z)-f^{\sa\cdot \sb}(z')\|}\right|
\le 2^{-n_\sb/3}\Delta(S_\sa, n_\sa) \, ,\end{equation}
The sought inequality follows by summing inequalities \eqref{Dist1'} and \eqref{Dist2}. \end{proof}
 In \cref{perfect}, we define the perfect words. We are going to show that the distortion of any perfect words admissible from $S\in \mathcal H$  is bounded by $20$. This implies immediately  \cref{distparfaitcommon}. 
\begin{prop}\label{distparfaitcommon2} 
For every perfect word $\sg\in \hat \sA^{(\N)}$ admissible from some $S\in \mathcal H$, it holds: 
\[\Delta(S_\sg, n_\sg)\le 20\; .\]
\end{prop}
\begin{proof}
Let $\sg= \sa_1\cdots \sa_m$ with $\sa_i\in \hat \sA$ for every $i$. Put $\sg_i:=\sa_1\cdots \sa_{i}$. 
 By \cref{compodist}, it holds:
\[\Delta(S_\sg,n_\sg)\le  \sum_{j=1}^m \Delta(S^{\sg_{i}}_{\sa_i}, n_{\sa_i}) 2^{-\frac{n_{\sa_{j+1}\cdots \sa_m}}3}
\; .\]
We now apply \cref{distortion_Y0} to bound the distortion of pieces in $\sY_0$ and \cref{roughdistbound} to bound those not in $\sY_0$.  This gives:
\[\Delta(S_\sg,n_\sg)\le  \sum_{1\le j\le m: \sa_j\in \sY_0}
3\cdot 2^{-\frac{n_{\sa_{j+1}\cdots \sa_m}}3} 
+ \sum_{1\le j\le m: \sa_j\notin \sY_0} 3\cdot 
 2^{ \frac{4 n_{\sa_j}-n_{\sa_{j+1}\cdots \sa_m}}3  }\; .\]
By perfect word's \cref{perfect},  if $\sa_j\notin \sY_0$, then $ 4 n_{\sa_j}\le \frac1{4} n_{\sa_{j+1}\cdots \sa_m}$.  Consequently:
\[\Delta(S_\sg,n_\sg)\le  \sum_{1\le j\le m: \sa_j\in \sY_0}
3\cdot 2^{-\frac{n_{\sa_{j+1}\cdots \sa_m}}3} 
+\sum_{1\le j\le m : \sa_j\notin \sY_0}
3\cdot  2^{-\frac{n_{\sa_{j+1}\cdots \sa_m}}4} 
\; .\]
And so
\[\Delta(S_\sg,n_\sg)\le  \sum_{1\le j\le m}
3\cdot 2^{-\frac{n_{\sa_{j+1}\cdots \sa_m}}4} 
\le\frac{3}{1-2^{-1/4}}<18.85<20
\; .\]
\end{proof}
The following will be useful for the proof of \cref{Crutial prop}.
\begin{prop}\label{distprop}
Let $\sg\in \sR$ be a strongly regular word. For every $x,x'\in Y_\sg$, and every $u,u'\in \chi_h$ unit vectors, it holds:
\[\frac{ \| D_zf^{n_\sg}(u)\|}{\| D_{z'}f^{n_\sg}(u')\|}\le 40\cdot 2^{29\cdot 2^{-\sqrt M} n_\sg}\]
\end{prop}
\begin{proof}
We are going to use the following shown below:
\begin{lemm}\label{estimepouretreparfait}
 Let $\sg=\sa_1\cdots \sa_p\in \sR$ be a strongly regular. There exists $q\le p$ such that:
\begin{enumerate}[(1)]
\item the word $\sg'= \sa_1\cdots \sa_q$ is perfect,
\item $n_{\sa_{q+1} \cdots \sa_p}\le 17\cdot 2^{ -\sqrt M}n_{\sg} $.
\end{enumerate}\end{lemm}
We will show for \cref{ForB}, that there exists a unit vector field $\partial_y'$ on $Y_{\sg'}$ with values in $\chi_v$ such that
 \begin{equation}\label{prop partialy'} \|\partial'_y    f^{n_{\sg'}}|Y_{\sg'}\|\le (K\sqrt b^{2/3})^{n_{\sg'}} 
\qand 
 \|\partial'_y \partial_x  f^{n_{\sg'}}|Y_{\sg'}\|\le K\sqrt b.\end{equation} 
By splitting a unit  vectors  $u$ and $ u'\in \chi_h$ in the basis $\partial_x $ and $\partial_y'$, and then using the contraction of $\partial_y$ and the expansion of $\partial_x$  by $ Df^{n_{\sg}}$ for every $z\in Y_{\sg}$, it comes:
\begin{equation}\label{fulldist1} \frac{ \| D_zf^{n_{\sg'}}(u)\|}{\| D_{z'}f^{n_{\sg'}}(u')\|}\le
\frac{ \| \partial_x f^{n_{\sg'}}(z)\|}{\| \partial_x f^{n_{\sg'}}(z')\|}+K\theta\; .\end{equation}
Now we use the \cref{distparfaitcommon2}, the second part of \eqref{prop partialy'}, and $\|\partial_x Df^{n_{\sg'}}|Y_\sg \|\ge 1$ to obtain:
\begin{equation}\label{fulldist2} \frac{ \| \partial_x f^{n_{\sg'}}(z)\|}{\| \partial_x f^{n_{\sg'}}(z')\|} +K\cdot \theta\le  20\cdot \leb I_\se + K \cdot \sqrt b+K\cdot \theta \le 40
\; .\end{equation}
Now we infer that $n_\sg-n_{\sg'}\le 17\cdot 2^{ -\sqrt M}n_{\sg}$.  Also the expansion of during this times is in $[2^{(n_\sg-n_{\sg'})/3}, 4^{n_\sg-n_{\sg'}}]$. The composition rules gives:
\begin{equation}\label{fulldist3}\frac{ \| D_zf^{n_{\sg}}(u)\|}{\| D_{z'}f^{n_{\sg}}(u')\|}
\le \frac{ \| D_zf^{n_{\sg'}}(u)\|}{\| D_{z'}f^{n_{\sg'}}(u')\|}\cdot 2^{\frac53  17\cdot 2^{ -\sqrt M}n_{\sg}}\; . \end{equation}
Then \eqref{fulldist1}, \eqref{fulldist2} and  \eqref{fulldist3} implies the inequality stated in the proposition.

%
%
\end{proof}
\begin{proof}[Proof of \cref{estimepouretreparfait}]
The proof is done by using the Pliss Lemma. To this end, for every 
$1\le s\le n_\sg$, we put $\sa(s):= \sa_{m+1}$ if $s\in [n_{\sa_1 \cdots  \sa_m}+1, n_{\sa_1 \cdots  \sa_{m+1}}] $ and we consider the sequence of numbers $(X(s))_{1\le s\le n_\sg}$  :
\[s\mapsto \left\{ \begin{array}{cl}  X(s) = 1& \text{if } \sa(s)\in Y_0\\
 X(s) = 1- 17 & \text{if }\sa(s)\notin Y_0\end{array}\right.\; . \]
By definition of the strong regular chain, we have:
\[\sum_{s=1}^{n_\sg}  X(s)=
\sum_{1\le m\le q : \sa_m\in \sY_0} n_{\sa_m}
+(1-17) \sum_{1\le m\le q : \sa_m\notin \sY_0} n_{\sa_m}
\]
\[=
\sum_{1\le m\le q} n_{\sa_m}
-17 \sum_{1\le m\le q : \sa_m\notin \sY_0} n_{\sa_m}
\ge n_\sg(1- 17 \cdot 2^{-\sqrt M})\]
Thus by Pliss' Lemma \ref{Pliss}, there exist $l\ge (1-17\cdot  2^{-\sqrt M}) n_\sg$ and a sequence of integers $1\le n_1< \cdots < n_l\le  n_\sg$ such that
\[\sum_{i=m+1}^{n_j} X(i) \ge 0 \text{ for all } m\le n_j\text{ and } j= 1,\dots , l.\]
We notice that $X(n_l)=1$ and so $\sa(n_l)=:\sa_{q}\in Y_0$ for a certain $q\le p$. If we take $n_l$ maximal, then $n_l=n_{\sa_1 \cdots  \sa_{q}}$. 
Then for every $m\le q$ it holds:
\[0\le (1-17) \sum_{m\le i\le q : \sa_i\notin \sY_0} n_{\sa_i} 
+\sum_{m\le i\le q : \sa_i\in \sY_0} n_{\sa_i}=-17  \sum_{m\le i\le q : \sa_i\notin \sY_0} n_{\sa_i} 
+\sum_{m\le i\le q } n_{\sa_i}\; .\]
Thus $\sg'=\sa_1 \cdots  \sa_{q}$ is perfect. Moreover $n_{\sg}-n_{\sg'} =n_{\sg}-n_l\le n_{\sg}-l\le 17 \cdot 2^{ -\sqrt M}n_{\sg} $. 
\end{proof}

\subsection{Invariance of vertical flat curves by iterations associated to pieces}
Let us show the second item of \cref{invariancecone}. Namely, given a piece $(Y,n)$, we want to show that $(Df^n|Y)$ pulls back $\chi_v$ into itself and  
$(D_z f^{n})^{-1}(\chi_v)\subset \chi_v\; .$
Moreover, for every vertical stretched curve $\mathcal C\subset Y_\se$ which is between both components of $f^n(\partial^s Y)$, the set  $(f^{n}|Y)^{-1}(\mathcal C)$ is a vertical,  stretched  curve. 
\medskip

The proof will use a technology developed in \cite{WY01}. Let us recall their settings. 

Let $M$ be a $2\times 2$ matrix. Assuming that $M$ is not a homothety, we say that a unit vector $e$ defines the \emph{most contracted direction} of $M$ if $\|Mu\|\ge \|Me\|$ for every unit vector $u$.
Given a sequence of matrices $(M_i)_{i\ge 1}$, we denote the matrix product by $M^{(i)}:= M_i\times \cdots \times M_1$. If $M^{(i)}$ is not a homothety, then $e_i$ denotes a unit vector defining the most contracted direction of $M^{(i)}$. 
\begin{lemm}[Lem. 2.1 and Cor. 2.1 p. 16 \cite{WY01}]\label{2.1}
Let $(M_i)_i$ be a sequence of $2\times 2$ matrices. Let $K_0, b, \kappa>0$ be such that $1\gg  b\ge 0$ and $\kappa\gg \sqrt b$. Assume that:
\begin{equation}\tag{$\dag$}  |det(M_i)|\le b\qand \|M_i\|\le K_0 \qand \|M^{(i)}\|\ge \kappa^i\quad \forall i\ge 1\; .\end{equation}
 Then $M^{(i)}$ is not a homothety for every $i$ and there exists $K$ depending only on $K_0$ such that the most contracted direction satisfies:
\begin{enumerate}[$(a)$]
\item $\| e_n\times e_1\|\le \frac{Kb}{\kappa^2}$ and $\| e_n\times e_{n+1}\|\le \frac{(Kb)^n }{\kappa^{2n}}$,
\item $\| M^{(i)}e_n\|\le \left(\frac{Kb}{\kappa^2}\right)^{i}$ for every $1\le i\le n$.
\end{enumerate}\end{lemm}
\begin{lemm}[Cor. 2.2 p.16  \cite{WY01}]\label{2.2}
Under the assumptions of \cref{2.1},  if $(M_i)_i=(M_i(s))_i$ is a sequence $2\times 2$ matrices depending $C^1$ on a parameter $s\in \R^3$ and such that for every $s$, the sequence $(M_i(s))_i$ satisfies $(\dag)$ and moreover $\|\partial_s  M_i(s)\|\le K_0^i$ and $\|\partial_s  \det\,  M_i(s)\|\le K_0^i b$, then it holds:
\begin{enumerate}[$(a)$]
\item $\| \partial_s  (e_n\times e_1)\|\le \frac{Kb}{\kappa^3}$,
\item $\| \partial_s M^{(i)}e_n\|\le \left(\frac{Kb}{\kappa^3}\right)^{i}$ for every $1\le i\le n$.
\end{enumerate}
\end{lemm}

\begin{proof}[Proof of \cref{invariancecone}.$(ii)$]
\label{proof:invariancecone2}
Let $(Y,n)$ be a piece in $Y_\se$ and put $\bar Y=f^n(Y)$.
 Let $\mathcal C$ be a vertical flat, stretched curve in $Y_\se$ which is between both components of $\partial_s \bar Y= f^n(\partial_s Y)$. We would like to show that the preimage $\mathcal C'$ of $\mathcal C$ by $f^n|Y$ is a  vertical, stretched curve. By \cref{piecedef}.$(ii)$ and the invariance of $\chi_h$ given by \cref{invariancecone}.$(i)$, the map $f^n|Y$ is transverse to $\mathcal C$. Thus $\mathcal C'$ is a curve. Moreover, by assumption, for every $y\in [-\theta, \theta]$, the curve $\mathcal C$ intersects the curve $f^n(Y\cap \R\times \{y\})$ at a unique point. Thus $\mathcal C'$ intersects $\R\times \{y\}$ at a unique point and so $\mathcal C'$ is a connected stretched $C^{1+Lip}$-curve. 
 
 To show that it is vertical, by density, we can assume $\mathcal C$ of class $C^2$. Then by transversality, the curve   $\mathcal C'$ is of class $C^2$.
Let $\mathcal F$ be the foliation of $\R\times [-\theta, \theta]$ whose leaves are all parallel to $\mathcal C$. Let $\pi$ be holonomy from along $\mathcal F$ to 
$\R\times \{0\}$. We notice that if $z\in \R\times [-\theta, \theta]$, the point $\pi(z)$ is equal to $x_0+u_0$ with 
$(x_0,0)$ the intersection point of $\mathcal C\cap \R\times \{0\}$ and $u_0\in \R$ such that $z\in \mathcal C+(u_0,0)$. We notice that $\pi$ is of class $C^2$ and $C^2$-$\theta$-close to the first coordinate projection.  We remark that $\mathcal C'$  is a fiber of 
\[\pi\circ f^{n}|Y\to \R\;. \]
Given $z\in Y$, we define the sequence of matrices 
\[\left\{\begin{array}{ll} 
A_i(z):= D_{f^{i-1}(z)} f& \text{if } 1\le i\le n,\\
A_i(z):= D_{f^{n}(z)} \pi & \text{if } n< i\; .\end{array}\right.\]
We remark that the hypotheses of \cref{2.1} and \cref{2.2} are satisfied with $K_0= 4$ and $\kappa=b^{1/6}$ by $(ii)$ of \cref{piecedef}. Also $e_{n+k}(z)=e_{n+1}(z)$ for every $k\ge 1$. 
As $e_1$ is $O(b)$-$C^1$-close to $(0,1)$, it comes:
\begin{equation}\label{ForA} |e_{n+1}(z)-(0,1)|\le \frac{ K b}{b^{2/6}}= { K b^{2/3}}
\qand |\partial_z e_{n+1}(z)|\le \frac{ K b}{b^{3/6}}=K\sqrt b \; .\end{equation}
Consequently the tangent space of $\mathcal C'$ is in $\chi_h$ and its curvature is smaller than $\theta$. In other words, $\mathcal C'$ is vertical. 

Using the same argument, by taking $\pi$ with fibers equal to line  all parallel to a vector $v\in \chi_v$, we obtain that  $\cD(f^n|Y)^{-1}(v)$ is in $\chi_v$. 
\end{proof}
\begin{rema}\label{ForB}
By the second part of \cref{2.1} and \cref{2.2}, it comes:
\begin{equation} |Df^{i}( e_{n+1})(z)|\le ({ K b^{2/3}})^{i}
\qand |\partial_z  Df^{i}(  e_{n+1})(z)|\le (K\sqrt b)^{i} \quad \forall i\le n\; .\end{equation} 
  \end{rema}
  Let us now prove the following lemma which was used in the proof of \cref{Pesin stable}.
  \begin{lemm}\label{Lamination Lipschitz} Let $f$ be $0$-strongly regular. 
Let $PH$ be the set of points $z\in \R^2$ such that:
\begin{equation}\tag{$\mathcal {PH}$}\|\partial_x f^k(z)\|\ge b^{k/6},\quad \forall k\ge 0\; .\end{equation}
Then there exist $K>0$ and  $e_\infty: PH\to \R^2$ which is continuous, and such that:
\begin{enumerate}[(1)]
\item The norm $\| e_\infty(z) -(0,1)\|$ is smaller than $Kb^{2/3}$ for every $z\in PH$.
\item For every $z, z'\in PH$, it holds $\|e_\infty(z') -  e_\infty(z)\|\le K\sqrt b\| z'-z\|$. 
\item The vector $e_\infty(z)$ is $(Kb)^{2k/3}$ contracted by $Df^k$, for every $k\ge 0$. 
\end{enumerate}
 \end{lemm}
 \begin{proof} 
 We use \cref{2.1}.(a) with $M_i(z)= D_{f^i(z)}f$ and $\kappa= b^{1/6}$. The lemma implies that $(e_n|PH)_n$ converges to a continuous vector field $(e_\infty| PH)$ which is $Kb^{2/3}$-$C^0$-close to $e_0| PH:= (0,1)$ as sated in \emph{(1)}.  Also by \cref{2.1}.(b) the vector $e_\infty(z)$ is $(Kb)^{2k/3}$ contracted by $f^k$, for every $k\ge 0$, as sated in \emph{(3)}.  

Let us now prove \emph{(2)}. Let $z, z$' be in $PH$ and let $n$ be maximal such that $d(z,z')\le b^{n/6}4^{-2n}$. Then for every $z''\in [z,z']$ and $k\ge 0$,  it holds  $d(f^k(z),f^k(z''))\le b^{n/6}4^{k-2n}$. Thus for every $m\le n$, we have:
\[\|\partial_x f^m(z'')-\partial_{x}f^m(z)\|= \|\sum_{k=1}^{m-1} D_{f^{k+1} (z'')}f^{n-k-1}\circ  (D_{f^k (z'')}f
-D_{f^k (z)}f)\circ \partial_{x}f^k(z)\|\]
\[\le \sum_{k=0}^{m-1} 4^{n-k-1}\cdot   b^{n/6}\cdot 4^{k-2n+1} \cdot 4^k
\le b^{n/6}\cdot  4^{-n} \frac{4^{m}-1}{4-1}
\le  \frac{b^{n/6}}3 \le  \frac{b^{m/6}}3\; .
\]
Thus for every $z''\in [z,z']$, it holds $\|\partial_x f^m(z'')\|\ge b^{m/6}/2$ for every $m\le n$.  By \cref{2.2}.(a), for every $z''\in [z,z']$, there exists a constant $K$ such that:
\[\|\frac{d}{dz}  (e_n(z'')\times e_1(z))\|\le K \sqrt{ b}.\]
As $e_1$ is $C^1$-$b$-close to $\partial_y$, it comes that $\frac d {dz''}e_n(z'')$ is  $ K\sqrt{b}$-small for every $z''\in [z,z']$. Now we compute: 
\[\|e_\infty (z)\times  e_\infty (z')\|\le
 \|e_\infty (z)\times e_n(z)\|+
  \|e_n (z)\times e_n(z')\|+  \|e_n (z')\times e_\infty (z')\|\]
  \[\le (Kb)^{2\max(n,1)/3} + K \sqrt{b}\cdot d(z,z')  + (Kb)^{2\max(n,1)/3}\le K\sqrt b \cdot d(z,z')\; . \]
\end{proof}

\label{pesinstableprop}
\begin{proof}[Proof of \cref{Pesin stable1}]
By \cref{Lamination Lipschitz}, the vector field  $e_\infty:PH\to \R^2$ defined therein is $K\sqrt b$-Lipschitz close to the constant vector field equal to $(0,1)$. We use the Whitney extension theorem \cite{Wh34, JLS86} to extend it to a Lipschitz vector field $\tilde e_\infty$ on $\R\times [-\theta, \theta]$ which is $b^{1/3}$-Lipschitz close to $(0,1)$. 

Let $\pi_f$ be the holonomy map from $\R\times [-\theta, \theta]$ to the transverse line $\R$. We observe that $\pi_f$ is $b^{1/3}$-Lipschitz close to the first coordinate projection. We notice that the first two items of the proposition are satisfied. 

To prove the third item, we recall that a puzzle piece $(Y,n)$, 
satisfies the following property of  \cref{piecedef}.$(ii)$.
For every $z\in \partial^s Y$ and unit vector $w\in \chi_h$ it holds:
\[\|D_zf^m(w)\|\ge b^{m/6}, \quad \forall m\le n\; .\]
Furthermore, $f^n(z)$ belongs to $\partial^s Y_\se$ by definition of the puzzle piece and so $f^m(z)\in \partial^s Y_\se$ for every $m\ge n$. Also  $D_zf^n(w)$ belongs to $\chi_h$ by \cref{invariancecone}$(ii)$, and so $D_z f^m(w)$ belongs to $\chi_h$ for every $m\ge n$. Consequently:
\[\|Df^m(w)\|\ge b^{m/6}, \quad \forall m\ge 0\; .\]
This implies that $\partial^s Y$ is included in $PH$. Furthermore, every unit vector $u$ tangent to $\partial^sY$ can be written as $w= s\cdot e_\infty+ u\cdot (1,0)$. Then $Df^m(u)$ is $b^{m/2}$ close to $u Df^m(1,0)$. As $Df^m(1,0)$ goes to infinity whereas $Df^m(w)$ goes to zero, it holds $u=0$. This proves that $\partial^s Y$ is tangent to $e_\infty$. The contraction of its tangent vectors follows from \cref{Lamination Lipschitz}.3.

Let us prove the last and fourth item of the proposition. Let $C:= ((Y_k,n_k))_{k\ge 0}$ be a decreasing sequence of puzzle pieces and put $W^s_{C}:= \bigcap_{k\ge 0} Y_k$. 
 For every $k$, let $[x_k^-, x_k^+]\subset I_\se$ be such that $Y_k= \pi_f^{-1}([x_k^-, x_k^+])$. We remark that:
\[ W^s_C= \bigcap_k Y_k= \bigcap_k\pi_f^{-1}([x_k^-, x_k^+])=\pi_f^{-1}(\bigcap_k [x_k^-, x_k^+])\; .\]
We notice that $(n_k)_k$ is increasing and so $n_k\to \infty$. 
By the expansion of $\chi_h$ given by \cref{piecedef}.$(i)$, the length of $[x_k^-, x_k^+]$ is at most $2^{-n_k/3}\leb(I_\se)$. 
Thus the decreasing intersection $\bigcap_k [x_k^-, x_k^+]$ is a single point $x_C\in I_\se$ and it holds:
\[ W^s_C =\pi_f^{-1}(\{x_C\})\; .\]
For every $k$, let $C_k$ be a component of $\partial^s Y_k$. We recall that $C_k$ is a flat stretched curve in $PH$ which is tangent to $e_\infty$. By continuity of $e_\infty$ and compactness of $PH$, the same occurs for any cluster value $C_\infty$ of $(C_k)_k$. Note that $C_\infty$ is included in  $W^s_C$. As $W^s_C$ is a vertical stretched curve, it holds $W^s_C= C_\infty$. Thus $W^s_C$ is included in $PH$ and tangent to $e_\infty$. The contraction of its tangent vectors follows from \cref{Lamination Lipschitz}.3.
\end{proof}

\section{Affine-like representation and parabolic products}\label{section AF}
The affine-like representation\footnote{The present section has been added after that Yoccoz told me his intention to add this formalism in \cite{Y15}.}  was introduced in \cite{PY01,PY09} to present the bounds on the iteration of the dynamics restricted to a piece. Actually, this way of representation is inspired from the generating function, and was intensively studied by Shilnikov \cite{Sh67} and his school in the dissipative context, as the \emph{Shilnikov variable or cross map}.

\subsection{Presentation of the estimates using the affine-like representation}
Let $(Y, n)$ be a  piece and $x_1\in I_\se$. Let $\mathcal C_{x_1}$ be  the preimage by $f^{n}|Y$ of the segment $\{x_1\}\times [-\theta, \theta]$. 
 As $Y_\se= \bigcup_{x_1\in I_\se} \{x_1\}\times [-\theta, \theta]$ contains $f^n(Y)$ we have:
\[ Y= \bigcup_{I_\se} \mathcal C_{x_1}\; .\]
Furthermore, if $(Y,n)$ is a puzzle piece,  the curves  $\mathcal C_{x_1}$ are all vertical and stretched by 
\cref{invariancecone}.$(ii)$. Here is an immediate consequence of \cref{ForB} on the proof of \cref{invariancecone}.$(ii)$:
\begin{coro}
For every $i\le n$ and $x_1\in I_\se$, the curve $\mathcal C_{x_1}$ is $b^{i/2}$ contracted by $f^i$: for every unit vector $u$ tangent to $\mathcal C_{x_1}$, it holds $\|Df^i(u)\|\le b^{i/2}$.
 \end{coro}

Given $y_0\in [-\theta, \theta]$, we denote by $S_Y(y_0)$  the segment $Y\cap \R\times \{y_0\}$. By \cref{invariancecone}.$(i)$, the curve $S^{Y}(y_0):= f^{n}(S_Y(y_0))$ is  horizontal.  We have:
\[f^n(Y)= \bigcup_{[-\theta, \theta]} S^Y(y_0)\; .\]

\begin{defi}[Affine-like representation of $(Y,n)$]\index{Affine-like representation}
The affine-like representation of a piece $(Y,n)$ is the pair of functions $(\cX_0,\cY_1)$ defined by:

The intersection point of $\mathcal C_{x_1}$ with
$\R\times \{y_0\}$ is $(\cX_0(y_0,x_1),y_0)$ and the intersection point of $\{x_1\}\times [-\theta, \theta]$ with $S^Y(y_0)$ is $(x_1,\cY_1(y_0,x_1))$. We notice that:
\begin{equation}
\tag{$AL$}
f^{n}(x_0,y_0)=(x_1,y_1)\Leftrightarrow 
(x_0,y_1)=(\cX_0,\cY_1)(y_0,x_1)\; .\end{equation}

\end{defi}
The domain of $(\cX_0,\cY_1)$ is the projection $\Delta$ on the 2nd and 3rd of the graph of $f^{n}|Y$. The following sheds light on many estimates already computed. 
\begin{prop}\label{AFrepresentation} There exists a universal constant $K>0$ such that the affine-like representation $(\cX_0,\cY_1)$ of a piece satisfies for every $(y_0, x_1)\in \Delta$:
\begin{enumerate}[(1)]
\item  $|\partial_{x_1} \cY_1 (y_0, x_1)|=  b $ and $|\partial_{x_1} \cX_0 (y_0, x_1)|\le \sqrt{1+\theta^2}\cdot 2^{-n/3}$, 
\item $|\partial_{y_0} \cX_0 (y_0, x_1)|\le Kb^{2/3}$ and $|\partial_{y_0} \cY_1 (y_0, x_1)|\le b^{n}\sqrt{1+\theta^2}\cdot 2^{-n/3}$,
\item $|\partial_{x_1} \partial_{y_0} \cX_0 (y_0, x_1)|\le K\sqrt b|\partial_{x_1}  \cX_0 (y_0, x_1)|$ and $|\partial^2_{y_0} \cX_0 (y_0, x_1)|\le  K\sqrt b$,
\item $|\partial_{x_1} \partial_{y_0} \cY_1 (y_0, x_1)|\le (K b)^{n/2}$
and 
$|\partial^2_{y_0} \cY_1 (y_0, x_1)|\le (Kb)^{n/2}$,
\item $|\partial^2_{x_1} \cY_1 (y_0, x_1)|\le K b$.

.
\end{enumerate}
\end{prop}
\begin{proof}
We remark that $(\partial_{x_1} \cX_0 (y_0, x_1), 0)\in \chi_h$ is sent into $(1, \partial_{x_1} \cY_1 (y_0, x_1))$ which has second coordinate at most $ b$ by
\cref{pour prod parabolic} p. \pageref{pour prod parabolic}. Thus $|
\partial_{x_1} \cY_1 |\le  b$. Also by piece's \cref{piecedef}.$(i)$, it holds $|\partial_{x_1} \cX_0 (y_0, x_1) | \cdot 2^{n/3}\le  \sqrt{1+\theta^2}$. This proves \emph{(1)}.
 The second part of (2) is given by the following:
\[| \partial_{y_0} \cY_1|= \left |\det \left ( \begin{array}{cc}
1& \partial_{x_1} \cY_1\\
0& \partial_{y_0} \cY_1\end{array}\right) \right|
\stackrel{(AL)}=|\det Df^{n}|\cdot 
\left |\det \left( \begin{array}{cc}
\partial_{x_1} \cX_0 &0\\
\partial_{y_0} \cX_0 &1\end{array}\right) \right|
\le 
b^{n}| \partial_{x_1} \cX_0|  
 \; .\]
With $e_{n+1}$ the contracted vector defined in the proof of \cref{invariancecone}  \cpageref{proof:invariancecone2}, it holds $e_{n+1}=(\partial_{y_0}\cX_0 ,1)$ and $Df^{n}(e_{n+1})= (0,\partial_{y_0} \cY_1)$.  By \cref{ForA} p. \pageref{ForA}, it holds 
$|e_{n+1}(z)-(0,1)|\le Kb^{2/3}$
and $ |\partial_z e_{n+1}(z)|\le K\sqrt b$. 
This implies  the first part of $(2)$ and  \emph{(3)}:
\begin{multline}
|  \partial_{y_0} \cX_0| =   | e_{n+1}(z)|\le K b^{2/3}\quad ,\quad  
|  \partial^2_{y_0} \cX_0| =   |\partial_{y_0} e_{n+1}(z)|\le K\sqrt b\\ \qand   |\partial_{x_1} \partial_{y_0} \cX_0| =   |\partial_{x_0} e_{n+1}(z)|\cdot 
 |\partial_{x_1} \cX_0| \le K\sqrt b |\partial_{x_1} \cX_0| \; .
\end{multline} 
Similarly \cref{ForB} \cpageref{ForB} implies \emph{(4)}:
 \[ |\partial_{y_0}^2 \cY_1| = |\partial_{y_0}  Df^{n}(  e_{n+1})(z)|\le (K\sqrt b)^{n} \; ,\]
 \[ |\partial_{x_1} \partial_{y_0} \cY_1| = |\partial_{x_0}  Df^{n}(  e_{n+1})(z)|\cdot |\partial_{x_1} \cX_0| \le (K\sqrt b)^{n}
 \; .\]
To get \emph{(5)}, we recall that by \cref{curvature}, for every $y_0\in [-\theta,\theta]$, the curvature of the curve $\{(x_1, \cY_1(x,y_0)): x_1\in I_\se\}$ is dominated by $b$. 
\end{proof}

\subsection{Contraction of the graph transform induced by a piece}
We are going to  prove the following which implies  \cref{contractionPP}:
\begin{prop}  \label{contractionPPplus }There exists a universal constant $K>0$, such that for every puzzle piece $(Y,n)$ in $Y_\se$, the map $S\in \mathcal H\mapsto S^Y:=f^{n}(S\cap Y)\in \mathcal H$ is $ (Kb)^{n/2}$ contracting.
\end{prop}
\begin{proof}
\label{proof Contraction of the graph transform}
 Let $\rho\in C^{1+Lip}(I_\se, \R)$ be a function whose graph $S$ belongs to $\cH$.   Let $\rho^Y\in C^{1+Lip}(I_\se, \R)$ be the function whose graph is $S^Y$. 
 We want to show that $\rho\mapsto \rho^Y$ is $(K b)^{n/2}$-contracting for the $C^1$-topology. By definition of $\mathcal H$ we have:
\begin{equation}\label{norm_rho} \|D\rho\|_{C^0}\le \theta\qand  Lip(D\rho)\le (1+K\theta) \theta\; .\end{equation}
Let $(\cX_0,\cY_1)$ be the affine-like representation of $(Y,n)$. We notice that $\rho^Y$ is implicitly defined by the system
\begin{equation}\label{solution implicit} \rho^Y(x_1)= \cY_1(y_0, x_1)\qand y_0= \rho \circ \cX_0(y_0, x_1)\; .\end{equation}
As the map $(y_0, \rho, x_1)\mapsto  \rho\circ \cX_0(y_0, x_1)$ is of class $C^{1+Lip}$, with its first derivative $\theta^2$-contracting by 
\cref{AFrepresentation}.(2) and \eqref{norm_rho},
 there is a $C^{1+Lip}$-function $\cY_0$ such that:
\[y_0= \cY_0(x_1, \rho)\Leftrightarrow y_0= \rho \circ \cX_0(y_0, x_1)\; .\]
\begin{multline}\label{tunetune} \partial_{x_1} \cY_0(x_1, \rho) = 
D \rho  \circ  \partial_{x_1}\cX_0(\cY_0( x_1,\rho ), x_1)+D\rho \cdot\partial_{y_0} \cX_0(\cY_0( x_1,\rho ), x_1) \partial_{x_1} \cY_0(x_1, \rho)\\
 \partial_\rho \cY_0(x_1, \rho)=  \partial \rho  \circ  \cX_0(\cY_0( x_1,\rho ), x_1)+ D\rho \cdot\partial_{y_0} \cX_0(\cY_0( x_1,\rho ), x_1) \partial_\rho \cY_0(x_1, \rho)
 \end{multline}
 
We observe that  $\rho^Y= x_1\mapsto  \cY_1(\cY_0(x_1, \rho),x_1)]$. By the two next lemmas $\rho\mapsto \rho^Y$ is the composition of two Lipschitz maps of  constants 
$\le (Kb)^{n/2}$ and $\le \theta(1+\theta)$. This implies  the proposition.\end{proof}
\begin{lemm}\label{contraction Y Lip Kbn/2}
The map $\rho\in C^1(I_\se, [-\theta,\theta])\mapsto [x_1\mapsto  \cY_1(\rho(x_1),x_1)]$ is $(Kb)^{n/2}$-Lipschitz. \end{lemm}
\begin{proof}
It suffices to observe that the linear map 
$\partial_\rho\mapsto  [x_1\mapsto \partial_{y_0}\cY_1(\rho(x_1), x_1)\cdot \partial\rho(x_1)]$ 
has norm at most $(Kb)^{n/2}$ because $ \partial_{y_0} \cY_1$ has $C^1$-norm at most $(Kb)^{n/2}$ by \cref{AFrepresentation} (2-4).
\end{proof}

\begin{lemm}\label{X0 lip} 
The map $x_1\mapsto \cY_0(x_1, \rho)$ is $\theta(1+\theta)$-$C^1$-small and the map  $\rho \mapsto [x_1\mapsto  \cY_0(x_1, \rho)]$ is $ (1+\theta)$-Lipschitz for the $C^1$-tolopogy.
\end{lemm}
\begin{proof}By density, we can assume that $\rho$ is of class $C^2$.  We have by \eqref{tunetune}:
\begin{eqnarray}
\partial_{x_1} \cY_0(x_1, \rho)&=& \frac{D \rho  \circ  \partial_{x_1}\cX_0(\cY_0( x_1,\rho ), x_1)}{1-  D\rho \cdot\partial_{y_0} \cX_0(\cY_0( x_1,\rho ), x_1)}\label{partialx1Y_0}
\\
 \partial_\rho \cY_0(x_1, \rho)&:&\partial \rho \mapsto  \frac{\partial \rho  \circ  \cX_0(\cY_0( x_1,\rho ), x_1)}{1-  D\rho \cdot\partial_{y_0} \cX_0(\cY_0( x_1,\rho ), x_1)}\end{eqnarray}
We recall that by \cref{AFrepresentation} $(2)$ and $(3)$, it holds 
$|\partial_{y_0}\cX_0|\le Kb^{2/3}$ and $|\partial_{x_1}\cX_0\le 1|$. Thus by \eqref{norm_rho} and \eqref{partialx1Y_0}, 
 $\partial_{x_1} \cY_0$ is at most $\theta(1+\theta)$. As the $C^0$ norm of $\cY_0$ at most $b$, we get the first statement of the lemma. Moreover, we observe that the norm of the derivative of $x_1\mapsto \cX_0(\cY_0(x_1,\rho),x_1)$ is at most $1$. Also  \cref{AFrepresentation} $(3)$, the $C^1$-norm of $\partial_{y_0}\cX_0$ is at most $K\sqrt b$ and so the $C^1$-norm of   $x_1\mapsto \partial_{y_0} \cX_0(\cY_0(x_1,\rho), x_1)$ is at most $K \sqrt b$. Now we infer that $\partial\rho$ 
  and $D\rho$ have $C^1$-norm at most $(1+K\theta)\theta$ by \eqref{norm_rho}. Thus the linear form $\partial_\rho \cY_0(x_1,\rho)$  has norm at most $1+K\theta$ and the  linear map  $\partial_\rho \mapsto [x_1\mapsto  D \cY_0(x_1, \rho)(\partial_\rho)]$  has norm  $\le (1+\theta)$ for the $C^1$-tolopogy. Form this we get the last statement of the lemma.
\end{proof}

\subsection{Graph transform induced by a parabolic operation}
\label{proof of disjoint box}
In this section, we are going to prove propositions \ref{graphtransformpara} and \ref{disjoint box}. 
Let $(Y,n)$ and $(Y',n')$ be two puzzle pieces such that $Y'\subset Y \subset Y_\se$  with $n'\le n +\max(M, n)$ such that $\tilde \cD(\boxdot (Y-Y')$ is not empty. In particular, this implies that $\partial^s Y$ and $\partial^s Y'$ does not have the same left component. 
Put $(\tilde Y, M+n)=(Y_\sr, M)\star (Y,n)$ and  $(\tilde Y', M+ n')=(Y_\sr, M)\star (Y',n')$. Similarly, the right components of 
$\partial^s\tilde Y$ and $\partial^s \tilde Y'$ are disjoint.

 Let us denote by $(\cX_0, \cY_1)$ the affine-like representation of $(\tilde Y, M+n)$. Given $S=graph\, \rho\in \cD(\boxdot (Y-Y'))$, we denote $\iota_\rho(x)= (x, \rho(x))$ for every $x\in I_\se$. 

Let $\epsilon\in \{-,+\}$ be such that 
 $graph\, \cX_0(\cdot , \epsilon \alpha_0)$ is the right component of $\partial^s \tilde Y$. We notice that $\epsilon $ is the sign 
 of $-\partial_{x_1} \cX_0$. Let $x_\boxdot\in I_\se$
 be such that the segment $[x_\boxdot, \epsilon \alpha_0]$ is minimal to satisfy that  $\{(\cX_0(y_0, x_1), y_0): y_0\in [-\theta, \theta], x_1\in [x_\boxdot, \epsilon\cdot  \alpha_0]\}$ contains the right component of $\tilde Y\setminus  \tilde Y'$.

\begin{lemm}\label{graph transform parabolic init}
 For every $S=graph\,  \rho\in \cD(\boxdot (Y-Y'))$, for every $x_1\in [x_\boxdot, \epsilon\cdot \alpha_0]$,  there are exactly two preimages $\cX_{-1}^-(x_1, \rho)<\cX_{-1}^+(x_1, \rho)$  in $ (f\circ \iota_\rho) ^{-1}(graph\, \cX_0(\cdot , x_1))$ and it holds for $\pm\in \{-,+\}$:
 \[S_{\boxdot_\pm  (Y-Y')}\subset \{\iota_\rho\circ  \cX_{-1}^\pm (x_1, \rho): x_1\in [x_\boxdot, \epsilon\cdot \alpha_0]\}\; .\]
\end{lemm}
We will prove this lemma below. Let us give the algorithm of extension. 
\begin{figure}[h]\label{extensioncourbe}
    \centering
        \includegraphics[width=10cm]{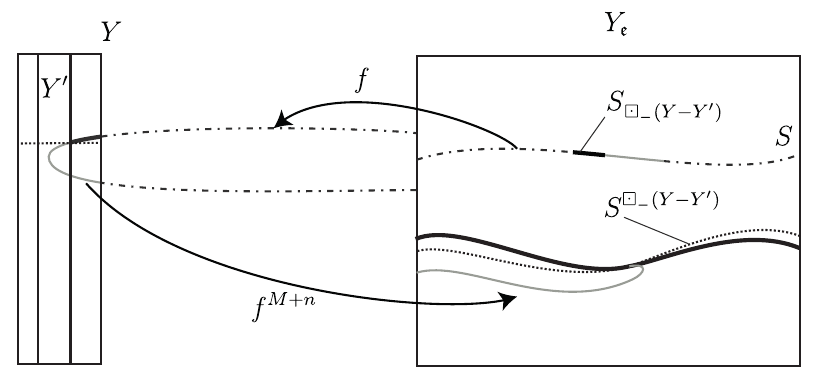}
    \caption{Extension $S^{\boxdot_-  (Y-Y')}$.}
\end{figure}
\begin{defi}[The extension algorithm]\label{extension algo} Let $S:= Graph\, \rho\in \cD(\boxdot_\pm (Y-Y'))$. Put $\cY_0^\pm$ be the $y$-coordinate of $f\circ \iota_\rho\circ  \cX_{-1}^\pm$ and put:
$$T^\pm (\rho): x_1\in [x_\boxdot, \epsilon\cdot \alpha_0 ] \mapsto 
\cY_1(\cY_0^\pm (x_1, \rho), x_1)\; .$$
Each function  $ T^\pm(\rho)$ is $C^{1+Lip}$. 
Let $\tilde T^\pm (\rho)$ be the unique $C^{1+Lip}$-extension of $T\pm(\rho)$ on $I_\se$ s.t. :
\[\partial^2_{x_1} \tilde T^\pm(\rho)(x_1)
= \partial^2_{x_1} \cY_1(\cY_0^\pm (x_\boxdot, \rho), x_1),\quad \forall x_1\in [-\epsilon\alpha_0 , x_\boxdot)
\]
We put $S^{\boxdot_\pm(Y-Y')}:= graph\, \tilde T^\pm(\rho)$. 
\end{defi}

We recall that $(\cX_{-1}, \rho\circ \cX_{-1})(x_1, \rho)\in S$ is sent by
$f$ to  $graph\, \cX_0(\cdot , x_1)$ and so to by $f^{n+M+1}$ to $\{x_1\}\times [-\theta, \theta]$, for every $x_1\in [x_\boxdot, \epsilon \cdot \alpha_0]$.  Thus the graph of $T^\pm(\rho)$ contains the image by $f^{M+1+n}$ of $S_{\boxdot_\pm(Y-Y')}$.  
\medskip 

We are going to show that $S^{\boxdot_\pm(Y-Y')}$ is a horizontal stretched curve and satisfies \cref{graphtransformpara}. Before this, let us show the following extension of  \cref{graph transform parabolic init}:
\begin{lemm}\label{graph transform parabolic init2}
Under the assumption of \cref{graph transform parabolic init}, it holds moreover with $N:=n'+M+2\aleph(n')$: 
\begin{enumerate}
\item $|\partial_{x_1}\cX^\pm_{-1}|$ is small compared to $2^{N}$,
\item The map $x_1\in [x_\boxdot, \epsilon\cdot \alpha_0 ] \mapsto \cY_0(x_1, \rho)$ has $C^0$- norm at most $b$, 
its $C^1$-norm is less than $2^N\cdot b$ and the Lipschitz constant of its derivative is  less than $ 8^N\cdot b$.
\item  $\rho\mapsto [x_1\mapsto \cX_1(x_1, \rho)]$ and  $\rho\mapsto [x_1\mapsto \cY_0(x_1, \rho)]$ are $o(b\cdot 32^{N})$-Lipschitz.
\end{enumerate}
\end{lemm}
\begin{rema}
We note that: $N\le M+\frac M{24} +n'+\frac{n'}{12}\le  3M +3n'$.
\end{rema}
Before proving this lemma let us state the following improvement of \cref{graphtransformpara}: 
\begin{prop}\label{prop pour graphtransformpara (i)}
For every $\pm\in \{-,+\}$, it holds:
\begin{enumerate}
\item for every  $S=graph\, \rho\in \tilde \cD(\boxdot (Y-Y'))$,   the graph  $S^{\boxdot_\pm (Y-Y')}$ 
of $\tilde T^\pm(\rho)$ is a stretched, horizontal curve.
\item The map $S\in \tilde \cD(\boxdot (Y-Y')) \mapsto S^{\boxdot_\pm (Y-Y')}$ is $o(b^{\frac{M+n+1}3})$ contracting.
\item The image $\mathcal I_Y$ of $S\in \mathcal H\mapsto S^Y =f^n(S\cap Y)$ and the image 
$\mathcal I_{\boxdot_\pm (Y-Y')}$ of $S\in \tilde \cD(\boxdot (Y-Y')) \mapsto S^{\boxdot_\pm (Y-Y')}$ are in a small ball of $\mathcal H$ of radius  $o(\theta b^{\frac{n}3})$.
\end{enumerate}
\end{prop} 
\begin{proof}[Proof of \cref{prop pour graphtransformpara (i)}]
 By density, we can assume that $\rho$ is of class $C^2$.
We notice that $T^\pm(\rho):x_1\in [x_\boxdot, \epsilon x_0] \mapsto \cY_1( \cY_0(x_1, \rho), x_1)$.
Thus it holds:
\begin{equation}\label{compo rules derivees}\partial_{x_1} T^\pm(\rho) = \partial_{y_0} \cY_1 \partial_{x_1} \cY_0+ \partial_{x_1} \cY_1
\text{ and }
\partial_{x_1}^2 T^\pm(\rho) = \partial_{y_0} \cY_1 \partial^2_{x_1} \cY_0+ \partial^2_{y_0} \cY_1 (\partial_{x_1} \cY_0)^2 +\partial_{x_1}^2 \cY_1
\; .\end{equation}
We now use \cref{AFrepresentation} to bound the derivatives of $\cY_1$ and \cref{graph transform parabolic init2}.2 to bound the derivative of $\cY_0$. It comes:
\[|\partial_{x_1} T^\pm(\rho)| \le   o(b^{n+M}) 2^N b+ b
\text{ and }
|\partial_{x_1}^2 T^\pm(\rho)| \le o(b^{n+M} ) 8^N \theta + (Kb)^{\frac{n+M}2} (2^N b)^2 +Kb 
\; .\]
Thus:
\[|\partial_{x_1} T^\pm(\rho) |\le  Kb
\qand
|\partial_{x_1}^2 T^\pm(\rho)| \le Kb 
\; .\]
As $|T^\pm(\rho)|\le b$, it comes that $T^\pm(\rho)$ has $C^{2}$-norm at most $Kb$. As $ \partial^2_{x_1} \cY_1(\cdot , 0)$ is at most $Kb$ by \cref{AFrepresentation}.5, it comes that 
$\partial_{x_1}^2\tilde T^\pm(\rho)(x_1)$ is at most $Kb$ for every $x_1\in I_\se$. Thus 
 $\partial_{x_1}\tilde T^\pm(\rho)(x_1)$  and $\tilde T^\pm(\rho)(x_1)$ are at most $Kb$ for every $x_1\in I_\se$.  In particular the graph $S^{\boxdot_\pm (Y-Y')}$  of $\tilde T^\pm(\rho)$ is horizontal and included in $Y_\se$ and so its  endpoints are in $\partial^sY_\se$. Thus $S^{\boxdot_\pm (Y-Y')}$  is stretched. \medskip 
 
To show the contraction $S\mapsto S^{\boxdot_\pm (Y-Y')}$, we notice that by the algorithm of extension, given $S=graph\, \rho$ and $\breve S=graph \, \breve \rho$ in 
$\tilde \cD(\boxdot (Y-Y')$, we have:
\[\max_{x\in I_\se} \| \partial_x \tilde T^\pm (\rho)  -\partial_x \tilde T^\pm (\breve \rho)  \|=  
\max_{x\in [x_\boxdot, \epsilon\alpha_0 ]} \| \partial_x  T^\pm (\rho)  -\partial_x   T^\pm (\breve \rho) \|\; .\]
Thus the contraction of $S\mapsto S^{\boxdot_\pm (Y-Y')}$ is at most $\leb\, I_\se$ times the contraction of $\rho\mapsto T^\pm (\rho)$ for  the $C^1$-topology.  Thus by \cref{graph transform parabolic init2}.3 and \cref{contraction Y Lip Kbn/2}, the latter is $(Kb)^{(M+n)/2}\cdot o(b 32^N)$-contracting. This proves the second statement of the proposition. \medskip

As the diameter of $\mathcal H$ is $4\theta$,  by the second statement the diameter of $\mathcal I_{\boxdot_\pm (Y-Y')}$  is  $o(\theta b^{n/3})$ and by \cref{contractionPPplus }, the diameter of $\mathcal I_Y$ is  at most $4\theta (Kb)^{n/2}=o(\theta b^{n/3})$. Thus to prove the last statement it suffices to show that  one $C^2$-curve $S=graph\, \rho  \in \tilde \cD(\boxdot (Y-Y'))$ has its image $S^{\boxdot_\pm (Y-Y')}$ which is $o(\theta b^{(n+M)/3})$-close to $\breve S^Y$ for $\breve S\in \mathcal H$. We take  $\breve S:= f^{M}(Y_\sr\cap \R\times \{0\})$. Then $\breve S^Y$ is the graph of $\mathcal Y_0(0,\cdot )$.  A to prove the third statement, it suffices to show that $x_1\in I_\se\mapsto \cY_1(0,x_1)$ is at a distance $o(\theta b^{n/3})$ of $\tilde T^\pm (\rho)$. 
For every $x_1\in [x_\boxdot, \epsilon \alpha_0]$, we have $|\cY_1(0,x_1)- T^\pm(\rho)(x_1)|\le b^{M+n}$ by contraction of the fibers of $\cX_0$ and by \cref{compo rules derivees}, it comes:
\begin{multline} \partial_{x_1} \cY_1(0,x_1)- \partial_{x_1} T^\pm(\rho)(x_1)= \partial_{y_0}\cY_1(\cY^\pm _0(x_1, \rho), x_1) 
\cdot \partial_{x_1}\cY^\pm _0(x_1, \rho)  \\
+\partial_{x_1 }\cY_1(\cY^\pm _0(x_1, \rho), x_1) - \partial_{x_1 }\cY_1(0, x_1) = o(b^{M+n+1}\cdot 2^N)+O( b\cdot \partial_{y_0} \partial_{x_1 }\cY_1)\; .\end{multline}
We now infer that $\partial_{y_0} \partial_{x_1 }\cY_1 =O( (Kb)^{(n+M)/2}$ by \cref{AFrepresentation}.4. Thus  $T^\pm(\rho)$ is $O(b^{(M+n)/2+1})$-close to $x_1\in [x_\boxdot, \epsilon \alpha_0]\mapsto \cY_1(0,x_1)$.  As the second derivative of $\tilde T^{\pm}(\rho) |  [-\epsilon \alpha_0,x_\boxdot)$ and $\cY_1(0, \cdot ) |  [-\epsilon \alpha_0,x_\boxdot)$ are equal it comes that the $C^1$-distance between $\tilde T^{\pm}(\rho)$ and $\tilde T^{\pm}(\rho)$ is small compared to $\theta b^{(n+M)/3}$.  \end{proof}

Before proving the lemmas \ref{graph transform parabolic init} and \ref{graph transform parabolic init2}, let us introduce a few notations. We put:
\[f(x,y)= (x^2+a+B_x(x,y), B_y(x,y))\]
Let $e:z\in Y\mapsto (e_x(z), 1)$ a vector field tangent to the fibers of $\cX_0$: $e_x(z)=\partial_{y_0} \cX_0(y_0, x_1)$ with $(y_0,x_1)$ the second and third coordinates of $(z, f^{M+n}(z))$.

Given $S=graph\, \rho\in \cH$, the following is the determinant between the tangent vector $\partial_x (f\circ \iota_\rho)(x)$ of $f(S)$ and the tangent vector $e$ of the fibers of $\cX_0$ at the corresponding point:
\[\Theta(x,\rho)=Df\circ \partial_x \iota_\rho(x)\times e\circ f\circ \iota_\rho(x)\; .\]
This function will play a key role in the proof below. 

\begin{proof}[Proof of \cref{graph transform parabolic init}] 
Let $S=graph\, \rho\in \tilde \cD(\boxdot(Y-Y'))$.   By density we may assume that $S$ and so $\rho$ are of class $C^2$.  We recall that $f(S)$ intersects exactly one component of $\partial^s Y'$ ,  and so exactly  one component of $\partial^s Y$. As  $e$ is tangent to $\partial^sY$, there exists $\xi$ such that $Df\circ D\iota_\rho(\xi)$ is tangent to $e$. This means $\Theta(\xi,\rho)=0$.
Let us develop $\Theta$:
\[\Theta(x,\rho) = 2x+\partial_x  (B_x\circ \iota_\rho)(x)- 
\partial_x  (B_y\circ \iota_\rho) (x)\cdot  e_x\circ   f\circ \iota_\rho(x)\; .\]
 We recall that $z\mapsto e_x(z)$ is $C^1$-$K\sqrt b $-small by  \eqref{ForA} p. \pageref{ForA}. This implies:
 
 \begin{fact} \label{THeta est C1 proche de 2x}
 The function $(x, \rho)\mapsto  \Theta(x,\rho) $ is $O(b)$-$C^1$-close to $(x, \rho)\mapsto 2x$. 
 \end{fact}
Consequently, at $\rho$ fixed,  the function $x\mapsto  \Theta(x,\rho) $ has a unique zero, and so the tangency point $\xi$ is unique. Moreover $\xi$ is dominated by $b$. 

Let $(\cX_0', \cY_1')$ be the affine-like representation of $(\tilde Y', M+n')$. Let $(\cX_0'', \cY_1'')$ be the affine-like representation of the puzzle piece $(Y'', n'')$ whose box is the right component of $\partial_{\aleph(n')} \tilde Y'$.  We remark that:
\[\min_{y_0\in [-\theta, \theta]} \leb(  Y''\cap \R\times \{y_0\})\ge \min |\partial_{x}  \cX_0''|\gg 4^{-N}\; ,\]
 because $n''\le N$, $\|Df\|\le 4$ and the inequality is large because $(Y'', n'')$  is obtained using  $\aleph(n')\gg 1$-times a $\star$-product with $(Y_{\ss_-}, n_{\ss_-})$ whose  horizontal expansion is bounded $2^{2\aleph(n')(1+\sqrt{M})}$ by \cref{coro pour greatly}.

 For $z=(x_0, y_0)\in \tilde Y'$, let $(e'_x(z),1)$ be the tangent space to the fiber of $\cX_0'$:
$e_x'(z) =  \partial_{y_0}\cX_0'({y_0,x'_1})$  with $(y_0, x_1')$ the second and third coordinate projection of $(z,f^{M+n'}(z))$.
 By \cref{2.1}.$(a)$ and \cref{2.2}.$(a)$, it holds:
 \begin{fact}\label{distance e et e'}
The vector fields  $e|\tilde Y $ and $e' $ are $(Kb)^{(M+n)/2}$-$C^1$-close. The curve $\cX_0(\{x_\boxdot\}\times [-\theta, \theta])  $ is $K b\cdot \theta^{n'}$-$C^1$-close to the right component of $\partial^s \tilde Y'$. 
 \end{fact}

 By definition of $\tilde \cD(\boxdot(Y-Y'))$, the curve  $S$ is $\theta^{n'}$-close to a certain curve $S'\in \mathcal H$ which is in critical position with $(Y',n')$. 
 Put $S'=:graph\, \rho'$. By \cref{defi Critical position} of the critical position, the set $f(S')$ intersects in the interior of $\tilde Y'\setminus \partial_{\aleph(n')} \tilde Y'$ but not the left component of $\partial_{\aleph(n')} \tilde Y'$. This means that there exists a point $\xi'$ such that:
 \begin{equation}\label{def xi'}\left\{\begin{array}{l} f\circ  \iota_{ \rho'}(\xi')\in cl(\tilde Y'\setminus \partial_{\aleph(n')} \tilde Y')\\
 0=\Theta'(\xi', \rho')\quad  \text{ with }  \quad 
 \Theta'(x, \rho')=   Df\circ D\iota_{\rho'}(x)\times e'\circ f\circ \iota_{\rho'}(x)\end{array}\right.
.\end{equation}
 \begin{fact}\label{fact sur distance des xi2}
The distance from $f\circ \iota_{\rho'} (\xi')\in cl(\tilde Y'\setminus \partial_{\aleph(n')} \tilde Y')$ to $\partial^s \tilde Y'$ satisfies:
\[d(f\circ \iota_{\rho'}(\xi'),\partial^s \tilde Y')
\ge   \min |\partial_{x}  \cX_0''|\gg 4^{-N}
 \; .\] 
\end{fact}

 As $e $ and $e' $ are $(Kb)^{(M+n)/2}$-$C^1$-close by fact \ref{distance e et e'}  and $\rho, \tilde \rho$  are $\theta^{n'}$-$C^1$-close, the functions $\Theta(\cdot , \rho)$ and $\Theta'(\cdot , \rho')$  are $K b\cdot \theta^{n'}$-close. This implies:
 \begin{fact}\label{fact sur distance des xi}
The points $\xi$ and $\xi'$ are $K b\cdot \theta^{n'}$-close.
\end{fact}
We notice that $4^{-N}$ is large compared to $b\theta^{n'}$. Thus by facts \ref{distance e et e'}, \ref{fact sur distance des xi2} and \ref{fact sur distance des xi}, it comes:
 \begin{fact}\label{fact sur distance des xi3}
The point $f\circ \iota_\rho (\xi)$ is at the left hand side of $\cX_0([\theta, \theta]\times \{x_\boxdot\})$ at a distance which satisfies:
\[d(f\circ \iota_\rho(\xi), \cX_0([\theta, \theta]\times \{x_\boxdot\}))\ge\frac1K  
\min |\partial_{x}  \cX_0''|\gg 4^{-N}
 \; .\] 
\end{fact}
As $\cX_0([\theta, \theta]\times \{x_\boxdot\}))$ is a vertical curve, the latter distance is proportional to:
\[\cX_0(B_y\circ \iota_\rho(\xi), x_\boxdot)-(
\xi^2+a+ B_x\circ \iota_\rho(\xi))\ge\frac1K  
\min |\partial_{x}  \cX_0''|\gg 4^{-N}\; .\]
 
Thus for every $x_1\in [x_\boxdot, \epsilon\cdot \alpha_0]$, it holds:
\[m:=\cX_0(B_y\circ \iota_\rho(\xi), x_1)-(
\xi^2+a+ B_x\circ \iota_\rho(\xi))\ge\frac1K  
\min |\partial_{x}  \cX_0''|\gg 4^{-N}\; .\]
The set $ (f\circ \iota_\rho) ^{-1}(graph\, \cX_0(\cdot , x_1))$ is formed by the zeros of the following function:
\begin{equation}\label{def Qx1rho} Q_{x_1, \rho}: x\mapsto  \cX_0(B_y\circ \iota_\rho(x), x_1)-(x^2+a+ B_x\circ \iota_\rho(x))\; .\end{equation}
 As the function $Q_{x_1, \rho}$ takes the value $m$ at $\xi$ and has derivative close to $x\mapsto -2x$, there are exactly two points $\cX_{-1}^-(x_1, \rho)<\cX_{-1}^+(x_1, \rho)$ such that 
\begin{equation}\label{def X-1pm} x_{-1 }= \cX_{-1}^\pm (x_1,\rho)
\Leftrightarrow  Q_{x_1, \rho}(x_{-1})=0\text{ and } \pm(x_{-1}-\xi)>0
\end{equation}
Moreover $\xi$ is between $
\cX_{-1}^-(x_1, \rho)$ and $\cX_{-1}^+(x_1, \rho)$  at a distance to them satisfying:
\begin{equation}\label{minoration Qx1rho} |\cX_{-1}^\pm(x_1, \rho) - \xi|\asymp  \sqrt {m}
\ge\frac1K  
\min \sqrt{|\partial_{x}  \cX_0''|}\gg 2^{-N}
\quad \text{where }\Theta(\xi, \rho)=0\; . 
\end{equation}
As $\cX_0([-\theta, \theta]\times [x_\boxdot, \epsilon \alpha_0])$ contains the right component of $\tilde Y\setminus \tilde Y'$,  we have that $S_{\boxdot_\pm  (Y-Y')}\subset \{\iota_\rho\circ  \cX_{-1}^\pm (x_1, \rho): x_1\in [x_\boxdot, \epsilon\cdot \alpha_0]\}$. \end{proof}
\begin{rema}\label{rema pour la prochaine preuve} By fact \ref{THeta est C1 proche de 2x} and \eqref{minoration Qx1rho} it holds for every $ x_1\in [x_\boxdot, \epsilon\cdot \alpha_0]$:
\[| \Theta(\cX_{-1}^\pm(x_1, \rho),\rho)|\ge\frac1K  
\min \sqrt{|\partial_{x}  \cX_0''|}\gg 2^{-N}\; .\]
\end{rema}
\begin{rema}\label{pour la prochaine remark} By fact \ref{fact sur distance des xi2} and since $\Theta'$ is $C^1$close to $(x, \rho) \mapsto 2x$ with a zero at $\xi'$, if $z=\iota_\rho (x)$ is sent to $\partial^s\tilde Y'$ by $f$, it holds:
 \[| \Theta'(x,\rho)|=|Df\circ D\iota_{\rho'}(x)\times e'\circ f\circ \iota_{\rho'}(x)|\gg 2^{-N}\; .\]
\end{rema}

\begin{proof}[Proof of \cref{graph transform parabolic init2}] 
We continue with the notations of the proof of the latter lemma. In particular, we still assume $\rho$  of class $C^2$ by density. Equation (\ref{def Qx1rho}) invites us to consider:
\begin{equation}\label{def Q} Q: (x_{-1}, x_1, \rho)\mapsto  \cX_0(B_y\circ \iota_\rho(x_{-1}), x_1)-(x^2_{-1}+a+ B_x\circ \iota_\rho(x_{-1}))\; .\end{equation}
We notice that if $x_{-1}=\cX_{-1}^\pm(x_1, \rho)$,  then $Q(x_{-1}, x_1, \rho)=0$  by \eqref{def X-1pm}, and it holds:
\[\partial_{x_{-1}}  Q(x_{-1}, x_1, \rho)=\Theta(x_{-1},\rho)\quad \text{at }x_{-1}=  \cX_{-1}^\pm(x_1, \rho)
\; .\]
Thus by \cref{rema pour la prochaine preuve} it holds:
\begin{equation}\label{min partialx-1 Q}
 |\partial_{x_{-1}}  Q(\cX_{-1}^\pm(x_1, \rho) , x_1, \rho)|\ge\frac1K  
\min \sqrt{|\partial_{x}  \cX_0''|}\gg 2^{-N}\; .
\end{equation} 
By the implicit function theorem, the function $(x_1, \rho)\mapsto \cX_{-1}^\pm(x_1, \rho)$ defined by \eqref{def X-1pm} is of class $C^1$ (for variations  $\partial \rho$ of $\rho$ in the Banach space of $C^1$-functions). Also we have:
\begin{eqnarray}
\partial_{x_1} \cX^\pm_{-1}& =& -\frac{\partial_{x_1}  Q(\cX^\pm_{-1}, x_1, \rho)}{\partial_{x_{-1}}  Q(\cX_{-1}^\pm , x_1, \rho)} 
\label{derivex1X-1}\; .\\
\partial_{\rho} \cX^\pm_{-1}&= &-\frac{\partial_{\rho }  Q(\cX^\pm_{-1}, x_1, \rho)}{\partial_{x_{-1}}  Q(\cX_{-1}^\pm , x_1, \rho)}
 \; .\label{deriverhoX-1}
\end{eqnarray}
We now use \eqref{min partialx-1 Q} and the fact that $\partial_{x_1}\cX_0$ is small by \cref{AFrepresentation}.$(1)$, to obtain the first item of the lemma:
\begin{equation}\label{bound derive x1 X-1}
|\partial_{x_1} \cX^\pm_{-1}|  \le K 
\frac{ {|\partial_{x_1} \cX_0(\cY_0, x_1)|}}{  
\min \sqrt{|\partial_{x}  \cX_0''|}} \ll 2^N
\end{equation}
We recall that:
\[\cY_0(x_1, \rho):= B_y\circ \iota_\rho(\cX^{\pm}_{-1}(x_1, \rho))\; .\]
The function $x_{-1}\mapsto B_y\circ \iota_\rho(_{-1})$ is $b$-$C^1$-small and its derivative is $Kb$-Lipschitz. This with \eqref{bound derive x1 X-1} implies the $C^1$-bound on $x_1\mapsto \cY_0(x_1, \rho)$ claimed in the  lemma. Also to obtain the bound on the Lipschitz constant of its derivative, it suffices to show that $x_1\mapsto \partial_{x_1} \cX_{-1}(x_1, \rho)$ is $8^N$-Lipschitz. We compute:
\begin{equation}
\partial_{x_1}^2 \cX^\pm_{-1}= -\frac{
\partial_{x_1}^2  Q(\cX^\pm_{-1}, x_1, \rho)}{\partial_{x_{-1}}  Q(\cX_{-1}^\pm , x_1, \rho)}+\frac{\partial_{x_1}  Q(\cX^\pm_{-1}, x_1, \rho)\cdot \partial_{x_1}(\partial_{x_{-1}}  Q(\cX_{-1}^\pm , x_1, \rho)
)
}{(\partial_{x_{-1}}  Q(\cX_{-1}^\pm , x_1, \rho))^2} \; .
\label{derivex1x1X-1}
\end{equation}
We have $\partial_{x_1}^2 Q= \partial_{x_1}^2 \cX_0$ which is small compared to $2^{4(M+n)/3}$ by 
\cref{roughdistbound} and \ref{AFrepresentation}.$(5)$. Thus by \eqref{min partialx-1 Q},  the first term of the above sum is small compared to $2^{7N/3}$. Then the following computation -- which uses  fact \ref{THeta est C1 proche de 2x} -- achieves the proof of the second item of the lemma:
\begin{multline}
|\partial_{x_1}^2 \cX^\pm_{-1}|\ll 2^{7N/3}+ 4^N |\partial_{x_1}(\partial_{x_{-1}}  Q(\cX_{-1}^\pm , x_1, \rho)
)|
= 2^{7N/3}+ 4^N |\partial_{x_1} \Theta( \cX_{-1}^\pm, \rho)|
\\
\le  2^{7N/3}+ 3\cdot 4^N  |\partial_{x_1} \cX_{-1}^\pm |\ll
2^{7N/3}+  8^N \; .
\label{derivex1x1X-1 bis}
\end{multline}
To prove the third and last item of the lemma, we notice that:
\[\partial_\rho \cY_0= DB_y (\partial_{\rho}\iota_\rho) (\cX^{\pm}_{-1}) + \partial_{x_{-1}}(B_y \circ \iota_\rho)\partial_\rho \cX^{\pm}_{-1}\; .\]
The first term of the sum is $b 2^N$-contracting for the $C^1$-topology. As $\partial_{x_{-1}}(B_y \circ \iota_\rho)$ is a $C^{1}$-function of $x_{-1}$ of norm bounded by $b$, it suffices to show that $\partial_\rho \cX^{\pm}_{-1}$ is $\ll 32^N$-contracting for the $C^{1}$-topology.  We compute this derivative using \eqref{deriverhoX-1}:
\[
\partial_{\rho} \cX^\pm_{-1}= \frac{\partial_{\rho }  Q(\cX^\pm_{-1}, x_1, \rho)}{\Theta(\cX^\pm_{-1}, \rho)}=
\frac{ (\partial_{y_0} \cX_0)\circ (DB_y)\circ (\partial_{\rho}\iota_\rho) (\cX^\pm _{-1}) - D B_x\circ (\partial_{\rho} \iota_\rho)(\cX^\pm_{-1})}{\Theta(\cX^\pm_{-1}, \rho)}
 \; .\]
The function $x_1\mapsto \Theta(\cX^\pm_{-1}(x_1, \rho), \rho)$ has $C^1$-norm bounded by $2^N$ and its values are of modulus $\ge 2^{-N}$. Thus its reciprocal has $C^1$-norm bounded by   $8^N$. We already saw that the operators $D B_x\circ (\partial_{\rho} \iota_\rho)(\cX^\pm_{-1})$ and $D B_y\circ (\partial_{\rho} \iota_\rho)(\cX^\pm_{-1})$ are $b\cdot 2^N$-contracting for the $C^1$-topology.  Finally, using \cref{AFrepresentation}-$(2-3)$,  the function $x_1\mapsto (\partial_{y_0} \cX_0(B_y\circ \iota_\rho(\cX^\pm _{-1}), x_1)$ has $C^1$-norm bounded by $\sqrt b$ plus $\sqrt b$ times the one of $B_y\circ \iota_\rho(\cX^\pm _{-1})$ which is smaller than $2^Nb$. This implies the sough inequality:
\[
Lip(\rho \mapsto \partial_{\rho} \cX^\pm_{-1})\ll
8^N
| \sqrt b (1+2^N b )b 2^N + b 2^N|\le K \cdot b\cdot  32^N
 \; .\]
 \end{proof}
 The following will be used for the proof of   \cref{disjoint box}
\begin{rema}\label{pour la prochaine preuve prop i} 
By using \cref{pour la prochaine remark} instead of \cref{rema pour la prochaine preuve} and the affine-like representation of $(\tilde Y', n'+M)$ instead the one of $(\cX_0, \cY_1)$, the same argument as for \eqref{bound derive x1 X-1} shows that if $z\in S\cap f^{-1}(\partial^s\tilde Y')$, it holds:
 \[\|\partial_x (f^{1+M+n'} \circ \iota_\rho)(x)\|\gg 2^{-N}\; .\]
\end{rema}
The following will be useful to prove \ref{Crutial prop}:
\begin{lemm}\label{pre-Crutial prop}
Let $\sc$ and $\sc'$ are words in $\sR$ such that the letter $\sa := \boxdot_{\pm}(\sc-\sc')$ belongs to $\hat \sA$. 
Then for every $S\in \tilde \cD(\boxdot(Y_\sc-Y_\sc'))$, $z\in  S_{\sa}$ and $u$ tangent o $S_\sa$,   it holds:
\[\left\{\begin{array}{cl} 
\|D_zf^{n_\sa}(u)\|\ge 2^{n_\sa/3}\| u \|&\text{ if }n_\sc\le 100\cdot M\, ,\\
\|D_zf^{n_\sa}(u)\|\ge 2^{2n_\sa/5}\| u \|&\text{ otherwise.}\end{array}\right.\; .\]
\end{lemm}
\begin{proof}[Proof of \cref{pre-Crutial prop}]
Let $\sc':= \sc\cdot \sd$. 
Let $(\cX_0, \cY_1)$ be the affine-like representation of $(Y_{\sr\cdot \sc}, n_{\sr\cdot \sc})$. 
Let $Y''$ be the right component of 
$\partial_{\aleph(n')} Y_{\sr\cdot \sc'}$. It is the box of a puzzle piece $(Y_{\sr\cdot \sc\cdot \sd'}, n_{\sr\cdot \sc\cdot \sd'})$ with $\sd' = \sd\cdot \ss_-^{\odot \aleph(n_{\sc'})}$ or $\sd' = \sd\cdot \ss_+\cdot  \ss_-^{\odot \aleph(n_{\sc'})}$. Let 
$(\cX_0'', \cY_1'')$ be the affine-like representation of $(Y_{\sr\cdot \sc\cdot \sd'}, n_{\sr\cdot \sc\cdot \sd'})$.  
 By \eqref{bound derive x1 X-1}, the expansion $\|D_zf^{n_\sa}(u)\|/\|u\|$ is at least:
\begin{equation}\label{resol E1} E=\frac{1}{K}
\frac{  
\min \sqrt{|\partial_{x}  \cX_0''|}} {\max |\partial_{x_1} \cX_0 |}
\ge \frac{1}{K}
\frac{\min_{Y_{\sr\cdot \sc}} \|Df^{n_{\sr\cdot \sc}}|\chi_h\|}{  
\max_{Y''} \sqrt{\|Df^{n_{\sr\cdot \sc\cdot \sd'}}|\chi_h\|}} \; .
\end{equation}
The following bound use the mapping cones property of the   pieces:
\begin{equation}\label{resol E2} \max_{Y''}  \|Df^{n_{\sr\cdot \sc\cdot \sd'}}|\chi_h\|
\le \max_{Y_{\sr\cdot \sc}} \|Df^{n_{\sr\cdot \sc}}|\chi_h\|\cdot 
\max_{Y_{\sd'}}  \|Df^{n_{\sd'}}|\chi_h\|\; .\end{equation}

Now we use the distortion bound of \cref{distprop}:
\begin{equation}\label{resol E3}
\max_{Y_{\sr\cdot \sc}} \|Df^{n_{\sr\cdot \sc}}|\chi_h\|\le 
\min_{Y_{\sr\cdot \sc}} \|Df^{n_{\sr\cdot \sc}}|\chi_h\|\cdot K\cdot 2^{29\cdot 2^{-\sqrt M} n_\sc}\; .
\end{equation}
Then \eqref{resol E1}, \eqref{resol E2} and \eqref{resol E3} yield:
\begin{equation}
E\ge \frac{1}{K}
\sqrt{\frac{  
\min_{Y_{\sr\cdot \sc}} \|Df^{n_{\sr\cdot \sc}}|\chi_h\|
}
{2^{29\cdot 2^{-\sqrt M} n_\sc} \max_{Y_{\sd'}}  \|Df^{n_{\sd'}}|\chi_h\|}}\; .
\end{equation}  
We recall that by \cref{Expansion SR piece} it holds 
$\min_{Y_{\sr\cdot \sc}} \|Df^{n_{\sr\cdot \sc}}|\chi_h\| \ge 2^{(n_\sc+2M)(1-1/\sqrt M)}$  and so:
$$E\ge \frac1K 
\frac{  2^{(n_\sc/2+M)(1-1/\sqrt M)-15\cdot 2^{-\sqrt M} n_\sc}}{\sqrt{{
 \max_{Y_{\sd'}}  \|Df^{n_{\sd'}}|\chi_h\|}}}\; .$$  
If $n_{\sc}<M 2^{\sqrt M-1}$, then $\sc\in \sY_0^{(\N)}$ and  $\sd$ belongs to $\sY_0$ by definition of the strongly regular words. 
Also as the word $\sd'$ is in $\sY_0^{(\N)}$ by \cref{Expansion SR piece} it holds 
$\max_{Y_{\sd'}}  \|Df^{n_{\sd'}}|\chi_h\| \le K2^{{n_{\sd'}}(1+1/\sqrt M)}$. Therefore:
  $$\log_2 (E)\ge   -K - \frac{n_{\sd'}}2(1+\frac{1}{\sqrt M }) + {(\frac {n_\sc}2+M)(1-\frac2{\sqrt M})}\; . $$  
We now bound $\frac{n_{\sd'}}{2}\le \frac{M}2 +\lfloor\frac{n_\sc+M}{12}+\frac M{24}\rfloor+1$. Thus  
\begin{multline*}\log_2 (E)\ge  
- \frac{M}2 - \lfloor\frac{n_\sc+M}{12}+\frac M{24}\rfloor 
   + (\frac{n_\sc}2+M) -K(\frac{n_\sc}{\sqrt M}+\sqrt M)\\
\ge 
\frac5{12} n_\sc+
\frac9{24} M -     K(\frac{n_\sc}{\sqrt M}+\sqrt M)\gg  \frac{n_\sc+M+1}3 =\frac{  n_\sa }{3}
   \; . \end{multline*}
Furthermore, if $n_\sc\in[ 100\cdot  M, M 2^{\sqrt M-1}]$, then $\log_2 (E)\ge  \frac{2 n_\sa }{5}$.
   \\
   If $n_{\sc}>M 2^{\sqrt M-1}$ then  $n_\sd\le 2^{1-\sqrt M}n_\sc$. Thus we obtain:
   $$\log_2 (E)\ge  \frac {n_\sc}2  - \frac{n_\sc}{12} +o(n_\sc)\gg   \frac{2n_\sa}{5}\; . $$ 
\end{proof}
\begin{proof}[Proof of \cref{disjoint box}]\label{proof disjoint box}
Let $(Y, n)$ and $(Y', n')$ be puzzle pieces. We want to show that if a parabolic product $\boxdot_\pm  (Y'- Y'')$ is pre-admissible from a piece  $(Y'', n'')$, then $Y''\boxdot_\pm  (Y- Y')$ is a box. Moreover if $Y''\boxdot_\mp  (Y- Y')$ is well defined, then it is disjoint from $Y''\boxdot_\pm  (Y- Y')$.
\medskip

For every $y\in [-\theta, \theta]$, let $S(y):=I_\se\times \{y\}$ and let $S_{Y''}(y):= S(y)\cap Y''$. 
We recall that $\partial^s Y''$ is the union $\bigcup_{y\in  [-\theta, \theta]} \partial S_{Y''}(y) $ of the
endpoints of $\partial S_{Y''}(y)$ of $S_{Y''}(y)$. Thus $\bigcup_{y\in  [-\theta, \theta]} \partial S_{Y''}(y) $ is formed by two arcs of $W^s(A)$. Likewise, $\bigcup_{y\in  [-\theta, \theta]} \partial f^{n''}(S_{Y''}(y)) $ is formed by two arcs of $W^s(A)$. If $\bar S(y)$ is an extension of $f^{n''}(S_{Y''}(y)) $ which depends continuously on $y$,  then $\bigcup_{y} \partial \bar S_{\boxdot_\pm  (Y'- Y'')}(y)$ is formed by two arcs of $W^s(A)$ which intersects transversely each $\bar S(y)$. 
If such an arc intersects 
 $\bigcup_{y\in  [-\theta, \theta]}   f^{n''}(S_{Y''}(y)) $, then it is included in it by coherence of the arc. By property $(i)$ of \cref{piecedef} of the piece $(Y'', n'')$, its pull back by $f^{n''}|Y''$ is an arc of $W^s(A)$. Moreover it contains an endpoint of each $S(y)\cap (f^{n''}|Y'')^{-1}(\bar S_{\boxdot_\pm  (Y'- Y'')})(y)$. Otherwise the corresponding endpoint of   $S(y)\cap (f^{n''}|Y'')^{-1}(\bar S_{\boxdot_\pm  (Y'- Y'')})(y)$ is the corresponding endpoint of $S_Y(y)$. Thus the following set is bounded by two segment of $\R\times \{-\theta, \theta\}$ and two arcs of $W^s(A)$:
  $$\bigcup_{y\in [-\theta, \theta]} S(y)\cap (f^{n''}|Y'')^{-1}(\bar S_{\boxdot_\pm  (Y'- Y'')})(y)\;.$$
As the latter set is equal to $Y''\boxdot_\pm (Y-Y')$, we just proved that the latter set is bounded by 
a segment of $\R\times \{\theta\}$, a segment of $\R\times \{-\theta\}$ and two disjoint arcs 
$\partial^s(Y''\boxdot_\pm (Y-Y'))$ 
of $W^s(A)$ joining the endpoints of them.
\medskip

Clearly, if  $Y''\boxdot_\mp (Y-Y')$ is well defined, this set intersects $S(y)$ at a segments sent by $f^{n''}$ to a segment of $\bar S$ disjoint from $\bar S_{\boxdot_\pm  (Y'- Y'')})(y)$ by \cref{graph transform parabolic init}. Thus $Y''\boxdot_\mp (Y-Y')$ and $Y''\boxdot_\pm (Y-Y')$ intersects $S(y)$ at two disjoint segments for every $y\in [-\theta, \theta]$. Consequently $Y''\boxdot_\mp (Y-Y')$ and $Y''\boxdot_\pm (Y-Y')$ are disjoint. 
\medskip 

Two show that $Y''\boxdot_\pm (Y-Y')$ is a box it suffices to show that $\partial^s(Y''\boxdot_\pm (Y-Y'))$ is formed by two vertical curves. Let $\mathcal C$ be a component of $\partial^s(Y''\boxdot_\pm (Y-Y'))$. If $\mathcal C$ is included in a component of $\partial^sY''$, then we are done. Otherwise $f^{n''+M+1}(\mathcal C)$ is included in $\partial^s Y$ or $\partial^sY'$. Thus by the proof of \cref{invariancecone}.$(ii)$ p. \pageref{proof:invariancecone2}, the verticality of these curves is the consequence of the next lemma. \end{proof} 
 \begin{lemm} 
For every $z\in \partial^s(Y''\boxdot_\pm (Y-Y'))$:
\begin{enumerate}[$(a)$]
\item if $f^{M+n''+1}(z)\in \partial^s  Y$, then $\|\partial_x f^{k}(z)\|\ge b^{k/6} $ for every $0\le k\le n''+1+M+n$.
\item if $f^{M+n''+1}(z)\in \partial^s   Y'$, then $\|\partial_x f^{k}(z)\|\ge b^{k/6} $ for every $0\le k\le n''+1+M+n'$.
\end{enumerate}
\end{lemm}
\begin{proof}
The cases $k\le n''$ are direct consequences of property $(ii)$ of \cref{piecedef} for $(Y'', n'')$. Furthermore,  $\partial_x f^{n''}$ has norm at least $2^{n''/3}$ by   property $(i)$ of \ref{piecedef}.  Also $\partial_x f^{n''}$ is tangent at $S^{Y''}\in \cD(\boxdot(Y-Y'))$.

In case $(a)$, by  \cref{graph transform parabolic init2}.1, we have $|\partial_{x_{1}}\cX_{-1}|\ll 
 2^{N}$ and so:
\[ \|(Df^{1+M+n})(\partial_x f^{n''}(z))\|\gg 2^{-N} \|\partial_x f^{n''}(z)\| \; .\]
As the norm of $Df$ is at most $4$, this implies:
\[ \|(Df^{k})(\partial_x f^{n''}(z))\|\ge 2^{-N-2(1+M+n)} \|\partial_x f^{n''}(z)\| \quad  \forall k\le 1+M+n.\]
As $M+1+n\le 2^M  n'' $ by pre-admissibility, it comes:
\[ \|(Df^{k})(\partial_x f^{n''}(z))\|\ge b^{n''/6} \|\partial_x f^{n''}(z)\| \quad  \forall k\le 1+M+n.\]
Then we achieve the proof of $(a)$ for $n'' \le k\le n''+n+1+M$ using $\|\partial_x f^{n''}(z)\|\ge 2^{n''/3}$. 

The case $(b)$ is proved likewise by using \cref{pour la prochaine preuve prop i} which implies:
 \[ \|(Df^{1+M+n'})(\partial_x f^{n''}(z))\|\gg 2^{-N} \|\partial_x f^{n''}(z)\| \; .\]
\end{proof}
A Corollary of the latter proof is:
\begin{coro}\label{i implies ii}
Let $(Y, n)$ and $(Y', n')$ be puzzle piece. If a parabolic product $\boxdot_\pm  (Y'- Y'')$ is pre-admissible from a piece  $(Y'', n'')$, and then the pair $(Y''\boxdot_\mp  (Y- Y'), n''+n_\boxdot+n)$ satisfies property $(ii)$ of \cref{piecedef}.
\end{coro}

\begin{lemm} \label{expansion hori prod para}
Let $\sb$ and $\sc$ be words in $\sR$ such that the letter $\sa := \boxdot_{\pm}(\sc-\sc')$ belongs to $\hat \sA$. 
Then for every $S\in \tilde \cD(\boxdot(Y_\sc-Y_\sc'))$, $z\in  S_{\sa}$ and $v\in T_z S_\sa$,   it holds:
\[ \|D_zf^{n_\sa}(v)\| \ge 2^{\frac{n_\sa-j}3 }\|D_zf^{j}(v)\|\; , \forall j\le n_\sa\; .\]
\end{lemm}
\begin{proof}
The case $j=0$ was shown in \cref{pre-Crutial prop}. Furthermore, we showed: 
\begin{equation}\label{le 2/5}  \text{ If }n_\sc\ge 100\cdot M\text{ then } \|D_zf^{n_\sa}(v)\| \ge 2^{\frac{2\cdot n_\sa }5 }\| v \|\; .\end{equation}
Put $n=n_\sc$.  Let $z_{-1}:=z$, $v_{-1}:=v$ and for every $j\in [0, M+n]$, put  $z_j=(x_j,y_j):= f^{1+j}(z_{-1})$ and let $v_{j}:= D_z f^{1+j}(v_{-1})$. We want to show that:
\[\log_2 \frac{\|v_{M+n}\|}{ \| v_j\|}\ge  {\frac{n+M-j}3 } , \quad \forall 0\le  j\le n+M\; .\]
If $n<100 M$ then $\sc$ is a product of pieces in $\sY_0$ by the definition of strongly regular words. As the vector  $v$ is expanded by $Df^{n+M+1}$ and since $\|Df\|\le 4$, it holds $\|v_0\|\ge 4^{-100M}$. As $b$ is small in function of $M$, this means that the $x$ coordinate of $z$ is greater than $4^{-100M+1}$, and such that $v_0$ belongs to $\chi_h$. Using the property $(i)$ of \ref{piecedef} of the pieces $(Y_\sr, M)\star (Y_\sc,n_\sc)$, it comes that $v_j$ is $2^{(n-j)/3}$-expanded by $Df^{n-j}$. 
\medskip

\noindent If $n\ge 100\cdot M$, then put $ j_0:=\frac{n+M+1}{50}-M>1$. Note that $\|v_0\|\le 1$ and $\|v_j\|\le 4^{j}$. By \eqref{le 2/5}, it holds:
\[\log_2 \frac{
\|v_{M+n}\| }{
\| v_j\|}\ge {\frac25 (n+M+1)-2j  }\ge {\frac{n+M-j}3 } +\frac{n+M+1-25 j}{15}\; .
\]
Thus the lemma is satisfied for every $j\le \frac{n+M+1}{25}$.  For the other $j$, we use:
\begin{lemm}\label{lemma pre ha times pour parabolic product} For every $j\ge j_0\ge 1$,  the sinus $\Theta$  of the angle between the vectors  $w_j:= \partial_x  f^{j}({z_0})$ and $v_j$ 
is  $\theta^{j+1}$-small.\end{lemm}
\begin{proof}
We notice that $\|w_0\|=1$ and $\|v_0\|\le 1$ because $z_0\in Y_\boxdot$. 
We use that $|det\, Df|<b$ and $\|w_j\|\ge b^{j/6}$ by property $(ii)$ of \cref{piecedef} of the piece  $(Y_\sr, M)\star (Y_\sc,n_\sc)$, to obtain:
\[|\Theta|:=  \frac{ \|v_{ j}\times   w_{ j}\|}{\|v_{ j}\|\cdot \|w_{ j}\|}
\le  b^j \frac{ \|v_0\times w_0\| }{\|v_{ j}\| b^{j/6}}\le  \frac{ b^{5{ j}/6}}{\|v_{ j}\|}
\; .\]
As  $\|v_{n+M}\|\ge 1$ and as $Df$ has norm  at most 4, it comes that $\|v_{ j}\|\ge  2^{ -2(n+M-j)}$. Now we infer that $j\ge j_0$ and so $50(j+M)\ge n+M$.  Thus 
$\|v_{ j}\|
\ge 2^{- 101\cdot  j}$ for every $j$. Consequently, it holds:
$|\Theta|\le  2^{100\cdot  (j +M)} b^{5j/6}\ll \theta^{j+1}
\; ,$ because   $b\ll \theta\ll 2^{100M}$.
\end{proof}
We are going to use the expansion of the iteration associated to the letter $\sa_i\in \hat \sA$ forming $\sc=\sa_1\cdots \sa_m$. Put $\sg_i:= \sa_0\cdot \sa_1\cdots \sa_i$, with $\sa_0=\sr$.  By definition of the strong regular word and $j_0$, for every $ j\ge \frac{n+M+1}{25}$, there  exists $i$ such that $n_{\sg_i}>  j\ge n_{\sg_{i-1}}\ge j_0$. We recall that $j_0>0$ and so $i\ge 1$. Now we are going to use the expansion associated to $\sa_i$ to show that  $v_j$ is $2^{(n_{\sg_i}-  j)/3}$-expanded by $D_{z_j}f^{n_{\sg_i}- j}$.  To this end, it suffices to show that $v_{\sg_{i-1}}$ is tangent to a curve $ S$ which belongs to the domain $\mathcal D_{\sa_{i}}$ equal to $\mathcal H$ if $\sa_i\in \sA_0$ and to $\tilde \cD(\boxdot (Y_{\sc_i}-Y_{\sc_i'}))$ if 
$\sa=\boxdot_\pm (\sc_i-\sc_i')$. Indeed such an expansion is satisfied at the tangent space of $ S$ by \cref{piecedef}.$(i)$ when $\sa_i\in \sA_0$ and be the last statement \ref{Crutial prop} when $\sa_i$ is parabolic. 

To show that $ S$ belongs to such domains, we work with the curve $S'=f^M(Y_\sr \cap \R\times \{y_0\})$ which contains $z_M=z_{n_{\sg_0}}$. 
As $\sc$ belongs to $\sR$, the curve $S'':= S^{\sa_1\cdots \sa_{i-1}}$ is the graph of a function $\rho''$ which is $b$-$C^1$-small by \cref{pour prod parabolic}, and if $\sa_i=\boxdot_\pm (\sc_i-\sc_i')$, the curve $S''$ belongs to  $\cD(\boxdot (Y_{\sc_i}-Y_{\sc_i'}))$.
We note that $z_{n_{\sg_{i-1}}}$ belongs to $S''$ and 
$w_{n_{\sg_{i-1}}}$ is tangent to $S''$ at $z_{n_{\sg_{i-1}}}$. 
By lemma \ref{lemma pre ha times pour parabolic product}, the angle $\Theta$ between $w_{n_{\sg_{i-1}}}$ and $v_{n_{\sg_{i-1}}}$ is smaller than $\theta^{n_{\sg_{i-1}}}\ll \theta^{n_{\sa_i}}$ (by the order inequality satisfied by the strongly regular words). This is smaller than $\theta^2$. Thus the following function:
 \[
\rho=x\in I_\se\mapsto y_{ n_{\sg_i}} +\int_{x_{ n_{\sg_i}}}^x (\Theta+ D\rho(t))dt\; .\]
has its graph $S$ which is in $\mathcal H$. Moreover if $\sa_i$ is a parabolic symbol, it is $\theta^{n_{\sa_i}}$-close to $S''$. Thus $S$ belongs to $\tilde \cD(\boxdot (Y_{\sc_i}-Y_{\sc_i'}))\Supset \cD(\boxdot (Y_{\sc_i}-Y_{\sc_i'}))$.
\end{proof}

\begin{proof}[Proof of \cref{Crutial prop}] \label{proof of Crutial prop}
Let $\sa :=\boxdot_{\pm}(\sc-\sc')\in \hat \sA$  with $\pm\in \{-,+\}$ such that $\sc, \sc'$ belongs to $\sR$. 
Then by \cref{lemma pre ha times pour parabolic product}  for every $S\in \tilde \cD(\boxdot(Y_\sc-Y_{\sc'}))$,  it holds:
\begin{equation}\label{h time pour para}
\forall z\in S_{\boxdot_{\pm}(\sc-\sc')}, \quad \forall u\in T_z S, \quad \forall k\le n_\sa, \quad 
\|D_zf^{n_\sa}(u)\|\ge 2^{k/3}\|D_zf^{n_\sa-k}(u)\|\; .\end{equation}
Moreover, if there exists  $\sg\in \tilde \sR$,  such that  the parabolic product $\boxdot_\pm (Y_\sc-Y_{\sc'})$ is {pre-admissible} for the piece $(Y_\sg, n_\sg)$, then 
$Y_\sg \boxdot_\pm (Y_{\sc}-Y_{\sc'})$ is a box by \cref{disjoint box}. Furthermore
the pair $(Y_\sg \boxdot_\pm (Y_{\sc}-Y_{\sc'}), n_\sg+n_\sa)$ 
 satisfies properties $(o)-(i)-(ii)$ of \ref{piecedef} of piece. Indeed the box $Y_\sg \boxdot_\pm (Y_{\sc}-Y_{\sc'})$
is sent by $f^{ n_\sg+n_\sa}\subset Y_\sc\subset Y_\se$ as claimed in \ref{piecedef}.$(o)$. The property  \ref{piecedef}.$(i)$ 
is a consequence of \eqref{h time pour para} and the  property  \ref{piecedef}.$(i)$ for $(Y_\sg, n_\sg)$. The property \ref{piecedef}.$(ii)$ was shown in \cref{i implies ii}. 
\end{proof} 

\section{Parameter dependence of pieces and curves }\label{section dependence piece with para}
\subsection{ Motion of the strongly regular puzzle pieces}

\begin{proof}[Proof of prop. \ref{prop9.3}]\label{proof prop9.3}
The proof of the proposition requires two lemmas on the parameter dependence of the affine-like representation of the pieces $(Y_i,n_i)$ of $f_a$. 
We recall that the affine-like representation of $(Y_i,n_i)$ is the pair $(\cX_i, \cY_{i+1})$ of functions  defined by:
\[f_a^{n_i}(x_i,y_i)=(x_{i+1}, y_{i+1})\Leftrightarrow 
(x_i= \cX_i(y_i,x_{i+1},  a), y_{i+1}= \cY_{i+1}(y_i,x_{i+1},  a))\]

\begin{lemm}\label{estimate simple paradep}
For every $(x_1,y_0)\in [\alpha(a),-\alpha(a)]\times [-\theta, \theta]$, it holds:
\begin{enumerate}[$(i)$]
\item The affine-like representation $(\cX_0,\cY_1)$ of $(Y_\sr, n_\sr)$ satisfies:
\[| \partial_a \cX_0(y_0, x_1,  a)-1/3| \le C M 4^{-M}
\; .\]
\item  The affine-like representation $(\cX'_0,\cY'_1)$ of $(Y_\ss, n_\ss)$ with $\ss\in \sY_0$ satisfies: 
\[| \partial_a \cX'_0(y_0, x_1, a)| \le C 2^{n_\ss}\; .
\]
\end{enumerate}
\end{lemm}
\begin{proof} 
Since $n_\sr= M$, the function $(x_1, y_0, a) \mapsto \cX_0(y_0, x_1,  a)$ is $O(4^Mb)$-close to $(x_1, y_0, a) \mapsto 
(P_a^M| I_\sr)^{-1}(x_1)$. Thus $(i)$ is a consequence of Proposition 5.4.(1) \cite{Y95} which states that for every $x_1\in [\alpha(a),-\alpha(a)]$:
\[| \partial_a(P_a^M| I_\sr)^{-1}(x_1)-1/3| \le C M 4^{-M}\; .\]

Likewise, for every  $\ss\in \sY_0$, it holds $n_\ss\le M$ and the function  $(x_1, y_0, a) \mapsto \cX_0(y_0,x_1,  a)$ is $O(4^Mb)$-close to $(x_1, y_0, a) \mapsto 
(P_a^{n_\ss}| I_\ss)^{-1}(x_1)$. Thus $(i)$ is a consequence of Proposition 5.4.(2) \cite{Y95} which states for every  $x_1\in [\alpha(a),-\alpha(a)]$
\[| \partial_a(P_a^{n_\ss}| I_\ss)^{-1}(x_1)| \le C 2^{n_\ss}\; .\]
\end{proof}

\begin{lemm}\label{weak estimate paradep}
For every piece $(Y,n)$, then the affine-like representation $(\cX_0,\cY_1)$  of $(Y, n)$ satisfies for every $(x_1,y_0)\in [\alpha(a),-\alpha(a)]\times [-\theta, \theta]$:
\[| \partial_a \cX_0(y_0,x_1,  a)| \le  4^{n}\qand | \partial_a \cY_1(y_0,x_1,  a)| \le  b 4^{2 n}\; .\]
\end{lemm}
\begin{proof}
Let $g_a$ be the first coordinate of $f^n_a$.  We have $g_a( \cX_0(y_0,x_1,a), y_0)= x_1$.  Thus it holds:
\[(\partial_x g_a)\cdot \partial_a \cX_0(y_0,x_1,a)+
(\partial_a g_a)( \cX_0(y_0,x_1,a), y_0)=0\; .\]
By property $(i)$ of \cref{piecedef}, the derivative $|\partial_x g_a|$ is at least $2^{n/3}(1-\theta)$. On the other hand, the differential of the map $( x, y, a)\mapsto ( f_a(x,y), a)$ has norm at most $4$ (in the operator norm associated to the sup $C^1$-norm on $I\times [-2,2]\times [-\theta, \theta] $). This gives the first estimate of the lemma.
Now we notice that $\cY_1( y_0,x_1, a)$ is the second coordinate of $f^n_a( \cX_0( y_0,x_1, a),y_0)$. As the $C^1$-norm of the second coordinate of $f^n_a$ is smaller than $b 4^{n-1}$, we get the sought estimate on  $\partial_a \cY_1$. 
\end{proof}
We recall that $(\cX_i, \cY_{i+1})$ denotes the  affine-like representation of 
$(Y_i, n_i)$. Let $(\cX_0^{(j)}, \cY^{(j)}_{j+1})$ be the affine-like representation of
$(Y_0, n_0)\star \cdots \star (Y_j, n_j)$ for every $j\le m$. 
 As $Y_\se= [\alpha_0(a),-\alpha_0(a)]\times [-\theta, \theta]$ for every $a$,
for every $m$ it holds:
\[\partial^s Y(f_a)= \bigcup_{y\in [-\theta, \theta]} \{(\cX^{(m)}_0(y_0,\alpha_0(a),  a),y), (\cX^{(m)}_0( y,-\alpha_0(a), a),y_0)\}\; .\]
We first get rid of the dependence of $\alpha_0$ on $a$. We have:
\[\partial_a  [\cX^{(m)}_0(y_0,\pm \alpha_0(a),  a)]= 
[\partial_a  \cX^{(m)}_0](y_0,\pm \alpha_0,  a)\pm \partial_x \cX^{(m)}_0( y_0, \pm \alpha_0,a)\cdot \partial_a \alpha_0\; .\]
Here $ \partial_a \alpha_0$ is bounded and by condition $(i)$ of puzzle piece's \cref{piecedef}, it holds $|\partial_x \cX^{(m)}_0|\le 2^{-(M+n)/3}(1+\theta)$.

We now estimate $\partial_a  \cX^{(m)}_0$ by induction on $m$; for $m=0$ this has been done in \cref{estimate simple paradep}.$(i)$. To prove the step $m-1\to m+1$, we consider the system:
\[\left\{\begin{array}{rl}
x_0 &= \cX_0^{(m-1)}(y_0,x_m, a)\\
y_m &= \cY_m^{(m-1)}(y_0,x_m, a)\\
x_m &= \cX_m (y_m,x_{m+1}, a)\\
y_{m+1}&= \cY_{m+1}(y_m, x_{m+1}, a)\end{array}\right.
\]
We would like to deduce from this system the expressions of $x_0$ and $y_{m+1}$ in function of $x_{m+1}$, $y_0$ and $a$. To this end, we are going to use the implicit function theorem through the following mapping:
\[\psi: \left[
\left( \begin{array}{c}
x_{m+1}\\
y_0 \\
a\end{array}\right), \left( \begin{array}{c}
x_m \\
y_m
\end{array}\right)\right] \mapsto 
\left( \begin{array}{c}
x_m - \cX_m (y_m,x_{m+1}, a)\\
y_m -\cY_m^{(m-1)}(y_0,x_m, a)\end{array}\right)\; .\]
By \cref{AFrepresentation}, the derivative $\partial_2\psi$ of $\psi$ following the second variable is equal to the identity plus a term smaller that $\theta+ \sqrt{1+\theta^2}2^{- n_m/3}$. Thus by the implicit function theorem, there are functions $\cY'_m$ and $\cX'_m$ such that:
\[  x_m = \cX_m'(y_0,x_{m+1}, a)\qand y_m= \cY_m'(y_0,x_{m+1}, a)\; .\] 
Furthermore, $\partial_a (\cX'_m, \cY_m') (y_0,x_m, a)= -(\partial_2 \psi)^{-1}\circ \partial_a \psi(y_0,x_m, a)$. Thus by Lemma \ref{estimate simple paradep}.$(ii)$ and \ref{weak estimate paradep}, it holds:
\begin{enumerate}[$(a)$]
\item If $(Y_m, n_m)$ is simple, then  $| \partial_a \cX'_m| \le C 2^{n_m}$,
\item If $(Y_m, n_m)$ is not simple, then $| \partial_a \cX'_m| \le C 4^{n_m}$ and with $n_m \le 2^{-\sqrt M/2} \sum_{j=1}^m n_j$.
\end{enumerate}
Now we observe that $\cX_0^{(m)}(y_0,x_{m+1},  a)= 
\cX_0^{(m-1)}(y_0, \cX_m'(y_0, x_{m+1}, a), a)$. Its derivative w.r.t. $a$ gives the following recurrence relation:
\[\partial_a \cX_0^{(m)}= \partial_a \cX_0^{(m-1)}+ \partial_x  \cX_0^{(m-1)} \cdot \partial_a \cX'_m\; .\]
As the horizontal expansion of $(Y_0, n_0)$  is of the order of $4^M$, it holds that $\partial_x  \cX_0^{(m-1)}$ is smaller than $C4^{-M}
 \cdot 2^{-\frac13 \sum_{1\le j< m} n_j}$, in both cases $(a)$ and $(b)$, we obtain:
\[|\partial_x  \cX_0^{(m)} \cdot \partial_a \cX'_m|\le C 2^{-M-\frac14 \sum_{1\le j< m} n_j}.\]
Plugging this estimate in the recurrence relation gives the estimate of the proposition. 
\end{proof}
\begin{rema}\label{prop9.3 forte}
The proof showed actually that under the assumptions of \cref{prop9.3}, the affine-like representation $(\cX_0, \cY_1)$ of $(Y(f_a), n)$ satisfies:
\[\partial_a \cX_0 = 1/3+ O(2^{-M}\]
\end{rema}
\subsection{Motion of the horizontal curves}
We are going to prove \cref{transversity de Wu}. 
\begin{defi}[Cone $\tilde \chi$] \index{$\tilde \chi$}
Let $\tilde \chi:=\{(u_x, u_y, u_z)\in \R^3: |u_y|\le \theta \cdot (|u_x|+|u_z|)\}$.
\end{defi}
Let us fix a family $(f_a)_a$ satisfying  the assumptions of \textsection \ref{setting}.  The following is a rephrasing of \cref{transversity de Wu}:
\begin{prop}\label{deplacementdescourbes}
 Let $\st\in \sA^{\Z^-}$ and let $\omega\subset  \R$ be an interval such that $\st\in \arr \sR (f_a)$ for every $a\in \omega$.
Let $\hat  \Sigma_u^\st:= \cup_{a\in \omega} \hat W_u ^\st(f_a)\times \{a\}$ and  $ \Sigma_u^\st:= \cup_{a\in \omega} W_u ^\st(f_a)\times \{a\}$. 
Then both  $\Sigma_u^\st$ and  $\hat \Sigma_u^\st$ are Lipschitz surfaces and their tangent vectors are in $\tilde \chi$. 
\end{prop}
\begin{proof}
Let $F:\; (x,y,a)\mapsto (f_{a}(x,y),a)$.  Let $\st = \cdots \sg_m\cdots \sg_0\in \arr \sR$, with $\sg_m\in \sR$ for every $m$. 
\medskip

Let us first prove that the tangent space of $\Sigma^\st_u$ is in $\tilde \chi$. 
  We notice that the following limit occurs in the $C^0$-topology by \cref{contractionPP}:
\[\Sigma^\st_u= \lim_{m\to -\infty} F^{n_{ \sg_m \cdots \sg_0}}\left (Y_{\sg_m \cdots \sg_0}\cap \bigcup_{a\in \omega} I_e(f_a)\times  \{0, a\}\right)\;  .\]
Thus it suffices to  apply the following lemma  with $(Y_{\sg_m \cdots \sg_0}, n_{\sg_m \cdots \sg_0})=: (Y,n)$:
\begin{lemm} \label{lemme de deplacementdescourbes}
If $(Y(f_a), n)$ is a piece which persists for every $a\in \omega$, 
then the following surface is $C^2$-surface and its tangent space is in $\tilde \chi$:
$$F^{n }\left (Y\cap \bigcup_{a\in \omega} I_e(f_a)\times  \{0, a\}\right)\; .$$ 
\end{lemm}
\begin{proof}This is equivalent to say that $\partial_x F^n(z, a)$ and   $\partial_{a} F^n(z, a)$ span a plan in $\tilde \chi$.
A normal vector to this plan is given by the vector product  
$\partial_u F^n(z, a)\wedge \partial_{a} F^n(z, a)$. It suffices to show that such a vector makes an angle $\ll  \theta$ with $(0,1,0)$. We develop:
\begin{multline}\partial_a F^{n}(z)\wedge  \partial_x F^n_a(z)
=(\partial_a  f^{n}_a(z), 1)\wedge (\partial_x f ^n_a(z), 0)\\
=(p_y\circ \partial_x f ^n_a(z) , -p_x\circ \partial_x f ^n_a(z), \partial_a  f^{n}_a(z) \times \partial_x f ^n_a(z))\; .\end{multline}
Thus the angle between $\partial_a F^{n}(z,a)\wedge  \partial_x F^n_a(z, 0)$ and the vector $(0,1,0)$  is bounded by:
\[ \frac{\| \partial_a  f^{n}_a(z)\times  \partial_x  f ^n_a(z)\|}{\| p_x\circ \partial_x f ^n_a(z)\|}+\frac{\|  p_y\circ \partial_x f ^n_a(z)\|}{\| p_x\circ \partial_x f ^n_a(z)\|}
\; .\]
By \cref{pour prod parabolic}, $\partial_x f ^n_a(z)$ makes an angle $\le b$ with $\partial_x$ and so the second term of this sum is at most $b$. This implies also that $\partial_x f ^n_a(z)$ is dominated by $p_x\circ \partial_x f ^n_a(z)$. Thus it suffices to show that the following is small compared to  $\theta$:  
\[ \frac{\| \partial_a  f^{n}_a(z)\times  \partial_x  f ^n_a(z)\|}{\|  \partial_x f ^n_a(z)\|}\; .
\]
An induction shows:
 \[\partial_a  f_a^n(z)= \sum_{j=1}^n D_{z_j}f^{n-j}(\partial_x^B(z_j)) ,\quad 
\mathrm{with}\; z_j:= f^j(z)\quad\mathrm{and}\quad \partial_x^B:= \partial_x +\partial_a B.\]
 We compute:
\[ \partial_a  f^{n}_a(z)\times  \partial_x f^n(z)= \sum_{j=1}^n D_{z_j}f^{n-j}(\partial_x^B(z_j))\times  D_zf^n(\partial_x )\; .\]
Since the Jacobian of $f$ is less than $b$, we have:
\[\|D_{z_j}f^{n-j}(\partial_x^B)\times  D_{z}f^n(\partial_x )\|\le b^{n-j}\|\partial_x  f^j(z)\|(1+b).\]  
By   \cref{piecedef}.$(i)$ of the piece $(Y,n)$, it holds $\|\partial_xf^j(z)\|\le 2^{- (n-j)/3} \|\partial_x f^n(z)\|$, so we have:
\[\|D_{z_j} f^{n-j}(\partial_x^B)\times  \partial_x  f^n(z)\|\le 2b^{n-j} 2^{-(n-j)/3} \|\partial_x   f^n(z)\|.\]
For $j=n$, since $\partial_x   f^n(z)$ is $ b$-close to  be horizontal by \cref{pour prod parabolic}, we have:
\[\|\partial_x^B(z_n) \times  \partial_x   f^n(z)\|\le  b \|\partial_x   f^n(z)\|,\]
 Thus 
\begin{equation*}\label{angle entre plan2}
\frac{\| \partial_a  f^{n}_a(z) \times \partial_x  f^n(z)\|}{  \|\partial_x   f^n(z)\|}\le 
 Kb + \sum_{j=0}^{n-1}  2b^{n-j} 2^{-(n-j)/3}  =K  b\ll  \theta \; . 
\end{equation*}
\end{proof}
When $\st$ is complete then  $\hat W^\st_u(f_a)=W^\st_u(f_a)$  for every $a$ by \cref{rigidPP}, and so $\Sigma^\st_u=\hat \Sigma^\st_u$. The latter lemma carries this case. Otherwise, an induction using the next  lemma achieves the proof of \cref{deplacementdescourbes}. 
\end{proof}
\begin{lemm} Given $\sb=\boxdot_\pm (\sc-\sc')\in \sA$ and a family  $(S_a)_a$ of horizontal stretched curves such that 
\begin{itemize}
\item  $S_a \in \tilde \cD_\sa(f_a)$ for every $a\in \omega$,
\item $\Sigma_0:= \bigcup_{a\in \omega} S_a\times \{a\}$ is Lipschitz with tangent space in $\tilde \chi$,
\end{itemize}
then $\hat \Sigma =\bigcup_{a\in \omega} S_a^\sb\times \{a\}$ is Lipschitz with tangent space in $\tilde \chi$. 
\end{lemm}
\begin{proof}By density, we can assume that $\Sigma_0$ is of class $C^1$. 
 
Let $(\cX_0, \cY_1)$ be the affine-like representation of the piece $(Y_{\sr\cdot \sc}, n_{\sr\cdot \sc})$, with $n_{\sr\cdot \sc}=M+n_\sc$. We recall that $\epsilon \in \{-, +\}$ and $x_\boxdot(a)$ such that the segment $J_a= [x_\boxdot(a), \epsilon \alpha_0]$ is minimal such that $\cX_0([-\theta, \theta]\times J_a)$ contains  the right component of $Y_{\sr\cdot \sc}(f_a)\setminus   Y_{\sr\cdot \sc'}(f_a)$. 
\begin{sublemm}\label{partial a x boxodot}
The map $a\mapsto  x_\boxdot (a)$ is $4^{M+n_\sc}$-Lipschitz. 
\end{sublemm}
\begin{proof}
By \cref{prop9.3}, the right component of $\partial^s Y_{\sr\cdot \sc'}(f_a)$ is the graph of a function $v_ a$, so that $\partial_a v_a$ is close to $1/3$. Moreover by \cref{prop9.3 forte}, $\partial_a \cX_0$ is $C^0$-close to $1/3$. For every $y_0\in [-\theta, \theta]$, we define $X_\boxdot(y_0, a)$ implicitly by:
\[v_a(y_0)= \cX_0(y_0, X_\boxdot(y_0, a), a)\; .\] 
It satisfies:
\[\partial_a v_a(y_0)- (\partial_a \cX_0)(y_0, x_\boxdot(y_0, a), a) = \partial_x \cX_0(y_0, X_\boxdot(y_0, a), a) \partial_a x_\boxdot(y_0, a)\; .\] 
We recall that $|\partial_x  \cX|$ is at least  $4^{-M-n_\sc}$. 
Thus $\partial_a X_\boxdot(y_0, a)$ is smaller than $4^{M+n_\sc}$. 
Consequently, $a\mapsto X_\boxdot(y_0, a)$ is $4^{M+n_\sc}$ Lipschitz.
Thus $a\mapsto x_\boxdot(a)$  is equal to the supremum or the infimum of the functions $\{X_\boxdot(y_0, a) :y_0\in [-\theta, \theta]\}$. Consequently, it is   $4^{M+n_\sc}$-Lipschitz. 
\end{proof}
Let $S_a$ be the graph of $\rho_a$ and let $S^\sb_a$ be the graph of $\bar \rho_a$. We recall that by \cref{extension algo} \cpageref{extension algo}:
\[\bar \rho_a:  x_1\in I_\se\mapsto \left\{\begin{array}{cl} \cY_1(\cY_0^\pm (x_1, \rho_a,a), x_1,a )& \text{if }x_1\in J_a,\\
\cY_1(\cY_0^\pm ( x_\boxdot(a), \rho_a,a), x_1,a )+\eta_a \cdot (x_\boxdot(a)-x_1) & \text{otherwise,} \end{array}\right. 
 \]
with $\cY_0^\pm (x_1, \rho_a,a)= B_{a\, y}\circ \iota_{\rho_a}\circ \cX_{-1}^\pm (x_1, \rho_a, a)$ and $\eta_a$ such that 
$\bar \rho_a$ is $C^1$. We notice that $\partial_x (B_{a\, y}\circ \iota_{\rho_a})$ and $\partial_a (B_{a\, y}\circ \iota_{\rho_a})$ are $Kb$-small. Thus for every $x_1\in J_a$:
\begin{equation}\label{born partialaY0} |\partial_a  \cY_0^\pm (x_1, \rho_a,a)|\le Kb |\partial_a \cX_{-1}^\pm (x_1, \rho_a,a)|
\; .\end{equation}
Also similarly to \eqref{derivex1X-1} \cpageref{derivex1X-1}, it holds:
\[
\partial_{a} \cX^\pm_{-1}= -\frac{\partial_{a}  Q(\cX^\pm_{-1}, x_1, \rho_a)}{\partial_{x_{-1}}  Q(\cX_{-1}^\pm , x_1, \rho_a)} 
\; . \]
We note that  $\partial_{a}  Q(\cX^\pm_{-1}, x_1, \rho_a)$ is bounded, thus \eqref{min partialx-1 Q} \cpageref{min partialx-1 Q} implies:
\begin{equation}\label{born partialaX-1}
|\partial_{a} \cX^\pm_{-1}|\le K\cdot 2^N, \quad \text{with }N=n_{\sc'}+\aleph(\sc').
\end{equation}
Consequently \eqref{born partialaY0} and \eqref{born partialaX-1} imply:
\begin{equation}\label{born partialaY0 bis} |\partial_a  \cY_0^\pm (x_1, \rho_a,a)|\le K\cdot  b \cdot 2^N \; .\end{equation}

By \cref{lemme de deplacementdescourbes} the derivative $\partial_a \cY_1$ is $Kb$-small. By \cref{AFrepresentation}.(2), we have $\partial_y \cY_1$ is small compared to $b^{n_{\sr\cdot \sc}}$. This gives the following bound for the $C^0$-topology:
\begin{equation}\label{bound ss extension} |\partial_a \bar \rho_a (x_1)|
= \|\partial_a  \cY_1 + \partial_{y_0} \cY_1 \cdot \partial_a \cY_0^\pm \|_{C^0}\le 
Kb+ b^{n_{\sr\cdot \sc}}\cdot b\cdot 2^N\le Kb, \quad \forall x_1\in J_a\; .\end{equation}

For $x_1\in I_\se\setminus J_a$, the Lipschitz constant of $a\mapsto \bar \rho_a(x_1)$ is bounded from above by $\partial_a(\cY_1\circ \cY_0)$ at the point $(x_\boxdot, \rho_a, a)$ -- which is bounded by $Kb$ in \cref{bound ss extension}-- plus the terms:
\begin{equation}\label{add terms}  |\partial_y \cY_1\cdot  \partial_{x_1} \cY_0^\pm |
\cdot Lip (x_\boxdot)+Lip( \eta_a )\cdot |x_\boxdot(a)-x_1| + \eta_a \cdot Lip (x_\boxdot) \; . \end{equation}
It remains only to show that the above expression  is bounded by $Kb$. 
We recall that $\partial_y \cY_1$ is small compared to $b^{n_{\sr\cdot \sc}}$ by \cref{AFrepresentation}.(2). Also $\partial_{x_1} \cY_0^\pm$ is smaller than $ Kb 8^N$ by \cref{graph transform parabolic init2}. We infer sublemma \ref{partial a x boxodot} to obtain  that the first term of \eqref{add terms} is bounded by:
\[ | \partial_y \cY_1 \cdot  \partial_{x_1} \cY_0^\pm \cdot Lip( x_\boxdot) |\le b^{n_{\sr\cdot \sc}}\cdot Kb 8^N\cdot 4^{M+n_\sc}\ll Kb
\; .\]
Thus it remains only to prove that the second and third terms of \eqref{add terms} are dominated by $b$. 
By continuity, the difference of slopes of $x_1\mapsto \cY_1(\cY_0^\pm (x_1, \rho_a,a), x_1,a )$ and 
$x_1   \mapsto  \cY_1(\cY^\pm_0( x_\boxdot(a), \rho_a,a), x_1,a )$ at $x_1= x_\boxdot $ is equal to $\eta_a$. The tangent vector to these two curves at $x_1= x_\boxdot $ are   $\partial_x (f^{n_\sc+M}\circ  f_a\circ \iota_{\rho_a}\circ \cX^\pm _{-1}) (x_\boxdot)$ and $\partial_x (f^{n_\sc+M})\circ (f_a\circ \iota_{\rho_a}\circ \cX^\pm _{-1}) (x_\boxdot)$. Thus  $\eta_a= \Phi(\Xi_a)$ with $\Phi$ a smooth function satisfying $\Phi(0)=0$ and: 
$$\Xi_a:=\frac{\|\partial_x f^{n_\sc+M }({z_\boxdot})\times  D_{z_\boxdot}f^{n_\sc+M }(w_a))\|}{\|\partial_x f^{n_\sc+M }(z_\boxdot )\| \cdot\| D_{z_\boxdot} f^{n_\sc+M }(w_a)\|}\; ,$$
 where $z_\boxdot = f_a\circ \iota_{\rho_a}\circ \cX^\pm _{-1}(x_\boxdot)$ and 
 $w_a:=\partial_x (f_a\circ \iota_{\rho_a}\circ \cX^\pm _{-1})(x_\boxdot)$. 
Thus to achieve the proof of the lemma,  it suffices to show that $\Xi_a$ is small compared to  $1/Lip(x_\boxdot)$ and that $Lip( \Xi_a)$ is small.  We have:
$$|\Xi_a|=|\det\, D_{z_\boxdot} f^{n_\sc+M}_a |\cdot  
\frac{  {\|\partial_x  \times   w_a \|}}
{\|\partial_x f_a^{n_\sc+M }(z_\boxdot )\| \cdot\| D_{z_\boxdot} f_a^{n_\sc+M }(w_a)\|} 
\; .$$ 
Using that $|\det Df|\le b$, property $(i)$ of \ref{piecedef} for the piece $(Y_{\sr\cdot \sc}, n_{\sr\cdot \sc})$, and \cref{expansion hori prod para}, it holds that $|\Xi_a |\ll b^{M+n_\sc}$. Thus by sub-lemma \ref{partial a x boxodot}, the third term of \eqref{add terms} is bounded by:
\[|\Xi_a \cdot Lip(x_\boxdot)|\le b^{M+n_\sc} \cdot 4^{M+n_\sc}\ll b\; .\]
This bound the third term of \eqref{add terms} by $Kb$, let us do the same with the second term. We notice that $z_\boxdot = (\cX_0^\pm(\cY_0(x_\boxdot, 
\rho_a,a), x_\boxdot,a), \cY_0(x_\boxdot, \rho_a,a))$. By   
\cref{prop9.3 forte}, the function $\cX_0$ is $C^1$-bounded. Also 
\eqref{born partialaY0 bis} bounds $\cY_0$ and sub-lemma \ref{partial a x boxodot} bounds $x_\boxdot$. 
This gives that $Lip( z_\boxdot)$ is at most $K^N$ for a universal constant $K$. By a similar proof as for \eqref{born partialaX-1}, we can bound $  \partial_a \partial_x \cX_{-1}$  by $K^N$,  and so we obtain $Lip( w_a)\le K^N$. As $(z,a)\mapsto \det D_z f_a$ is $b$-$C^2$-small and $f$ is $4$-$C^2$-small, it holds that $a\mapsto \det\, D_{z_\boxdot} f^{n_\sc+M} _a$ is $b^{(n_\sc+M) }\cdot   K^N $-small.  Consequently, $Lip \,  \eta_a\asymp Lip\, \Xi_a $ is smaller that $b$. 
\end{proof}

\subsection{Proof of \cref{cardSN}}
\label{apendix combi} 
Put $R(N):= \{(n_j)\in (\Z\setminus \{-1,0,1\})^{(\N)} : 
\sum | n_j | =N\}$.  We notice that:
\[E(N) := \bigcup_{r=  [\max( \delta N,M/2) ]+1, n,t} E(N,r,n,t),\quad \text{with }\]
\[E(N,r,n,t):=\left\{(n_j)_{j}\in R(N):\; r=\sum_{j\le n:\; |n_j|>\frac M2} |n_j|,\; t=|\{j\le n:\; |a_j|>\frac M 2\}|\right\}.\]

\begin{proof}  We recall the proof of \cite{BM13}. 
Put $M':= \lfloor M/2\rfloor $. 

\begin{lemm}  For $r$, $N$, $t$ and $n$ fixed, the number of possible choices of $P= \{j\le n:\; |a_j|>M'\}$ is ${n \choose t}$ is ${n \choose t}\le (eNM'/2r)^{r/M'}$.\end{lemm} 
 \begin{proof}  We remark that   $t \le r/M'$. Also $r< n$ and $M'\ge 2$, thus  $r/M' < n/2$ and ${n \choose t}\le{n \choose r/M'} \le  (enM'/r)^{r/M'}\le (eNM'/2r)^{r/M'}$.
\end{proof}
For given  $r$, $t$ and $P:= \{j\le n:\; |a_j|>M'\}$, the number of possible choices of $(a_j)_{j\in P}$ is bounded by $2^t$ (to choose the sign of $a_j$) times the number of solutions of $x_1+x_2+\cdots+x_t=r$ (to choose the absolute value), so it is $2^t{r \choose t}\le 2^t{r \choose r/M'}\le (2eM')^{r/M'}$.

Therefore,  given  $N$, $r$, $n$ and $t$, the number of possible choices of $\{a_j, j\le n:\; |a_j|>M'\}$ is at most $(e^2N M'^2/r )^{r/M'}$.

The number of choices of $(a_j, j\le n:\; |a_j|\le M')$ is bounded by $2^{N-r}$. Indeed an induction shows that the number $C_k$ of finite sequences of integers in $\mathbb Z\setminus \{-1,0,1\}$ whose absolute value sum is equal to $ k\ge 0 $ is less than $2^{k}$, since $C_{k+1}= 2+\sum_{i=2}^{k-2} 2 C_{k-i}$.  
  
So, the cardinality of $E(N,r,n,t)$ with $r= \lfloor \delta N\rfloor +1$ is at most 
 \[(e^2 M'^2 N/r)^{r/M'} \cdot 2^{N-r}\le 2^{N-r}(e^2 M'^2/\delta)^{r/M'}
 \]\[\le 2^N(e^2 M'^2 2^{-M'}/\delta)^{r/M'}=2^N(e^2 M'^2 2^{-M'}/\delta)^{(\delta N+1)/M'}.\]
Given $N$ there are at most $N^3$ choices for $(t,r,n)$ 
 and so:
\[Card\; \mathcal E(N)\le N^3 2^N(e^2 M'^2 2^{-M'}/\delta)^{(\delta N+1)/M'}
=  N^3 2^{N(1-\delta)}(e^2 M'^2/\delta)^{(\delta N+1)/M'}.\]\end{proof}
} }
 
\clearpage
{\small\addcontentsline{toc}{part}{Index}\printindex}
\clearpage
\bibliographystyle{alpha}
\bibliography{referenceshen}

\end{otherlanguage}
\end{document}